\documentclass[12pt,letterpaper,notitlepage]{amsart}

\usepackage{fullpage}

\usepackage[all]{xy}

\usepackage{verbatim}

\usepackage{amssymb}
\usepackage{amsmath}
\usepackage{amsthm}
\usepackage{amsfonts}

\usepackage{float}

\usepackage[small]{caption}
\usepackage[lofdepth,lotdepth,captionskip=10pt]{subfig}

\usepackage{pstricks}
\usepackage{epsfig}
\usepackage{pst-node}
\usepackage[dvips={-h tir_____.pfb}]{auto-pst-pdf}

\usepackage{enumitem}

\usepackage[linktocpage=true]{hyperref}

\usepackage[alphabetic,msc-links,nobysame,lite]{amsrefs}

\newcommand{\twostars}{}

\newcommand{\aext}{\measuredangle}
\newcommand{\mapstop}[1]{\mathop{\to}\limits^{#1}}
\newcommand{\inv}{^{-1}}

\newcommand{\bbC}{\mathbb{C}}
\newcommand{\bbD}{\mathbb{D}}

\newcommand{\bbG}{\mathbb{G}}
\newcommand{\bbH}{\mathbb{H}}

\newcommand{\bbR}{\mathbb{R}}
\newcommand{\bbS}{\mathbb{S}}
\newcommand{\bbT}{\mathbb{T}}

\newcommand{\bbZ}{\mathbb{Z}}

\newcommand{\calC}{\mathcal{C}}
\newcommand{\calD}{\mathcal{D}}

\newcommand{\calK}{\mathcal{K}}

\newcommand{\calP}{\mathcal{P}}

\theoremstyle{plain}

\newtheorem{lemma}{Lemma}[section]
\newtheorem*{lemma*}{Lemma}

\newtheorem{claim}[lemma]{Claim}
\newtheorem*{claim*}{Claim}
\newtheorem{conjecture}[lemma]{Conjecture}
\newtheorem*{conjecture*}{Conjecture}

\newtheorem*{corollary*}{Corollary}

\newtheorem*{fact*}{Fact}

\newtheorem*{facts*}{Facts}
\newtheorem{observation}[lemma]{Observation}
\newtheorem*{observation*}{Observation}
\newtheorem{proposition}[lemma]{Proposition}
\newtheorem*{proposition*}{Proposition}

\newtheorem*{question*}{Question}
\newtheorem{theorem}[lemma]{Theorem}
\newtheorem*{theorem*}{Theorem}

\theoremstyle{definition}

\newtheorem{definition}[lemma]{Definition}
\newtheorem*{definition*}{Definition}

\newtheorem*{example*}{Example}

\newtheorem*{remark*}{Remark}

\newtheorem*{remarks*}{Remarks}

\newcommand{\possibletildea}{\ensuremath{\diamondsuit}}
\newcommand{\possibletildeb}{\ensuremath{\heartsuit}}
\newcommand{\possiblea}{\ensuremath{\spadesuit}}
\newcommand{\possibleb}{\ensuremath{\clubsuit}}

\theoremstyle{plain}

\newtheorem{normalization}[lemma]{Normalization}

\newtheorem{propositionfree}{Proposition}
\newtheorem{theoremfree}{Theorem}

\newtheorem{additivity}[lemma]{Index Additivity Lemma}

\newtheorem{cil}[lemma]{Circle Index Lemma}
\newtheorem{dut}[lemma]{Discrete Uniformization Theorem}

\newtheorem{katt}[lemma]{Koebe--Andreev--Thurston Theorem}
\newtheorem{kc}[lemma]{Koebe Conjecture}

\newtheorem{threep}[lemma]{Three Point Prescription Lemma}

\newtheorem{mainrigidity}[lemma]{Main Rigidity Theorem}
\newtheorem{mainuniformization}[lemma]{Main Uniformization Theorem}
\newtheorem{mainindex}[lemma]{Main Index Theorem}
\newtheorem*{mainindex*}{Main Index Theorem (weak form)}

\theoremstyle{definition}

\newtheorem{keycase}{Case}

\newtheorem{libcase}{Case}

\newtheorem{mooncase}{Case}
\newtheorem{moonsubcase}{Sub-case}[mooncase]

\newtheorem{starcase}{Case}

\begin{document}

\title{Rigidity of thin disk configurations, via fixed-point index}
\author{Andrey M. Mishchenko}
\email{misHchea@umich.edu}
\date{\today}
\thanks{The author was partially supported by NSF grants DMS-0456940, DMS-0555750, DMS-0801029, DMS-1101373.  This article is adapted from the Ph.D.\ thesis \cite{mishchenko-thesis} of the author.  MSC2010 subject classification: 52C26.}

\begin{abstract}

We prove some rigidity theorems for configurations of closed disks.  First, fix two collections $\calC$ and $\tilde \calC$ of closed disks in the Riemann sphere $\hat\bbC$, sharing a contact graph which (mostly-)triangulates $\hat\bbC$, so that for all corresponding pairs of intersecting disks $D_i,D_j\in \calC$ and $\tilde D_i,\tilde D_j\in\tilde\calC$ we have that the overlap angle between $D_i$ and $D_j$ agrees with that between $\tilde D_i$ and $\tilde D_j$.  We require the extra condition that the collections are \emph{thin}, meaning that no pair of disks of $\calC$ meet in the interior of a third, and similarly for $\tilde\calC$.  Then $\calC$ and $\tilde \calC$ differ by a M\"obius or anti-M\"obius transformation.  We also prove the analogous statements for collections of closed disks in the complex plane $\bbC$, and in the hyperbolic plane $\bbH^2$.

Our method of proof is elementary and self-contained, relying only on plane topology arguments and manipulations by M\"obius transformations.  In particular, we generalize a fixed-point argument which was previously applied by Schramm and He to prove the analogs of our theorems in the circle-packing setting, that is, where the disks in question are pairwise interiorwise disjoint.  It was previously thought that these methods of proof depended too crucially on the pairwise interiorwise disjointness of the disks for there to be a hope for generalizing them to the setting of configurations of overlapping disks.

We end by stating some open problems and conjectures, including conjectured generalizations both of our main result and of our main technical theorem.  Specifically, we conjecture that our thinness condition is unnecessary in the statements of our main theorems.
\end{abstract}

\maketitle

\tableofcontents

\section{Introduction\twostars}
\label{section introduction}

A \emph{circle packing} is defined to be a collection of pairwise interiorwise disjoint metric closed disks in the Riemann sphere $\hat\bbC$.  We will always consider $\hat\bbC$ to have the usual constant curvature $+1$ spherical metric.  The \emph{contact graph} of a circle packing $\calP$ is the graph $G$ having a vertex for every disk of $\calP$, so that two vertices of $G$ are connected by an edge if and only if the corresponding disks of $\calP$ meet.  If $\calP$ is a locally finite circle packing in $\hat\bbC$, then clearly its contact graph is simple and planar.  A graph is \emph{simple} if it has no loops and no repeated edges.  If a circle packing $\calP$ has contact graph $G$ then we say that $\calP$ \emph{realizes} $G$.  It turns out that the converse also holds: if $G$ is a simple planar graph, then there is a circle packing in $\hat\bbC$ having $G$ as its contact graph.  This well-known result, known as the Circle Packing Theorem, is originally due to Koebe, first appearing in \cite{koebe-1936}.

The Circle Packing Theorem settles the question of existence of circle packings in $\hat\bbC$.  It is then natural to ask for rigidity statements.  In the same article, Koebe states a theorem equivalent to the following:

\begin{katt}
\label{katt}
Let $G$ be the 1-skeleton of a triangulation of the 2-sphere $\bbS^2$.  Then the circle packing realizing $G$ is unique, up to M\"obius and anti-M\"obius transformations.
\end{katt}

\noindent A \emph{triangulation} of a topological surface $S$ is a collection of triangular faces, each of which is a topological closed disk, so that two given faces are glued either along a single edge, or at a single vertex, or not at all, and so that there are no gluings along the boundary of any one fixed triangle, such that that the resulting object is homeomorphic to $S$.  An anti-M\"obius transformation is the composition of a M\"obius transformation with $z\mapsto \bar z$.  M\"obius and anti-M\"obius transformations send circles to circles and preserve contact graphs, so the rigidity given by Theorem \ref{katt} is the best possible.

After Koebe, the Circle Packing Theorem and Theorem \ref{katt} were for a long time forgotten.  They were reintroduced to the mathematical community at large in the 1970s by Thurston\footnote{At the International Congress of Mathematicians, Helsinki, 1978, according to \cite{MR1303402}*{p.\ 135}.}.  There he discussed his methods of proof based on Andreev's characterization of finite-volume hyperbolic polyhedra given in \cite{MR0273510}.  The best source we are aware of for Thurston's original work on this topic is his widely circulated lecture notes, \cite{thurston-gt3m-notes}*{Section 13.6}.

Thurston later conjectured\footnote{In his address at the International Symposium in Celebration of the Proof of the Bieberbach Conjecture, Purdue University, March 1985, according to \cite{MR1207210}*{p.\ 371}.} that the Riemann mapping can be approximated by circle packings.  The subsequent proof of this conjecture by Rodin and Sullivan in \cite{MR906396} confirmed the importance of circle packing to complex analysis.  A flurry of research in the area followed, and circle packing has since found applications in many other areas, for example, in combinatorics, hyperbolic 3-manifolds, probability, and geometric analysis.  A list of references for successful applications of circle packing to other areas appears for example in \cite{MR2884870}*{Section 2.2}.\medskip

%Andreev's results are too technical to state here, although we discuss a powerful generalization of his results due to Rivin in Section \ref{sec:rivin}.

It is natural to ask for rigidity statements in the spirit of Theorem \ref{katt} in geometries besides the spherical one, specifically in Euclidean and hyperbolic geometries.  This line of investigation led Schramm, and later He, to the following theorem:

\begin{dut}
\label{dut}
Let $G$ be the 1-skeleton of a triangulation of a topological open disk.  Suppose that $\calP$ and $\tilde\calP$ are circle packings realizing $G$, so that $\calP$ is locally finite in $\bbG$ and $\tilde\calP$ is locally finite in $\tilde\bbG$, where each of $\bbG$ and $\tilde\bbG$ is equal to one of $\bbC$ and $\bbH^2$.  Then $\bbG = \tilde \bbG$.  Furthermore, the packings $\calP$ and $\tilde\calP$ differ by a Euclidean similarity if $\bbG = \tilde\bbG = \bbC$, and by a hyperbolic isometry if $\bbG = \tilde\bbG = \bbH^2$.
\end{dut}

\noindent From now on, we consider the hyperbolic plane $\bbH^2$ to be identified with the open unit disk $\bbD \subset \bbC$ via the Poincar\'e embedding, and embed $\bbC \subset \hat\bbC$ via usual stereographic projection.  Then $\bbH^2 \cong \bbD \subset \bbC \subset \hat\bbC = \bbC \cup \{\infty\}$.  Furthermore, a metric closed disk in $\bbH^2$ embeds into a metric closed disk in $\bbD\subset \bbC$ under the Poincar\'e embedding.  Also, a metric closed disk in $\bbC$ is identified with a metric closed disk in $\hat\bbC$ under stereographic projection.  For clarity we remark that metric centers of disks are in general not preserved under the Poincar\'e embedding, nor under stereographic projection.

The first complete proof of Theorem \ref{dut} was given by Schramm in \cite{MR1076089}, using only elementary plane topology arguments.    Then, in \cite{MR1207210}, He and Schramm implicitly reinterpreted the method of \cite{MR1076089} as a fixed-point argument.  This approach turned out to be quite powerful, and allowed them to prove much more general statements about domains in $\hat\bbC$ whose boundary components are circles and points.  In particular, they prove the Koebe conjecture for domains having countably many boundary components.  They also prove an existence statement for circle packings in $\bbC$ and $\bbH^2$, to go along with the rigidity of Theorem \ref{dut}.  We discuss the results of \cite{MR1207210} in more detail in Section \ref{sec related work} on related work.  Other proofs of Theorem \ref{dut} have since been found, which we discuss briefly also in Section \ref{sec related work}.\medskip

In this article, we generalize the fixed-point arguments used in \cite{MR1207210}, and implicitly in \cite{MR1076089}, to prove generalizations of Theorems \ref{katt} and \ref{dut} to collections of disks whose interiors may overlap.  It was previously thought that those arguments depended too crucially on the pairwise interiorwise disjointness of the disks for there to be hope of generalizing them in this direction\footnote{For example, see comments in \cite{MR1680531}*{p.\ 3}, made by one of the authors of \cite{MR1207210}}.  Specifically, we prove rigidity and uniformization theorems for so-called \emph{thin disk configurations}:

\begin{definition}
\label{def thin disk conf}
A \emph{disk configuration} is a collection of metric closed disks on the Riemann sphere $\hat\bbC$, so that no disk of the collection is contained in any other, but with no other conditions.  A disk configuration is called \emph{thin} if no three disks of the configuration have a common point.  The \emph{contact graph} of a disk configuration $\calC$ is a graph with a vertex for every disk of $\calC$, so that two vertices share an edge if and only if the corresponding disks meet.

Suppose that $G = (V,E)$ is a graph, with vertex set $V$ and edge set $E$, so that $G$ is the contact graph of the disk configuration $\calC = \{D_v\}_{v\in V}$.  Let $\Theta : E \to [0,\pi)$ be so that if $\left<u,v\right>$ is an edge of $G$, then $\aext(D_u, D_v) = \Theta\left<u,v\right>$, with $\aext(\cdot, \cdot)$ defined as in Figure \ref{fig:aext}.  Then $(G,\Theta)$ is called the \emph{incidence data} of $\calC$, and $\calC$ is said to \emph{realize} $(G,\Theta)$.
\end{definition}

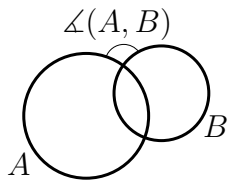
\begin{figure}[t]
\centering
% Generated with LaTeXDraw 2.0.8
% Fri Apr 13 00:20:29 EDT 2012
% \usepackage[usenames,dvipsnames]{pstricks}
% \usepackage{epsfig}
% \usepackage{pst-grad} % For gradients
% \usepackage{pst-plot} % For axes
\scalebox{1} % Change this value to rescale the drawing.
{
\begin{pspicture}(0,-1.1167188)(3.2890625,1.1567187)
\pscircle[linewidth=0.04,dimen=outer](1.18,-0.26671875){0.85}
\pscircle[linewidth=0.04,dimen=outer](2.18,0.03328125){0.65}
\psarc[linewidth=0.01](1.68,0.43328124){0.25}{37.874985}{161.56505}
\usefont{T1}{ptm}{m}{n}
\rput(0.28453124,-0.90671873){$A$}
\usefont{T1}{ptm}{m}{n}
\rput(2.9045312,-0.40671876){$B$}
\usefont{T1}{ptm}{m}{n}
\rput(1.5945313,0.9532812){$\aext(A,B)$}
\end{pspicture} 
}
\caption
{
\label{fig:aext}
The definition of $\aext(A,B)$, the \emph{external intersection angle} or \emph{overlap angle} between two closed disks $A$ and $B$.
}
\end{figure}

\noindent The main rigidity and uniformization result of this paper is the following theorem:

\begin{mainuniformization}
\label{mainuniformization}
Let $G$ be the 1-skeleton of a triangulation of a topological open disk.  Suppose that $\calC$ and $\tilde\calC$ are thin disk configurations, locally finite in $\bbG$ and $\tilde\bbG$, respectively, where each of $\bbG$ and $\tilde\bbG$ is equal to one of $\bbC$ and $\bbH^2$, so that $\calC$ and $\tilde\calC$ realize the same incidence data $(G,\Theta)$.  Then $\bbG = \tilde\bbG$, and $\calC$ and $\tilde\calC$ differ by a Euclidean similarity if $\bbG = \tilde\bbG = \bbC$, or by a hyperbolic isometry if $\bbG = \tilde\bbG = \bbH^2$.
\end{mainuniformization}

\noindent We also prove the following closely related theorem, using the same techniques:

\begin{mainrigidity}
\label{mainrigidity}
\label{rigid in sphere main}
Let $G$ be the 1-skeleton of a triangulation of the 2-sphere $\bbS^2$.  Suppose that $\calC$ and $\tilde\calC$ are thin disk configurations in $\hat\bbC$ realizing the same incidence data $(G,\Theta)$.  Then $\calC$ and $\tilde\calC$ differ by a M\"obius or an anti-M\"obius transformation.
\end{mainrigidity}

\noindent Although its statement has never appeared in the literature, Theorem \ref{mainrigidity} is not new in the sense that it follows as a corollary of known results, for example Rivin's characterization of ideal hyperbolic polyhedra, c.f.\ Section \ref{subsec hyp poly}.  We discuss this further in Section \ref{sec related work} on related work.  No counterexamples are known to Theorems \ref{mainuniformization} and \ref{mainrigidity} if the thinness condition is omitted from their statements, and we conjecture that the theorems continue to hold in this case.  More details are given in Section \ref{sec conjectures}.

\medskip

This article is organized as follows.  First, we give a brief survey of related work in Section \ref{sec related work}.  Then, in Section \ref{section fixed-point index preliminaries}, we introduce the so-called \emph{fixed-point index}, which will be the essential technical tool in our proofs of the main rigidity results of this article, Theorems \ref{mainrigidity} and \ref{mainuniformization}.

In Section \ref{sec rigid in plane}, we apply fixed-point index to prove Theorems \ref{katt} and \ref{dut} on rigidity and uniformization of classical circle packings, via the arguments of \cite{MR1207210}.  We include Section \ref{sec rigid in plane} for the following reasons.  First, He and Schramm in \cite{MR1207210} work in a much more technical setting, and do not actually work out proofs of Theorems \ref{katt} and \ref{dut}: rather, they describe how such proofs may be obtained by adapting, in a non-trivial way, their proof of the countably-connected case of the Koebe Conjecture.  The hope is that that the exposition given in Section \ref{sec rigid in plane} works to isolate the main ideas of \cite{MR1207210}, and to clarify what is required to generalize those ideas to our setting.

In Section \ref{sec main index}, we state our main technical result, which we call the Index Theorem \ref{mainindex} and sketch its proof.  The proofs of Theorems \ref{mainrigidity} and \ref{mainuniformization} require some elementary lemmas from plane geometry, and these are proved in Section \ref{sec geom lems}.  In Section \ref{sec:rigidity proof}, we prove our main rigidity results, Theorems \ref{mainrigidity} and \ref{mainuniformization}, using the Index Theorem \ref{mainindex}.  Sections \ref{chap:topo confo}--\ref{sec proof prop} are spent completing the proof of the Index Theorem \ref{mainindex}.  We conclude with a discussion of related open questions and generalizations of our results in Section \ref{sec conjectures}.\medskip

\noindent {\bf Acknowledgments.}  Thanks to my Ph.\ D.\ advisor Jeff Lagarias, for helpful comments on many portions of my dissertation, from which this article is adapted.  Thanks to Kai Rajala and Karen Smith for reading and commenting on an early version of the proofs in this article.  Thanks to Jordan Watkins for many fruitful discussions, especially for pointing us strongly in the direction of Section \ref{chap:torus}, greatly simplifying that portion of the exposition.  Thanks to Mario Bonk for helpful comments on this article.

\section{Related work\twostars}
\label{sec related work}

\subsection{Koebe uniformization\twostars}

Circle packings are closely related to classical complex analysis.  As we have already mentioned, it was conjectured by Thurston, and proved by Rodin and Sullivan in \cite{MR906396}, that circle packings can be used to approximate the Riemann mapping, in some precise sense.  Conversely, theorems in circle packing can sometimes be proved via applications of results of classical complex analysis.  For example, Koebe first discovered circle packing while researching what is now known as the Koebe Conjecture, posed in \cite{koebe-1908}*{p.\ 358}:

\begin{kc}
\label{kc}
Every domain $\Omega \subset \hat\bbC$ is biholomorphically equivalent to a circle domain.
\end{kc}

\noindent A \emph{circle domain} is a connected open subset of $\hat\bbC$ all of whose boundary components are circles and points.  In the same article, Koebe himself gave a construction, via iterative applications of the Riemann mapping, biholomorphically uniformizing an $\Omega$ having finitely many boundary components to a circle domain.  Later, in \cite{koebe-1936}, he used this uniformization to prove that any finite simple planar graph $G$ admits a circle packing realizing it.  His construction approximates the desired circle packing by first arranging disjoint not-necessarily-round compact sets roughly according to the contact pattern demanded by $G$, then uniformizing the resulting complementary region to a circle domain.  The desired circle packing is then obtained as a limit.

There is an existence statement associated to the rigidity statement of the Koebe--Andreev--Thurston Theorem \ref{katt}: if $G$ is the 1-skeleton of a triangulation of $\bbS^2$, then $G$ is finite, simple, and planar, so there exists a circle packing in $\hat\bbC$ realizing $G$.  It is natural to ask for an analogous existence statement to go along with the Discrete Uniformization Theorem \ref{dut}.  In \cite{MR1207210} He and Schramm prove that if $G$ is the 1-skeleton of a triangulation of a topological open disk, then there exists a locally finite circle packing in one of $\bbC$ and $\bbH^2$ which realizes $G$.  In the same article, they also prove the existence of the uniformizing map described in the Koebe Conjecture \ref{kc} for countably connected domains, that is, domains having countably many boundary components.  The two existence proofs are closely intertwined, both appealing to fixed-point arguments at some crucial points.\medskip

Sometimes when the Koebe Conjecture \ref{kc} is stated, a statement of uniqueness of the uniformizing biholomorphism, up to postcomposition by M\"obius transformations, is included as part of the conjecture.  The article \cite{MR1207210} establishes this rigidity portion of the conjecture as well, for countably connected domains.  The main idea of this rigidity proof is visible, adapted to the setting of circle packings, in Section \ref{sec rigid in plane}.  A sketch of the proof is given in \cite{MR2884870}*{Theorem 2.11}.

In the case of uncountably connected domains, the uniqueness part of the Koebe Conjecture \ref{kc} is well-known to be false.  A counterexample can be obtained by ``placing a nonzero Beltrami differential supported on [a Cantor set of non-zero area] and solving the Beltrami equation to obtain a quasiconformal map which is conformal outside the Cantor set,'' as noted by \cite{MR1207210}*{p.\ 370}.  The existence portion of the Koebe Conjecture \ref{kc} is still open in this case.

\subsection{Existence statements for collections of disks with overlaps\twostars}
\label{oopsies}

Given that there are existence statements to go along with both Theorems \ref{katt} and \ref{dut}, it is also natural to ask for analogous existence statements to go along with the main results of this paper, Theorems \ref{mainrigidity} and \ref{mainuniformization}.  No such existence results are presently available.

There are many non-trivial conditions on incidence data which are necessary for the existence of a disk configuration in the Riemann sphere realizing that data.  For instance, it is not hard to show that given $n$ disks $D_i$, with $i\in \bbZ/n\bbZ$, so that $D_i$ and $D_j$ meet if and only if $i=j\pm 1$, we have that $\sum_{i=1}^n \aext(D_i, D_{i+1}) < (n-2)\pi$; see Figure \ref{fig:observation angles drawing} on p.\ \pageref{fig:observation angles drawing}.  In general, conditions on $(G,\Theta)$ which force extraneous contacts are not well understood.  An example of such a condition is when $G$ contains a closed $n$-cycle consisting of distinct edges $e_1,\ldots,e_n$ so that $\sum_{i=1}^n \Theta(e_i) \ge (n-2)\pi$: in this case, by the earlier discussion, for there to be any hope of a positive answer to the existence question for the data $(G,\Theta)$, there must be at least one additional contact among the vertices which are the endpoints of the $e_i$.

The general existence question for configurations of disks realizing certain given incidence data appears not to have been studied much.  Presently, the main obstruction to obtaining theorems in this vein is not in finding proofs, but in finding the correct statements.  For example: as we mentioned, the proof given in \cite{MR1207210} of existence of circle packings having contact graphs triangulating a topological open disk relies at crucial points on fixed-point arguments.  Our Main Index Theorem \ref{mainindex} would exactly fill the gaps in the fixed-point portions of what would be the generalizations of those arguments to our setting, that of thin disk configurations.  In general, the methods used to prove existence of circle packings are quite robust and varied, and at least some of these methods are likely to generalize nicely to the setting of collections of disks with overlap, if the correct statements to be proved were known.  For further discussion on the general question of existence of disk configurations realizing given incidence data $(G,\Theta)$, see \cite{mishchenko-thesis}*{Section 2.9}.  The special case when $\Theta$ is uniformly bounded above by $\pi/2$ is much simpler than the general situation, and is discussed further in Section \ref{subsec hyp poly} and especially Section \ref{subsec vel}.

\subsection{Hyperbolic polyhedra\twostars}
\label{subsec hyp poly}

Configurations of disks on $\hat\bbC$ are closely related to hyperbolic polyhedra.  For example, given a collection of disks covering $\hat\bbC$, we may construct a hyperbolic polyhedron by cutting out the half-spaces which are bounded at $\partial \bbH^3 = \hat\bbC$ by the disks in our collection.  This construction can be used to translate theorems on hyperbolic polyhedra to theorems about circle packings or disk configurations, and vice versa.

%Let $\calP$ be a circle packing in $\hat\bbC$ whose contact graph is the 1-skeleton of a triangulation of $\bbS^2$.  The complementary components $\hat\bbC\setminus \cup_{D\in \calP} D$ are curvilinear triangles: let $\calP^* = \{D^*_U\}$, where $U$ ranges over the connected components of $\hat\bbC\setminus \cup_{D\in \calP} D$, and $D^*_U$ is the metric closed disk in $\hat\bbC$ which contains $U$ and whose boundary circle $\partial D^*_U$ passes through the corners of $U$.  Then $\calP^*$ is called the \emph{dual packing} to $\calP$.  It is not hard to show that if $D\in \calP$ and $D^*\in \calP^*$ meet, then $\aext(D,D^*) = \pi/2$.  Identify $\hat\bbC$ with the boundary at infinity of hyperbolic 3-space $\bbH^3$.  Let $P$ be obtained by removing from $\bbH^3$ the open hyperbolic half-spaces which are bounded at $\hat\bbC = \partial \bbH^3$ by the interiors of the disks of $\calP$ and of $\calP^*$.  Then $P$ is an ideal hyperbolic polyhedron, having finite volume.  Also, the interior angle between two faces of $P$ sharing an edge is always exactly $\pi/2$.

In \cite{MR0273510}, Andreev gives a characterization of finite-volume hyperbolic polyhedra satisfying the condition that every two faces sharing an edge meet at an interior angle of at most $\pi/2$.  In particular, the combinatorics and interior angles of such a polyhedron completely determine it, up to hyperbolic isometry of $\bbH^3$.  From this one may deduce Theorem \ref{katt}.  This was the approach originally taken by Thurston.  For the details of the construction, see \cite{thurston-gt3m-notes}*{Section 13.6}.

Rivin has worked extensively on generalizations of Andreev's characterization theorems.  In particular, he has given a complete characterization of ideal hyperbolic polyhedra all of whose vertices lie on $\hat\bbC = \partial \bbH^3$, with no requirements on the incidence angles of the faces; see \citelist{\cite{MR1283870}*{Theorem 14.1 (rigidity)} \cite{MR1370757}*{Theorem 0.1 (existence)} \cite{MR1985831}*{(generalizations)}}.  Our Theorem \ref{mainrigidity} can be obtained as a corollary of his.  The full strength of our Theorem \ref{mainrigidity} cannot be obtained from Andreev's results, because of the bound on the interior angles $\Theta(e)$ at the edges $\{e\}$ of the polyhedra in his hypotheses.  Rivin remarks that in the setting of hyperbolic polyhedra, the restriction $\Theta \le \pi/2$ is a very strong one\footnote{See comments in \cite{MR1370757}*{p.\ 52}}.  However, interestingly, there are few places in the present article where a corresponding upper bound of $\pi/2$ on the overlap angles of our disks would simplify the arguments significantly.\medskip

No proofs of existence nor of rigidity statements for circle packings having contact graphs triangulating a topological open disk have been obtained directly via these or similar theorems on hyperbolic polyhedra.  A major obstruction is that the ``polyhedron'' constructed via Thurston's methods from an infinite circle packing in $\hat\bbC$ typically has infinite volume.\medskip

Rivin's theorem characterizing ideal hyperbolic polyhedra can be directly translated into a statement about configurations of disks on $\hat\bbC = \partial \bbH^3$.  One may then hope to generalize this translated statement to the higher-genus setting.  Bobenko and Springborn have done exactly this in \cite{MR2022715}*{Theorem 4}, where they prove an existence and uniqueness statement for disk configurations on positive-genus closed Riemann surfaces.  Their proof uses variational principles.

One may hope that because Rivin's characterization of ideal hyperbolic polyhedra includes an existence component, we may obtain an existence statement about disk configurations to go along with our Rigidity Theorem \ref{mainrigidity} from the direct translation of Rivin's results.  However, the existence portion of this translation does not take as input the contact graph $G$ of the disk configuration $\calC$ which is eventually realized, rather taking a certain planar subgraph of $G$.  There is no known method for computing the eventual contact graph $G$ obtained this way from only the allowed input to the translated theorem, although by the rigidity portion we know that $G$ is completely determined by said input.  This itself may be a difficult problem: for a heuristic argument explaining why, see \cite{mishchenko-thesis}*{Section 2.9.6}.  This subtle issue also underlies Bobenko and Springborn's results, although it is not directly addressed by those authors.

\subsection{Vertex extremal length and modulus\twostars}
\label{subsec vel}

A discretized version of classical conformal modulus, equivalently extremal length, of a curve family has been used extensively to prove circle packing theorems.  One of the earliest such applications was by He and Schramm in \cite{MR1331923}.  There, using so-called vertex extremal length, given a graph $G$ which is the 1-skeleton of a triangulation of a topological open disk, they reprove the existence of a locally finite circle packing realizing $G$ in exactly one of $\bbC$ and $\bbH^2$.  They also give discrete-analytic conditions on $G$, for example the recurrence or transience of a simple random walk on $G$, which determine whether the circle packing realizing $G$ lives naturally in $\bbC$ or in $\bbH^2$.

These ideas have been generalized by He to the setting of disk configurations with overlaps.  In \cite{MR1680531} he proves a generalization of Theorem \ref{dut} using similar methods for configurations of disks whose overlap angles are bounded above by $\pi/2$.  He also includes an existence statement.  In the same paper, he wrote\footnote{See comments in \cite{MR1680531}*{p.\ 2}.} that he intended to generalize his techniques further to handle the case of arbitrary overlap angles, but he never published any work doing so.

\subsection{Disk configurations in other Riemann surfaces\twostars}

Given a triangulation $X$ of an open or closed oriented topological surface $S$ without boundary, it is possible to find a complete constant curvature Riemannian metric $d$ on $S$, and a circle packing $\calP$ in $(S,d)$, whose contact graph is the 1-skeleton of $X$.  To see why, first lift $X$ to a triangulation of the universal cover of $S$, allowing us to obtain a periodic circle packing in a simply connected constant curvature surface, one of $\hat\bbC$, $\bbC$, or $\bbH^2$.  Quotienting by this periodicity gives our desired circle packing realizing $X$ in some complete constant curvature Riemann surface $(S,d)$.  Note that the metric $d$ and the packing $\calP$ in $(S,d)$ are both essentially uniquely determined by $X$, by the rigidity of Theorem \ref{dut}.  This construction is well-known, appearing for example in \cite{MR1087197}, and later in modified form in \cite{MR1207210}*{Section 8}.

There is no obstruction to applying the same argument in the more general setting of disk configurations with overlaps.  For example, applying our Theorem \ref{mainuniformization}, if $\calC$ and $\tilde\calC$ are thin disk configurations realizing the same incidence data, living in complete constant curvature Riemann surfaces $R$ and $\tilde R$ respectively, then $R$ and $\tilde R$ are conformally isomorphic.  Thus it is generally sufficient to prove an existence or rigidity statement for disk configurations in $\bbC$ or $\bbH^2$ to get analogous statements in multiply connected surfaces.\medskip

Some authors have studied related theorems in higher genus surfaces directly.  For example, Thurston himself proved an existence theorem for disk configurations with overlap angles bounded above by $\pi/2$ on closed finite-genus surfaces without boundary, in \cite{thurston-gt3m-notes}*{Theorem 13.7.1}.  He did not give a rigidity statement.  As was already mentioned, Bobenko and Springborn proved an existence and uniqueness theorem for circle patterns on closed finite-genus surfaces without boundary, in \cite{MR2022715}.  Both of these proofs use variational principles.

\subsection{Further references\twostars}

Today many proofs are known of the Circle Packing Theorem and of Theorem \ref{katt}.  For example, some proofs using variational principles appear in \citelist{\cite{MR1106755} \cite{MR1189006} \cite{MR1283870}}.  Thurston describes how Theorem \ref{katt} may be obained from Mostow--Prasad rigidity in \cite{thurston-gt3m-notes}*{Proof of 13.6.2}.  There is a short, clean proof of Theorem \ref{katt} in \cite{MR2884870}*{Section 2.4.3}, which is attributed to Oded Schramm.  The earliest published version of this same argument that we are aware of appears in \cite{MR1680531}*{Section 2}.

Many of these arguments generalize readily to prove our Theorem \ref{mainrigidity}.  However, few of them have been adapted to prove Theorem \ref{dut}, and it therefore appears unlikely that they will work to prove our Theorem \ref{mainuniformization} either.
\medskip

The fixed-point index techniques we use here have recently been used by Merenkov to prove rigidity statements for Sierpinski carpets, see \cite{MR2900233}*{Section 12}.\medskip

Theorem \ref{dut} in the case where $G$ has uniformly bounded-above vertex degree is proved in \cite{MR1087197}, by a modification of an argument by Rodin and Sullivan given in \cite{MR906396}*{Appendix 1}.  The proof uses quasi-conformal mapping theory.  It appears unlikely that this method of proof can be generalized to the unbounded-valence case.  However, it would likely prove the bounded-valence case of our Theorem \ref{mainuniformization} just as well.\medskip

A short survey of parts of the area of circle packing which nowadays are considered classical is given by Sachs in \cite{MR1303402}.  It also gives a rough outline of the history of circle packing through 1994.  The first half of \cite{MR2884870} is an excellent survey by Rohde focused on the contributions of Oded Schramm.  Rohde also gives a long list of successful applications of circle packing to other areas of math in his Section 2.2.  Stephenson's book \cite{MR2131318} provides a more detailed, elementary, and mostly self-contained introduction to the area and could serve as a kind of ``first course in circle packing.'' Stephenson's methods of proof in this book are essentially adapted from those of He and Schramm in \cite{MR1207210}.  Finally, the introduction of the Ph.\ D.\ thesis \cite{mishchenko-thesis}*{Chapter 2} of the present author, from which this article is adapted, gives a fairly thourough survey of the area of circle packing.

\section{Fixed-point index preliminaries\twostars}
\label{section fixed-point index preliminaries}

\phantomsection
\label{def pos orient}
A \emph{Jordan curve} is a homeomorphic image of a topological circle $\bbS^1$ in the complex plane $\bbC$.  A \emph{Jordan domain} is a bounded open set in $\bbC$ with Jordan curve boundary.  We use the term \emph{closed Jordan domain} to refer to the closure of a Jordan domain.  Suppose that a Jordan curve appears as the boundary of a set $X\subset \bbC$ having non-empty interior.  Then the \emph{positive orientation} of $\partial X$ with respect to $X$ is defined as usual.  That is, the interior of $X$ stays to the left as we traverse $\partial X$ in what we call the \emph{positive} direction.  As a rule, if we write a Jordan curve as $\partial X$, where $X$ is an open or closed Jordan domain or the complement thereof, then we will take that to mean that $\partial X$ is oriented positively with respect to $X$, unless otherwise noted.  In particular, if $X$ is an open or closed Jordan domain, then the positive orientation induced on $\partial X$ is the counterclockwise one as usual.

We now define the \emph{fixed-point index}, our main technical tool.  The rest of this section will consist of the proofs of several fundamental lemmas on fixed-point index.\medskip

%\begin{jct}
%\label{jct}
%A Jordan curve $\gamma$ separates the plane into two connected components, one bounded and one unbounded.  The bounded one is a topological open disk and the unbounded one is a punctured topological open disk.  The boundary of either is exactly $\gamma$.
%\end{jct}

%\noindent See \cite{MR1160354}*{Theorem 10.2}.

\begin{definition}
\label{def fixed-point index}
Let $\gamma$ and $\tilde \gamma$ be oriented Jordan curves.  Let $\phi:\gamma \to \tilde\gamma$ be a homeomorphism which is fixed-point-free and orientation-preserving.  We call such a homeomorphism \emph{indexable}.  Then $\{\phi(z) - z\}_{z\in \gamma}$ is a closed curve in $\bbC$ which misses the origin.  It has a natural orientation induced by traversing $\gamma$ according to its orientation.  We define the \emph{fixed-point index} of $\phi$, denoted $\eta(\phi)$, to be the winding number of $\{\phi(z) - z\}_{z\in \gamma}$ around the origin.
\end{definition}

Intuitively, fixed-point index counts the following.  Suppose that $\Phi : K \to \tilde K$ is a homeomorphism of closed Jordan domains having only isolated fixed points, which restricts to an indexable $\phi : \partial K \to \partial \tilde K$.  Then $\eta(\phi)$ counts the number of fixed points, with signed multiplicity, of $\Phi$.  For more discussion on the history and broader relevance of fixed-point index, see \cite{MR1207210}*{Section 2}.  Every integer, positive or negative, occurs as a fixed-point index.\medskip

Our first lemma says that the fixed-point index between two (round) circles is always non-negative:

\begin{cil}
\label{cil}
\label{lem:circle index lemma}
Let $K$ and $\tilde K$ be closed Jordan domains in $\bbC$, and let $\phi:\partial K \to \partial \tilde K$ be an indexable homeomorphism.  Then the following hold.
\begin{enumerate}
\item The homeomorphism $\phi^{-1} : \partial \tilde K  \to \partial K$ is indexable with $\eta(\phi)=\eta(\phi^{-1})$.
\item If $K\subseteq \tilde K$ or $\tilde K\subseteq K$, then $\eta(\phi)=1$.
\item If $K$ and $\tilde K$ have disjoint interiors, then $\eta(\phi) = 0$.
\item If $\partial K$ and $\partial \tilde K$ intersect in exactly two points, then $\eta(\phi)\ge 0$.
\end{enumerate}
As a consequence of the above, if $K$ and $\tilde K$ are metric closed disks in the plane, then $\eta(\phi) \ge 0$.
\end{cil}

\noindent This lemma can be found in \cite{MR1207210}*{Lemma 2.2}.  There it is indicated that the same lemma appeared earlier in \cite{MR0051934}.  We sketch the proof of Lemma \ref{lem:circle index lemma} given in \cite{MR1207210}*{Lemma 2.2}:\medskip 

\noindent \emph{``Proof.''} (1) By definition $\eta(f\inv)$ is the winding number of $\{f\inv(\tilde z) - \tilde z\}_{\tilde z\in \partial \tilde K}$ around the origin, which is equal to the winding number of $\{z - f(z)\}_{z\in \partial K}$ around the origin under the coordinate change $f(z) = \tilde z$.  But the winding number around the origin of a cloesd curve $\{\gamma(t)\}_{t\in \bbS^1}$ which misses $0$ is equal to the winding number around the origin of $\{-\gamma(t)\}_{t\in \bbS^1}$.

Part (2) is believable if we imagine $K$ to be ``very small,'' and contained in $\tilde K$.  Then the endpoint $z$ of the vector $f(z) - z$ does not move very much as $z$ traverses $\partial K$, while the endpoint $f(z)$ of the same vector ``winds once positively around $K$.''  Part (3) is believable for similar reasons if we imagine $K$ and $\tilde K$ to be very far away from each other.  These ideas can be made into proofs via simple homotopy arguments.

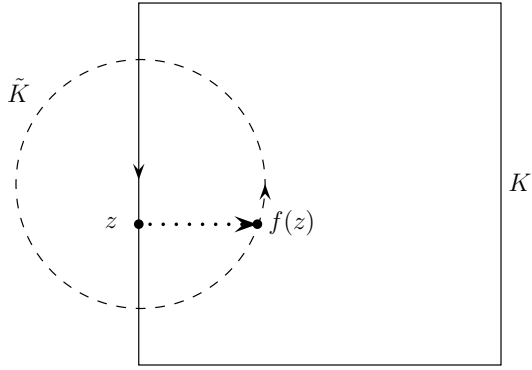
\begin{figure}[t]
\centering
% Generated with LaTeXDraw 2.0.8
% Tue Nov 29 12:16:26 EST 2011
% \usepackage[usenames,dvipsnames]{pstricks}
% \usepackage{epsfig}
% \usepackage{pst-grad} % For gradients
% \usepackage{pst-plot} % For axes
\scalebox{.8} % Change this value to rescale the drawing.
{
\begin{pspicture}(0,-3.01)(9.509063,3.01)
\psframe[linewidth=0.02,dimen=middle](8.81,3.01)(2.79,-3.01)
\pscircle[linewidth=0.02,linestyle=dashed,dash=0.16cm 0.16cm,dimen=middle](2.82,0.0){2.07}
\psline[linewidth=0.07cm,linestyle=none]{->}(4.89,-0.07)(4.89,0.00)
\psline[linewidth=0.07cm,linestyle=none]{->}(2.79,0.17)(2.79,0.07)
\psline[linewidth=0.06cm,linestyle=dotted,dotsep=0.16cm,arrowsize=0.15291667cm 2.0,arrowlength=1.4,arrowinset=0.4]{->}(2.79,-0.67)(4.76,-0.67)
\usefont{T1}{ptm}{m}{n}
\rput(9.124531,0.02){$K$}
\usefont{T1}{ptm}{m}{n}
\rput(0.7945312,1.56){$\tilde K$}
\psdots[dotsize=0.16](4.76,-0.67)
\psdots[dotsize=0.16](2.79,-0.67)
\usefont{T1}{ptm}{m}{n}
\rput(5.3345314,-0.64){$f(z)$}
\usefont{T1}{ptm}{m}{n}
\rput(2.3345313,-0.64){$z$}
\end{pspicture} 
}
\caption
{
\label{fig:lem two int sq circ}
Closed Jordan domains whose boundaries cross transversely at exactly two points, after an isotopy.  Whenever $f(z) - z \in \bbR_+$, for instance as shown, then $f(z) - z$ is locally turning counter-clockwise.
}
\end{figure}

For part (4) we may assume without loss of generality by parts (2) and (3) that $\partial K$ and $\partial \tilde K$ meet transversely at both of their intersection points, so after applying an isotopy to $\bbC$ we have that $K$ and $\tilde K$ are the square and circle depicted in Figure \ref{fig:lem two int sq circ}, c.f.\ Lemma \ref{prop:index invariance}.  We ask ourselves when the vector $f(z) - z$ can possibly point in the positive real direction, as in Figure \ref{fig:lem two int sq circ}.  If $z\in \partial K$ does not lie in the interior of $\tilde K$, the vector $f(z) - z$ has either an imaginary component, or a negative real component.  Similarly if $f(z)\in \partial \tilde K$ does not lie inside of $K$, then $f(z) - z$ has a negative real component.  We conclude that the only way that $f(z) - z$ can be real and positive is if $z$ lies along $\partial K$ in the interior of $\tilde K$ and $f(z)$ lies along $\partial \tilde K$ inside of $K$.  But in this case because of the orientations on $\partial K$ and $\partial \tilde K$, the vector $f(z) - z$ is locally turning in the positive direction.  Thus whenever the curve $\{f(z) - z\}_{z\in \partial K}$ crosses the positive real axis it is turning in the positive direction, so this curve's total winding number around the origin cannot be negative.\qed\medskip

Our next lemma says essentially that fixed-point indices ``add nicely'':

\begin{additivity}
\label{lem:indices add}
\label{additivity}
\label{ial}
Suppose that $K$ and $L$ are interiorwise disjoint closed Jordan domains which meet along a single positive-length Jordan arc $\partial K \cap \partial L$, similarly for $\tilde K$ and $\tilde L$.  Then $K\cup L$ and $\tilde K \cup \tilde L$ are closed Jordan domains.

Let $\phi_K:\partial K\to \partial \tilde K$ and $\phi_L:\partial L\to \partial\tilde L$ be indexable homeomorphisms.  Suppose that $\phi_K$ and $\phi_L$ agree on $\partial K \cap \partial L$.  Let $\phi:\partial (K\cup L) \to \partial (\tilde K \cup \tilde L)$ be induced via restriction to $\phi_K$ or $\phi_L$ as necessary.  Then $\phi$ is an indexable homeomorphism and $\eta(\phi) = \eta(\phi_L) + \eta(\phi_K)$.
\end{additivity}

\begin{proof}
\label{ex:indices add}
By the definition of the fixed-point index, we have that $\eta(\phi_L)$ is equal to $1/2\pi$ times the change in argument of the vector $f(z) - z$, as $z$ traverses $\partial K$ once in the positive direction, and similarly for $\eta(\phi_K)$ and $\eta(\phi)$.  The orientation induced on $\partial K \cap \partial L$ by the positive orientation on $\partial K$ is opposite to the one induced by the positive orientation on $\partial L$, so as $z$ varies positively in $\partial K$ and in $\partial \tilde K$ the contributions to the sum $\eta(\phi_K) + \eta(\phi_L)$ along $\partial K \cap \partial L$ exactly cancel.

If we consider the alternative interpretation of the fixed-point index of $\phi$ to be counting the number of fixed points with signed multiplicity of a homeomorphic extension of $\phi$ to all of $K\cup L$, and similarly for $\eta(\phi_K)$ and $\eta(\phi_L)$, then the lemma is also clear.
\end{proof}\medskip

\phantomsection
\label{def:gen pos jordan curves}
Moving on, we make a definition.  Let $K$ and $\tilde K$ be closed Jordan domains.  We say that $K$ and $\tilde K$ are in \emph{transverse position} if $\partial K$ and $\partial \tilde K$ cross wherever they meet.  More precisely, we say that $K$ and $\tilde K$ are in \emph{transverse position} if for any $z\in \partial K \cap \partial \tilde K$, there is an open neighborhood $U$ of $z$ and a homeomorphism $\phi:U\to \bbD$ to the open unit disk sending $\partial K \cap U$ to $\bbR \cap \bbD$ and sending $\partial \tilde K\cap U$ to $i\bbR \cap \bbD$.\medskip

We now state our next fundamental lemma about fixed-point index.  This lemma says essentially that we may almost always prescribe the images of three points on $\partial K$ in $\partial \tilde K$, and obtain an indexable homeomorphism $\partial K \to \partial \tilde K$ with non-negative fixed-point index, which respects this prescription.

\begin{threep}
\label{threep}
\label{3p}
\label{tppl}
\label{lem:3p}
\label{lem:three points}
Let $K$ and $\tilde K$ be closed Jordan domains in transverse position.  Let $z_1,z_2,z_3\in \partial K\setminus \partial \tilde K$ appear in positively oriented order around $\partial K$, and similarly $\tilde z_1,\tilde z_2, \tilde z_3\in \partial \tilde K\setminus \partial K$.  Then there is an indexable homeomorphism $\phi:\partial K \to \partial \tilde K$ sending $z_i \mapsto \tilde z_i$ for $i=1,2,3$, so that $\eta(\phi) \ge 0$.
\end{threep}

\noindent A version of this lemma is stated in \cite{MR2131318}*{Lemma 8.14}, with the following heuristic argument: by Carath\'eodory's theorem (see \cite{MR1511735}), because $K$ and $\tilde K$ are closed Jordan domains, any Riemann mapping from the interior of $K$ to that of $\tilde K$ extends homeomorphically to a map $\Phi:K \to \tilde K$.  Furthermore $\Phi$ may be chosen so that $\Phi : z_i \mapsto \tilde z_i$ for $i=1,2,3$.  Fix such a $\Phi$, and let $\phi : \partial K \to \partial \tilde K$ be defined by restriction.

Suppose that $\phi : \partial K \to \partial \tilde K$ does not have any fixed points.  It is automatically orientation-preserving because a Riemann mapping always is.  Suppose also that $\partial K$ and $\partial \tilde K$ are piecewise smooth.  Then, using the standard complex analysis definition of winding number, we have that:
\[
\eta(\phi) = \oint_{\{\phi(z) - z\}_{z\in \partial K}} \frac{dw}{w} = \oint_{\partial K} \frac{ \Phi'(w) - 1 }{ \Phi(w) - w }dw
\]
Then by the standard Argument Principle, the second integral counts the number of zeros minus the number of poles of $\Phi(z) - z$ in the interior of $K$, but $\Phi(z) - z$ is holomorphic there, thus has no poles there, so this integral is non-negative.\medskip

Actually $\Phi'$ is undefined on $\partial K$, because $\Phi$ is not holomorphic in a neighborhood of $K$, so the second integral does not quite make sense.  There is another more serious issue, which is that in general $\phi : \partial K \to \partial \tilde K$ may have many fixed points, and it is not clear how to get rid of them.  The argument given in \cite{MR2131318}*{Lemma 8.14} does not address these two issues.  We give an original elementary inductive proof of Lemma \ref{threep}, using only plane topology arguments, in \cite{mishchenko-thesis}*{Section 3.5}.  The proof is not hard, but is lengthy to include here.  This lemma fails if we try to prescribe the images of four points.  For a counterexample, see \cite{MR1207210}*{Figure 2.2}.

\section{Rigidity proofs in the circle packing case\twostars}
\label{sec rigid in plane}

In this section we prove rigidity theorems for circle packings which are special cases of our main rigidity results on thin disk configurations.  The arguments here are adapted from those in \cite{MR1207210}.\medskip

The following well-known and easy-to-check lemma will be implicit in our discussion below, although we will not refer to it directly:

\begin{lemma}
Let $G = (V,E)$ be a 3-cycle.  Let $\calP = \{D_v\}_{v\in V}$ and $\tilde \calP = \{\tilde D_v\}_{v\in V}$ be circle packings in $\hat\bbC$ having contact graph $G$.  Then any M\"obius transformation sending the three tangency points of $\calP$ to those of $\tilde\calP$ in fact identifies $\calP$ and $\tilde\calP$.
\end{lemma}

\noindent (Let $C^*$ and $\tilde C^*$ be the circles passing through the tangency points of $\calP$ and those of $\tilde\calP$, respectively.  Then one can show that $C^*$ meets every $\partial D_v$ orthogonally, similarly $\tilde C^*$ and the $\partial\tilde D_v$, and the lemma follows.)\medskip

Before moving on, it will help to have some definitions.  First, suppose that $\calP$ is a circle packing locally finite in $\bbG$, where $\bbG$ is equal to one of $\hat\bbC, \bbC, \bbH^2$.  Then $\calP$ induces an embedding of its contact graph $G$ in $\bbG$ by placing every vertex $v$ of $G$ at the metric center of its associated disk $D_v\in \calP$, and connecting the centers $u,v$ of touching disks $D_u, D_v$ with a geodesic arc passing through the point $D_u\cap D_v$.  We call this the \emph{geodesic embedding of $G$ in $\bbG$ induced by $\calP$}.  Next, two circle packings $\calP$ and $\tilde\calP$ are in \emph{general position} if the following hold:
\begin{itemize}
\item Every pair of disks $D\in \calP$ and $\tilde D \in \tilde \calP$ are in transverse position as closed Jordan domains.  For a reminder of what this means see p.\ \pageref{def:gen pos jordan curves}.
\item If $z$ is an intersection point of two distinct disks of $\calP$, then $z$ does not lie on $\partial \tilde D$ for any $\tilde D\in \tilde\calP$.  Similarly if $\tilde z$ is an intersection point of two distinct disks of $\tilde\calP$, then $z$ does not lie on $\partial D$ for any $D\in \calP$.
\end{itemize}\medskip

We now proceed to the proofs of our rigidity theorems on circle packings.  First:

\begin{theorem}
\label{rigid in sphere}
\label{thm cp rigid in sphere}
Suppose that $\calP$ and $\tilde\calP$ are circle packings in $\hat\bbC$, sharing a contact graph $G$ which is the 1-skeleton of a triangulation of the 2-sphere $\bbS^2$.  Then $\calP$ and $\tilde\calP$ differ by a M\"obius or an anti-M\"obius transformation.
\end{theorem}

\begin{proof}
First, recall that if $G$ is the 1-skeleton of a triangulation of $\bbS^2$, then there are exactly two ways to embed $G$ in $\bbS^2$, up to orientation-preserving self-homeomorphism of $\bbS^2$.  Therefore we may suppose without loss of generality, by applying $z\mapsto \bar z$ to one of the two packings if necessary, that the geodesic embeddings of $G$ in $\hat\bbC$ induced by $\calP$ and by $\tilde\calP$ are images of one another by orientation-preserving self-homeomorphisms of $\hat\bbC$.  In our next three proofs we note this preliminary procedure simply by saying that we may suppose without loss of generality that $\calP$ and $\tilde\calP$ have the same \emph{orientation}.

Then we proceed by contradiction, supposing that there is no M\"obius transformation sending $\calP$ to $\tilde\calP$.  The first step of the proof is to normalize $\calP$ and $\tilde\calP$ in a convenient way.  In particular, we apply M\"obius transformations so that the following holds:

\begin{normalization}
\label{lem normalize three disks of cp in sphere}
There are disks $D_a,D_b,D_c\in \calP$ and $\tilde D_a, \tilde D_b, \tilde D_c\in\tilde\calP$, where $a,b,c$ are distinct vertices of the common contact graph $G$ of $\calP$ and $\tilde\calP$, so that the following hold:
\begin{itemize}
\item One of $D_v$ and $\tilde D_v$ is contained in the interior of the other, for all $v = a,b,c$.
\item The point $\infty\in \hat\bbC$ lies in the interior of $D_a \cap \tilde D_a$.
\item The packings $\calP$ and $\tilde\calP$ are in general position.
\end{itemize}
\end{normalization}

\noindent We do not prove here that Normalization \ref{lem normalize three disks of cp in sphere} is possible, because it is a special case of the stronger normalization we construct in detail in the proof of our Main Rigidity Theorem \ref{mainrigidity}.  See Figure \ref{fig normalize three disks of cp on sphere} for a model drawing of the present situation.

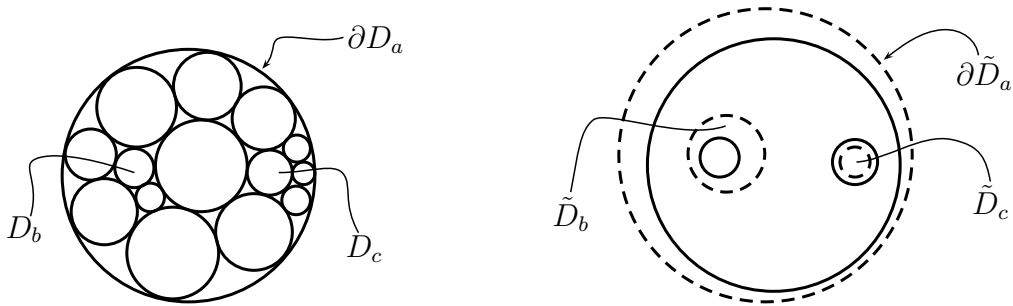
\begin{figure}[t]
\centering
\subfloat
{
% Generated with LaTeXDraw 2.0.8
% Fri Sep 28 14:56:11 CDT 2012
% \usepackage[usenames,dvipsnames]{pstricks}
% \usepackage{epsfig}
% \usepackage{pst-grad} % For gradients
% \usepackage{pst-plot} % For axes
\scalebox{1} % Change this value to rescale the drawing.
{
\begin{pspicture}(0,-1.9367187)(6.2090626,1.9767188)
\pscircle[linewidth=0.04,dimen=outer](2.69,-0.23671874){1.7}
\pscircle[linewidth=0.04,dimen=outer](2.94,0.9532812){0.47}
\pscircle[linewidth=0.04,dimen=outer](3.69,0.5232813){0.44}
\pscircle[linewidth=0.04,dimen=outer](2.86,-0.10671875){0.63}
\pscircle[linewidth=0.04,dimen=outer](3.77,-0.19671875){0.32}
\pscircle[linewidth=0.04,dimen=outer](4.12,-0.56671876){0.21}
\pscircle[linewidth=0.04,dimen=outer](2.48,-1.2667187){0.63}
\pscircle[linewidth=0.04,dimen=outer](2.0,0.67328125){0.55}
\pscircle[linewidth=0.04,dimen=outer](3.56,-0.9867188){0.53}
\pscircle[linewidth=0.04,dimen=outer](4.13,0.12328125){0.2}
\pscircle[linewidth=0.04,dimen=outer](4.22,-0.20671874){0.17}
\pscircle[linewidth=0.04,dimen=outer](1.39,0.04328125){0.36}
\pscircle[linewidth=0.04,dimen=outer](1.57,-0.71671873){0.46}
\pscircle[linewidth=0.04,dimen=outer](2.18,-0.52671874){0.21}
\pscircle[linewidth=0.04,dimen=outer](1.97,-0.13671875){0.28}
\usefont{T1}{ptm}{m}{n}
\rput(0.50453126,-.9467187){$D_b$}
\usefont{T1}{ptm}{m}{n}
\rput(5.0245314,-1.2067188){$D_c$}
\usefont{T1}{ptm}{m}{n}
\rput(5.1845315,1.5732812){$\partial D_a$}
\psbezier[linewidth=0.02,arrowsize=0.05291667cm 2.0,arrowlength=1.4,arrowinset=0.4]{<-}(3.67,1.1832813)(3.87,1.4232812)(3.75,1.5832813)(3.95,1.6032813)(4.15,1.6232812)(4.49,1.5232812)(4.75,1.6032813)
\psbezier[linewidth=0.02](3.89,-0.19671875)(4.31,-0.19671875)(4.63,-0.05671875)(4.75,-0.21671875)(4.87,-0.37671876)(4.73,-0.5967187)(4.95,-0.99671876)
\psbezier[linewidth=0.02](1.97,-0.19671875)(1.55,-0.05671875)(1.01,0.08328125)(0.71,0.00328125)(0.41,-0.07671875)(0.73,-0.27671874)(0.59,-0.77671874)
\end{pspicture} 
}
}
\subfloat
{
% Generated with LaTeXDraw 2.0.8
% Fri Sep 28 15:08:24 CDT 2012
% \usepackage[usenames,dvipsnames]{pstricks}
% \usepackage{epsfig}
% \usepackage{pst-grad} % For gradients
% \usepackage{pst-plot} % For axes
\scalebox{1} % Change this value to rescale the drawing.
{
\begin{pspicture}(0,-1.97)(8.309063,1.97)
\pscircle[linewidth=0.04,dimen=outer](3.71,-0.13){1.7}
\pscircle[linewidth=0.04,dimen=outer](4.79,-0.09){0.32}
\pscircle[linewidth=0.04,dimen=outer](2.99,-0.03){0.28}
\usefont{T1}{ppl}{m}{n}
\rput(6.594531,-0.62){$\tilde D_c$}
\pscircle[linewidth=0.04,linestyle=dashed,dimen=outer](3.08,0.02){0.53}
\pscircle[linewidth=0.04,linestyle=dashed,dimen=outer](4.79,-0.09){0.22}
\pscircle[linewidth=0.04,linestyle=dashed,dimen=outer](3.6,0.0){1.97}
\usefont{T1}{ptm}{m}{n}
\rput(1.0145313,-0.78){$\tilde D_b$}
\usefont{T1}{ptm}{m}{n}
\rput(6.514531,1.1){$\partial \tilde D_a$}
\psbezier[linewidth=0.02](3.07,0.39)(2.51,0.19)(1.49,0.61)(1.31,0.47)(1.13,0.33)(1.11,0.03)(1.01,-0.49)
\psbezier[linewidth=0.02](4.81,-0.09)(5.19,-0.07)(5.67,0.29)(6.07,0.21)(6.47,0.13)(6.53,-0.05)(6.59,-0.35)
\psbezier[linewidth=0.02,arrowsize=0.05291667cm 2.0,arrowlength=1.4,arrowinset=0.4]{<-}(5.17,1.27)(5.43,1.47)(5.67,1.65)(5.95,1.67)(6.23,1.69)(6.41,1.51)(6.43,1.33)
\end{pspicture} 
}
}
\caption{
\label{fig normalize three disks of cp on sphere}
The packing $\calP$, with $\infty$ in the interior of $D_a$; and the interaction between the disks $D_a,D_b,D_c$ and $\tilde D_a,\tilde D_b, \tilde D_c$ after our normalization.
}
\end{figure}

An \emph{interstice} of the packing $\calP$ is defined to be a connected component of $\hat\bbC \setminus \cup_{D\in \calP} D$.  That is, the interstices of $\calP$ are the curvilinear triangles which make up the complement of the packing.  If we write $F$ to denote the set of faces of the triangulation of $\bbS^2$ having $G$ as its 1-skeleton, then the interstices of $\calP$ are in natural bijection with the faces $F$.  (If we embed $G$ via the embedding induced by $\calP$, then every face of the resulting triangulation of $\hat\bbC$ contains precisely one interstice.)  We write $T_f$ to denote the interstice of $\calP$ corresponding to the face $f\in F$.  The interstices $\tilde T_f$ of $\tilde\calP$ are defined analogously.  Note also that there is a natural correspondence from the corners of $T_f$ to those of $\tilde T_f$, for a given $f\in F$.  For every $f\in F$, fix an indexable homeomorphism $\phi_f : \partial T_f \to \partial \tilde T_f$ which identifies corresponding corners with $\eta(\phi_f) \ge 0$.  We may do so by the Three Point Prescription Lemma \ref{tppl}.  We remark that it is also important at this step that $\calP$ and $\tilde\calP$ have the same orientation, as per the first paragraph of this proof.  Then the homeomorphisms $\phi_f$ induce, by restriction, indexable homeomorphisms $\phi_v : \partial D_v \to \partial \tilde D_v$ between the boundary circles of the disks $D_v\in \calP$ and $\tilde D_v\in \tilde\calP$.  Here $v$ ranges over the vertex set $V$ of $G$.

Orient $\partial D_a$ and $\partial \tilde D_a$ positively with respect to the open disks they bound in $\bbC$.  We remark for clarity that this is the opposite of the positive orientation on $\partial D_a$ and $\partial D_a$ with respect to the interiors of $D_a$ and $\tilde D_a$, in $\hat\bbC$.  Then $\eta(\phi_v) = 1$ for all $v=a,b,c$, by the Circle Index Lemma \ref{cil}, because of our Normalization \ref{lem normalize three disks of cp in sphere}.  On the other hand, by the Index Additivity Lemma \ref{ial}, we have the following:

\[
1 = \eta(\phi_a) = \sum_{f\in F} \eta(\phi_f) + \sum_{v\in V\setminus \{a\}} \eta(\phi_v)
\]

\noindent Every $\eta(\phi_f)$ in the first sum is non-negative by construction, and every $\eta(\phi_v)$ in the second sum is non-negative by the Circle Index Lemma \ref{cil}.  Also, we have contributions from $\eta(\phi_b) = 1$ and $\eta(\phi_c) = 1$ to second sum, so it must be at least 2, giving us the desired contradiction.
\end{proof}

The proofs of our other rigidity and uniformization theorems for circle packings are adapted from the proof of Theorem \ref{rigid in sphere} using similar ideas.  We give these proofs now without further comment: 

\begin{theorem}
\label{thm:no luv}
\label{libun}
There cannot be two circle packings $\calP$ and $\tilde \calP$ sharing a contact graph $G$ which is the 1-skeleton of a triangulation of a topological open disk, so that one of $\calP$ and $\tilde\calP$ is locally finite in $\bbC$ and the other is locally finite in the open unit disk $\bbD$, or equivalently the hyperbolic plane $\bbH^2\cong \bbD$.
\end{theorem}

\begin{proof}
We again proceed by contradiction, supposing that $\calP$ is locally finite in $\bbC$ and $\tilde\calP$ is locally finite in the open unit disk $\bbD$.  As before we apply $z\mapsto \bar z$ to one of the packings if necessary to ensure that the geodesic embeddings of $G$ in $\bbC$ and in $\bbH^2 \cong \bbD \subset \bbC$ induced by $\calP$ and $\tilde\calP$ respectively are identified via some orientation-preserving homeomorphism $\bbC \to \bbD$, ensuring that $\calP$ and $\tilde\calP$ have the same orientation.  This time we normalize by applying orientation-preserving Euclidean similarities to $\calP$ so that the following holds:

\begin{normalization}
\label{lem normalize three disks of cp in both planes}
There are disks $D_a, D_b\in \calP$ and $\tilde D_a,\tilde D_b \in \tilde \calP$, where $a,b$ are distinct vertices of the common contact graph $G$ of $\calP$ and $\tilde\calP$, so that the following hold:
\begin{itemize}
\item One of $D_v$ and $\tilde D_v$ is contained in the interior of the other, for all $v = a,b$.
\item The packings $\calP$ and $\tilde\calP$ are in general position.
\end{itemize}
\end{normalization}

\noindent (We give a detailed construction of a stronger normalization in the proof of Theorem \ref{uniformization main}.)

Let $X = (V,E,F)$ be a triangulation of a topological open disk with vertices $V$, edges $E$, and faces $F$, considered only up to its combinatorics, so that the 1-skeleton $(V,E)$ of $X$ is $G$.  We define the interstices $T_f$ and $\tilde T_f$ as before, and again fix $\phi_f : \partial T_f \to \partial \tilde T_f$ having $\eta(\phi_f) \ge 0$.  For every $v\in V$ we again write $\phi_v : \partial D_v \to \partial \tilde D_v$ for the indexable homeomorphism induced by restriction to the $\phi_f$.

Let $(V_0, E_0, F_0) = X_0 \subset X$ be a subtriangulation of $X$, so that $X_0$ is a triangulation of a topological closed disk, and so that $\bbD \subset \bigcup_{v\in V_0} D_v \cup \bigcup_{f\in F_0} T_f$.  Call this total union $K$, and define $\tilde K$ analogously as $\tilde K = \bigcup_{v\in V_0} \tilde D_v \cup \bigcup_{f\in F_0} \tilde T_f$.  Let $\phi_K : \partial K \to \partial \tilde K$ be the indexable homeomorphism induced by restriction to the $\phi_v$.  Then $\eta(\phi_K) = 1$ by the Circle Index Lemma \ref{cil}, because $\tilde K \subset \bbD \subset K$.  On the other hand, by the Index Additivity Lemma \ref{ial}:

\[
1 = \eta(\phi_K) = \sum_{f\in F_0} \eta(\phi_f) + \sum_{v\in V_0} \eta(\phi_v)
\]

\noindent As before, in the first sum every $\eta(\phi_f) \ge 0$ by construction and in the second sum every $\eta(\phi_v) \ge 0$ by the Circle Index Lemma \ref{cil}.  Also the second sum has contributions from $\eta(\phi_a) = 1$ and $\eta(\phi_b) = 1$, again giving us a contradiction as desired.
\end{proof}

\begin{theorem}
\label{libhyp}
\label{rigid in hyp plane}
Suppose that $\calP$ and $\tilde\calP$ are circle packings locally finite in the open unit disk $\bbD$, equivalently the hyperbolic plane $\bbH^2\cong \bbD$, so that $\calP$ and $\tilde\calP$ share a contact graph $G$ which is the 1-skeleton of a triangulation of a topological open disk.  Then $\calP$ and $\tilde\calP$ differ by a hyperbolic isometry, that is, a M\"obius or anti-M\"obius transformation fixing $\bbD \cong \bbH^2$ set-wise.
\end{theorem}

\begin{proof}
As usual, we may suppose that $\calP$ and $\tilde \calP$ have the same orientation.  Proceeding by contradiction, we then normalize by orientation-preserving Euclidean similarities so that the following holds:

\begin{normalization}
\label{lem normalize three disks of cp in hyp plane}
There are disks $D_a,D_b \in \calP$ and $\tilde D_a, \tilde D_b \in \tilde\calP$, where $a,b$ are distinct vertices of the common contact graph $G$ of $\calP$ and $\tilde\calP$, so that the following hold:
\begin{itemize}
\item One of $D_v$ and $\tilde D_v$ is contained in the interior of	the other, for all $v=a,b$.
\item Letting $D$ and $\tilde D$ denote the images of the open unit disk $\bbD$ under the normalizations applied to $\calP$ and $\tilde\calP$ respectively, we have that one of $D$ and $\tilde D$ is contained in the interior of the other.
\item The packings $\calP$ and $\tilde\calP$ are in general position.
\end{itemize}
\end{normalization}

\noindent (We give a detailed construction of a stronger normalization in the proof of Theorem \ref{rigid in hyp plane main}.)

Let $X = (V,E,F)$ be a triangulation of a topological open disk, considered up to its combinatorics, having 1-skeleton $G = (V,E)$.  We define all of $T_f, \tilde T_f, \phi_f, \phi_v$ as before.

Suppose without loss of generality, by interchanging the roles of $\calP$ and $\tilde\calP$ if necessary, that $\tilde D$ is contained in the interior of $D$.  Let $(V_0, E_0, F_0) = X_0 \subset X$ be a subtriangulation of $X$, so that $X_0$ is a triangulation of a topological closed disk, and so that $\tilde D \subset \bigcup_{v\in V_0} D_v \cup \bigcup_{f\in F_0} T_f$.  Call this total union $K$, and define $\tilde K$ analogously as before, again getting $\tilde K \subset \tilde D \subset K$.  We obtain the desired contradiction as in the proof of Theorem \ref{thm:no luv}.
\end{proof}

\begin{theorem}
\label{rigid in complex plane}
\label{libcp}
Suppose that $\calP$ and $\tilde\calP$ are circle packings locally finite in $\bbC$, sharing a contact graph $G$ which is the 1-skeleton of a triangulation of a topological open disk.  Then $\calP$ and $\tilde\calP$ differ by a Euclidean similarity.
\end{theorem}

\begin{proof}
As usual, we may suppose that $\calP$ and $\tilde\calP$ have the same orientation.  We proceed by contradiction, and begin by normalizing $\calP$ and $\tilde\calP$ by M\"obius transformations so that the following holds:

\begin{normalization}
There are disks $D_a,D_b,D_c\in \calP$ and $\tilde D_a,\tilde D_b,\tilde D_c\in \tilde\calP$, where $a,b,c$ are distinct vertices of the common contact graph $G$ of $\calP$ and $\tilde\calP$, so that the following hold:
\begin{itemize}
\item One of $D_v$ and $\tilde D_v$ is contained in the interior of the other, for all $v = a,b,c$.
\item The point $\infty\in \hat\bbC$ lies in the interior of $D_a \cap \tilde D_a$.
\item Letting $z_\infty$ and $\tilde z_\infty$ denote the images of $\infty$ under the normalizations applied to $\calP$ and $\tilde\calP$ respectively, we have that $z_\infty \ne \tilde z_\infty$.
\item The packings $\calP$ and $\tilde\calP$ are in general position.
\end{itemize}
\end{normalization}

\noindent (We give a detailed construction of a stronger normalization in the proof of Theorem \ref{rigid in plane main}.)

We define all of $X=(V,E,F), T_f, \tilde T_f, \phi_f, \phi_v$ as before.  Let $U$ and $\tilde U$ be small disjoint open neighborhoods of $z_\infty$ and $\tilde z_\infty$ respectively, and let $(V_0, E_0, F_0) = X_0 \subset X$ be a sub-triangulation of $X$ so that the following hold:

\begin{itemize}
\item We have that $X_0$ is a triangulation of a topological closed disk.
\item Setting $L = \bigcup_{v\not\in V_0} D_v \cup \bigcup_{f\not \in F_0} T_f$, and defining $\tilde L$ analogously, we have that $L\subset U$ and $\tilde L\subset \tilde U$.
\end{itemize}

\noindent Then the $\phi_v$ induce, via restriction, an indexable homeomorphism $\phi_L : \partial L \to \partial \tilde L$, with $\eta(\phi_L) = 0$ by the Circle Index Lemma \ref{cil}, because $U\supset L$ and $\tilde U\supset \tilde L$ are disjoint.

We orient $\partial D_a$ and $\partial \tilde D_a$ positively with respect to the open disks they bound in $\bbC$, as in the proof of Theorem \ref{rigid in sphere}.  Then by the Index Additivity Lemma \ref{ial}, we have:

\[
1 = \eta(\phi_a) = \sum_{f\in F_0} \eta(\phi_f) + \sum_{v\in V_0} \eta(\phi_v) + \eta(\phi_L)
\]

\noindent The first sum is non-negative and the second sum is at least 2 as usual, and $\eta(\phi_L) = 0$, so we get our desired contradiction.
\end{proof}

\section{Our main technical result, the Index Theorem\twostars}
\label{sec main index}

Let $\calK = \{K_1,\ldots, K_n\}$ and $\tilde \calK = \{\tilde K_1,\ldots,\tilde K_n\}$ be collections of closed Jordan domains.  We denote $\partial \calK = \partial \cup_{i=1}^n K_i$, similarly $\partial \tilde \calK = \cup_{i=1}^n \tilde K_i$.  A homeomorphism $\phi:\partial \calK \to \partial \tilde \calK$ is called \emph{faithful} if whenever we restrict $\phi$ to $K_j \cap \partial \calK$ we get a homeomorphism $K_j \cap \partial \calK \to \tilde K_j \cap \partial \tilde \calK$.  The particular choice of indices of $K_i$ and $\tilde K_i$ is important in determining whether a given homeomorphism is faithful, so we consider the labeling to be part of the information of the collections.  Note that in general $\partial \calK$ and $ \partial \tilde \calK$ need not be homeomorphic, and even if they are homeomorphic there may still be no faithful homeomorphism between them.

We now give a weak form of our main technical result, to illustrate the manner in which we generalize the Circle Index Lemma \ref{cil}.  It may be helpful to recall Definition \ref{def thin disk conf} on p.\ \pageref{def thin disk conf}.

\begin{mainindex*}
Let $\calD = \{ D_1,\ldots,D_n \}$ and $\tilde\calD = \{\tilde D_1,\ldots,\tilde D_n\}$ be finite thin disk configurations in the plane $\bbC$ realizing the same incidence data $(G,\Theta)$.  Let $\phi:\partial \calD \to \partial \tilde \calD$ be a faithful indexable homeomorphism.  Then $\eta(\phi)\ge 0$.
\end{mainindex*}

\noindent This follows from the full statement of the Main Index Theorem \ref{thm:main technical}.  To get rid of the general position hypothesis from the statement of Theorem \ref{thm:main technical}, we need a lemma, which we do not prove in this article, which says that the fixed-point index of a homeomorphism is invariant under a small perturbation of its domain or range, see \citelist{\cite{mishchenko-thesis}*{Lemma 3.3} \cite{MR2131318}*{Lemma 8.11}}.

Note that our current Definition \ref{def fixed-point index} of fixed-point index is not strong enough to accomodate the theorem statement we just gave.  This is because $\cup_{D\in \calD} D$ need not be a closed Jordan domain.  In light of the Index Additivity Lemma \ref{ial}, it is clear how to adapt Definition \ref{def fixed-point index} to suit our needs.  In particular:

\begin{definition}
\label{def fixed point index multiply connected}
Suppose that $K$ is a union of finitely many closed disks in $\bbC$, some of which may intersect.  Suppose also that $\partial K$ is oriented positively with respect to $K$, meaning as usual that the interior of $K$ stays to the left as we traverse $\partial K$ in what we call the positive direction.  Then $\partial K$ decomposes, possibly in more than one way, as a union of finitely many oriented Jordan cuves $\gamma_1,\ldots,\gamma_n$ any two of which meet only at finitely many points.  Some of the $\gamma_i$ will be oriented positively with respect to the finite Jordan domains they bound in $\bbC$, some negatively.  Suppose $\tilde K$ is another finite union of closed disks, with $\partial \tilde K$ similarly decomposing as $\tilde\gamma_1,\ldots,\tilde\gamma_n$, and that $\phi:\partial K \to \partial \tilde K$ is a fixed-point-free orientation-preserving homeomorphism which extends to a homeomorphism $K\to \tilde K$, and which identifies $\gamma_i$ with $\tilde\gamma_i$ for $1\le i \le n$.  Then we define $\eta(\phi) = \sum_{i=1}^n \eta(\gamma_i \mapstop{\phi} \tilde\gamma_i)$.
\end{definition}

\noindent Here we write $\gamma_i\mapstop{\phi} \tilde\gamma_i$ to denote the restriction of $\phi$ to $\gamma_i\to \tilde\gamma_i$.  We will continue to use this notational convention in the future.  We remark that in the definition, the decomposition of $\partial K$ into $\gamma_1,\ldots,\gamma_n$ may not be unique, similarly for $\partial \tilde K$.  We leave it as an exercise for the reader to verify that the same value for the fixed-point index is obtained regardless of which decomposition is chosen.  We remark also that the natural generalization of our Index Additivity Lemma \ref{ial} continues to hold, and leave this as an exercise as well.  Definition \ref{def fixed point index multiply connected} will be general enough to completely accommodate the statement of our Main Index Theorem \ref{thm:main technical}.\medskip

To give the full statement of our Main Index Theorem \ref{mainindex}, we need one more definition:

\begin{definition}
\label{def subsumptive}
Let $\calD = \{D_1,\ldots,D_n\}$ and $\tilde\calD = \{\tilde D_1,\ldots,\tilde D_n\}$ be finite collections of closed disks in the complex plane $\bbC$, sharing a contact graph $G$, with $D_i\in \calD$ corresponding to $\tilde D_i \in \tilde \calD$ for all $1\le i\le n$.  A subset $I\subset \{1,\ldots,n\}$ is called \emph{subsumptive} if
\begin{itemize}
\item either $D_i\subset \tilde D_i$ for every $i\in I$, or $\tilde D_i\subset D_i$ for every $i\in I$, and
\item the set $\cup_{i\in I} D_i$ is connected, equivalently the set $\cup_{i\in I} \tilde D_i$ is connected.
\end{itemize}
Let $I$ be a subsumptive subset of $\{1,\ldots,n\}$.  Then $I$ is called \emph{isolated} if there is no $i\in I$ and $j\in \{1,\ldots,n\}\setminus I$ so that one of $D_i\cap D_j$ and $\tilde D_i \cap \tilde D_j$ contains the other.  The collections $\{D_i\}_{i\in I}$ and $\{\tilde D_i\}_{i\in I}$ together are called a \emph{pair of subsumptive subconfigurations} of $\calD$ and $\tilde\calD$.  The pair is called \emph{isolated} if $I$ is isolated.
\end{definition}

The main technical result of this article is the following theorem:

\begin{mainindex}
\label{mainindex}
\label{thm:main technical}
Let $\calD = \{ D_1,\ldots,D_n \}$ and $\tilde\calD = \{\tilde D_1,\ldots,\tilde D_n\}$ be finite thin disk configurations in the complex plane $\bbC$, in general position, realizing the same incidence data $(G,\Theta)$, with $D_i\in \calD$ corresponding to $\tilde D_i \in \tilde \calD$ for all $1\le i\le n$.  Let $\phi:\partial \calD \to \partial \tilde \calD$ be a faithful indexable homeomorphism.  Then $\eta(\phi)$ is at least the number of maximal isolated subsumptive subsets of $\{1,\ldots,n\}$.  In particular $\eta(\phi) \ge 0$.
\end{mainindex}

\noindent For an example, look ahead to Figure \ref{fig:graph example} on p.\ \pageref{fig:graph example}.  There we know that $\eta(\phi) \ge 1$ for $\phi$ satisfying the hypotheses of Theorem \ref{thm:main technical}.  We discuss possible generalizations of our Main Index Theorem \ref{mainindex} at the end of Section \ref{sec conjectures}.\medskip

We will now prove Theorem \ref{thm:main technical}, modulo four propositions.  We give the complete statements of these propositions in the running text of the proof, and number them according to where they are found with their own proofs in this article.

\begin{proof}[Proof of Theorem \ref{thm:main technical}]
We first need to make some preliminary definitions and observations.  We say that two closed disks \emph{overlap} if their interiors meet.  Suppose that $D_i\ne  D_j$ overlap.  Then the \emph{eye} between them is defined to be $E_{ij} = E_{ji} = D_i\cap D_j$.  When we quantify over the eyes $E_{ij}$ of $\calD$, we keep in mind that $E_{ij} = E_{ji}$ and treat this as a single case.  The \emph{eyes} of $\tilde \calD$ are defined analogously.  A homeomorphism $\epsilon_{ij}:\partial E_{ij}\to \partial \tilde E_{ij}$ is called \emph{faithful} if it restricts to homeomorphisms $D_i \cap \partial E_{ij} \to \tilde D_i \cap \partial \tilde E_{ij}$ and $D_j \cap \partial E_{ij} \to \tilde D_j \cap \partial \tilde E_{ij}$.

We first note that for every eye $E_{ij}$ there exists a faithful indexable homeomorphism $\epsilon_{ij}:\partial E_{ij} \to \partial \tilde E_{ij}$.  The only way that there could fail to exist any faithful fixed-point-free homeomorphisms $\partial E_{ij} \to \partial \tilde E_{ij}$ is if a pair of corresponding points in $\partial D_i \cap \partial D_j$ and $\partial \tilde D_i \cap \partial \tilde D_j$ coincide, which cannot happen by the general position hypothesis on $\calD$ and $\tilde\calD$.  Furthermore, however they are chosen, the homeomorphisms $\epsilon_{ij}$ agree with one another because their domains are disjoint, and every $\epsilon_{ij}$ agrees with $\phi$ on $\partial E_{ij} \cap \partial \calD$ because of the faithfulness conditions on $\phi$ and on the $\epsilon_{ij}$.

For every $E_{ij}$ pick a faithful indexable $\epsilon_{ij}$.  For $i\in \{1,\ldots,n\}$ let $\delta_i:\partial D_i\to \partial \tilde D_i$ be the function induced by restricting to $\phi$ or to the $\epsilon_{ij}$, as necessary.  It is routine to check that $\delta_i$ defined this way is an indexable homeomorphism.  The following observation serves as a good intuition builder, and will be appealed to later in our proof:

\begin{observation}
\label{obsA}
$\displaystyle
\eta(\phi) = \sum_{i=1}^n \eta(\delta_i) - \sum_{E_{ij}} \eta(\epsilon_{ij})$
\end{observation} 

\noindent The second sum is taken over all eyes $E_{ij}$ of $\calD$.  This observation follows from the Index Additivity Lemma \ref{lem:indices add}: notice that $\eta(\epsilon_{ij})$ is exactly double-counted in the sum $\eta(\delta_i) + \eta(\delta_j)$.  We remark now that as we hope to prove in particular that $\eta(\phi) \ge 0$, one of our main strategies will be to try to get $\epsilon_{ij}$ so that $\eta(\epsilon_{ij}) = 0$.  Recall that we always have $\eta(\delta_i) \ge 0$ by the Circle Index Lemma \ref{cil}.

If $I\subset \{1,\ldots,n\}$ then let $\calD_I = \{D_i : i\in I\}$, similarly $\tilde\calD_I$.  We denote by $\phi_I:\partial \calD_I\to \partial\tilde \calD_I$ the function obtained by restriction to $\phi$ or to the $\epsilon_{ij}$, as necessary.  Then $\phi_I$ is a faithful indexable homeomorphism.  We make another observation.

\begin{observation}
\label{mainB}
Let $I,J\subset \{1,\ldots,n\}$ be disjoint and non-empty, satisfying $I\sqcup J = \{1,\ldots,n\}$.  Then by the Index Additivity Lemma \ref{ial} we get
\[
\eta(\phi) = \eta(\phi_I) + \eta(\phi_J) - \sum \eta(\epsilon_{ij})
\]
where the sum is taken over all $E_{ij}$ so that $i\in I, j\in J$.
\end{observation}
\medskip

We now proceed to the main portion of our proof.  The proof is by induction on $n$ the number of disks in each of our configurations $\calD$ and $\tilde\calD$.  The base case $n=1$ follows from the Circle Index Lemma \ref{cil}, so we suppose from now on that $n\ge 2$.  We begin with a simplifying observation that gives us access to our main propositions:

\begin{observation}
\label{mainA}
Suppose that $D_j\setminus \cup_{i\ne j} D_i =: d_j$ and $\tilde D_j\setminus \cup_{i\ne j} \tilde D_i =: \tilde d_j$ are disjoint for some $j$.  Then we are done by induction.
\end{observation}

\noindent To see why, observe the following.  First, if neither of $D_j$ and $\tilde D_j$ contains the other, then $j$ does not belong to any subsumptive subset of $\{1,\ldots,n\}$, so letting $I = \{1,\ldots,n\}\setminus \{j\}$, we observe that the lower bound we wish to prove on $\eta(\phi)$ is the same as the lower bound we get on $\eta(\phi_I)$ by our induction hypothesis.  Then by the Index Additivity Lemma \ref{ial}, we get $\eta(\phi) = \eta(\phi_I) + \eta(\partial d_j \mapstop{\phi, \epsilon_{ij}} \partial \tilde d_j ) = \eta(\phi_I)$.  Here $\partial d_j \mapstop{\phi, \epsilon_{ij}} \partial \tilde d_j$ denotes the indexable homeomorphism induced by restriction to $\phi$ and to the $\epsilon_{ij}$, as necessary.  The fixed-point index of this homeomorphism is $0$ because $d_j$ and $\tilde d_j$ are disjoint.

On the other hand, suppose that one of $D_j$ and $\tilde D_j$ contains the other.  We will be done by the same argument if we can show that the number of maximal isolated subsumptive subsets of $\{1,\ldots,n\}$ is the same as the number of maximal isolated subsumptive subsets of $\{1,\ldots,n\} \setminus \{j\}$.  Suppose without loss of generality that $\tilde D_j\subset D_j$.  Because $d_j$ and $\tilde d_j$ are disjoint, it follows that there must be an $i\ne j$ so that $\tilde D_j\subset D_i$.  It is also not hard to see that if $E_{jk}$ and $\tilde E_{jk}$ are eyes one of which contains the other, then we must have $k = i$ and $\tilde E_{ij} \subset E_{ij}$.  Let $J$ be the maximal subsumptive subset of $\{1,\ldots,n\}$ containing $j$.  If $\tilde D_i\not\subset D_i$, then $i\not\in J$, but $\tilde E_{ij}\subset E_{ij}$, so $J$ is not isolated.  In fact $J = \{j\}$, so we are done by induction.  To see why note that if $k\in J$ is different from $j$ then $\tilde D_k\subset D_k$, so $\tilde E_{jk} \subset E_{jk}$, so $k = i$ by the earlier discussion, contradicting $\tilde D_i\not\subset D_i$.  So, finally, suppose that $\tilde D_i\subset D_i$, so $i\in J$.  If $J$ fails to be isolated, then it does so at a $k\in J$ different from $j$, and $J\setminus \{j\}$ is a maximal subsumptive subset of $\{1,\ldots,n\} \setminus \{j\}$, both by the preceding argument.  Thus $J\setminus \{j\}$ is isolated in $\{1,\ldots,n\} \setminus \{j\}$ if and only if $J$ is isolated in $\{1,\ldots,n\}$.  An element of $\{1,\ldots,n\}$ may belong to at most one maximal subsumptive subset of $\{1,\ldots,n\}$, so once again we are done by induction.  This completes the proof of Observation \ref{mainA}.\medskip

We therefore assume without loss of generality via Observation \ref{mainA} for the remainder of the proof the weaker statement that $D_j\setminus D_i$ and $\tilde D_j\setminus \tilde D_i$ meet, for all $i,j$.

The following proposition will be key in our induction step:

\renewcommand{\thepropositionfree}{\ref{prop:nontrivial eye int then done}}
\begin{propositionfree}
\label{prop:nontrivial eye int then done a}
Let $\{A,B\}$ and $\{\tilde A,\tilde B\}$ be pairs of overlapping closed disks in the complex plane $\bbC$, in general position.  Suppose that neither of $E = A\cap B$ and $\tilde E = \tilde A \cap \tilde B$ contains the other.  Suppose further that $A\setminus B$ and $\tilde A \setminus \tilde B$ meet, and that $B\setminus A$ and $\tilde B\setminus \tilde A$ meet.  Then there is a faithful indexable homeomorphism $\epsilon:\partial E\to\partial \tilde E$ satisfying $\eta(\epsilon) = 0$.
\end{propositionfree}

For the remainder of the proof, suppose that for every eye $E_{ij}$ of $\calD$, we have chosen our faithful indexable $\epsilon_{ij}$ so that $\eta(\epsilon_{ij}) = 0$ whenever neither of $E_{ij}$ and $\tilde E_{ij}$ contains the other, and necessarily so that $\eta(\epsilon_{ij}) = 1$ otherwise.  Then for example if for no $i,j$ is it the case that one of $E_{ij}$ and $\tilde E_{ij}$ contains the other, then we are done by Observation \ref{obsA}.  Alternatively, if there exist disjoint non-empty $I,J\subset \{1,\ldots,n\}$ so that $I\sqcup J = \{1,\ldots,n\}$, and so that for every $i\in I, j\in J$ we have that neither of $E_{ij}$ and $\tilde E_{ij}$ contains the other, then we are done by induction and Observation \ref{mainB}.\medskip

Our next key proposition is the following:

\renewcommand{\thepropositionfree}{\ref{prop:black box 2}}
\begin{propositionfree}
\label{prop:black box 2 a}
Let $\calD = \{ D_1,\ldots,D_n \}$ and $\tilde\calD = \{\tilde D_1,\ldots,\tilde D_n\}$ be as in the statement of Theorem \ref{thm:main technical}.  Let $I$ be a maximal nonempty subsumptive subset of $\{1,\ldots,n\}$.  Then there is at most one pair $i\in I$, $j\in \{1,\ldots,n\}\setminus I$ so that one of $E_{ij} = D_i\cap D_j$ and $\tilde E_{ij} = \tilde D_i\cap \tilde D_j$ contains the other.
\end{propositionfree}

\noindent This proposition says that maximal subsumptive configurations are always at least ``almost'' isolated, and, together with Proposition \ref{prop:no contained loops}, will allow us to excise maximal subsumptive configurations from $\calD$ and $\tilde\calD$ in the style of Observation \ref{mainB}, to complete our proof by induction.  We explain how in more detail shortly.\noindent

There is a potential problem: we would like to say that if $I\subset \{1,\ldots,n\}$ is subsumptive, implying that one of $\cup_{i\in I} D_i$ and $\cup_{i\in I}\tilde D_i$ contains the other, then $\eta(\phi_I) = 1$.  However, \emph{a priori}, this may fail, for example see Figure \ref{fig:contained loop}.  Our next proposition addresses this issue:

\renewcommand{\thepropositionfree}{\ref{prop:no contained loops}}
\begin{propositionfree}
\label{prop:no contained loops a}
Let $n\ge 3$ be an integer.  Let $\{D_i : i \in \bbZ/n\bbZ\}$ and $\{\tilde D_i : i\in \bbZ/n\bbZ\}$ be thin collections of closed disks in the plane $\bbC$, in general position, so that the following conditions hold.
\begin{itemize}
\item We have that $\tilde D_i$ is contained in the interior of $D_i$ for all $i$.
\item The disk $D_i$ overlaps with $D_{i\pm 1}$, and the disk $\tilde D_i$ overlaps with $\tilde D_{i\pm 1}$, for all $i$.
\item If $D_i$ and $D_j$ meet, then $i=j\pm 1$.
\end{itemize}
Then $\sum_{i\in \bbZ/n\bbZ} \aext(\tilde D_i,\tilde D_{i+1}) < \sum_{i\in \bbZ/n\bbZ} \aext(D_i,D_{i+1})$.  In particular, for some $i$ we must have $\aext(D_i,D_{i+1}) \ne \aext(\tilde D_i,\tilde D_{i+1})$.
\end{propositionfree}

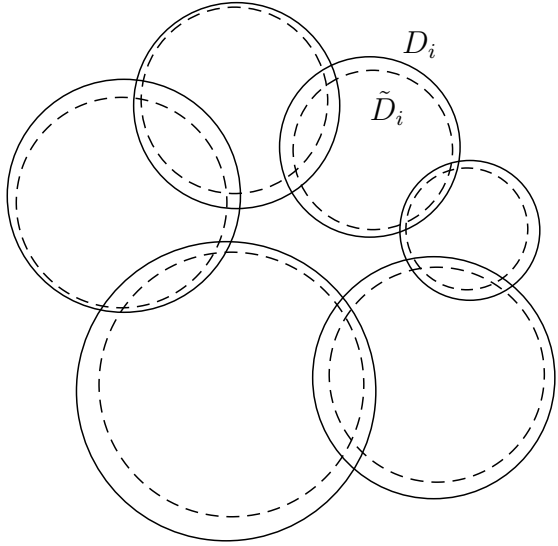
\begin{figure}[t]
\centering
% Generated with LaTeXDraw 2.0.8
% Mon Dec 05 17:25:27 EST 2011
% \usepackage[usenames,dvipsnames]{pstricks}
% \usepackage{epsfig}
% \usepackage{pst-grad} % For gradients
% \usepackage{pst-plot} % For axes
\scalebox{1} % Change this value to rescale the drawing.
{
\begin{pspicture}(0,-3.59)(7.3,3.59)
\pscircle[linewidth=0.02,dimen=outer](1.56,1.01){1.56}
\pscircle[linewidth=0.02,dimen=outer](3.06,2.21){1.38}
\pscircle[linewidth=0.02,dimen=outer](4.83,1.66){1.21}
\pscircle[linewidth=0.02,dimen=outer](6.16,0.55){0.94}
\pscircle[linewidth=0.02,dimen=outer](5.68,-1.41){1.62}
\pscircle[linewidth=0.02,dimen=outer](2.92,-1.59){2.0}
\pscircle[linewidth=0.02,linestyle=dashed](1.53,0.92){1.41}
\pscircle[linewidth=0.02,linestyle=dashed](3.03,2.28){1.25}
\pscircle[linewidth=0.02,linestyle=dashed](4.87,1.62){1.07}
\pscircle[linewidth=0.02,linestyle=dashed](6.12,0.57){0.82}
\pscircle[linewidth=0.02,linestyle=dashed](5.72,-1.37){1.44}
\pscircle[linewidth=0.02,linestyle=dashed](2.99,-1.5){1.77}
\usefont{T1}{ptm}{m}{n}
\rput(5.494531,3.02){$D_i$}
\usefont{T1}{ptm}{m}{n}
\rput(5.0645313,2.18){$\tilde D_i$}
\end{pspicture} 
}
\caption
{
\label{fig:contained loop}
Two closed chains of disks with $\tilde D_i\subsetneq D_i$ for all $i$.  The $D_i$ are drawn solid and the $\tilde D_i$ dashed.  Proposition \ref{prop:no contained loops} implies that $\aext(\tilde D_i,\tilde D_{i+1}) \ne \aext(D_i,D_{i+1})$ for some $i$.
}
\end{figure}

Thus suppose that $I$ is a maximal nonempty subsumptive subset of $\{1,\ldots,n\}$.  Then by Proposition \ref{prop:no contained loops} and the Circle Index Lemma \ref{cil}, we have that $\cup_{i\in I} D_i$ and $\cup_{i\in I}\tilde D_i$ are closed Jordan domains, so $\eta(\phi_I) = 1$.  Let $J = \{1,\ldots,n\} \setminus I$.  If $I$ is isolated in $\{1,\ldots,n\}$, then every $\epsilon_{ij} = 0$, with $i\in I$ and $j\in J$.  Also $J$ has one fewer maximal isolated subsumptive subset than does $\{1,\ldots,n\}$.  Thus we are done by induction and Observation \ref{mainB}.  On the other hand, suppose that $I$ is not isolated.  Then it is not hard to see that $J$ has as many maximal isolated subsumptive subsets as does $\{1,\ldots,n\}$.  Also, by Proposition \ref{prop:black box 2}, there is exactly one eye $E_{ij}$ with $i\in I$ and $j\in J$ so that $\eta(\epsilon_{ij}) = 1$, and for all the others we have $\eta(\epsilon_{ij}) = 0$.  Again, we are done by induction and Observation \ref{mainB}.\medskip

We now state our final key proposition in the proof of Theorem \ref{mainindex}:

\renewcommand{\thepropositionfree}{\ref{prop:black box 3}}
\begin{propositionfree}
\label{prop:black box 3 a}
Let $\calD = \{ D_1,\ldots,D_n \}$ and $\tilde\calD = \{\tilde D_1,\ldots,\tilde D_n\}$ be as in the statement of Theorem \ref{thm:main technical}, and so that for all $i,j$ the sets $D_i\setminus D_j$ and $\tilde D_i\setminus \tilde D_j$ meet.  Suppose that there is no $i$ so that one of $D_i$ and $\tilde D_i$ contains the other.  Suppose that for every pair of disjoint non-empty subsets $I,J\subset \{1,\ldots,n\}$ so that $I\sqcup J = \{1,\ldots,n\}$, there exists an eye $E_{ij}$, with $i\in I$ and $j\in J$, so that one of $E_{ij}$ and $\tilde E_{ij}$ contains the other.  Then for every $i$ we have that any faithful indexable homeomorphism $\delta_i:\partial D_i \to \partial \tilde D_i$ satisfies $\eta(\delta_i) \ge 1$.  Furthermore there is a $k$ so that $D_i$ and $D_k$ overlap for all $i$, and so that one of $E_{ij}$ and $\tilde E_{ij}$ contains the other if and only if either $i=k$ or $j=k$.
\end{propositionfree}

\noindent Unless one of our earlier propositions has already finished off the proof of Theorem \ref{mainindex} by induction, the hypotheses of Proposition \ref{prop:black box 3 a} hold, and we are done by Observation \ref{obsA}.
\end{proof}

We need to establish Propositions \ref{prop:no contained loops}, \ref{prop:black box 2}, \ref{prop:nontrivial eye int then done}, and \ref{prop:black box 3}.  We establish Propositions \ref{prop:no contained loops} and \ref{prop:black box 2} next, in Section \ref{sec:subsume}.  Their proofs are quick and elementary, and some ingredients of their proofs are used in the proofs of our main rigidity theorems.  We then prove our main rigidity theorems in Section \ref{sec:rigidity proof}, using Theorem \ref{mainindex}.  The proofs of Propositions \ref{prop:nontrivial eye int then done} and \ref{prop:black box 3} take up most of the rest of the article.

\section{Subsumptive collections of disks\twostars}
\label{sec:subsume}
\label{sec geom lems}

In this section we prove some lemmas, and Propositions \ref{prop:no contained loops} and \ref{prop:black box 2}, having to do with subsumptive configurations of disks.\medskip

First, we establish some geometric facts, starting with the following important observation, which is illustrated in Figure \ref{fig:observation angles drawing}:

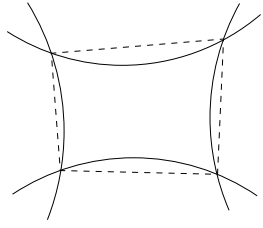
\begin{figure}[t]
\centering
% Generated with LaTeXDraw 2.0.8
% Sat Apr 21 15:35:05 EDT 2012
% \usepackage[usenames,dvipsnames]{pstricks}
% \usepackage{epsfig}
% \usepackage{pst-grad} % For gradients
% \usepackage{pst-plot} % For axes
\scalebox{.5} % Change this value to rescale the drawing.
{
\begin{pspicture}(8,-2.7389712)(18.225996,3.0760415)
\psarc[linewidth=0.02](14.83,-6.9){5.83}{55.72038}{123.52237}
\psarc[linewidth=0.02](14.53,7.06){5.67}{237.30272}{310.41534}
\psarc[linewidth=0.02](23.11,0.0){6.25}{150.94539}{203.25116}
\psarc[linewidth=0.02](6.49,-0.36){6.49}{338.68607}{31.75948}
\pspolygon[linewidth=0.02,linestyle=dashed,dash=0.16cm 0.16cm](12.64,1.71)(17.22,2.09)(17.06,-1.51)(12.88,-1.41)
\end{pspicture}
}
\caption
{
\label{fig:observation angles drawing}
A complementary component of the union of four disks as in Observation \ref{obs:four disks interiorwise disjoint}.  The sum of the angles inside of the dashed honest quadrilateral is exactly $2\pi$.  This sum is greater than the sum of the external intersection angles of the disks.
}
\end{figure}

\begin{observation}
\label{obs:four disks interiorwise disjoint}
Suppose that $D_1,D_2,D_3,D_4$ are metric closed disks in $\bbC$, so that there is a bounded connected component $U$ of $\bbC\setminus \cup_{i=1}^4 D_i$ which is a curvilinear quadrilateral, whose boundary $\partial U$ decomposes as the union of four circular arcs, one taken from each of $\partial D_1, \partial D_2, \partial D_3, \partial D_4$.  Suppose that as we traverse $\partial U$ positively, we arrive at $\partial D_1,\allowbreak\partial D_2,\allowbreak\partial D_3,\allowbreak\partial D_4$ in that order.  Then $\sum_{i=1}^4 \aext(D_i,D_{i+1}) < 2\pi$, where we consider $D_5 = D_1$.
\end{observation}

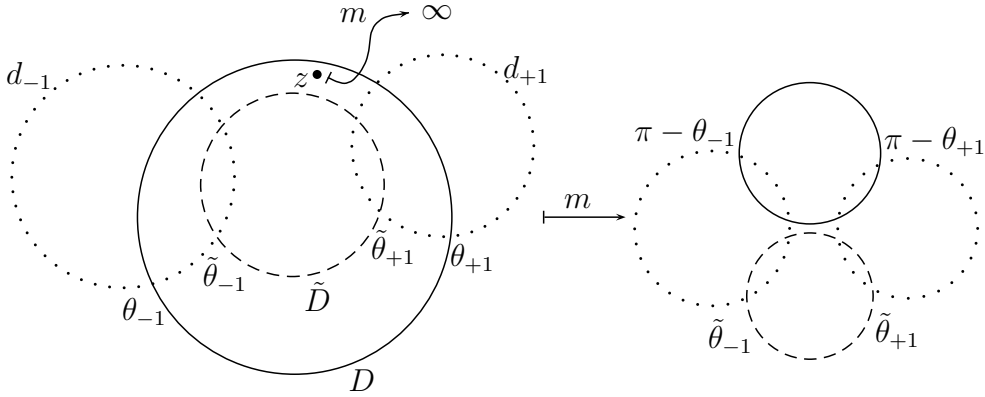
\begin{figure}[t]
\centering
% Generated with LaTeXDraw 2.0.8
% Sat Apr 21 15:57:28 EDT 2012
% \usepackage[usenames,dvipsnames]{pstricks}
% \usepackage{epsfig}
% \usepackage{pst-grad} % For gradients
% \usepackage{pst-plot} % For axes
\scalebox{1} % Change this value to rescale the drawing.
{
\begin{pspicture}(0,-2.653125)(14.569062,2.653125)
\pscircle[linewidth=0.02,dimen=outer](4.29,-0.2803125){2.1}
\pscircle[linewidth=0.02,linestyle=dashed](4.26,0.1496875){1.23}
\pscircle[linewidth=0.04,linestyle=dotted,dotsep=0.16cm,dimen=outer](2.01,0.2596875){1.5}
\pscircle[linewidth=0.04,linestyle=dotted,dotsep=0.16cm,dimen=outer](6.26,0.6696875){1.23}
\pscircle[linewidth=0.02,dimen=outer](11.14,0.5696875){0.95}
\pscircle[linewidth=0.02,linestyle=dashed](11.14,-1.3303125){0.85}
\pscircle[linewidth=0.04,linestyle=dotted,dotsep=0.16cm,dimen=outer](9.84,-0.4303125){1.05}
\pscircle[linewidth=0.04,linestyle=dotted,dotsep=0.16cm,dimen=outer](12.44,-0.4303125){0.95}
\psline[linewidth=0.02cm,tbarsize=0.07055555cm 5.0,arrowsize=0.05291667cm 2.0,arrowlength=1.4,arrowinset=0.4]{|->}(7.59,-0.2803125)(8.69,-0.2803125)
\usefont{T1}{ptm}{m}{n}
\rput(5.074531,2.4096875){$m$}
\usefont{T1}{ptm}{m}{n}
\rput(8.054531,-0.0703125){$m$}
\usefont{T1}{ptm}{m}{n}
\rput(6.184531,2.4496875){$\infty$}
\usefont{T1}{ptm}{m}{n}
\rput(5.184531,-2.4503126){$D$}
\usefont{T1}{ptm}{m}{n}
\rput(4.594531,-1.3303125){$\tilde D$}
\usefont{T1}{ptm}{m}{n}
\rput(0.76453125,1.5896875){$d_{-1}$}
\usefont{T1}{ptm}{m}{n}
\rput(7.364531,1.6696875){$d_{+1}$}
\usefont{T1}{ptm}{m}{n}
\rput(2.2945313,-1.5503125){$\theta_{-1}$}
\usefont{T1}{ptm}{m}{n}
\rput(6.6545315,-0.8303125){$\theta_{+1}$}
\usefont{T1}{ptm}{m}{n}
\rput(3.3645313,-1.0503125){$\tilde \theta_{-1}$}
\usefont{T1}{ptm}{m}{n}
\rput(5.6045313,-0.6903125){$\tilde\theta_{+1}$}
\usefont{T1}{ptm}{m}{n}
\rput(4.3745313,1.5496875){$z$}
\psdots[dotsize=0.12](4.59,1.6196876)
\usefont{T1}{ptm}{m}{n}
\rput(9.484531,0.8296875){$\pi - \theta_{-1}$}
\usefont{T1}{ptm}{m}{n}
\rput(12.804531,0.7296875){$\pi - \theta_{+1}$}
\usefont{T1}{ptm}{m}{n}
\rput(10.084531,-1.8703125){$\tilde\theta_{-1}$}
\usefont{T1}{ptm}{m}{n}
\rput(12.304531,-1.7903125){$\tilde\theta_{+1}$}
\pscustom[linewidth=0.02]
{
\newpath
\moveto(4.73,1.5796875)
\lineto(4.81,1.5396875)
\curveto(4.85,1.5196875)(4.92,1.4996876)(4.95,1.4996876)
\curveto(4.98,1.4996876)(5.035,1.5196875)(5.06,1.5396875)
\curveto(5.085,1.5596875)(5.135,1.6146874)(5.16,1.6496875)
\curveto(5.185,1.6846875)(5.225,1.7596875)(5.24,1.7996875)
\curveto(5.255,1.8396875)(5.275,1.9246875)(5.28,1.9696875)
\curveto(5.285,2.0146875)(5.295,2.0996876)(5.3,2.1396875)
\curveto(5.305,2.1796875)(5.335,2.2496874)(5.36,2.2796874)
\curveto(5.385,2.3096876)(5.44,2.3546875)(5.47,2.3696876)
\curveto(5.5,2.3846874)(5.575,2.4046874)(5.62,2.4096875)
\curveto(5.665,2.4146874)(5.735,2.4246874)(5.81,2.4396875)
}
\psline[linewidth=0.02,linestyle=none]{|-}(4.73,1.5796875)(4.81,1.5396875)
\psline[linewidth=0.02,linestyle=none]{->}(5.665,2.4146874)(5.81,2.4396875)
\end{pspicture} 
}
\caption
{
\label{fig:mobius to infinity}
A M\"obius transformation chosen to prove Lemma \ref{lem:meat}.
}
\end{figure}

We use Observation \ref{obs:four disks interiorwise disjoint} to prove the following key lemma, illustrated in Figure \ref{fig:mobius to infinity}:

\begin{lemma}
\label{lem:meat}
Let $d_{-1}, d_{+1}, D, \tilde D$ be closed disks in $\bbC$, so that $\tilde D$ is contained in the interior of $D$, so that both of $D$ and $\tilde D$ meet both of $d_{-1}$ and $d_{+1}$, and so that $d_{-1} \cap d_{+1} \cap D$ is empty.  Suppose that neither of $d_{-1}$ and $d_{+1}$ is contained in $D$.  We denote $\theta_{-1} = \aext(D, d_{-1})$ and $\tilde \theta_{-1} = \aext(\tilde D, d_{-1})$, defining $\theta_{+1}$ and $\tilde \theta_{+1}$ analogously.  Then $ \tilde\theta_{-1} + \tilde\theta_{+1} < \theta_{-1} + \theta_{+1}$.
\end{lemma}

\begin{proof}
Suppose first that $d_{-1}$ and $d_{+1}$ are disjoint, as in Figure \ref{fig:mobius to infinity}.  Let $z$ be a point in the interior of $D\setminus (\tilde D\cup d_{-1} \cup d_{+1})$, and let $m$ be a M\"obius transformation sending $z$ to $\infty$.  Then $m$ inverts the disk $D$ but none of the disks $\tilde D, d_{-1}, d_{+1}$.  Because $m$ preserves angles we get $(\pi - \theta_{-1}) + (\pi - \theta_{+1}) + \tilde \theta _{-1} + \tilde\theta_{+1} < 2\pi$ by Observation \ref{obs:four disks interiorwise disjoint}, and the desired inequality follows.  The case where $d_{-1}$ and $d_{+1}$ meet outside of $D$ is proved identically.
\end{proof}

The following follows as a corollary of Lemma \ref{lem:meat}, by applying a suitable M\"obius transformation:

\begin{lemma}
\label{lem:finlandia}
Let $\{A,B\}$ and $\{\tilde A,\tilde B\}$ be pairs of overlapping closed disks in the plane $\bbC$, in general position, so that $\aext(A,B) = \aext(\tilde A,\tilde B)$.  Suppose that $\tilde A$ is contained in the interior of $A$ and that $\tilde B$ is contained in the interior of $B$.  Suppose also that neither $\tilde A\subset B$ nor $\tilde B\subset A$.  Then $2\aext(A,B) = 2\aext(\tilde A,\tilde B) < \aext(\tilde A, B) + \aext(A, \tilde B)$.
\end{lemma}

\noindent In particular, it works to apply a M\"obius transformation sending a point in the interior of $B\setminus (A \cup \tilde A\cup \tilde B)$ to $\infty$.\medskip

We proceed to our final preliminary geometric lemma, illustrated in Figure \ref{fig:mogwai}:

\begin{figure}[t]
\centering
% Generated with LaTeXDraw 2.0.8
% Sun Jul 15 09:05:01 EDT 2012
% \usepackage[usenames,dvipsnames]{pstricks}
% \usepackage{epsfig}
% \usepackage{pst-grad} % For gradients
% \usepackage{pst-plot} % For axes
\scalebox{1} % Change this value to rescale the drawing.
{
\begin{pspicture}(0,-1.97)(9.5,1.945)
\definecolor{color5245b}{rgb}{0.8,0.8,0.8}
\pscircle[linewidth=0.03,dimen=outer](2.08,-0.01){1.16}
\pscircle[linewidth=0.03,dimen=outer](3.41,-0.04){1.21}
\psarc[linewidth=0.03,fillcolor=color5245b](3.64,0.07){1.0}{120.96375}{248.1986}
\psarc[linewidth=0.018,linestyle=dashed,dash=0.16cm 0.16cm,fillcolor=color5245b](3.64,0.07){1.0}{105.9454}{272.2026}
\psarc[linewidth=0.02,linestyle=dotted,dotsep=0.16cm,fillcolor=color5245b](3.64,0.07){1.0}{94.57392}{285.06848}
\psdots[dotsize=0.1](0.96,-0.27)
\psbezier[linewidth=0.03,fillcolor=color5245b]{|->}(0.8,-0.31)(0.42,-0.41)(0.16,-0.55)(0.08,-0.71)(0.0,-0.87)(0.16,-1.09)(0.52,-1.29)
%\psline[linestyle=none,linewidth=0.04cm]{|-}(0.8,-0.31)(0.42,-0.41)
%\psline[linestyle=none,linewidth=0.05cm]{->}(0.16,-1.09)(0.52,-1.29)
\psline[linewidth=0.03cm,fillcolor=color5245b,tbarsize=0.07055555cm 5.0,arrowsize=0.05291667cm 2.0,arrowlength=1.4,arrowinset=0.4]{|*->}(5.3,0.13)(6.3,0.13)
\psline[linewidth=0.03cm,fillcolor=color5245b](8.3,1.93)(8.3,-1.87)
\pscircle[linewidth=0.03,dimen=outer](8.1,0.13){1.4}
\psbezier[linewidth=0.02,fillcolor=color5245b]{<-}(3.26,-0.91)(3.2,-1.25)(3.46,-1.39)(3.34,-1.55)(3.22,-1.71)(3.04,-1.55)(2.8,-1.63)
%\psline[linestyle=none,linewidth=0.04cm]{<-}(3.26,-0.91)(3.2,-1.25)
\psarc[linewidth=0.03,fillcolor=color5245b](8.9,0.17){1.0}{107.525566}{248.1986}
\psarc[linewidth=0.018,linestyle=dashed,dash=0.16cm 0.16cm,fillcolor=color5245b](8.9,0.17){1.0}{105.9454}{272.2026}
\psarc[linewidth=0.02,linestyle=dotted,dotsep=0.16cm,fillcolor=color5245b](8.9,0.17){1.0}{94.2364}{276.58194}
\pscircle[linewidth=0.0139999995,linestyle=dashed](8.18,0.17){1.0}
\usefont{T1}{ptm}{m}{n}
\rput(1.11,1.115){$A$}
\usefont{T1}{ptm}{m}{n}
\rput(4.48,1.035){$B$}
\usefont{T1}{ptm}{m}{n}
\rput(2.52,-1.765){$C$}
\usefont{T1}{ptm}{m}{n}
\rput(0.72,-0.045){$z$}
\usefont{T1}{ptm}{m}{n}
\rput(5.87,-1.345){$C'$}
\usefont{T1}{ptm}{m}{n}
\rput(7.95,-1.785){$A$}
\usefont{T1}{ptm}{m}{n}
\rput(9.0,-1.265){$B$}
\usefont{T1}{ptm}{m}{n}
\rput(6.22,1.315){$C$}
\psbezier[linewidth=0.02,fillcolor=color5245b]{<-}(7.98,0.67)(7.64,1.11)(7.8755884,1.65)(7.5205884,1.69)(7.1655884,1.73)(7.082059,1.29)(6.56,1.29)
%\psline[linestyle=none,linewidth=0.04cm]{<-}(7.98,0.67)(7.64,1.11)
\psbezier[linewidth=0.02,fillcolor=color5245b]{<-}(7.44,-0.57)(7.06,-0.91)(7.28,-1.33)(6.96,-1.37)(6.64,-1.41)(6.54,-1.25)(6.14,-1.31)
%\psline[linestyle=none,linewidth=0.04cm]{<-}(7.44,-0.57)(7.06,-0.91)
\usefont{T1}{ptm}{m}{n}
\rput(0.77,-1.465){$\infty$}
\end{pspicture} 
}
\caption
{
\label{fig:mogwai}
A M\"obius transformation chosen to prove Lemma \ref{lem:mogwai}.
}
\end{figure}
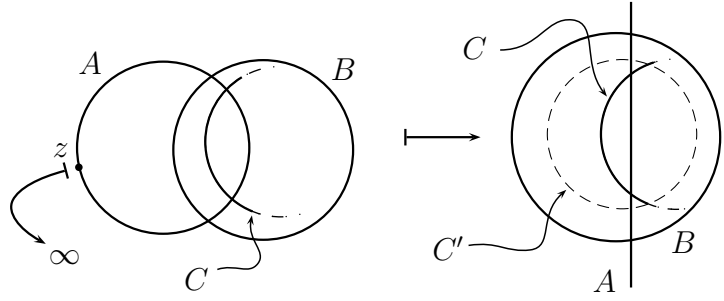

\begin{lemma}
\label{lem:mogwai}
Let $A$, $B$, $C$ be closed disks, none of which is contained in any other.  Suppose that $A$ and $C$ overlap, with $A\cap C\subset B$.  Then $A$ and $B$ overlap, with $\aext(A,C) < \aext(A,B)$.
\end{lemma}

\begin{proof}
Let $z\in \partial A \setminus B$.  Because of the hypothesis that $A\cap C\subset B$, we have that $z\not \in C$.  Apply a M\"obius transformation sending $z\mapsto \infty$ so that $A$ becomes the left half-plane.  Because $z\not \in B,C$ we have that $B$ and $C$ remain closed disks after this transformation.  Let $C'$ be the closed disk so that $\aext(A,C') = \aext(A,B)$, and so that $C$ and $C'$ have the same Euclidean radius and the same vertical Euclidean coordinate.  Then $C'\subset A$.  Also, notice that $C$ is obtained from $C'$ by a translation to the right or to the left.  In fact it must be a translation to the right, because the points $\partial B \cap \partial C$ must lie in the complement of $A$, which is the right-half plane.  But we see that $\aext(A,C')$ is monotone decreasing as $C'$ slides to the right.
\end{proof}
\medskip

We now proceed to the proofs of Propositions \ref{prop:no contained loops} and \ref{prop:black box 2}.  We restate them here for the convenience of the reader.  Our first proposition was illustrated in Figure \ref{fig:contained loop} on p.\ \pageref{fig:contained loop}:

\begin{proposition}
\label{prop:no contained loops}
Let $n\ge 3$ be an integer.  Let $\{D_i : i \in \bbZ/n\bbZ\}$ and $\{\tilde D_i : i\in \bbZ/n\bbZ\}$ be thin collections of closed disks in the plane $\bbC$, in general position, so that the following conditions hold.
\begin{itemize}
\item We have that $\tilde D_i$ is contained in the interior of $D_i$ for all $i$.
\item The disk $D_i$ overlaps with $D_{i\pm 1}$, and the disk $\tilde D_i$ overlaps with $\tilde D_{i\pm 1}$, for all $i$.
\item If $D_i$ and $D_j$ meet, then $i=j\pm 1$.
\end{itemize}
Then $\sum_{i\in \bbZ/n\bbZ} \aext(\tilde D_i,\tilde D_{i+1}) < \sum_{i\in \bbZ/n\bbZ} \aext(D_i,D_{i+1})$.  In particular, for some $i$ we must have $\aext(D_i,D_{i+1}) \ne \aext(\tilde D_i,\tilde D_{i+1})$.
\end{proposition}

\begin{proof}
Note first that for $\aext(D_i, D_{i+1})$ to be well-defined, we still need to show that neither $D_i \subset D_{i+1}$ nor $D_{i+1}\subset D_i$.  The same is true for $\aext(\tilde D_i, \tilde D_{i+1})$.  Suppose for contradiction that $D_i\subset D_{i+1}$.  Then $D_{i-1}\cap D_i\subset D_{i+1}$, contradicting our hypotheses.  By symmetry we get that $D_{i+1}\not\subset D_i$.  The proof that $\aext(\tilde D_i, \tilde D_{i+1})$ is well-defined is identical.

To finish off the proof, we apply Lemma \ref{lem:meat} twice.  In both cases we will let $D = D_i$ and $\tilde D = \tilde D_i$.  First let $d_{-1} = D_{i-1}$ and $d_{+1} = D_{i+1}$.  This gives:%
\begin{equation}%
\label{eq11}%
\aext(D_{i-1},\tilde D_i) + \aext(D_{i+1},\tilde D_i) < \aext(D_{i-1},D_i) + \aext(D_{i+1},D_i)
\end{equation}%
\noindent Next let $d_{-1} = \tilde D_{i-1}$ and $d_{+1} = \tilde D_{i+1}$.  This gives:%
\begin{equation}%
\label{eq22}%
\aext(\tilde D_{i-1},\tilde D_i) + \aext(\tilde D_{i+1},\tilde D_i) < \aext(D_i,\tilde D_{i-1}) + \aext(D_i,\tilde D_{i+1})
\end{equation}%
\noindent If we let $i$ range over $\bbZ/n\bbZ$, the sum of the terms on the left side of equation \ref{eq11} is equal to the sum of the terms on the right side of equation \ref{eq22}.  The desired inequality follows.
\end{proof}

\begin{proposition}
\label{prop:black box 2}
Let $\calD = \{ D_1,\ldots,D_n \}$ and $\tilde\calD = \{\tilde D_1,\ldots,\tilde D_n\}$ be as in the statement of Theorem \ref{thm:main technical}, configurations in $\bbC$ which are thin and in general position, realizing the same incidence data.  Suppose there is some pair $D_i$ and $\tilde D_i$, one of which contains the other.  Let $I$ be a maximal nonempty subsumptive subset of $\{1,\ldots,n\}$.  Then there is at most one pair $i\in I$, $j\in \{1,\ldots,n\}\setminus I$ so that $D_i$ and $D_j$ overlap and one of $E_{ij} = D_i\cap D_j$ and $\tilde E_{ij} = \tilde D_i\cap \tilde D_j$ contains the other.
\end{proposition}

\begin{proof}
Suppose from now on, without loss of generality, that $\tilde D_i\subset D_i$ for all $i\in I$.

First, let $H_{\mathrm{u}}$ be the undirected simple graph defined as follows: the vertex set is $I$, and there is an edge between $i$ and $j$ if and only if $D_i$ and $D_j$ overlap.  Observe:

\begin{observation}
The graph $H_{\mathrm{u}}$ is connected and is a tree.
\end{observation}

\noindent This follows from Proposition \ref{prop:no contained loops} and the general position hypothesis.

Next, let $H$ be the directed graph so that $\left<i\to j\right>$ is an edge of $H$ if and only if:
\begin{itemize}
\item we have that $\left<i,j\right>$ is an edge of $H_{\mathrm{u}}$, and
\item either $\aext(\tilde D_i, D_j) > \aext(D_i,D_j)$ or $\tilde D_i\subset D_j$.
\end{itemize}
If $\left<i\to j\right>$ is an edge of $H$ then we call $\left<i\to j\right>$ an edge \emph{pointing away from $i$ in $H$}.  The idea is that if $\left<i\to j\right>$ is an edge in $H$ then the disk $\tilde D_i\subset D_i$ is ``shifted towards $D_j$ in $D_i$.''  See Figure \ref{fig:graph example} for an example.\medskip

\begin{figure}[t]
\centering
% Generated with LaTeXDraw 2.0.8
% Mon Apr 23 22:02:47 EDT 2012
% \usepackage[usenames,dvipsnames]{pstricks}
% \usepackage{epsfig}
% \usepackage{pst-grad} % For gradients
% \usepackage{pst-plot} % For axes
\scalebox{1} % Change this value to rescale the drawing.
{
\begin{pspicture}(0,-2.17)(7.6,2.17)
\pscircle[linewidth=0.02,dimen=outer](4.5,-0.33){1.1}
\pscircle[linewidth=0.02,dimen=outer](5.65,0.12){0.85}
\pscircle[linewidth=0.02,dimen=outer](6.25,1.12){0.85}
\pscircle[linewidth=0.02,dimen=outer](6.95,1.52){0.65}
\pscircle[linewidth=0.02,dimen=outer](6.4,-0.73){0.8}
\pscircle[linewidth=0.02,dimen=outer](3.2,0.27){1.1}
\pscircle[linewidth=0.02,dimen=outer](1.45,0.12){1.45}
\pscircle[linewidth=0.02,linestyle=dashed](6.34,1.17){0.7}
\pscircle[linewidth=0.02,linestyle=dashed](5.72,0.23){0.68}
\pscircle[linewidth=0.02,linestyle=dashed](6.25,-0.58){0.53}
\pscircle[linewidth=0.02,linestyle=dashed](4.63,-0.24){0.87}
\pscircle[linewidth=0.02,linestyle=dashed](3.34,-0.79){1.1}
\pscircle[linewidth=0.02,linestyle=dashed](1.64,-0.77){1.4}
\pscircle[linewidth=0.02,linestyle=dashed](6.83,1.44){0.45}
\psdots[dotsize=0.12](6.5,-0.83)
\psdots[dotsize=0.12](5.7,0.17)
\psdots[dotsize=0.12](6.2,1.07)
\psdots[dotsize=0.12](7.1,1.57)
\psdots[dotsize=0.12](4.6,-0.23)
\psline[linewidth=0.046cm,arrowsize=0.05291667cm 3.0,arrowlength=1.4,arrowinset=0]{->}(4.6,-0.23)(5.7,0.17)
\psline[linewidth=0.046cm,arrowsize=0.05291667cm 3.0,arrowlength=1.4,arrowinset=0]{->}(6.5,-0.83)(5.7,0.17)
\psline[linewidth=0.046cm,arrowsize=0.05291667cm 3.0,arrowlength=1.4,arrowinset=0]{->}(5.7,0.17)(6.2,1.07)
\psline[linewidth=0.046cm,arrowsize=0.05291667cm 3.0,arrowlength=1.4,arrowinset=0]{<->}(6.2,1.07)(7.1,1.57)
\end{pspicture} 
}

\caption
{
\label{fig:graph example}
The directed graph $H$ associated to a maximal subsumptive subconfigurations.  The solid disks are the $D_i$ and the dashed disks are the $\tilde D_i$.  The graph $H_{\mathrm{u}}$ can be obtained by undirecting every edge.
}
\end{figure}
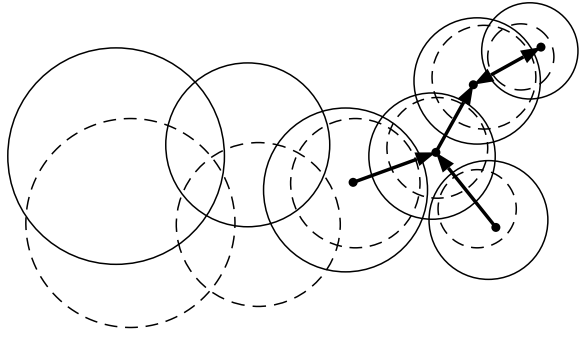

We now make a series of observations about $H$ and $H_{\mathrm{u}}$.  First:

\begin{observation}
\label{oj1}
If $\left<i,j\right>$ is an edge in $H_{\mathrm{u}}$ then at least one of $\left<i\to j\right>$ and $\left<j\to i\right>$ is an edge in $H$, and possibly both are.
\end{observation}

\noindent This follows from Lemma \ref{lem:finlandia}.

\begin{observation}
\label{oj2}
For every $i\in I$, there is at most one edge $\left<i\to j\right>$ in $H$ pointing away from $i$.
\end{observation}

\noindent This follows from Lemma \ref{lem:meat}, with $D = D_i$, $\tilde D = \tilde D_i$, $d_{-1} = D_j$, $d_{+1} = D_k$, for $j,k\in I$ so that $D_i$ overlaps with both $D_j$ and $D_k$.

\begin{observation}
\label{oj3}
Let $\left<i_1,i_2,\ldots,i_m\right>$ be a simple path in $H_{\mathrm{u}}$, meaning that $\left<i_\ell,i_{\ell+1}\right>$ is an edge in $H_{\mathrm{u}}$ for all $1\le \ell<m$ and that $i_\ell$ and $i_{\ell'}$ are distinct for $\ell\ne \ell'$.  Suppose that $\left<i_{m-1}\to i_m\right>$ is an edge in $H$.  Then $\left<i_\ell\to i_{\ell+1}\right>$ is an edge in $H$ for $1\le \ell<m$.
\end{observation}

\noindent This follows from Observations \ref{oj1} and \ref{oj2}, and induction.

\begin{observation}
\label{oj4}
There is at most one $i\in I$ so that there is no edge pointing away from $i$ in $H$.
\end{observation}

\noindent This follows from Observations \ref{oj1}, \ref{oj2}, and \ref{oj3}, because $H_{\mathrm{u}}$ is connected.  If there is an $i$ as in the statement of Observation \ref{oj4}, then we call this $i$ the \emph{sink} of the subsumptive subset $I\subset \{1,\ldots,n\}$.\medskip

Having established all we need to about $H$, we are ready to make two final observations which will complete the proof of Proposition \ref{prop:black box 2}.  First:

\begin{observation}
\label{lem:klm}
Let $i\in I$.  Then there is at most one $1\le j\le n$ different from $i$ so that $D_i$ and $D_j$ overlap and either $\tilde D_i\subset D_j$ or $\aext(D_i,D_j) < \aext(\tilde D_i,D_j)$.
\end{observation}

\noindent This follows from Lemma \ref{lem:meat} in the same way as does Observation \ref{oj2}.  Next:

\begin{observation}
\label{obs:mogwai2}
Suppose that $i$ and $j$ are as in the last sentence of the statement of Proposition \ref{prop:black box 2}.  Thus we have that $i\in I$ and $j\in \{1,\ldots,n\}\setminus I$ so that $D_i$ and $D_j$ overlap and $\tilde E_{ij}\subset E_{ij}$.  Then $\aext(D_i,D_j) = \aext(\tilde D_i,\tilde D_j) < \aext(\tilde D_i, D_j)$.
\end{observation}

\noindent This follows by an application of Lemma \ref{lem:mogwai} with $\tilde D_i = A$, $D_j = B$, and $\tilde D_j = C$.  Thus if $i$ and $j$ are as in the statement of Proposition \ref{prop:black box 2}, then $i$ is the unique sink of $H$.  Furthermore by Observations \ref{lem:klm} and \ref{obs:mogwai2} there is no $k\in \{1,\ldots,n\}\setminus I$ different from $j$ so that $D_i$ and $D_k$ overlap and so that one of $E_{ik}$ and $\tilde E_{ik}$ contains the other.  Proposition \ref{prop:black box 2} follows.
\end{proof}

The following lemmas will be helpful in the next section, and it is best to get them out of the way now:

\begin{lemma}
\label{lem megabus}
Let $\calD$ and $\tilde\calD$ be as in the statement of Proposition \ref{prop:black box 2}.  Suppose the disks $D_i \in \calD$ and $\tilde D_i\in \tilde\calD$ are so that $\tilde D_i$ is contained in the interior of $D_i$.  Suppose finally that for every $j\in \{1,\ldots,n\}$ different from $i$, so that $D_i$ and $D_j$ meet, we have either that $\tilde D_i$ is disjoint from $D_j$, or that $\aext(\tilde D_i, D_j) < \aext(D_i, D_j)$.  Then $i$ is the unqiue sink of some maximal isolated subsumptive subset of $\{1,\ldots,n\}$.
\end{lemma}

\noindent Lemma \ref{lem megabus} is really just an observation.  Let $I$ be the subsumptive subset of $\{1,\ldots,n\}$ containing $i$.  Define the directed graph $H$ as in the proof of Proposition \ref{prop:black box 2}.  Then by definition of $H$ there is no edge pointing away from $i$ in $H$.\medskip

\noindent The next lemma is an easy corollary of Lemma \ref{lem megabus}:

\begin{lemma}
\label{lem megabus corollary}
Let $\calD$ and $\tilde\calD$ be as in the statement of Proposition \ref{prop:black box 2}.  Suppose that the disks $D_i\in \calD$ and $\tilde D_i\in \tilde\calD$ have coinciding Euclidean centers.  Then $i$ is the unique sink of some maximal isolated subsumptive subset of $\{1,\ldots,n\}$.
\end{lemma}

\begin{proof}
Note that by the general position hypothesis the disks $D_i$ and $\tilde D_i$ cannot be equal.  We may suppose without loss of generality in our proof that $\tilde D_i \subset D_i$.  Then the lemma follows from Lemma \ref{lem megabus} because it is easy to see that if closed disks $A$ and $B$ in $\bbC$ overlap, so that neither is contained in the other, then $\aext(A,B)$ is monotone decreasing as we shrink $B$ by a contraction about its Euclidean center.
\end{proof}

\section{Proofs of our main rigidity theorems\twostars}
\label{sec:rigidity proof}

In this section we prove our main rigidity results using our Main Index Theorem \ref{mainindex}.  The main idea of the proofs we will see here is similar to the main idea of the proofs of the circle packing rigidity theorems given in Section \ref{sec rigid in plane}.  It may be helpful to review those now.  It may also be helpful to recall Definition \ref{def thin disk conf} on p.\ \pageref{def thin disk conf}.  The normalizations we construct were inspired by those of Merenkov, given in \cite{MR2900233}*{Section 12}.\medskip

The following lemma will be implicit in much of our discussion below:

\begin{lemma}
\label{kit kat bar}
Let $G = (V,E)$ be a 3-cycle, and $\Theta:G\to [0,\pi)$.  Then there is precisely one triple $\{D_v\}_{v\in V}$ of disks in $\hat\bbC$, up to action by M\"obius transformations, realizing the incidence data $(G,\Theta)$.
\end{lemma}

\noindent This is not hard to prove, and we leave it as an exercise.\medskip

The first rigidity theorem we prove here is Theorem \ref{mainrigidity}, restated here for the reader's convenience:

\renewcommand{\thetheoremfree}{\ref{mainrigidity}}
\begin{theoremfree}
Let $\calC$ and $\tilde\calC$ be thin disk configurations in $\hat\bbC$ realizing the same incidence data $(G,\Theta)$, where $G$ is the 1-skeleton of a triangulation of the 2-sphere $\bbS^2$.  Then $\calC$ and $\tilde\calC$ differ by a M\"obius or an anti-M\"obius transformation.
\end{theoremfree}

\begin{proof}
We begin by applying $z\mapsto \bar z$ to one of the configurations, if necessary, to ensure that the geodesic embeddings of $G$ in $\hat\bbC$ induced by $\calP$ and $\tilde\calP$ differ by an orientation-preserving self-homeomorphism of $\hat\bbC$.  For a reminder of the meaning of \emph{geodesic embedding}, see the proof of Theorem \ref{rigid in sphere}, on p.\ \pageref{rigid in sphere}.  The proof then proceeds by contradiction, supposing that there is no M\"obius transformation identifying $\calC$ and $\tilde\calC$.

First, note that we may suppose without loss of generality that there is a vertex $a$ of $G$ so that no disk of $\calC \setminus \{D_a\}$ overlaps with $D_a$, that is, that every contact between $D_a$ and another disk $D_v\in \calC$ is a tangency.  Then necessarily the same holds for $\tilde D_a$.  Every face $f$ of the triangulation of $\hat\bbC$ coming from the geodesic embedding of $G$ induced by $\calC$ contains exactly one interstice.  Index these faces by $F$, and write $T_f$ to denote the interstice of $\calC$ contained in the face corresponding to $f\in F$.  We define the interstices $\tilde T_f$ of $\tilde\calC$ analogously.  Pick an interstice $T_f$ of $\calC$, and let $D$ be the metric closed disk of largest spherical radius whose interior fits inside of $T_f$.  Let $\tilde D$ be constructed analogously for the corresponding interstice $\tilde T_f$ of $\tilde\calC$.  Each of the disks $D$ and $\tilde D$ is internally tangent to all three sides of its respective interstice $T_f$ or $\tilde T_f$.  Then it is not hard to show using Lemma \ref{kit kat bar} that any M\"obius transformation sending $\calC$ to $\tilde\calC$ will send $T_f$ to $\tilde T_f$, thus also $D$ to $\tilde D$, so $\calC$ and $\tilde\calC$ are M\"obius equivalent if and only if $\calC \cup \{D\}$ and $\tilde\calC \cup \{\tilde D\}$ are.  It is therefore harmless to add $D$ and $\tilde D$ to our configurations if necessary.

If there is no disk $D_b\in \calC$ which does not meet $D_a$, then $G$ is the 1-skeleton of a tetrahedron and it is easy to check Theorem \ref{rigid in sphere main} by hand using Lemma \ref{kit kat bar}.  Thus suppose that $D_b\in \calC$ is disjoint from $D_a$.

We now apply a series of M\"obius transformations, to explicitly describe a normalization on $\calC$ and $\tilde\calC$ in terms of one non-negative real parameter $\varepsilon\ge 0$:

\begin{enumerate}
\item First ensure that $\infty$ lies in the interiors of $D_a$ and of $D_a$, so that the circles $\partial D_a$ and $\partial D_b$ are concentric when considered in $\bbC$, and so that $\partial \tilde D_a$ and $\partial \tilde D_b$ are concentric when considered in $\bbC$.  Apply orientation-preserving Euclidean similarities so that $D_b$ and $\tilde D_b$ are both equal to the closed unit disk $\bar\bbD$.  Then the Euclidean centers in $\bbC$ of the circles $\partial D_a,\partial D_b,\partial \tilde D_a,\partial \tilde D_b$ all coincide.  The disks $D_b$ and $\tilde D_b$ are equal, and the disks $D_a$ and $\tilde D_a$ may be equal or unequal.
\item Pick a vertex $c$ of $G$, so that the disks $D_c$ and $\tilde D_c$ differ either in Euclidean radii or in the distances of their Euclidean centers from the origin, or both.  Such a $c$ must certainly exist: for instance, if $\tilde D_a\subset D_a$, then we may pick $\tilde D_c$ meeting $D_a$.  If $D_a = \tilde D_a$ and such a $c$ did not exist then we could argue via Lemma \ref{kit kat bar} that $\calC$ and $\tilde\calC$ differ by a rotation, after all of the normalizations applied thus far.
\item Apply a rotation to both packings so that the Euclidean centers of $D_c$ and $\tilde D_c$ both lie on the positive real axis.  Then apply a positive non-trivial dilation about the origin to one of the two packings, so that the Euclidean centers of $D_c$ and $\tilde D_c$ coincide.
\item At this point, the Euclidean centers of $\partial D_v$ and $\partial \tilde D_v$ coincide, for all $v = a,b,c$.  Because we applied a non-trivial positive dilation to one of the packings in the previous step, we have that $\partial D_b \ne \partial \tilde D_b$.  For either $v = a,c$, the disks $D_v$ and $\tilde D_v$ may be equal or unequal.  Regardless, our final step is to apply a dilation to $\calP$ by $1+\varepsilon$ about the common Euclidean center of $\partial D_c$ and $\partial \tilde D_c$.  Call the resulting normalization $\mathrm{N}(\varepsilon)$.
\end{enumerate}

\noindent Note that there clearly is an open interval $(0, \ldots)$, having one of its endpoints at $0$, of positive values that $\varepsilon$ may take so that after applying $\mathrm{N}(\varepsilon)$, we have that one of $D_v$ and $\tilde D_v$ is contained in the interior of the other, for all $v = a,b,c$.  For only finitely many of these values is it the case that $\calC$ and $\tilde\calC$ fail to end up in general position.\medskip

Denote $\calD = \calC \setminus \{D_a\}$ and $\tilde\calD = \tilde\calC \setminus \{\tilde D_a\}$.  Then $\calD$ and $\tilde\calD$ are thin disk configurations in $\bbC$ realizing the same incidence data.  Let $G_\calD$ denote their common contact graph, having vertex set $V_\calD$.  Then we have the following:

\begin{observation}
\label{claim to finish me}
If $\varepsilon > 0$ is sufficiently small, then after appling $\mathrm{N}(\varepsilon)$, we have that each of $b$ and $c$ belongs to a different maximal isolated subsumptive subset of the common index set $V_\calD$ of $\calD$ and $\tilde\calD$.
\end{observation}

To see why, first note that $D_b$ and $\tilde D_b$ are unequal and concentric in $\bbC$ under the normalization $\mathrm{N}(0)$.  Thus by the argument we used to prove Lemma \ref{lem megabus corollary}, for any $v\in V_\calD$ so that $D_v$ and $D_b$ meet, we have that either $\tilde D_b$ is disjoint from $D_v$ or $\aext(\tilde D_b, D_v) < \aext(D_b, D_v)$.  If we consider all of the disks to vary under $\mathrm{N}(\varepsilon)$, then these angles are continuous in the variable $\varepsilon$, so for some small interval $[0, \ldots)$ they continue to hold.  For all $\varepsilon$ in this interval $b$ will be the unique sink of a maximal isolated subsumptive subset of $V_\calD$ by Lemma \ref{lem megabus}.  Next, recall that $D_c$ and $\tilde D_c$ are concentric under any $\mathrm{N}(\varepsilon)$, and are unequal for all but one value $\varepsilon$.  Then $c$ is the unqiue sink of a maximal isolated subsumptive subset of $V_\calD$ by Lemma \ref{lem megabus corollary}, and Observation \ref{claim to finish me} follows.\medskip

We are now ready to obtain the desired contradiction to complete the proof of Theorem \ref{rigid in sphere main}.  Pick $\varepsilon>0$ sufficiently small as per Claim \ref{claim to finish me}, so that in addition $\calC$ and $\tilde\calC$ are in general position, and so that one of $D_v$ and $\tilde D_v$ is contained in the interior of the other for all $v = a,b,c$.  For every pair of corresponding interstices $T_f$ and $\tilde T_f$ of the packings $\calC$ and $\tilde\calC$, let $\phi_f : \partial T_f \to \partial \tilde T_f$ be an indexable homeomorphism identifying corresponding corners, satisfying $\eta(\phi_f) \ge 0$.  We may do so by the Three Point Prescription Lemma \ref{tppl}.  Then the $\phi_f$ induce a faithful indexable homeomorphism $\phi_\calD : \partial \calD \to \partial \tilde \calD$.  By our choice of $\varepsilon$ and Claim \ref{claim to finish me}, and by our Main Index Theorem \ref{mainindex}, we have that $\eta(\phi_\calD) \ge 2$.  On the other hand, orient $\partial D_a$ and $\partial \tilde D_a$ positively with respect to the open disks they bound in $\bbC$.  This is the opposite of the positive orientation on them with respect to $D_a$ and $\tilde D_a$.  Then $\eta(\phi_a) = 1$ by the Circle Index Lemma \ref{cil}.  Then we get a contradiction, because $\eta(\phi_a) = \eta(\phi_\calD) + \sum_{f\in F} \eta(\phi_f)$ by the Index Additivity Lemma \ref{ial}, and $\eta(\phi_f) \ge 0$ for all $f$ by construction.
\end{proof}

We next prove our Main Uniformization Theorem \ref{mainuniformization}.  We break the statement of Theorem \ref{mainuniformization} into three theorems, and prove each of these separately.  The proofs are adapted from the proof of Theorem \ref{mainrigidity} in exactly the same way that the proof of the constitutent theorems of Theorem \ref{dut} were adapted from the proof of Theorem \ref{katt}, so we will not give the full details.  Instead, we will construct in detail the appropriate normalization to start the proof, and omit the last part of each proof, where the contradiction is obtained.

\begin{theorem}
\label{uniformization main}
There do not exist thin disk configurations $\calC$ and $\tilde\calC$ realizing the same incidence data $(G,\Theta)$, where $G$ is the 1-skeleton of a triangulation of a topological open disk, so that $\calC$ is locally finite in $\bbC$ and $\tilde\calC$ is locally finite in the open unit disk $\bbD$, equivalently the hyperbolic plane $\bbH^2 \cong \bbD$.
\end{theorem}

\begin{proof}
This proof proceeds by contradiction, supposing that $\calC$ is locally finite in $\bbC$ and $\tilde\calC$ is locally finite in $\bbD$.  Apply $z\mapsto \bar z$ to one of the configurations, if necessary, to ensure that the geodesic embeddings of $G$ in $\bbC$ and $\bbD$ induced by $\calP$ and $\tilde\calP$ respectively differ by an orientation-preserving homeomorphism $\bbC\to \bbD$.  We now apply a series of orientation-preserving Euclidean similarities, to explicitly describe a normalization on $\calC$ and $\tilde\calC$ in terms of one non-negative real parameter $\varepsilon\ge 0$:

\begin{enumerate}
\item First, pick $D_a\in \calC$ and $\tilde D_a\in \tilde\calC$, and apply translations to both configurations, and a scaling to $\calC$, so that $D_a$ and $\tilde D_a$ coincide, and are centered at the origin.
\item Pick disks $D_b\in \calC$ and $\tilde D_b\in \tilde\calC$ which differ either in their Euclidean radii or in the distances of their Euclidean centers from the origin, or both.  We may obviously do so.  Apply a rotation about the origin to both configurations so that the Euclidean centers of $D_b$ and $\tilde D_b$ both lie on the positive real axis, and then apply a non-trivial dilation about the origin to one of the configurations so that the Euclidean centers of $D_b$ and $\tilde D_b$ coincide.
\item At this point $D_a$ and $\tilde D_a$ are unequal, but are concentric in $\bbC$, and $D_b$ and $\tilde D_b$ are concentric in $\bbC$, and may be equal or unequal.  As our last step, we dilate $\calP$ by a factor of $1+\varepsilon$ about the common Euclidean center of $D_b$ and $\tilde D_b$.  Denote the resulting normalization $\mathrm{N}(\varepsilon)$.
\end{enumerate}

\noindent The rest of the proof proceeds in the same way as did the proof of Theorem \ref{libun}.
\end{proof}

\begin{theorem}
\label{rigid in plane main}
Let $\calC$ and $\tilde\calC$ be thin disk configurations realizing the same incidence data $(G,\Theta)$, where $G$ is the 1-skeleton of a triangulation of a topological open disk, so that both $\calC$ and $\tilde\calC$ are locally finite in $\bbC$.  Then $\calC$ and $\tilde\calC$ differ by a Euclidean similarity.
\end{theorem}

\begin{proof}
Apply $z\mapsto \bar z$ to one of the configurations, if necessary, to ensure that the geodesic embeddings of $G$ in $\bbC$ induced by $\calP$ and $\tilde\calP$ differ by an orientation-preserving self-homeomorphism of $\bbC$.  Suppose for contradiction that $\calC$ and $\tilde\calC$ do not differ by any orientation-preserving Euclidean similarity.  They therefore do not differ by any M\"obius transformation.  We now apply a series of M\"obius transformations, to explicitly describe a normalization on $\calC$ and $\tilde\calC$ in terms of one non-negative real parameter $\varepsilon\ge 0$:

\begin{enumerate}
\item We first argue as in the proof of Theorem \ref{mainrigidity} that we may assume without loss of generality that we may take $D_a\in \calC$ and $\tilde D_a\in \tilde\calC$ so that no other disk $D_v\in \calC \setminus \{D_a\}$ overlaps with $D_a$.
\item Pick $b\in V\setminus \{a\}$ so that $b$ does not share an edge with $a$ in $G$.  Apply M\"obius transformations so that $\infty$ lies in the interiors of both $D_a$ and $\tilde D_a$, and so that all of the circles $\partial D_a, \partial D_b, \partial \tilde D_a, \partial \tilde D_b$ have their Euclidean centers at the origin.  Apply a Euclidean scaling to $\calC$ so that $D_b$ and $\tilde D_b$ coincide.  At this point, the disks $D_a$ and $\tilde D_a$ may be equal or unequal.
\item We argue as before that we may pick $c\in V\setminus \{a,b\}$ so that $D_c$ and $\tilde D_c$ differ in their Euclidean radii, the distances of their Euclidean centers from the origin, or both.  Apply rotations about the origin so that the Euclidean centers of $D_c$ and $\tilde D_c$ lie on the positive real axis, and apply a scaling to $\calC$ so that the Euclidean centers of $D_c$ and $\tilde D_c$ coincide.
\item At this point the disks $D_a$ and $\tilde D_a$ may be equal or unequal, the disks $D_b$ and $\tilde D_b$ are unequal, and the disks $D_c$ and $\tilde D_c$ may be equal or unequal.  All of $\partial D_a, \partial D_b, \partial \tilde D_a, \partial \tilde D_b$ are centered at the origin, and $D_c$ and $\tilde D_c$ are concentric in $\bbC$.  As the last step of our normalization, apply a dilation by a factor of $1+\varepsilon$ about the common Euclidean center of $D_b$ and $\tilde D_b$ to $\calC$.  Denote the resulting normalization $\mathrm{N}(\varepsilon)$.
\end{enumerate}

\noindent The rest of the proof proceeds in the same way as the proof of Theorem \ref{libcp}.
\end{proof}

\begin{theorem}
\label{rigid in hyp plane main}
Let $\calC$ and $\tilde\calC$ be thin disk configurations realizing the same incidence data $(G,\Theta)$, where $G$ is the 1-skeleton of a triangulation of a topological open disk, so that both $\calC$ and $\tilde\calC$ are locally finite in the open unit disk $\bbD$, equivalently in the hyperbolic plane $\bbH^2 \cong \bbD$.  Then $\calC$ and $\tilde\calC$ differ by a hyperbolic isometry.
\end{theorem}

\begin{proof}
As always, apply $z\mapsto \bar z$ to one of the configurations if necessary, to ensure that the geodesic embeddings of $G$ in $\bbD$ induced by $\calP$ and $\tilde\calP$ differ by an orientation-preserving self-homeomorphism of $\bbD$.  Suppose for contradiction that $\calC$ and $\tilde\calC$ do not differ by any orientation-preserving hyperbolic isometry.  They therefore do not differ by any M\"obius transformation.  We now apply a series of M\"obius transformations, to explicitly describe a normalization on $\calC$ and $\tilde\calC$ in terms of one non-negative real parameter $\varepsilon\ge 0$:

\begin{enumerate}
\item First, there must be disks $D_a\in \calC$ and $\tilde D_a\in \tilde\calC$ which have different hyperbolic radii in $\bbD \cong \bbH^2$, otherwise the two configurations coincide by elementary arguments.  Apply hyperbolic isometries to both configurations so that $D_a$ and $\tilde D_a$ are centered at the origin, and apply a Euclidean scaling centered at the origin to $\calC$, so that $D_a$ and $\tilde D_a$ coincide.
\item Pick disks $D_b\in \calC$ and $\tilde D_b\in \tilde\calC$ which differ in their Euclidean radii, or the distances of their Euclidean centers from the origin, or both.  Apply rotations centered at the origin so that the Euclidean centers of $D_b$ and $\tilde D_b$ lie on the positive real axis, and apply a Euclidean scaling centered at the origin to one configuration so that the Euclidean centers of $D_b$ and $\tilde D_b$ coincide.
\item At this point the disks $D_a$ and $\tilde D_a$ are unequal, and are concentric in $\bbC$, the disks $D_b$ and $\tilde D_b$ are concentric in $\bbC$, and may be equal or unequal.  Also, denoting by $D$ and $\tilde D$ the images of $\bbD$ under the normalizations applied thus far to $\calC$ and $\tilde\calC$ respectively, we have that $D$ and $\tilde D$ are centered at the origin, and may be equal or unequal.  As the last step of our normalization, apply a dilation by a factor of $1+\varepsilon$ about the common Euclidean center of $D_b$ and $\tilde D_b$ to $\calC$.  Denote the resulting normalization $\mathrm{N}(\varepsilon)$.
\end{enumerate}

\noindent The rest of the proof proceeds in the same way as the proof of Theorem \ref{libhyp}.
\end{proof}

\section{Topological configurations\twostars}
\label{chap:topo confo}

Suppose that $X_1,\ldots,X_n$ and $ X'_1,\ldots, X'_n$ are all subsets of $\bbC$.  Then we say that the collections $\{X_1,\ldots,X_n\}$ and $\{X'_1,\ldots, X'_n\}$ are in the same \emph{topological configuration} if there is an orientation-preserving homeomorphism $\varphi:\bbC \to \bbC$ so that $\varphi(X_i) =  X'_i$ for all $1\le i\le n$.  In practice the collections of objects under consideration will not be labeled $X_i$ and $ X'_i$, but there will be some natural bijection between the collections.  Then our requirement is that $\varphi$ respects this natural bijection.

The following lemma says that when working with fixed-point index, we need to consider our Jordan domains only ``up to topological configuration.''

\begin{lemma}
\label{prop:index invariance}
Suppose $K$ and $\tilde K$ are closed Jordan domains.  Let $\phi:\partial K \to \partial \tilde K$ be an indexable homeomorphism.  Suppose that $K'$ and $\tilde K'$ are also closed Jordan domains, so that $\{K,\tilde K\}$ and $\{K',\tilde K'\}$ are in the same topological configuration, via the homeomorphism $\psi:\bbC \to \bbC$.  Let $\phi':\partial K'\to \partial \tilde K'$ be induced in the natural way, explicitly as $\phi' = \psi|_{\partial \tilde K} \circ \phi \circ \psi\inv|_{\partial K'}$.  Then $\phi'$ is indexable and $\eta(\phi) = \eta(\phi')$.
\end{lemma}

\noindent This follows via homotopy arguments from the well-known fact that every orientation-preserving homeomorphism $\bbC\to \bbC$ is homotopic to the identity map via homeomorphisms.\medskip

The following proposition limits the relevant topological configurations that two disks may be in to finitely many possibilities, which by Lemma \ref{prop:index invariance} reduces every subsequent proof to at worst a case-by-case analysis:

\begin{proposition}
\label{prop:possible for all}
Suppose that $\{A,B\}$ and $\{\tilde A,\tilde B\}$ are pairs of overlapping closed disks in the plane $\bbC$ in general position.  Suppose that $A\setminus B$ meets $\tilde A\setminus \tilde B$, that $A\cap B$ meets $\tilde A\cap \tilde B$, and that $B\setminus A$ meets $\tilde B\setminus \tilde A$.  Then given any three of the disks $A, B, \tilde A, \tilde B$, the topological configuration of those three disks is one of those depicted in Figures \ref{fig:possible} and \ref{fig:possibleii}.
\end{proposition}

\begin{figure}[t]
\begin{minipage}{0.48\linewidth}
\centering
\makeatletter
\def\p@subfigure{\possibletildea}
\makeatother
\subfloat[]
{\label{fig:possible3a}
% Generated with LaTeXDraw 2.0.8
% Fri Apr 06 22:15:26 EDT 2012
% \usepackage[usenames,dvipsnames]{pstricks}
% \usepackage{epsfig}
% \usepackage{pst-grad} % For gradients
% \usepackage{pst-plot} % For axes
\scalebox{.65} % Change this value to rescale the drawing.
{%
\begin{pspicture}(0,-0.7867187)(3.5690625,0.82671875)
\pscircle[linewidth=0.02,dimen=outer](2.57,-0.08671875){0.7}
\pscircle[linewidth=0.02,dimen=outer](1.97,-0.08671875){0.3}
\pscircle[linewidth=0.02,linestyle=dashed,dash=0.12cm 0.06cm,dimen=outer](1.47,-0.08671875){0.7}
\usefont{T1}{ppl}{m}{n}
\rput(1.4845313,-0.07671875){$A$}
\usefont{T1}{ppl}{m}{n}
\rput(3.1845312,0.62328124){$B$}
\usefont{T1}{ppl}{m}{n}
\rput(0.7945312,0.62328124){$\tilde A$}
\end{pspicture} 
}
}
\subfloat[]
{\label{fig:possible3b}
% Generated with LaTeXDraw 2.0.8
% Fri Apr 06 22:14:33 EDT 2012
% \usepackage[usenames,dvipsnames]{pstricks}
% \usepackage{epsfig}
% \usepackage{pst-grad} % For gradients
% \usepackage{pst-plot} % For axes
\scalebox{.65} % Change this value to rescale the drawing.
{
\begin{pspicture}(0,-0.7867187)(3.5690625,0.82671875)
\pscircle[linewidth=0.02,dimen=outer](2.57,-0.08671875){0.7}
\pscircle[linewidth=0.02,dimen=outer](1.57,-0.08671875){0.5}
\pscircle[linewidth=0.02,linestyle=dashed,dash=0.12cm 0.06cm,dimen=outer](1.47,-0.08671875){0.7}
\usefont{T1}{ppl}{m}{n}
\rput(1.2845312,-0.07671875){$A$}
\usefont{T1}{ppl}{m}{n}
\rput(3.1845312,0.62328124){$B$}
\usefont{T1}{ppl}{m}{n}
\rput(0.7945312,0.62328124){$\tilde A$}
\end{pspicture} 
}
}
\subfloat[]
{\label{fig:possible3c}
% Generated with LaTeXDraw 2.0.8
% Fri Apr 06 22:20:20 EDT 2012
% \usepackage[usenames,dvipsnames]{pstricks}
% \usepackage{epsfig}
% \usepackage{pst-grad} % For gradients
% \usepackage{pst-plot} % For axes
\scalebox{.65} % Change this value to rescale the drawing.
{
\begin{pspicture}(0,-0.88671875)(3.0690625,0.9267188)
\pscircle[linewidth=0.02,dimen=outer](1.97,-0.18671875){0.7}
\pscircle[linewidth=0.02,dimen=outer](0.97,-0.18671875){0.7}
\pscircle[linewidth=0.02,linestyle=dashed,dash=0.12cm 0.06cm,dimen=outer](1.47,-0.18671875){0.7}
\usefont{T1}{ppl}{m}{n}
\rput(0.28453124,0.42328125){$A$}
\usefont{T1}{ppl}{m}{n}
\rput(2.6845312,0.42328125){$B$}
\usefont{T1}{ppl}{m}{n}
\rput(1.4445312,0.72328126){$\tilde A$}
\end{pspicture} 
}
}

\subfloat[]
{\label{fig:possible3d}
% Generated with LaTeXDraw 2.0.8
% Fri Apr 06 22:21:30 EDT 2012
% \usepackage[usenames,dvipsnames]{pstricks}
% \usepackage{epsfig}
% \usepackage{pst-grad} % For gradients
% \usepackage{pst-plot} % For axes
\scalebox{.65} % Change this value to rescale the drawing.
{
\begin{pspicture}(0,-0.88671875)(3.0690625,0.9267188)
\pscircle[linewidth=0.02,dimen=outer](1.97,-0.18671875){0.7}
\pscircle[linewidth=0.02,dimen=outer](0.97,-0.18671875){0.7}
\pscircle[linewidth=0.02,linestyle=dashed,dash=0.12cm 0.06cm,dimen=outer](1.47,0.21328124){0.3}
\usefont{T1}{ppl}{m}{n}
\rput(0.28453124,0.42328125){$A$}
\usefont{T1}{ppl}{m}{n}
\rput(2.6845312,0.42328125){$B$}
\usefont{T1}{ppl}{m}{n}
\rput(1.4445312,0.72328126){$\tilde A$}
\end{pspicture} 
}
}
\subfloat[\label{fig:possible3e}]
{
% Generated with LaTeXDraw 2.0.8
% Fri Apr 06 22:17:56 EDT 2012
% \usepackage[usenames,dvipsnames]{pstricks}
% \usepackage{epsfig}
% \usepackage{pst-grad} % For gradients
% \usepackage{pst-plot} % For axes
\scalebox{.65} % Change this value to rescale the drawing.
{
\begin{pspicture}(0,-0.7367188)(3.0690625,0.77671874)
\pscircle[linewidth=0.02,dimen=outer](1.97,-0.03671875){0.7}
\pscircle[linewidth=0.02,dimen=outer](0.97,-0.03671875){0.7}
\pscircle[linewidth=0.02,linestyle=dashed,dash=0.12cm 0.06cm,dimen=outer](1.47,-0.43671876){0.3}
\usefont{T1}{ppl}{m}{n}
\rput(0.28453124,0.5732812){$A$}
\usefont{T1}{ppl}{m}{n}
\rput(2.6845312,0.5732812){$B$}
\usefont{T1}{ppl}{m}{n}
\rput(0.99453125,-0.32671875){$\tilde A$}
\end{pspicture} 
}
}
\subfloat[\label{fig:possible3f}]
{
% Generated with LaTeXDraw 2.0.8
% Fri Apr 06 22:19:21 EDT 2012
% \usepackage[usenames,dvipsnames]{pstricks}
% \usepackage{epsfig}
% \usepackage{pst-grad} % For gradients
% \usepackage{pst-plot} % For axes
\scalebox{.65} % Change this value to rescale the drawing.
{
\begin{pspicture}(0,-0.7867187)(3.5690625,0.82671875)
\pscircle[linewidth=0.02,dimen=outer](2.57,-0.08671875){0.7}
\pscircle[linewidth=0.02,linestyle=dashed,dash=0.12cm 0.06cm,dimen=outer](1.57,-0.08671875){0.5}
\pscircle[linewidth=0.02,dimen=outer](1.47,-0.08671875){0.7}
\usefont{T1}{ppl}{m}{n}
\rput(1.2845312,-0.07671875){$\tilde A$}
\usefont{T1}{ppl}{m}{n}
\rput(3.2845312,0.52328124){$B$}
\usefont{T1}{ppl}{m}{n}
\rput(0.7945312,0.52328124){$A$}
\end{pspicture} 
}
}

\subfloat[\label{fig:possible3g}]
{
% Generated with LaTeXDraw 2.0.8
% Fri Apr 06 22:12:46 EDT 2012
% \usepackage[usenames,dvipsnames]{pstricks}
% \usepackage{epsfig}
% \usepackage{pst-grad} % For gradients
% \usepackage{pst-plot} % For axes
\scalebox{.65} % Change this value to rescale the drawing.
{
\begin{pspicture}(0,-0.7367188)(3.0690625,0.77671874)
\pscircle[linewidth=0.02,dimen=outer](1.97,-0.03671875){0.7}
\pscircle[linewidth=0.02,dimen=outer](0.97,-0.03671875){0.7}
\pscircle[linewidth=0.02,linestyle=dashed,dash=0.12cm 0.06cm,dimen=outer](1.47,-0.03671875){0.3}
\usefont{T1}{ptm}{m}{n}
\rput(0.28453124,0.5732812){$A$}
\usefont{T1}{ptm}{m}{n}
\rput(2.6845312,0.5732812){$B$}
\usefont{T1}{ptm}{m}{n}
\rput(0.94453125,-0.02671875){$\tilde A$}
\end{pspicture} 
}
}
\subfloat[\label{fig:possible3h}]
{
% Generated with LaTeXDraw 2.0.8
% Fri Apr 06 22:23:27 EDT 2012
% \usepackage[usenames,dvipsnames]{pstricks}
% \usepackage{epsfig}
% \usepackage{pst-grad} % For gradients
% \usepackage{pst-plot} % For axes
\scalebox{.65} % Change this value to rescale the drawing.
{
\begin{pspicture}(0,-0.83671874)(3.6690626,0.87671876)
\pscircle[linewidth=0.02,dimen=outer](2.67,-0.13671875){0.7}
\pscircle[linewidth=0.02,linestyle=dashed,dash=0.12cm 0.06cm,dimen=outer](1.57,-0.13671875){0.7}
\pscircle[linewidth=0.02,dimen=outer](2.17,-0.13671875){0.6}
\usefont{T1}{ppl}{m}{n}
\rput(2.1345313,0.67328125){$A$}
\usefont{T1}{ppl}{m}{n}
\rput(3.3345314,0.4732812){$B$}
\usefont{T1}{ppl}{m}{n}
\rput(0.8945312,0.4732812){$\tilde A$}
\end{pspicture} 
}
}
%{
%	\captionsetup{labelformat=empty}
	\caption*
	{%
		\possibletildea
	}
%	\addtocounter{figure}{-1}
%}
%\addtocounter{figure}{1}
\end{minipage}
\begin{minipage}{0.48\linewidth}
\setcounter{subfigure}{0}
\centering
\makeatletter
\def\p@subfigure{\possibletildeb} 
\makeatother
\subfloat[]
{\label{fig:possible4a}
% Generated with LaTeXDraw 2.0.8
% Fri Apr 06 22:15:26 EDT 2012
% \usepackage[usenames,dvipsnames]{pstricks}
% \usepackage{epsfig}
% \usepackage{pst-grad} % For gradients
% \usepackage{pst-plot} % For axes
\scalebox{.65} % Change this value to rescale the drawing.
{
\begin{pspicture}(0,-0.7867187)(3.5690625,0.82671875)
\pscircle[linewidth=0.02,dimen=outer](2.57,-0.08671875){0.7}
\pscircle[linewidth=0.02,dimen=outer](1.97,-0.08671875){0.3}
\pscircle[linewidth=0.02,linestyle=dashed,dash=0.12cm 0.06cm,dimen=outer](1.47,-0.08671875){0.7}
\usefont{T1}{ppl}{m}{n}
\rput(1.4845313,-0.07671875){$B$}
\usefont{T1}{ppl}{m}{n}
\rput(3.1845312,0.62328124){$A$}
\usefont{T1}{ppl}{m}{n}
\rput(0.7945312,0.62328124){$\tilde B$}
\end{pspicture} 
}
}
\subfloat[]
{\label{fig:possible4b}
% Generated with LaTeXDraw 2.0.8
% Fri Apr 06 22:14:33 EDT 2012
% \usepackage[usenames,dvipsnames]{pstricks}
% \usepackage{epsfig}
% \usepackage{pst-grad} % For gradients
% \usepackage{pst-plot} % For axes
\scalebox{.65} % Change this value to rescale the drawing.
{
\begin{pspicture}(0,-0.7867187)(3.5690625,0.82671875)
\pscircle[linewidth=0.02,dimen=outer](2.57,-0.08671875){0.7}
\pscircle[linewidth=0.02,dimen=outer](1.57,-0.08671875){0.5}
\pscircle[linewidth=0.02,linestyle=dashed,dash=0.12cm 0.06cm,dimen=outer](1.47,-0.08671875){0.7}
\usefont{T1}{ppl}{m}{n}
\rput(1.2845312,-0.07671875){$B$}
\usefont{T1}{ppl}{m}{n}
\rput(3.1845312,0.62328124){$A$}
\usefont{T1}{ppl}{m}{n}
\rput(0.7945312,0.62328124){$\tilde B$}
\end{pspicture} 
}
}
\subfloat[]
{\label{fig:possible4c}
% Generated with LaTeXDraw 2.0.8
% Fri Apr 06 22:20:20 EDT 2012
% \usepackage[usenames,dvipsnames]{pstricks}
% \usepackage{epsfig}
% \usepackage{pst-grad} % For gradients
% \usepackage{pst-plot} % For axes
\scalebox{.65} % Change this value to rescale the drawing.
{
\begin{pspicture}(0,-0.88671875)(3.0690625,0.9267188)
\pscircle[linewidth=0.02,dimen=outer](1.97,-0.18671875){0.7}
\pscircle[linewidth=0.02,dimen=outer](0.97,-0.18671875){0.7}
\pscircle[linewidth=0.02,linestyle=dashed,dash=0.12cm 0.06cm,dimen=outer](1.47,-0.18671875){0.7}
\usefont{T1}{ppl}{m}{n}
\rput(0.28453124,0.42328125){$B$}
\usefont{T1}{ppl}{m}{n}
\rput(2.6845312,0.42328125){$A$}
\usefont{T1}{ppl}{m}{n}
\rput(1.4445312,0.72328126){$\tilde B$}
\end{pspicture} 
}
}

\subfloat[]
{\label{fig:possible4d}
% Generated with LaTeXDraw 2.0.8
% Fri Apr 06 22:21:30 EDT 2012
% \usepackage[usenames,dvipsnames]{pstricks}
% \usepackage{epsfig}
% \usepackage{pst-grad} % For gradients
% \usepackage{pst-plot} % For axes
\scalebox{.65} % Change this value to rescale the drawing.
{
\begin{pspicture}(0,-0.88671875)(3.0690625,0.9267188)
\pscircle[linewidth=0.02,dimen=outer](1.97,-0.18671875){0.7}
\pscircle[linewidth=0.02,dimen=outer](0.97,-0.18671875){0.7}
\pscircle[linewidth=0.02,linestyle=dashed,dash=0.12cm 0.06cm,dimen=outer](1.47,0.21328124){0.3}
\usefont{T1}{ppl}{m}{n}
\rput(0.28453124,0.42328125){$B$}
\usefont{T1}{ppl}{m}{n}
\rput(2.6845312,0.42328125){$A$}
\usefont{T1}{ppl}{m}{n}
\rput(1.4445312,0.72328126){$\tilde B$}
\end{pspicture} 
}
}
\subfloat[]
{\label{fig:possible4e}
% Generated with LaTeXDraw 2.0.8
% Fri Apr 06 22:17:56 EDT 2012
% \usepackage[usenames,dvipsnames]{pstricks}
% \usepackage{epsfig}
% \usepackage{pst-grad} % For gradients
% \usepackage{pst-plot} % For axes
\scalebox{.65} % Change this value to rescale the drawing.
{
\begin{pspicture}(0,-0.7367188)(3.0690625,0.77671874)
\pscircle[linewidth=0.02,dimen=outer](1.97,-0.03671875){0.7}
\pscircle[linewidth=0.02,dimen=outer](0.97,-0.03671875){0.7}
\pscircle[linewidth=0.02,linestyle=dashed,dash=0.12cm 0.06cm,dimen=outer](1.47,-0.43671876){0.3}
\usefont{T1}{ppl}{m}{n}
\rput(0.28453124,0.5732812){$B$}
\usefont{T1}{ppl}{m}{n}
\rput(2.6845312,0.5732812){$A$}
\usefont{T1}{ppl}{m}{n}
\rput(0.99453125,-0.32671875){$\tilde B$}
\end{pspicture} 
}
}
\subfloat[]
{\label{fig:possible4f}
% Generated with LaTeXDraw 2.0.8
% Fri Apr 06 22:19:21 EDT 2012
% \usepackage[usenames,dvipsnames]{pstricks}
% \usepackage{epsfig}
% \usepackage{pst-grad} % For gradients
% \usepackage{pst-plot} % For axes
\scalebox{.65} % Change this value to rescale the drawing.
{
\begin{pspicture}(0,-0.7867187)(3.5690625,0.82671875)
\pscircle[linewidth=0.02,dimen=outer](2.57,-0.08671875){0.7}
\pscircle[linewidth=0.02,linestyle=dashed,dash=0.12cm 0.06cm,dimen=outer](1.57,-0.08671875){0.5}
\pscircle[linewidth=0.02,dimen=outer](1.47,-0.08671875){0.7}
\usefont{T1}{ppl}{m}{n}
\rput(1.2845312,-0.07671875){$\tilde B$}
\usefont{T1}{ppl}{m}{n}
\rput(3.2845312,0.52328124){$A$}
\usefont{T1}{ppl}{m}{n}
\rput(0.7945312,0.52328124){$B$}
\end{pspicture} 
}
}

\subfloat[]
{\label{fig:possible4g}
% Generated with LaTeXDraw 2.0.8
% Fri Apr 06 22:12:46 EDT 2012
% \usepackage[usenames,dvipsnames]{pstricks}
% \usepackage{epsfig}
% \usepackage{pst-grad} % For gradients
% \usepackage{pst-plot} % For axes
\scalebox{.65} % Change this value to rescale the drawing.
{
\begin{pspicture}(0,-0.7367188)(3.0690625,0.77671874)
\pscircle[linewidth=0.02,dimen=outer](1.97,-0.03671875){0.7}
\pscircle[linewidth=0.02,dimen=outer](0.97,-0.03671875){0.7}
\pscircle[linewidth=0.02,linestyle=dashed,dash=0.12cm 0.06cm,dimen=outer](1.47,-0.03671875){0.3}
\usefont{T1}{ptm}{m}{n}
\rput(0.28453124,0.5732812){$B$}
\usefont{T1}{ptm}{m}{n}
\rput(2.6845312,0.5732812){$A$}
\usefont{T1}{ptm}{m}{n}
\rput(0.94453125,-0.02671875){$\tilde B$}
\end{pspicture} 
}
}
\subfloat[]
{\label{fig:possible4h}
% Generated with LaTeXDraw 2.0.8
% Fri Apr 06 22:23:27 EDT 2012
% \usepackage[usenames,dvipsnames]{pstricks}
% \usepackage{epsfig}
% \usepackage{pst-grad} % For gradients
% \usepackage{pst-plot} % For axes
\scalebox{.65} % Change this value to rescale the drawing.
{
\begin{pspicture}(0,-0.83671874)(3.6690626,0.87671876)
\pscircle[linewidth=0.02,dimen=outer](2.67,-0.13671875){0.7}
\pscircle[linewidth=0.02,linestyle=dashed,dash=0.12cm 0.06cm,dimen=outer](1.57,-0.13671875){0.7}
\pscircle[linewidth=0.02,dimen=outer](2.17,-0.13671875){0.6}
\usefont{T1}{ppl}{m}{n}
\rput(2.1345313,0.67328125){$B$}
\usefont{T1}{ppl}{m}{n}
\rput(3.3345314,0.4732812){$A$}
\usefont{T1}{ppl}{m}{n}
\rput(0.8945312,0.4732812){$\tilde B$}
\end{pspicture} 
}
}%
	\caption*
	{%
		\possibletildeb
	}%
%\addtocounter{figure}{1}
\end{minipage}
{
\caption
{
\label{fig:possible}
The relevant topological configurations of $\tilde A$ and of $\tilde B$, relative to $A, B$.
}
}
\end{figure}

\begin{figure}[t]
\begin{minipage}{0.49\linewidth}
\centering
\makeatletter
\def\p@subfigure{\possiblea} 
\makeatother
\subfloat[]
{\label{fig:possible5a}
% Generated with LaTeXDraw 2.0.8
% Fri Apr 06 22:15:26 EDT 2012
% \usepackage[usenames,dvipsnames]{pstricks}
% \usepackage{epsfig}
% \usepackage{pst-grad} % For gradients
% \usepackage{pst-plot} % For axes
\scalebox{.65} % Change this value to rescale the drawing.
{
\begin{pspicture}(0,-0.7867187)(3.5690625,0.82671875)
\pscircle[linewidth=0.02,linestyle=dashed,dash=0.12cm 0.06cm,dimen=outer](2.57,-0.08671875){0.7}
\pscircle[linewidth=0.02,linestyle=dashed,dash=0.12cm 0.06cm,dimen=outer](1.97,-0.08671875){0.3}
\pscircle[linewidth=0.02,dimen=outer](1.47,-0.08671875){0.7}
\usefont{T1}{ppl}{m}{n}
\rput(1.4845313,-0.07671875){$\tilde A$}
\usefont{T1}{ppl}{m}{n}
\rput(3.1845312,0.62328124){$\tilde B$}
\usefont{T1}{ppl}{m}{n}
\rput(0.7945312,0.62328124){$ A$}
\end{pspicture} 
}
}
\subfloat[]
{\label{fig:possible5b}
% Generated with LaTeXDraw 2.0.8
% Fri Apr 06 22:14:33 EDT 2012
% \usepackage[usenames,dvipsnames]{pstricks}
% \usepackage{epsfig}
% \usepackage{pst-grad} % For gradients
% \usepackage{pst-plot} % For axes
\scalebox{.65} % Change this value to rescale the drawing.
{
\begin{pspicture}(0,-0.7867187)(3.5690625,0.82671875)
\pscircle[linewidth=0.02,linestyle=dashed,dash=0.12cm 0.06cm,dimen=outer](2.57,-0.08671875){0.7}
\pscircle[linewidth=0.02,linestyle=dashed,dash=0.12cm 0.06cm,dimen=outer](1.57,-0.08671875){0.5}
\pscircle[linewidth=0.02,dimen=outer](1.47,-0.08671875){0.7}
\usefont{T1}{ppl}{m}{n}
\rput(1.2845312,-0.07671875){$\tilde A$}
\usefont{T1}{ppl}{m}{n}
\rput(3.1845312,0.62328124){$\tilde B$}
\usefont{T1}{ppl}{m}{n}
\rput(0.7945312,0.62328124){$ A$}
\end{pspicture} 
}
}
\subfloat[]
{\label{fig:possible5c}
% Generated with LaTeXDraw 2.0.8
% Fri Apr 06 22:20:20 EDT 2012
% \usepackage[usenames,dvipsnames]{pstricks}
% \usepackage{epsfig}
% \usepackage{pst-grad} % For gradients
% \usepackage{pst-plot} % For axes
\scalebox{.65} % Change this value to rescale the drawing.
{
\begin{pspicture}(0,-0.88671875)(3.0690625,0.9267188)
\pscircle[linewidth=0.02,linestyle=dashed,dash=0.12cm 0.06cm,dimen=outer](1.97,-0.18671875){0.7}
\pscircle[linewidth=0.02,linestyle=dashed,dash=0.12cm 0.06cm,dimen=outer](0.97,-0.18671875){0.7}
\pscircle[linewidth=0.02,dimen=outer](1.47,-0.18671875){0.7}
\usefont{T1}{ppl}{m}{n}
\rput(0.28453124,0.42328125){$\tilde A$}
\usefont{T1}{ppl}{m}{n}
\rput(2.6845312,0.42328125){$\tilde B$}
\usefont{T1}{ppl}{m}{n}
\rput(1.4445312,0.72328126){$ A$}
\end{pspicture} 
}
}

\subfloat[]
{\label{fig:possible5d}
% Generated with LaTeXDraw 2.0.8
% Fri Apr 06 22:21:30 EDT 2012
% \usepackage[usenames,dvipsnames]{pstricks}
% \usepackage{epsfig}
% \usepackage{pst-grad} % For gradients
% \usepackage{pst-plot} % For axes
\scalebox{.65} % Change this value to rescale the drawing.
{
\begin{pspicture}(0,-0.88671875)(3.0690625,0.9267188)
\pscircle[linewidth=0.02,linestyle=dashed,dash=0.12cm 0.06cm,dimen=outer](1.97,-0.18671875){0.7}
\pscircle[linewidth=0.02,linestyle=dashed,dash=0.12cm 0.06cm,dimen=outer](0.97,-0.18671875){0.7}
\pscircle[linewidth=0.02,dimen=outer](1.47,0.21328124){0.3}
\usefont{T1}{ppl}{m}{n}
\rput(0.28453124,0.42328125){$\tilde A$}
\usefont{T1}{ppl}{m}{n}
\rput(2.6845312,0.42328125){$\tilde B$}
\usefont{T1}{ppl}{m}{n}
\rput(1.4445312,0.72328126){$ A$}
\end{pspicture} 
}
}
\subfloat[]
{\label{fig:possible5e}
% Generated with LaTeXDraw 2.0.8
% Fri Apr 06 22:17:56 EDT 2012
% \usepackage[usenames,dvipsnames]{pstricks}
% \usepackage{epsfig}
% \usepackage{pst-grad} % For gradients
% \usepackage{pst-plot} % For axes
\scalebox{.65} % Change this value to rescale the drawing.
{
\begin{pspicture}(0,-0.7367188)(3.0690625,0.77671874)
\pscircle[linewidth=0.02,linestyle=dashed,dash=0.12cm 0.06cm,dimen=outer](1.97,-0.03671875){0.7}
\pscircle[linewidth=0.02,linestyle=dashed,dash=0.12cm 0.06cm,dimen=outer](0.97,-0.03671875){0.7}
\pscircle[linewidth=0.02,dimen=outer](1.47,-0.43671876){0.3}
\usefont{T1}{ppl}{m}{n}
\rput(0.28453124,0.5732812){$\tilde A$}
\usefont{T1}{ppl}{m}{n}
\rput(2.6845312,0.5732812){$\tilde B$}
\usefont{T1}{ppl}{m}{n}
\rput(0.99453125,-0.32671875){$ A$}
\end{pspicture} 
}
}
\subfloat[]
{\label{fig:possible5f}
% Generated with LaTeXDraw 2.0.8
% Fri Apr 06 22:19:21 EDT 2012
% \usepackage[usenames,dvipsnames]{pstricks}
% \usepackage{epsfig}
% \usepackage{pst-grad} % For gradients
% \usepackage{pst-plot} % For axes
\scalebox{.65} % Change this value to rescale the drawing.
{
\begin{pspicture}(0,-0.7867187)(3.5690625,0.82671875)
\pscircle[linewidth=0.02,linestyle=dashed,dash=0.12cm 0.06cm,dimen=outer](2.57,-0.08671875){0.7}
\pscircle[linewidth=0.02,dimen=outer](1.57,-0.08671875){0.5}
\pscircle[linewidth=0.02,linestyle=dashed,dash=0.12cm 0.06cm,dimen=outer](1.47,-0.08671875){0.7}
\usefont{T1}{ppl}{m}{n}
\rput(1.2845312,-0.07671875){$ A$}
\usefont{T1}{ppl}{m}{n}
\rput(3.2845312,0.52328124){$\tilde B$}
\usefont{T1}{ppl}{m}{n}
\rput(0.7945312,0.52328124){$\tilde A$}
\end{pspicture} 
}
}

\subfloat[]
{\label{fig:possible5g}
% Generated with LaTeXDraw 2.0.8
% Fri Apr 06 22:12:46 EDT 2012
% \usepackage[usenames,dvipsnames]{pstricks}
% \usepackage{epsfig}
% \usepackage{pst-grad} % For gradients
% \usepackage{pst-plot} % For axes
\scalebox{.65} % Change this value to rescale the drawing.
{
\begin{pspicture}(0,-0.7367188)(3.0690625,0.77671874)
\pscircle[linewidth=0.02,linestyle=dashed,dash=0.12cm 0.06cm,dimen=outer](1.97,-0.03671875){0.7}
\pscircle[linewidth=0.02,linestyle=dashed,dash=0.12cm 0.06cm,dimen=outer](0.97,-0.03671875){0.7}
\pscircle[linewidth=0.02,dimen=outer](1.47,-0.03671875){0.3}
\usefont{T1}{ptm}{m}{n}
\rput(0.28453124,0.5732812){$\tilde A$}
\usefont{T1}{ptm}{m}{n}
\rput(2.6845312,0.5732812){$\tilde B$}
\usefont{T1}{ptm}{m}{n}
\rput(0.94453125,-0.02671875){$ A$}
\end{pspicture} 
}
}
\subfloat[]
{\label{fig:possible5h}
% Generated with LaTeXDraw 2.0.8
% Fri Apr 06 22:23:27 EDT 2012
% \usepackage[usenames,dvipsnames]{pstricks}
% \usepackage{epsfig}
% \usepackage{pst-grad} % For gradients
% \usepackage{pst-plot} % For axes
\scalebox{.65} % Change this value to rescale the drawing.
{
\begin{pspicture}(0,-0.83671874)(3.6690626,0.87671876)
\pscircle[linewidth=0.02,linestyle=dashed,dash=0.12cm 0.06cm,dimen=outer](2.67,-0.13671875){0.7}
\pscircle[linewidth=0.02,dimen=outer](1.57,-0.13671875){0.7}
\pscircle[linewidth=0.02,linestyle=dashed,dash=0.12cm 0.06cm,dimen=outer](2.17,-0.13671875){0.6}
\usefont{T1}{ppl}{m}{n}
\rput(2.1345313,0.67328125){$\tilde A$}
\usefont{T1}{ppl}{m}{n}
\rput(3.3345314,0.4732812){$\tilde B$}
\usefont{T1}{ppl}{m}{n}
\rput(0.8945312,0.4732812){$ A$}
\end{pspicture} 
}
}
\caption*%
	{%
		\possiblea
	}%
%\addtocounter{figure}{1}
\end{minipage}
\begin{minipage}{0.49\linewidth}%
\setcounter{subfigure}{0}
\centering
\makeatletter
\def\p@subfigure{\possibleb} 
\makeatother
\subfloat[]
{\label{fig:possible6a}
% Generated with LaTeXDraw 2.0.8
% Fri Apr 06 22:15:26 EDT 2012
% \usepackage[usenames,dvipsnames]{pstricks}
% \usepackage{epsfig}
% \usepackage{pst-grad} % For gradients
% \usepackage{pst-plot} % For axes
\scalebox{.65} % Change this value to rescale the drawing.
{
\begin{pspicture}(0,-0.7867187)(3.5690625,0.82671875)
\pscircle[linewidth=0.02,linestyle=dashed,dash=0.12cm 0.06cm,dimen=outer](2.57,-0.08671875){0.7}
\pscircle[linewidth=0.02,linestyle=dashed,dash=0.12cm 0.06cm,dimen=outer](1.97,-0.08671875){0.3}
\pscircle[linewidth=0.02,dimen=outer](1.47,-0.08671875){0.7}
\usefont{T1}{ppl}{m}{n}
\rput(1.4845313,-0.07671875){$\tilde B$}
\usefont{T1}{ppl}{m}{n}
\rput(3.1845312,0.62328124){$\tilde A$}
\usefont{T1}{ppl}{m}{n}
\rput(0.7945312,0.62328124){$ B$}
\end{pspicture} 
}
}
\subfloat[]
{\label{fig:possible6b}
% Generated with LaTeXDraw 2.0.8
% Fri Apr 06 22:14:33 EDT 2012
% \usepackage[usenames,dvipsnames]{pstricks}
% \usepackage{epsfig}
% \usepackage{pst-grad} % For gradients
% \usepackage{pst-plot} % For axes
\scalebox{.65} % Change this value to rescale the drawing.
{
\begin{pspicture}(0,-0.7867187)(3.5690625,0.82671875)
\pscircle[linewidth=0.02,linestyle=dashed,dash=0.12cm 0.06cm,dimen=outer](2.57,-0.08671875){0.7}
\pscircle[linewidth=0.02,linestyle=dashed,dash=0.12cm 0.06cm,dimen=outer](1.57,-0.08671875){0.5}
\pscircle[linewidth=0.02,dimen=outer](1.47,-0.08671875){0.7}
\usefont{T1}{ppl}{m}{n}
\rput(1.2845312,-0.07671875){$\tilde B$}
\usefont{T1}{ppl}{m}{n}
\rput(3.1845312,0.62328124){$\tilde A$}
\usefont{T1}{ppl}{m}{n}
\rput(0.7945312,0.62328124){$ B$}
\end{pspicture} 
}
}
\subfloat[]
{\label{fig:possible6c}
% Generated with LaTeXDraw 2.0.8
% Fri Apr 06 22:20:20 EDT 2012
% \usepackage[usenames,dvipsnames]{pstricks}
% \usepackage{epsfig}
% \usepackage{pst-grad} % For gradients
% \usepackage{pst-plot} % For axes
\scalebox{.65} % Change this value to rescale the drawing.
{
\begin{pspicture}(0,-0.88671875)(3.0690625,0.9267188)
\pscircle[linewidth=0.02,linestyle=dashed,dash=0.12cm 0.06cm,dimen=outer](1.97,-0.18671875){0.7}
\pscircle[linewidth=0.02,linestyle=dashed,dash=0.12cm 0.06cm,dimen=outer](0.97,-0.18671875){0.7}
\pscircle[linewidth=0.02,dimen=outer](1.47,-0.18671875){0.7}
\usefont{T1}{ppl}{m}{n}
\rput(0.28453124,0.42328125){$\tilde B$}
\usefont{T1}{ppl}{m}{n}
\rput(2.6845312,0.42328125){$\tilde A$}
\usefont{T1}{ppl}{m}{n}
\rput(1.4445312,0.72328126){$ B$}
\end{pspicture} 
}
}

\subfloat[]
{\label{fig:possible6d}
% Generated with LaTeXDraw 2.0.8
% Fri Apr 06 22:21:30 EDT 2012
% \usepackage[usenames,dvipsnames]{pstricks}
% \usepackage{epsfig}
% \usepackage{pst-grad} % For gradients
% \usepackage{pst-plot} % For axes
\scalebox{.65} % Change this value to rescale the drawing.
{
\begin{pspicture}(0,-0.88671875)(3.0690625,0.9267188)
\pscircle[linewidth=0.02,linestyle=dashed,dash=0.12cm 0.06cm,dimen=outer](1.97,-0.18671875){0.7}
\pscircle[linewidth=0.02,linestyle=dashed,dash=0.12cm 0.06cm,dimen=outer](0.97,-0.18671875){0.7}
\pscircle[linewidth=0.02,dimen=outer](1.47,0.21328124){0.3}
\usefont{T1}{ppl}{m}{n}
\rput(0.28453124,0.42328125){$\tilde B$}
\usefont{T1}{ppl}{m}{n}
\rput(2.6845312,0.42328125){$\tilde A$}
\usefont{T1}{ppl}{m}{n}
\rput(1.4445312,0.72328126){$ B$}
\end{pspicture} 
}
}
\subfloat[]
{\label{fig:possible6e}
% Generated with LaTeXDraw 2.0.8
% Fri Apr 06 22:17:56 EDT 2012
% \usepackage[usenames,dvipsnames]{pstricks}
% \usepackage{epsfig}
% \usepackage{pst-grad} % For gradients
% \usepackage{pst-plot} % For axes
\scalebox{.65} % Change this value to rescale the drawing.
{
\begin{pspicture}(0,-0.7367188)(3.0690625,0.77671874)
\pscircle[linewidth=0.02,linestyle=dashed,dash=0.12cm 0.06cm,dimen=outer](1.97,-0.03671875){0.7}
\pscircle[linewidth=0.02,linestyle=dashed,dash=0.12cm 0.06cm,dimen=outer](0.97,-0.03671875){0.7}
\pscircle[linewidth=0.02,dimen=outer](1.47,-0.43671876){0.3}
\usefont{T1}{ppl}{m}{n}
\rput(0.28453124,0.5732812){$\tilde B$}
\usefont{T1}{ppl}{m}{n}
\rput(2.6845312,0.5732812){$\tilde A$}
\usefont{T1}{ppl}{m}{n}
\rput(0.99453125,-0.32671875){$ B$}
\end{pspicture} 
}
}
\subfloat[]
{\label{fig:possible6f}
% Generated with LaTeXDraw 2.0.8
% Fri Apr 06 22:19:21 EDT 2012
% \usepackage[usenames,dvipsnames]{pstricks}
% \usepackage{epsfig}
% \usepackage{pst-grad} % For gradients
% \usepackage{pst-plot} % For axes
\scalebox{.65} % Change this value to rescale the drawing.
{
\begin{pspicture}(0,-0.7867187)(3.5690625,0.82671875)
\pscircle[linewidth=0.02,linestyle=dashed,dash=0.12cm 0.06cm,dimen=outer](2.57,-0.08671875){0.7}
\pscircle[linewidth=0.02,dimen=outer](1.57,-0.08671875){0.5}
\pscircle[linewidth=0.02,linestyle=dashed,dash=0.12cm 0.06cm,dimen=outer](1.47,-0.08671875){0.7}
\usefont{T1}{ppl}{m}{n}
\rput(1.2845312,-0.07671875){$ B$}
\usefont{T1}{ppl}{m}{n}
\rput(3.2845312,0.52328124){$\tilde A$}
\usefont{T1}{ppl}{m}{n}
\rput(0.7945312,0.52328124){$\tilde B$}
\end{pspicture} 
}
}

\subfloat[]
{\label{fig:possible6g}
% Generated with LaTeXDraw 2.0.8
% Fri Apr 06 22:12:46 EDT 2012
% \usepackage[usenames,dvipsnames]{pstricks}
% \usepackage{epsfig}
% \usepackage{pst-grad} % For gradients
% \usepackage{pst-plot} % For axes
\scalebox{.65} % Change this value to rescale the drawing.
{
\begin{pspicture}(0,-0.7367188)(3.0690625,0.77671874)
\pscircle[linewidth=0.02,linestyle=dashed,dash=0.12cm 0.06cm,dimen=outer](1.97,-0.03671875){0.7}
\pscircle[linewidth=0.02,linestyle=dashed,dash=0.12cm 0.06cm,dimen=outer](0.97,-0.03671875){0.7}
\pscircle[linewidth=0.02,dimen=outer](1.47,-0.03671875){0.3}
\usefont{T1}{ptm}{m}{n}
\rput(0.28453124,0.5732812){$\tilde B$}
\usefont{T1}{ptm}{m}{n}
\rput(2.6845312,0.5732812){$\tilde A$}
\usefont{T1}{ptm}{m}{n}
\rput(0.94453125,-0.02671875){$ B$}
\end{pspicture} 
}
}
\subfloat[]
{\label{fig:possible6h}
% Generated with LaTeXDraw 2.0.8
% Fri Apr 06 22:23:27 EDT 2012
% \usepackage[usenames,dvipsnames]{pstricks}
% \usepackage{epsfig}
% \usepackage{pst-grad} % For gradients
% \usepackage{pst-plot} % For axes
\scalebox{.65} % Change this value to rescale the drawing.
{
\begin{pspicture}(0,-0.83671874)(3.6690626,0.87671876)
\pscircle[linewidth=0.02,linestyle=dashed,dash=0.12cm 0.06cm,dimen=outer](2.67,-0.13671875){0.7}
\pscircle[linewidth=0.02,dimen=outer](1.57,-0.13671875){0.7}
\pscircle[linewidth=0.02,linestyle=dashed,dash=0.12cm 0.06cm,dimen=outer](2.17,-0.13671875){0.6}
\usefont{T1}{ppl}{m}{n}
\rput(2.1345313,0.67328125){$\tilde B$}
\usefont{T1}{ppl}{m}{n}
\rput(3.3345314,0.4732812){$\tilde A$}
\usefont{T1}{ppl}{m}{n}
\rput(0.8945312,0.4732812){$ B$}
\end{pspicture} 
}
}
\caption*%
	{%
		\possibleb
	}%
%\addtocounter{figure}{1}
\end{minipage}
{
\caption
{
\label{fig:possibleii}
The relevant topological configurations of $A$ and of $B$, relative to $\tilde A, \tilde B$.
}
}
\end{figure}

\noindent We will often make reference to the configurations depicted in Figures \ref{fig:possible} and \ref{fig:possibleii}.  If the appropriate three-disk subset of $\{A,B,\tilde A,\tilde B\}$ is in a topological configuration depicted in one of these figures, we will indicate this simply by saying that the corresponding configuration \emph{occurs}, for example that \ref{fig:possible3a} occurs.\medskip

\begin{proof}[Proof of Proposition \ref{prop:possible for all}]
Note that by the symmetries involved, it suffices to prove that $\{A, B, \tilde A\}$ must be in one of the topological configurations on the \possibletildea{} side of Figure \ref{fig:possible}.  Therefore we restrict our attention to this case from now on.

The following observation, which is an easy exercise, will be the key to our proof:

\begin{observation}
\label{obs circ in plane}
Fix $\ell_1$ and $\ell_2$ to be unequal straight lines in $\bbC$ both of which pass through the origin.  The lines $\ell_1$ and $\ell_2$ divide the plane into four regions, which we loosely refer to as \emph{quasi-quadrants}.  If $C$ is a variable metric circle in $\bbC$ which is not allowed to pass through the origin, nor to be tangent to either of $\ell_1$ and $\ell_2$, then the topological configuration of $\{C, \ell_1, \ell_2\}$ is uniquely determined by which of the four quasi-quadrants the circle $C$ passes through.  Note also that $C$ cannot pass through two diagonally opposite quasi-quadrants without passing through at least one of the two remaining quasi-quadrants.
\end{observation}

\noindent Then the idea of the proof of Proposition \ref{prop:possible for all} is to apply a M\"obius transformation sending one of the two points of $\partial A \cap \partial B$ to $\infty$.  The images of the circles $\partial A$ and $\partial B$ will act as the lines $\ell_1$ and $\ell_2$ of Observation \ref{obs circ in plane}, and $\partial \tilde A$ will act as $C$.

We make one preliminary notational convention.  First, orient $\partial A$ and $\partial B$ positively as usual, and let $\{u,v\} = \partial A\cap \partial B$.  Label $u$ and $v$ so that $u$ is the point of $\partial A \cap \partial B$ where $\partial A$ enters $B$, and $v$ is the point of $\partial A\cap \partial B$ where $\partial B$ enters $A$.  See Figure \ref{fig:u v ex} for an example.

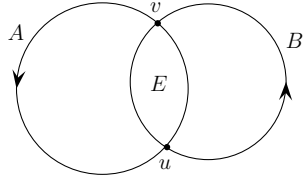
\begin{figure}[t]
\centering
% Generated with LaTeXDraw 2.0.8
% Mon Dec 19 23:14:10 EST 2011
% \usepackage[usenames,dvipsnames]{pstricks}
% \usepackage{epsfig}
% \usepackage{pst-grad} % For gradients
% \usepackage{pst-plot} % For axes
\scalebox{.7} % Change this value to rescale the drawing.
{
\begin{pspicture}(0,-1.6795312)(6.0290623,1.7067188)
\pscircle[linewidth=0.02,dimen=middle](1.92,-0.03671875){1.63}
\pscircle[linewidth=0.02,dimen=middle](3.93,0.09328125){1.48}
\psline[linewidth=0.1,linestyle=none,arrows=->](0.295,0)(0.29,-0.03671875)
\psline[linewidth=0.1,linestyle=none,arrows=->](5.38,-1)(5.41,0.09328125)
\psdots[dotsize=0.12](3.15,-1.1467188)
\psdots[dotsize=0.12](2.97,1.2032812)
\usefont{T1}{ptm}{m}{n}
\rput(0.28453124,1.0232812){$A$}
\usefont{T1}{ptm}{m}{n}
\rput(5.5745312,0.88328123){$B$}
\usefont{T1}{ptm}{m}{n}
\rput(2.9945313,0.06328125){$E$}
\usefont{T1}{ptm}{m}{n}
\rput(3.1145313,-1.4767188){$u$}
\usefont{T1}{ptm}{m}{n}
\rput(2.9345312,1.5032812){$v$}
\end{pspicture} 
}
\caption
{
\label{fig:u v ex}
The definitions of $u$ and $v$ in terms of the orientations on $\partial A$ and $\partial B$.
}
\end{figure}

Ultimately, we would like to say that if we fix overlapping $A$ and $B$, letting $\tilde A$ vary under the constraint that the general position hypothesis is not violated, then the topological configuration of $\{A, B, \tilde A\}$ is uniquely determined by two pieces of information:

\begin{itemize}
\item whether or not $v\in \tilde A$, and
\item which of the four regions $A\cap B, A\setminus B, B\setminus A, \bbC\setminus (A\cup B)$ the circle $\partial \tilde A$ passes through.
\end{itemize}

\begin{figure}[t]
\centering
\subfloat
{\label{fig:abcdefa}
% Generated with LaTeXDraw 2.0.8
% Wed Apr 04 18:04:27 EDT 2012
% \usepackage[usenames,dvipsnames]{pstricks}
% \usepackage{epsfig}
% \usepackage{pst-grad} % For gradients
% \usepackage{pst-plot} % For axes
\scalebox{.7} % Change this value to rescale the drawing.
{
\begin{pspicture}(0,-1.8)(4.0690627,1.8)
\pscircle[linewidth=0.02,dimen=outer](1.07,0.8){1.0}
\pscircle[linewidth=0.02,dimen=outer](2.87,0.8){1.0}
\pscircle[linewidth=0.02,linestyle=dashed](1.97,-0.8){1.0}
\usefont{T1}{ptm}{m}{n}
\rput(0.28453124,-0.19){$A$}
\usefont{T1}{ptm}{m}{n}
\rput(3.6845312,-0.19){$B$}
\usefont{T1}{ptm}{m}{n}
\rput(3.0945313,-1.49){$\tilde A$}
\end{pspicture} 
}
}\qquad
\subfloat
{\label{fig:abcdefb}
% Generated with LaTeXDraw 2.0.8
% Wed Apr 04 18:05:14 EDT 2012
% \usepackage[usenames,dvipsnames]{pstricks}
% \usepackage{epsfig}
% \usepackage{pst-grad} % For gradients
% \usepackage{pst-plot} % For axes
\scalebox{.7} % Change this value to rescale the drawing.
{
\begin{pspicture}(0,-1.8)(3.9890625,1.8)
\pscircle[linewidth=0.02,dimen=outer](1.07,-0.8){1.0}
\pscircle[linewidth=0.02,dimen=outer](2.87,-0.8){1.0}
\pscircle[linewidth=0.02,linestyle=dashed](1.97,0.8){1.0}
\usefont{T1}{ptm}{m}{n}
\rput(0.28453124,0.21){$A$}
\usefont{T1}{ptm}{m}{n}
\rput(3.5845313,0.21){$B$}
\usefont{T1}{ptm}{m}{n}
\rput(3.0945313,1.51){$\tilde A$}
\end{pspicture} 
}
}
\caption
{
\label{fig:nice description counter}
Two different topological configurations of three disks $\{A,B,\tilde A\}$, where $\partial \tilde A$ passes through the same components of $\bbC\setminus (\partial A\cup \partial B)$ in both cases.  We see that $\tilde A$ and $A\cap B$ do not meet in either case, so this example should not worry us too much in light of the hypotheses of Proposition \ref{prop:possible for all}.
}
\end{figure}
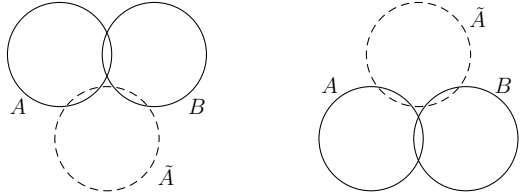

\noindent Unfortunately, this is not completely true.  There is a minor obstruction, illustrated in Figure \ref{fig:nice description counter}.  However, we will see that this is the only possible obstruction, and the nice classification described in this paragraph otherwise holds:

\begin{lemma}
\label{backpack}
Let $\{A,B\}$ and $\{\tilde A,\tilde B\}$ be pairs of overlapping metric closed disks in general position, similarly $\{A', B'\}$ and $\{\tilde A', \tilde B'\}$.  Define $u'$ and $v'$ for $A'$ and $B$' as we defined $u$ and $v$ for $A$ and $B$ in Figure \ref{fig:u v ex}.  Suppose that $v\in \tilde A$ if and only if $v' \in \tilde A'$, and that the subset of the four regions $A'\cap B', A'\setminus B', B'\setminus A', \bbC\setminus (A'\cup B')$ through which $\partial \tilde A'$ passes corresponds to the subset of the regions $A\cap B, A\setminus B, B\setminus A, \bbC\setminus (A\cup B)$ through which $\partial \tilde A$ passes, in the natural way.  Then $\{A, B, \tilde A\}$ and $\{A', B', \tilde A'\}$ are in the same topological configuration, unless one triple is arranged as in the left side of Figure \ref{fig:nice description counter} and the other is arranged as in the right side of the same figure.
\end{lemma}

\begin{proof}
Let $m$ be a M\"obius transformation sending $v$ to $\infty$ and $u$ to the origin, letting $\ell_1 = m(\partial A)$ and $\ell_2 = m(\partial B)$.  By the general position hypothesis, letting $C = m(\partial \tilde A)$ we have that $C$ is a circle as in the statement of Observation \ref{obs circ in plane}.  Similarly, let $m'$ be a M\"obius transformation sending $v'$ to $\infty$ and $u'$ to the origin.  Then by Observation \ref{obs circ in plane} we may apply an orientation-preserving homeomorphism $\psi : \bbC \to \bbC$ so that $\ell_1 = m(\partial A) = \psi \circ m' (\partial A')$, similarly $\ell_2 = m(\partial B) = \psi \circ m' (\partial B')$ and $C = m(\partial \tilde A) = \psi \circ m' (\partial \tilde A')$.

Now, if we can choose $\psi$ so that $m(\infty) = \psi\circ m'(\infty)$, then we will be done, because $m^{-1} \circ \psi \circ m'$ will be an orientation-preserving homeomorphism $\bbC \to \bbC$ identifying $\{A, B, \tilde A\}$ with $\{A', B', \tilde A'\}$.  Clearly there is such a $\psi$ so long as $\bbC \setminus m(A \cup B \cup \tilde A)$ equivalently $\bbC \setminus m(A' \cup B' \cup \tilde A')$ are connected.  This happens if and only if $\bbC \setminus (A \cup B \cup \tilde A)$ equivalently $\bbC \setminus (A' \cup B' \cup \tilde A')$ are connected, and it is easy to show that this fails only for the two configurations shown in Figure \ref{fig:nice description counter}.  The lemma follows.
\end{proof}

We can now complete the proof by exhaustion.  We will break the proof into two major cases, depending on whether $v\in \tilde A$ or $v\not\in \tilde A$.  We will not make reference to Observation \ref{obs circ in plane} again, so we overload terminology, using the term quasi-quadrants from now on to refer to the four regions $A\cap B, A\setminus B, B\setminus A, \bbC\setminus (A\cup B)$.

The following observation will be our source of contradictions to the hypotheses of Proposition \ref{prop:possible for all}:

\begin{observation}
\label{obsC}
Suppose that the hypotheses of Proposition \ref{prop:possible for all} hold.  Then we have
\begin{itemize}
\item that $\tilde A$ meets both $A\setminus B$ and $A\cap B$, and
\item that $B\setminus A$ is not contained in $\tilde A$.
\end{itemize}
\end{observation}

\noindent To see why, note that if $\tilde A$ does not meet $A\setminus B$, then $\tilde A$ cannot possibly meet $\tilde A \setminus \tilde B$.  We must have that $\tilde A$ meets $A\cap B$ for a similar reason.  Also, if $B\setminus A$ is contained in $\tilde A$, then $B\setminus A$ cannot possibly meet $\tilde B\setminus \tilde A$.

In the forthcoming case analysis, we will rely on the reader to supply his own drawings of the cases which we throw out by Observation \ref{obsC}.  In general it is not hard to draw an example of a configuration $\{A,B,\tilde A\}$ given the type of hypotheses we write down below, and once a single example is drawn Lemma \ref{backpack} ensures that it is typically the only one, up to topological configuration.

\begin{libcase}
$v\in \tilde A$
\end{libcase}

We now consider the possibilities depending on how many of the quasi-quadrants $\partial \tilde A$ passes through.  If it passes through only one, then it is easy to see that it must be $\bbC \setminus (A\cup B)$, otherwise we would violate $v\in \tilde A$.  Then $A\cup B \subset \tilde A$, in particular $B\setminus A \subset \tilde A$, so we may ignore this possibility by Observation \ref{obsC}.

Next, suppose that $\partial \tilde A$ passes through exactly two (necessarily adjacent) quasi-quadrants.  Then which two it hits is exactly determined by which one of the four arcs $\partial A \cap B, \partial A \setminus B, \partial B \cap A, \partial B \setminus A$ it hits.  It cannot be $\partial A \cap B$ nor $\partial B \cap A$ without violating $v \in \tilde A$.  If it is $\partial A \setminus B$ then $B \subset \tilde A$, so we may ignore this possibility by Observation \ref{obsC}.  The remaining possibility is represented as \ref{fig:possible3b}.

Now suppose that $\partial \tilde A$ passes through exactly three quasi-quadrants.  For brevity we will indicate which three it hits by saying instead which one it misses.  If it misses $B\setminus A$, then $B\setminus A \subset \tilde A$, so we throw this case out Observation \ref{obsC}.  Next, it cannot miss $\bbC\setminus (A\cup B)$ without violating $v\in \tilde A$.  The two remaining cases are represented in \ref{fig:possible3a} and \ref{fig:possible3c}.

Finally, the case where $\partial \tilde A$ passes through all four quasi-quadrants is drawn in \ref{fig:possible3d}.

\begin{libcase}
$v\not\in \tilde A$
\end{libcase}

Suppose first that $\partial \tilde A$ hits exactly one of the quasi-quadrants.  Then $\tilde A$ is contained in that quasi-quadrant.  But $\tilde A$ must meet at least two quasi-quadrants, by Observation \ref{obsC}, if the hypotheses of Proposition \ref{prop:possible for all} hold.

Next, suppose that $\partial \tilde A$ meets exactly two quasi-quadrants.  Again we indicate which two by saying which one of the four arcs $\partial A \cap B, \partial A \setminus B, \partial B \cap A, \partial B \setminus A$ it hits.  If it is $\partial B \setminus A$ or $\partial A\cap B$, then $\tilde A$ is disjoint from $A \setminus B$, and if it is $\partial A \setminus B$, then $\tilde A$ is disjoint from $A\cap B$.  Thus we throw these cases out by Observation \ref{obsC}.  The final possibility is represented in \ref{fig:possible3f}.

Suppose now that $\partial \tilde A$ meets exactly three quasi-quadrants.  As before we indicate which three by indicating which one it misses.  If it misses $A \setminus B$, then we will get that $A\setminus B$ and $\tilde A$ are disjoint.  If it misses $A\cap B$, then the disks are in one of the configurations of Figure \ref{fig:nice description counter}, in which case $\tilde A$ and $A\cap B$ are disjoint.  We throw these cases out by Obsevation \ref{obsC}.  The remaining two cases are depicted in \ref{fig:possible3g} and \ref{fig:possible3h}.

Last, if $\partial \tilde A$ passes through all four quasi-quadrants, and $v\not\in \tilde A$, then the disks are configured as in \ref{fig:possible3e}.\medskip

\noindent This completes the proof of Proposition \ref{prop:possible for all}.
\end{proof}

\section{Preliminary topological lemmas\twostars}
\label{chap:assorted}

In the section after this one, we will introduce a tool, called \emph{torus parametrization}, for working with fixed-point index.  This tool will handle most of our cases for us relatively painlessly, but for some special cases we will need extra lemmas.  This section is devoted to the statements and proofs of those lemmas.  We also state and prove some simplifying facts that greatly cut down the number of cases we will eventually need to check.\medskip

First:

\begin{lemma}
\label{lem:gen pos lemma}
Suppose $K$ and $\tilde K$ are closed Jordan domains in transverse position.  Then $\partial K$ and $\partial \tilde K$ meet a finite, even number of times, by compactness and the transverse position hypothesis.  In particular:

Suppose that $z\in \partial K \cap \partial \tilde K$.  Orient $\partial K$ and $\partial \tilde K$ positively with respect to $K$ and $\tilde K$ as usual.  Then one of the following two mutually exclusive possibilities holds at the point $z$.
\begin{enumerate}
\item \label{case:lem:orientations:1} The curve $\partial \tilde K$ is entering $K$, and the curve $\partial K$ is exiting $\tilde K$.
\item \label{case:lem:orientations:2} The curve $\partial K$ is entering $\tilde K$, and the curve $\partial \tilde K$ is exiting $K$.
\end{enumerate}
Thus as we traverse $\partial K$, we alternate arriving at points of $\partial K \cap \partial \tilde K$ where (\ref{case:lem:orientations:1}) occurs and those where (\ref{case:lem:orientations:2}) occurs, and the same holds as we traverse $\partial \tilde K$.
\end{lemma}

\noindent This is easy to check with a simple drawing.\medskip

Our next lemma characterizes the ways in which two convex closed Jordan domains may intersect:

\begin{lemma}
\label{prop:topo confo compact conv}
Let $K$ and $\tilde K$ be convex closed Jordan domains in transverse position, so that $\partial K$ and $\partial \tilde K$ meet $2M>0$ times.  Suppose that $K'$ and $\tilde K'$ are also convex closed Jordan domains in transverse position so that $\partial K$ and $\partial \tilde K$ meet $2M>0$ times.  Then $\{K,\tilde K\}$ and $\{K',\tilde K'\}$ are in the same topological configuration.
\end{lemma}

\begin{figure}[t]
\centering
% Generated with LaTeXDraw 2.0.8
% Fri Dec 09 18:30:53 EST 2011
% \usepackage[usenames,dvipsnames]{pstricks}
% \usepackage{epsfig}
% \usepackage{pst-grad} % For gradients
% \usepackage{pst-plot} % For axes
\scalebox{1} % Change this value to rescale the drawing.
{
\begin{pspicture}(0,-3.4667188)(8.129063,2.8717186)
\psdots[dotsize=0.12](4.81,0.45828125)
\psline[linewidth=0.02cm,arrowsize=0.05291667cm 2.0,arrowlength=1.4,arrowinset=0.4]{->}(4.81,0.45828125)(1.05,-1.9817188)
\psline[linewidth=0.02cm,linestyle=dotted,dotsep=0.16cm](4.81,0.45828125)(6.13,0.45828125)
\psarc[linewidth=0.02](4.81,0.45828125){0.25}{0.0}{212.7}
\usefont{T1}{ptm}{m}{n}
\rput(5.134531,2.6682813){$P_1$}
\usefont{T1}{ptm}{m}{n}
\rput(1.6945312,-0.35171875){$P_2$}
\usefont{T1}{ptm}{m}{n}
\rput(6.494531,-2.2917187){$P_3$}
\usefont{T1}{ptm}{m}{n}
\rput(2.7645311,2.0682812){$\tilde P_1$}
\usefont{T1}{ptm}{m}{n}
\rput(2.9445312,-2.1717188){$\tilde P_2$}
\usefont{T1}{ptm}{m}{n}
\rput(7.5245314,-0.89171875){$\tilde P_3$}
\pscustom[linewidth=0.02]
{
\newpath
\moveto(6.65,2.3682814)
\lineto(6.62,2.3782814)
\curveto(6.545,2.3782814)(6.415,2.3782814)(6.36,2.3782814)
\curveto(6.305,2.3782814)(6.175,2.3782814)(6.1,2.3782814)
\curveto(6.025,2.3782814)(5.875,2.3782814)(5.8,2.3782814)
\curveto(5.725,2.3782814)(5.555,2.3782814)(5.46,2.3782814)
\curveto(5.365,2.3782814)(5.205,2.3732812)(5.14,2.3682814)
\curveto(5.075,2.3632812)(4.91,2.3432813)(4.81,2.3282812)
\curveto(4.71,2.3132813)(4.545,2.2782812)(4.48,2.2582812)
\curveto(4.415,2.2382812)(4.285,2.1932812)(4.22,2.1682813)
\curveto(4.155,2.1432812)(4.015,2.0932813)(3.94,2.0682812)
\curveto(3.865,2.0432813)(3.695,1.9832813)(3.6,1.9482813)
\curveto(3.505,1.9132812)(3.275,1.8382813)(3.14,1.7982812)
\curveto(3.005,1.7582812)(2.83,1.7032813)(2.79,1.6882813)
\curveto(2.75,1.6732812)(2.67,1.6432812)(2.63,1.6282812)
\curveto(2.59,1.6132812)(2.435,1.5282812)(2.32,1.4582813)
\curveto(2.205,1.3882812)(2.05,1.2532812)(2.01,1.1882813)
\curveto(1.97,1.1232812)(1.92,0.8932812)(1.91,0.72828126)
\curveto(1.9,0.56328124)(1.92,0.24828126)(1.95,0.09828125)
\curveto(1.98,-0.05171875)(2.11,-0.38171875)(2.21,-0.56171876)
\curveto(2.31,-0.74171877)(2.47,-1.0067188)(2.53,-1.0917188)
\curveto(2.59,-1.1767187)(2.71,-1.3417188)(2.77,-1.4217187)
\curveto(2.83,-1.5017188)(2.945,-1.6467187)(3.0,-1.7117188)
\curveto(3.055,-1.7767187)(3.21,-1.9467187)(3.31,-2.0517187)
\curveto(3.41,-2.1567187)(3.605,-2.3367188)(3.7,-2.4117188)
\curveto(3.795,-2.4867187)(3.945,-2.5867188)(4.0,-2.6117187)
\curveto(4.055,-2.6367188)(4.165,-2.6967187)(4.22,-2.7317188)
\curveto(4.275,-2.7667189)(4.385,-2.8167188)(4.44,-2.8317187)
\curveto(4.495,-2.8467188)(4.685,-2.8367188)(4.82,-2.8117187)
\curveto(4.955,-2.7867188)(5.21,-2.7217188)(5.33,-2.6817188)
\curveto(5.45,-2.6417189)(5.66,-2.5417187)(5.75,-2.4817188)
\curveto(5.84,-2.4217188)(6.0,-2.2867188)(6.07,-2.2117188)
\curveto(6.14,-2.1367188)(6.295,-1.9617188)(6.38,-1.8617188)
\curveto(6.465,-1.7617188)(6.575,-1.6167188)(6.6,-1.5717187)
\curveto(6.625,-1.5267187)(6.695,-1.4117187)(6.74,-1.3417188)
\curveto(6.785,-1.2717187)(6.865,-1.1367188)(6.9,-1.0717187)
\curveto(6.935,-1.0067188)(7.015,-0.8617188)(7.06,-0.78171873)
\curveto(7.105,-0.70171875)(7.18,-0.5417187)(7.21,-0.46171874)
\curveto(7.24,-0.38171875)(7.285,-0.23171875)(7.3,-0.16171876)
\curveto(7.315,-0.09171875)(7.345,0.05328125)(7.36,0.12828125)
\curveto(7.375,0.20328125)(7.405,0.35828125)(7.42,0.43828124)
\curveto(7.435,0.5182812)(7.45,0.6982812)(7.45,0.79828125)
\curveto(7.45,0.8982813)(7.435,1.0682813)(7.42,1.1382812)
\curveto(7.405,1.2082813)(7.365,1.3482813)(7.34,1.4182812)
\curveto(7.315,1.4882812)(7.245,1.6482812)(7.2,1.7382812)
\curveto(7.155,1.8282813)(7.08,1.9682813)(7.05,2.0182812)
\curveto(7.02,2.0682812)(6.955,2.1582813)(6.92,2.1982813)
\curveto(6.885,2.2382812)(6.815,2.2982812)(6.78,2.3182812)
\curveto(6.745,2.3382812)(6.71,2.3632812)(6.65,2.3682814)
}
\pscustom[linewidth=0.02,linestyle=dashed,dash=0.16cm 0.16cm]
{
\newpath
\moveto(6.47,1.7682812)
\lineto(6.31,1.8782812)
\curveto(6.23,1.9282813)(6.095,2.0082812)(6.04,2.0382812)
\curveto(5.985,2.0682812)(5.865,2.1332812)(5.8,2.1682813)
\curveto(5.735,2.2032812)(5.6,2.2582812)(5.53,2.2782812)
\curveto(5.46,2.2982812)(5.32,2.3382812)(5.25,2.3582811)
\curveto(5.18,2.3782814)(5.025,2.4132812)(4.94,2.4282813)
\curveto(4.855,2.4432812)(4.695,2.4582813)(4.62,2.4582813)
\curveto(4.545,2.4582813)(4.385,2.4532812)(4.3,2.4482813)
\curveto(4.215,2.4432812)(4.035,2.4032812)(3.94,2.3682814)
\curveto(3.845,2.3332813)(3.68,2.2532814)(3.61,2.2082813)
\curveto(3.54,2.1632812)(3.38,2.0432813)(3.29,1.9682813)
\curveto(3.2,1.8932812)(3.045,1.7482812)(2.98,1.6782813)
\curveto(2.915,1.6082813)(2.79,1.4582813)(2.73,1.3782812)
\curveto(2.67,1.2982812)(2.57,1.1282812)(2.53,1.0382812)
\curveto(2.49,0.9482812)(2.415,0.73828125)(2.38,0.61828125)
\curveto(2.345,0.49828124)(2.28,0.28828126)(2.25,0.19828124)
\curveto(2.22,0.10828125)(2.17,-0.05171875)(2.15,-0.12171875)
\curveto(2.13,-0.19171876)(2.095,-0.35171875)(2.08,-0.44171876)
\curveto(2.065,-0.53171873)(2.05,-0.69671875)(2.05,-0.77171874)
\curveto(2.05,-0.8467187)(2.085,-1.0067188)(2.12,-1.0917188)
\curveto(2.155,-1.1767187)(2.27,-1.3267188)(2.35,-1.3917187)
\curveto(2.43,-1.4567188)(2.585,-1.5717187)(2.66,-1.6217188)
\curveto(2.735,-1.6717187)(2.875,-1.7467188)(2.94,-1.7717187)
\curveto(3.005,-1.7967187)(3.175,-1.8567188)(3.28,-1.8917187)
\curveto(3.385,-1.9267187)(3.62,-1.9867188)(3.75,-2.0117188)
\curveto(3.88,-2.0367188)(4.155,-2.0767188)(4.3,-2.0917187)
\curveto(4.445,-2.1067188)(4.735,-2.1267188)(4.88,-2.1317186)
\curveto(5.025,-2.1367188)(5.255,-2.1417189)(5.34,-2.1417189)
\curveto(5.425,-2.1417189)(5.66,-2.1117187)(5.81,-2.0817187)
\curveto(5.96,-2.0517187)(6.165,-2.0017188)(6.22,-1.9817188)
\curveto(6.275,-1.9617188)(6.425,-1.8917187)(6.52,-1.8417188)
\curveto(6.615,-1.7917187)(6.75,-1.6967187)(6.79,-1.6517187)
\curveto(6.83,-1.6067188)(6.91,-1.4767188)(6.95,-1.3917187)
\curveto(6.99,-1.3067187)(7.055,-1.0917188)(7.08,-0.96171874)
\curveto(7.105,-0.83171874)(7.105,-0.50171876)(7.08,-0.30171874)
\curveto(7.055,-0.10171875)(7.015,0.17328125)(7.0,0.24828126)
\curveto(6.985,0.32328126)(6.95,0.47328126)(6.93,0.54828125)
\curveto(6.91,0.62328124)(6.865,0.80328125)(6.84,0.90828127)
\curveto(6.815,1.0132812)(6.77,1.1732812)(6.75,1.2282813)
\curveto(6.73,1.2832812)(6.69,1.3832812)(6.67,1.4282813)
\curveto(6.65,1.4732813)(6.615,1.5482812)(6.6,1.5782813)
\curveto(6.585,1.6082813)(6.555,1.6632812)(6.54,1.6882813)
\curveto(6.525,1.7132813)(6.5,1.7532812)(6.47,1.7682812)
}
\usefont{T1}{ptm}{m}{n}
\rput(0.90453124,-2.2717187){$R_\theta$}
\usefont{T1}{ptm}{m}{n}
\rput(7.324531,2.1082811){$K$}
\usefont{T1}{ptm}{m}{n}
\rput(5.994531,1.6082813){$\tilde K$}
\usefont{T1}{ptm}{m}{n}
\rput(4.684531,0.92828125){$\theta$}
\usefont{T1}{ptm}{m}{n}
\rput(4.914531,0.12828125){$w$}
%\psdots[dotsize=0.12](6.89,0.78328127)
%\psdots[dotsize=0.12](7.45,1.0632813)
%\psdots[dotsize=0.12](4.49,-2.8667188)
%\psdots[dotsize=0.12](2.45,0.84328127)
\usefont{T1}{ptm}{m}{n}
%\rput(7.7945313,1.0532813){$u$}
\usefont{T1}{ptm}{m}{n}
%\rput(6.5845313,0.63328123){$\tilde u$}
\usefont{T1}{ptm}{m}{n}
%\rput(2.7845314,0.7732813){$\tilde u'$}
\usefont{T1}{ptm}{m}{n}
%\rput(4.534531,-3.0667188){$u'$}
\psdots[dotsize=0.12](2.51,-1.03671876)
\psdots[dotsize=0.12](2.21,-1.2367187)
\usefont{T1}{ptm}{m}{n}
\rput(2.8645312,-1.1467187){$z_\theta$}
\usefont{T1}{ptm}{m}{n}
\rput(1.7945312,-1.1667187){$\tilde z_\theta$}
\end{pspicture} 
}
\caption
{
\label{fig:p q increase nice example}
Two convex closed Jordan domains $K$ and $\tilde K$ in transverse position, with boundaries meeting at six points.  As $\theta$ varies positively, the ray $R_{\theta}$ scans around the boundaries of both $K$ and $\tilde K$ positively.
}
\end{figure}
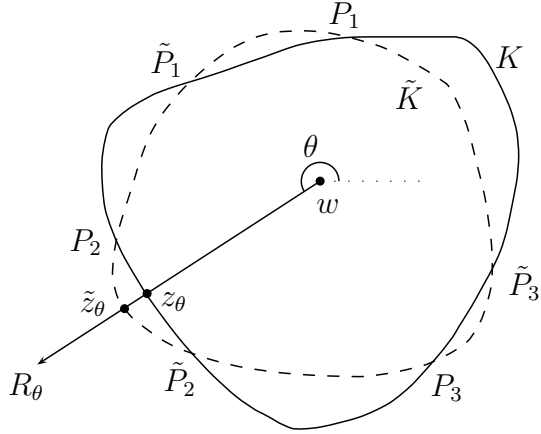

\begin{proof}
For the following construction, see Figure \ref{fig:p q increase nice example}.  Let $w$ be a common interior point of $K$ and $\tilde K$.  Let $R_\theta$ be the ray emanating from the point $w$ at an angle of $\theta$ from the positive real direction.  Let $P_i$ be the points of $\partial K \cap \partial \tilde K$ where $\partial K$ is entering $\tilde K$, and let $\tilde P_i$ be those where $\partial \tilde K$ is entering $K$.  Define $w', R'_\theta, P'_i, \tilde P'_i$ analogously for $K'$ and $\tilde K'$.  Identify $\bbS^1$ with the interval $[0,2\pi]$ with its endpoints identified, and define a homeomorphism $\bbS^1 \to \bbS^1$, denoting the image of $\theta\in [0,2\pi]$ by $\theta'$, so that $R_\theta$ hits a point $P_i$ if and only if $R'_{\theta'}$ hits a point $P'_i$, similarly for $\tilde P_i$ and $\tilde P'_i$.  Define homeomorphisms $R_\theta \to R'_{\theta'}$ piecewise linearly on the components $R_\theta \setminus (\partial K \cup \partial \tilde K) \to R'_{\theta'} \setminus (\partial K' \cup \partial \tilde K')$.  Then these homeomorphisms glue to an orientation-preserving homeomorphism $\bbC \to \bbC$ sending $\{K,\tilde K\}$ to $\{K',\tilde K'\}$.
\end{proof}

Lemma \ref{prop:topo confo compact conv} is very much false if we omit the condition that the Jordan domains are convex.  Which \emph{a priori} topological configurations can occur for two Jordan curves in transverse position is a poorly understood question, and is known as the study of \emph{meanders}\footnote{Thanks to Thomas Lam for informing us of the topic of meander theory}.  We are fortunate that our setting is nice enough that a statement like Lemma \ref{prop:topo confo compact conv} is possible.  The clean construction we use in our proof is due to Nic Ford and Jordan Watkins.\medskip

We now take a moment to introduce some notation we use throughout the rest of the article.  Let $\gamma$ be an oriented Jordan curve.  Let $a,b\in \gamma$ be distinct.  Then $[a\to b]_\gamma$ is the oriented closed sub-arc of $\gamma$ starting at $a$ and ending at $b$.  Then for example $[a\to b]_\gamma \cap [b\to a]_\gamma = \{a,b\}$ and $ [a\to b]_{\gamma} \cup [b\to a]_\gamma = \gamma$.

Throughout the rest of this section, let $\{A,B\}$ and $\{\tilde A,\tilde B\}$ be pairs of overlapping closed disks in general position.  We label $\{u,v\} = \partial A\cap \partial B$ as in the preceding section, see Figure \ref{fig:u v ex} on p.\ \pageref{fig:u v ex} for a reminder.  Label $\tilde u$ and $\tilde v$ analogously.  We denote $E = A\cap B$ and $\tilde E = \tilde A\cap \tilde B$, and loosely refer to these as \emph{eyes}.  The rest of this section consists of the proofs of an assortment of lemmas about these disks, which we give without further comment.

\begin{lemma}
\label{lem1}
The Jordan curves $\partial E$ and $\partial \tilde E$ meet exactly 0, 2, 4, or 6 times.
\end{lemma}

\begin{proof}
That they meet an even number of times is a consequence of the general position hypothesis.  There is an immediate upper bound of 8 meeting points because each of $\partial E$ and $\partial \tilde E$ is the union of two circular arcs.  Suppose for contradiction that $\partial E$ and $\partial \tilde E$ meet 8 times.  Thus every meeting point of one of the circles $\partial A$ and $\partial B$ with one of $\partial \tilde A$ and $\partial \tilde B$ lies in $\partial E\cap \partial \tilde E$.  It follows that $\partial (A\cup B)$ does not meet $\partial (\tilde A \cup \tilde B)$.  But these are Jordan curves, so then we have either that one of $A\cup B$ and $\tilde A\cup \tilde B$ contains the other, or that they are disjoint.  They cannot be disjoint because $\partial E$ and $\partial \tilde E$ meet (8 times) by hypothesis, so suppose without loss of generality that $\tilde A\cup \tilde B$ is contained in $A\cup B$, in particular in its interior by the general position hypothesis.  Then the sub-arc $\partial E \cap \partial A$ must enter the region $\tilde A\cup \tilde B$ somewhere, so that it may intersect $\partial \tilde E$, a contradiction.
\end{proof}

\begin{figure}[t]
\centering
\subfloat[]
{
\scalebox{.8}
{
\begin{pspicture}(-2.2,2.2)(2.2,-2)
\usefont{T1}{ppl}{m}{n}
%\psdots(-2,0)
\rput(-2.3,0){$\tilde E$}
\psdots(-1,-1.73)
\rput(-1,-2){$v$}
\psdots(1,-1.73)
\rput(1,-2){$\tilde u$}
%\psdots(2,0)
\rput(2.3,0){$E$}
\psdots(1,1.73)
\rput(1,2){$\tilde v$}
\psdots(-1,1.73)
\rput(-1,2){$u$}
\pspolygon[linestyle=dashed](-2,0)(1,-1.73)(1,1.73)
\pspolygon(-1,-1.73)(2,0)(-1,1.73)
%\psdots(0,0)
\end{pspicture}
}
}\qquad
\subfloat[]
{
\scalebox{.8}
{
\begin{pspicture}(-2.2,2.2)(2.2,-2)
\usefont{T1}{ppl}{m}{n}
%\psdots(-2,0)
\rput(-2.3,0){$\tilde E$}
\psdots(-1,-1.73)
\rput(-1,-2){$v$}
\psdots(1,-1.73)
\rput(1,-2){$\tilde v$}
%\psdots(2,0)
\rput(2.3,0){$E$}
\psdots(1,1.73)
\rput(1,2){$\tilde u$}
\psdots(-1,1.73)
\rput(-1,2){$u$}
\pspolygon[linestyle=dashed](-2,0)(1,-1.73)(1,1.73)
\pspolygon(-1,-1.73)(2,0)(-1,1.73)
%\psdots(0,0)
\end{pspicture}
}
}
\caption
{
\label{fig:possible sixers}
The possible relevant topological configurations for two generally positioned eyes whose boundaries meet at six points.  There are two remaining possibilities, not depicted, obtained by simultaneously swapping $u$ with $v$ and $\tilde u$ with $\tilde v$, which are irrelevant by symmetry .
}
\end{figure}
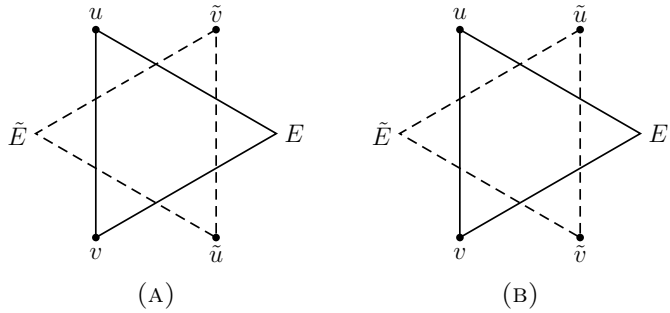

\begin{lemma}
\label{lem2}
Suppose that $\partial E$ and $\partial \tilde E$ meet 6 times, that $A\setminus B$ and $\tilde A\setminus \tilde B$ meet, and that $B\setminus A$ and $\tilde B\setminus \tilde A$ meet.  Then $\{ E, u, v, \tilde E, \tilde u, \tilde v\}$ are in one of the two topological configurations represented in Figure \ref{fig:possible sixers}, up to possibly simultaneously swapping $u$ with $v$ and $\tilde u$ with $\tilde v$.
\end{lemma}

\begin{proof}
By Lemma \ref{prop:topo confo compact conv}, if $\partial E$ and $\partial \tilde E$ meet 6 times then they are in the topological configuration shown in Figure \ref{fig:labeled jew star}.  We denote by $\epsilon_i$ the connected components of $\partial E \setminus \partial \tilde E$, and by $\tilde\epsilon_i$ the connected components of $\partial \tilde E\setminus \partial E$, labeled as in Figure \ref{fig:labeled jew star}.  We consider the indices of the $\epsilon_i$ and $\tilde\epsilon_i$ only modulo $6$.  For example, we write $\epsilon_{2+5} = \epsilon_1$.

\begin{figure}[t]
\centering
\scalebox{.8}{
\begin{pspicture}(-2.2,2.2)(2.2,-2)
%\psdots(-2,0)
\rput(-2.3,0){$\tilde\epsilon_2$}
%\psdots(-1,-1.73)
\rput(-1,-2){$\epsilon_3$}
%\psdots(1,-1.73)
\rput(1,-2){$\tilde \epsilon_4$}
%\psdots(2,0)
\rput(2.3,0){$\epsilon_5$}
%\psdots(1,1.73)
\rput(1,2){$\tilde \epsilon_6$}
%\psdots(-1,1.73)
\rput(-1,2){$\epsilon_1$}

\rput(-.7,0){$\epsilon_2$}
\rput(.7,0){$\tilde\epsilon_5$}

\rput(.4,.67){$\epsilon_6$}
\rput(-.4,.64){$\tilde\epsilon_1$}
\rput(.4,-.67){$\epsilon_4$}
\rput(-.4,-.67){$\tilde\epsilon_3$}

\pspolygon[linestyle=dashed](-2,0)(1,-1.73)(1,1.73)
\pspolygon(-1,-1.73)(2,0)(-1,1.73)
\end{pspicture}
}
\caption
{
\label{fig:labeled jew star} The components of $\partial E\setminus \partial \tilde E$ and $\partial \tilde E\setminus \partial E$ for two transversely positioned convex closed Jordan domains $E$ and $\tilde E$ meeting at six points.  The solid curve represents $\partial E$, and the dashed curve represents $\partial \tilde E$.
}
\end{figure}
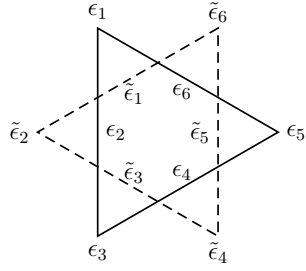

\noindent Proposition \ref{prop:possible for all} allows us to make the following observation:

\begin{observation}
\label{lem:not even one lol}
Neither $\tilde u$ nor $\tilde v$ may lie in $E$, and neither $u$ nor $v$ may lie in $\tilde E$.
\end{observation}

\noindent To see why, note that if $\partial E$ and $\partial \tilde E$ meet six times, then by the pigeonhole principle at least one of $\partial A$ and $\partial B$ must meet $\partial \tilde E$ at least three times.  Thus at least one of \ref{fig:possible5g} and \ref{fig:possible6g} must occur.  Thus $\tilde u$ and $\tilde v$ lie outside of at least one of $A$ and $B$, but $E = A\cap B$, thus neither $\tilde u$ nor $\tilde v$ lies in $E$.  The other part follows identically.

Thus we may assume that $u\in \epsilon_1$.  Then $v$ lies along $\epsilon_3$ or $\epsilon_5$.  By relabeling the $\epsilon_i$ and switching the roles of $u$ and $v$ as necessary, we may assume that $v\in \epsilon_3$.  Our proof will be done once we show that neither $\tilde u$ nor $\tilde v$ may lie along $\tilde \epsilon_2$.  Suppose for contradiction that $\tilde u$ lies along $\tilde \epsilon_2$.  Then $\tilde v$ lies along either $\tilde \epsilon_4$ or $\tilde\epsilon_6$.  If $\tilde v\in \tilde\epsilon_4$, then the circular arc $[v\to u]_{\partial E}$ meets the circular arc $[\tilde v\to \tilde u]_{\partial \tilde E}$ three times, a contradiction.  Similarly, if $\tilde v\in \tilde\epsilon_6$, then the circular arc $[v\to u]_{\partial E}$ meets the circular arc $[\tilde u\to \tilde v]_{\partial \tilde E}$ three times, also a contradiction.  Thus $\tilde u\not\in \tilde\epsilon_2$.  The argument is the same if we had initially let $\tilde v\in \tilde\epsilon_2$.
\end{proof}

\begin{lemma}
\label{lem3}
The following four statements hold.
\begin{enumerate}
\item \label{lem4a} If $[\tilde u\to \tilde v]_{\partial \tilde E}\subset A$ and $[\tilde v\to \tilde u]_{\partial \tilde E}$ meets $\partial A$, then $B\setminus A$ and $\tilde B\setminus \tilde A$ are disjoint.
\item If $[\tilde v\to \tilde u]_{\partial \tilde E}\subset B$ and $[\tilde u\to \tilde v]_{\partial \tilde E}$ meets $\partial B$, then $A\setminus B$ and $\tilde A\setminus \tilde B$ are disjoint.
\item If $[u\to v]_{\partial E}\subset \tilde A$ and $[v\to u]_{\partial E}$ meets $\partial \tilde A$, then $\tilde B\setminus \tilde A$ and $B\setminus A$ are disjoint.
\item If $[v\to u]_{\partial E}\subset \tilde B$ and $[u\to v]_{\partial E}$ meets $\partial \tilde B$, then $\tilde A\setminus \tilde B$ and $A\setminus B$ are disjoint.
\end{enumerate}
\end{lemma}

\begin{proof}
We prove only (1), as (2), (3), (4) are symmetric restatements of it.  Suppose the hypotheses of (1) hold.  Then both $\tilde u$ and $\tilde v$ lie in $A$.  Thus the circular arc $[\tilde v\to \tilde u]_{\partial \tilde E}$ meets $\partial A$ either exactly twice or not at all, in fact exactly twice because of the hypotheses.  But $[\tilde v\to \tilde u]_{\partial \tilde E} = [\tilde v\to \tilde u]_{\partial \tilde B}$.  Thus $[\tilde u\to \tilde v]_{\partial \tilde B}$ does not meet $\partial A$, and has its endpoints lying in $A$, so $[\tilde u\to \tilde v]_{\partial \tilde B}\subset A$.

From our definitions of $\tilde u$ and $\tilde v$, it is easy to check that $\partial (\tilde B\setminus \tilde A)$ is the union of the arcs $[\tilde u\to \tilde v]_{\partial \tilde B}$ and $[\tilde u\to \tilde v]_{\partial \tilde E}$.  It follows that $\partial (\tilde B\setminus \tilde A)$ is contained in $A$.  Thus $\tilde B\setminus \tilde A$ is contained in $A$, and so is disjoint from $B\setminus A$.
\end{proof}

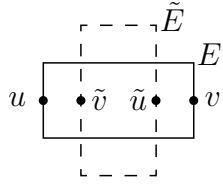
\begin{figure}[t]
\centering
% Generated with LaTeXDraw 2.0.8
% Tue Apr 10 00:24:47 EDT 2012
% \usepackage[usenames,dvipsnames]{pstricks}
% \usepackage{epsfig}
% \usepackage{pst-grad} % For gradients
% \usepackage{pst-plot} % For axes
\scalebox{1} % Change this value to rescale the drawing.
{
\begin{pspicture}(0,-1.1367188)(3.3490624,1.1767187)
\psframe[linewidth=0.02,dimen=middle](2.67,0.36328125)(0.67,-0.63671875)
\psframe[linewidth=0.02,linestyle=dashed,dash=0.16cm 0.16cm,dimen=middle](2.17,0.86328125)(1.17,-1.1367188)
\psdots[dotsize=0.12](0.67,-0.13671875)
\psdots[dotsize=0.12](2.67,-0.13671875)
\psdots[dotsize=0.12](1.17,-0.13671875)
\psdots[dotsize=0.12](2.17,-0.13671875)
\usefont{T1}{ptm}{m}{n}
\rput(2.3845313,0.97328126){$\tilde E$}
\usefont{T1}{ptm}{m}{n}
\rput(2.8745313,0.47328126){$E$}
\usefont{T1}{ptm}{m}{n}
\rput(0.32453125,-0.12671874){$u$}
\usefont{T1}{ptm}{m}{n}
\rput(2.9245312,-0.12671874){$v$}
\usefont{T1}{ptm}{m}{n}
\rput(1.4145313,-0.12671874){$\tilde v$}
\usefont{T1}{ptm}{m}{n}
\rput(1.9345313,-0.12671874){$\tilde u$}
\end{pspicture} 
}
\caption
{
\label{fig:topo confo four}
A topological configuration of two eyes which guarantees that $A\setminus B$ and $\tilde A\setminus \tilde B$ do not meet, and that $B\setminus A$ and $\tilde B\setminus \tilde A$ do not meet.
}
\end{figure}

\begin{lemma}
\label{lem4}
Suppose $\{E, u, v, \tilde E, \tilde u, \tilde v\}$ are in the topological configuration depicted in Figure \ref{fig:topo confo four}.  Then $A\setminus B$ and $\tilde A\setminus \tilde B$ do not meet, and $B\setminus A$ and $\tilde B\setminus \tilde A$ do not meet.
\end{lemma}

\begin{proof}
\noindent The curves $\partial \tilde A\setminus \partial \tilde E$ and $\partial \tilde B\setminus \partial \tilde E$ both have $\tilde u$ and $\tilde v$ as their endpoints and otherwise avoid $\tilde E$.  Thus each must cross $\partial E$ twice.  These four crossings together with the points $\partial E\cap \partial \tilde E$ accounts for all eight possible intersection points between $\partial A \cup \partial B$ and $\partial \tilde A\cup \partial \tilde B$.  Thus the arc $[\tilde v\to \tilde u]_{\partial \tilde E}$ does not meet $\partial B$.  Because this arc meets $B\supset E$, we conclude that $[\tilde v\to \tilde u]_{\partial \tilde E}$ is contained in $B$.  Note that $[\tilde u\to\tilde v]_{\partial \tilde E}$ meets $\partial B$.  Thus by part (\ref{lem4a}) of Lemma \ref{lem4} we get that $A\setminus B$ and $\tilde A\setminus \tilde B$ are disjoint.  That $B\setminus A$ and $\tilde B\setminus \tilde A$ are disjoint follows by symmetry.
\end{proof}

\begin{lemma}
\label{lem5}
Suppose $u\in \tilde E$ and $\tilde u\in E$.  Then $A\setminus B$ and $\tilde A\setminus \tilde B$ do not meet, or $B\setminus A$ and $\tilde B\setminus \tilde A$ do not meet.
\end{lemma}

\begin{proof}
Suppose for contradiction that $u\in \tilde E$ and $\tilde u\in E$, but that $A\setminus B$ and $\tilde A\setminus \tilde B$ meet, and that $B\setminus A$ and $\tilde B\setminus \tilde A$ meet.

\begin{observation}
\label{lem:no contained coeyes}
Neither of $B\setminus A$ and $\tilde B\setminus \tilde A$ contains the other.
\end{observation}

\noindent To see why this is true, note that $u$ is in fact an interior point of $\tilde E$ by the general position hypothesis, and that $u\in \partial (B\setminus A)$.  Thus $B\setminus A$ meets the exterior of $\tilde B\setminus \tilde A$.  A similar argument gives that $\tilde B\setminus \tilde A$ meets the exterior of $B\setminus A$.

We are supposing for contradiction that $B\setminus A$ and $\tilde B\setminus \tilde A$ meet, and by Observation \ref{lem:no contained coeyes} neither of them contains the other.  Thus if we can show that $\partial (B\setminus A)$ and $\partial (\tilde B\setminus \tilde A)$ do not meet we will have derived a contradiction, as desired.

Note that Proposition \ref{prop:possible for all} applies.  This allows us to make the following observation.

\begin{observation}
\label{observersz}
Either \ref{fig:possible3a} or \ref{fig:possible3e} occurs, either \ref{fig:possible4a} or \ref{fig:possible4d} occurs, either \ref{fig:possible5a} or \ref{fig:possible5e} occurs, and either \ref{fig:possible6a} or \ref{fig:possible6d} occurs.
\end{observation}

\noindent We prove that either \ref{fig:possible3a} or \ref{fig:possible3e} occurs, and the other parts of the observation follow similarly.  Because $\tilde E\subset \tilde A$, we may eliminate any candidate topological configurations where $u\not\in \tilde A$.  This eliminates \ref{fig:possible3d}, \ref{fig:possible3f}, \ref{fig:possible3g}, and \ref{fig:possible3h}.  Next, because $\tilde u\in \partial \tilde A$, we may eliminate any candidate topological configurations where $\partial \tilde A$ does not meet $E$, as this would preclude $\tilde u\in E$.  This eliminates \ref{fig:possible3b} and \ref{fig:possible3c}, leaving us with only the two claimed possibilities.  Thus the remainder of our proof breaks into cases as follows.

\begin{mooncase}
\label{case:win1}
Suppose that both \ref{fig:possible3a} and \ref{fig:possible4a} occur.
\end{mooncase}

\noindent Then $\partial (A\setminus B)$ is contained in $\tilde A$, and $\partial (B\setminus A)$ is contained in $\tilde B$.  Thus $\partial (A\setminus B) \cup \partial (B\setminus A)$ is contained in $\tilde A\cup \tilde B$.  But $\partial (A\cup B)$ is contained in $\partial (A\setminus B) \cup \partial (B\setminus A)$, thus in $\tilde A\cup \tilde B$.  We conclude that $A\cup B\subset \tilde A\cup \tilde B$.  Now $\tilde u\in \partial (\tilde A\cup \tilde B)$ and $E\subset A\cup B$, so by the general position hypothesis we get a contradiction to $\tilde u\in E$.

\begin{mooncase}
Suppose that both \ref{fig:possible3a} and \ref{fig:possible4d} occur.
\end{mooncase}

\noindent Then $u\in \tilde E$ and $v\in \tilde A\setminus \tilde B$.  One of the following two sub-cases occurs.

\begin{moonsubcase}
Suppose that \ref{fig:possible5a} occurs.
\end{moonsubcase}

\noindent Then $\partial A$ does not meet $\tilde A\setminus \tilde B$.  But $v$ lies on $\partial A$, contradicting $v\in \tilde A\setminus \tilde B$.

\begin{moonsubcase}
Suppose that \ref{fig:possible5e} occurs.  Then one of \ref{fig:possible6a} and  \ref{fig:possible6d} occurs.
\end{moonsubcase}

\noindent From \ref{fig:possible5e} and that $u\in \tilde E$ and $v\in \tilde A \setminus \tilde B$, it follows that $\partial (B\setminus A) \cap \partial A = [u\to v]_{\partial E}$ does not meet $\partial ( \tilde B \setminus \tilde A)$.  If \ref{fig:possible6a} occurs, then $\partial B\supset \partial (B\setminus A) \cap \partial B$ does not meet $\partial (\tilde B\setminus \tilde A)$.  If \ref{fig:possible6d} occurs, then via $u\in \tilde E$ and $v\in \tilde A \setminus \tilde B$ we get that $\partial (B\setminus A) \cap \partial B = [u\to v]_{\partial B}$ does not meet $\partial (\tilde B\setminus \tilde A)$.  In either case $\partial(B\setminus A)$ and $\partial (\tilde B\setminus \tilde A)$ do not meet, giving us a contradiction.

\vskip\topsep

\noindent Cases (1) and (2) together rule out \ref{fig:possible4a}, \ref{fig:possible5a}, and \ref{fig:possible6a} by symmetry, so the only remaining case is the following.

\begin{mooncase}
Suppose that \ref{fig:possible3e}, \ref{fig:possible4d}, \ref{fig:possible5e}, and \ref{fig:possible6d} occur.
\end{mooncase}

\noindent By \ref{fig:possible3e} and \ref{fig:possible4d} we have that $u\in \tilde E$ and $v\in \bbC\setminus (\tilde A\cup \tilde B)$.  Then from \ref{fig:possible5e} and \ref{fig:possible6d} we get that neither $\partial (B\setminus A)\cap \partial A = [u\to v]_{\partial A}$ nor $\partial (B\setminus A)\cap \partial B = [u\to v]_{\partial B}$ meets $\partial (\tilde B\setminus \tilde A)$, again giving us the desired contradiction.
\end{proof}

\section{Torus parametrization\twostars}
\label{chap:torus}

In this section we introduce a tool, called \emph{torus parametrization} that allows us to work with fixed-point index combinatorially.  This will allow us to systematically and relatively painlessly handle the remaining case analysis.\medskip

Let $K$ and $\tilde K$ be closed Jordan domains in transverse position, so that $\partial K$ and $\partial \tilde K$ meet at $2M\ge 0$ points, with boundaries oriented as usual.  Let $\partial K \allowbreak \cap\allowbreak  \partial \tilde K =\allowbreak  \{P_1,\allowbreak \ldots,\allowbreak P_M,\allowbreak \tilde P_1,\ldots,\allowbreak \tilde P_M\}$, where $P_i$ and $\tilde P_i$ are labeled so that at every $P_i$ we have that $\partial K$ is entering $\tilde K$, and at every $\tilde P_i$ we have that $\partial \tilde K$ is entering $K$.  Imbue $\bbS^1$ with an orientation and let $\kappa:\partial K\to \bbS^1$ and $\tilde \kappa : \partial\tilde K \to \bbS^1$ be orientation-preserving homeomorphisms.  We refer to this as fixing a \emph{torus parametrization} for $K$ and $\tilde K$.

We consider a point $(x,y)$ on the 2-torus $\bbT = \bbS^1\times \bbS^1$ to be parametrizing simultaneously a point $\kappa\inv(x) \in \partial K$ and a point $\tilde\kappa\inv(y) \in \partial \tilde K$.  We denote by $p_i\in \bbT$ be the unique point $(x,y)\in \bbT$ satisfying $\kappa\inv(x) = \tilde\kappa\inv(y) = P_i$, similarly $\tilde p_i\in \bbT$.  Note that by the transverse position hypothesis no pair of points in $\{p_1,\ldots,p_M, \tilde p_1,\ldots,\tilde p_M\}$ share a first coordinate, nor a second coordinate.

Suppose we pick $(x_0,y_0)\in \bbS^1\times \bbS^1$.  Then we may draw an image of $\bbT = \bbS^1\times \bbS^1$ by letting $\{x_0\}\times \bbS^1$ be the vertical axis and letting $\bbS^1 \times \{y_0\}$ be the horizontal axis.  Then we call $(x_0,y_0)$ a \emph{base point} for the drawing.  See Figure \ref{tpex1} for an example.

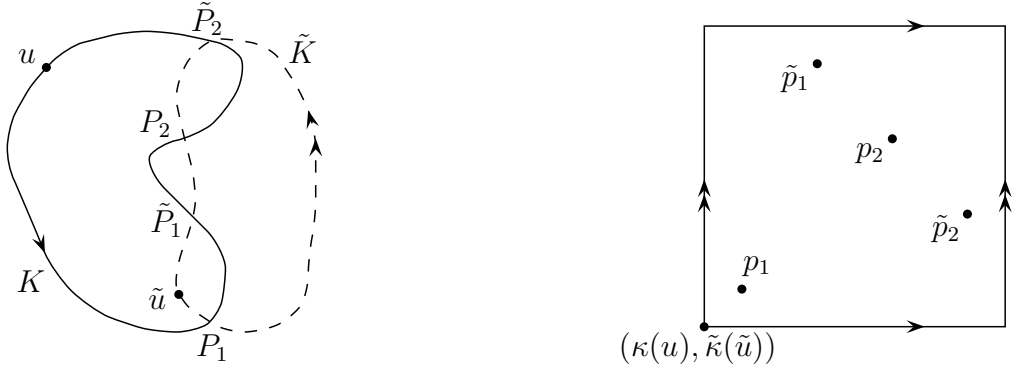
\begin{figure}[t]
\centering
\subfloat
{
% Generated with LaTeXDraw 2.0.8
% Wed Apr 18 11:36:21 EDT 2012
% \usepackage[usenames,dvipsnames]{pstricks}
% \usepackage{epsfig}
% \usepackage{pst-grad} % For gradients
% \usepackage{pst-plot} % For axes
\scalebox{1} % Change this value to rescale the drawing.
{
\begin{pspicture}(0,-2.4092188)(4.8290625,2.4092188)
\pscustom[linewidth=0.02]
{
\newpath
\moveto(0.11,0.03578125)
\lineto(0.26,-0.31421876)
\curveto(0.335,-0.48921874)(0.46,-0.7892187)(0.51,-0.9142187)
\curveto(0.56,-1.0392188)(0.735,-1.2892188)(0.86,-1.4142188)
\curveto(0.985,-1.5392188)(1.235,-1.7142187)(1.36,-1.7642188)
\curveto(1.485,-1.8142188)(1.71,-1.8892188)(1.81,-1.9142188)
\curveto(1.91,-1.9392188)(2.135,-1.9642187)(2.26,-1.9642187)
\curveto(2.385,-1.9642187)(2.585,-1.9142188)(2.66,-1.8642187)
\curveto(2.735,-1.8142188)(2.835,-1.6642188)(2.86,-1.5642188)
\curveto(2.885,-1.4642187)(2.91,-1.2392187)(2.91,-1.1142187)
\curveto(2.91,-0.9892188)(2.785,-0.7392188)(2.66,-0.6142188)
\curveto(2.535,-0.48921874)(2.26,-0.21421875)(2.11,-0.06421875)
\curveto(1.96,0.08578125)(1.86,0.31078124)(1.91,0.38578126)
\curveto(1.96,0.46078125)(2.135,0.56078124)(2.26,0.5857813)
\curveto(2.385,0.61078125)(2.61,0.7107813)(2.71,0.78578126)
\curveto(2.81,0.86078125)(2.985,1.0857812)(3.06,1.2357812)
\curveto(3.135,1.3857813)(3.16,1.6107812)(3.11,1.6857812)
\curveto(3.06,1.7607813)(2.91,1.8607812)(2.81,1.8857813)
\curveto(2.71,1.9107813)(2.485,1.9607812)(2.36,1.9857812)
\curveto(2.235,2.0107813)(1.985,2.0357811)(1.86,2.0357811)
\curveto(1.735,2.0357811)(1.51,2.0107813)(1.41,1.9857812)
\curveto(1.31,1.9607812)(1.11,1.9107813)(1.01,1.8857813)
\curveto(0.91,1.8607812)(0.71,1.7357812)(0.61,1.6357813)
\curveto(0.51,1.5357813)(0.36,1.3607812)(0.31,1.2857813)
\curveto(0.26,1.2107812)(0.16,1.0357813)(0.11,0.93578124)
\curveto(0.06,0.8357813)(0.01,0.61078125)(0.01,0.48578125)
\curveto(0.01,0.36078125)(0.035,0.18578126)(0.11,0.03578125)
}
\psline[linestyle=none,linewidth=0.06]{->}(0.335,-0.48921874)(0.51,-0.9142187)
\pscustom[linewidth=0.02,linestyle=dashed,dash=0.16cm 0.16cm]
{
\newpath
\moveto(4.11,0.23578125)
\lineto(4.11,0.48578125)
\curveto(4.11,0.61078125)(4.06,0.86078125)(4.01,0.98578125)
\curveto(3.96,1.1107812)(3.86,1.3107812)(3.81,1.3857813)
\curveto(3.76,1.4607812)(3.66,1.6107812)(3.61,1.6857812)
\curveto(3.56,1.7607813)(3.385,1.8607812)(3.26,1.8857813)
\curveto(3.135,1.9107813)(2.91,1.9357812)(2.81,1.9357812)
\curveto(2.71,1.9357812)(2.51,1.7857813)(2.41,1.6357813)
\curveto(2.31,1.4857812)(2.235,1.2107812)(2.26,1.0857812)
\curveto(2.285,0.9607813)(2.335,0.73578125)(2.36,0.6357812)
\curveto(2.385,0.53578126)(2.435,0.33578125)(2.46,0.23578125)
\curveto(2.485,0.13578124)(2.51,-0.08921875)(2.51,-0.21421875)
\curveto(2.51,-0.33921874)(2.435,-0.63921875)(2.36,-0.81421876)
\curveto(2.285,-0.9892188)(2.235,-1.2642188)(2.26,-1.3642187)
\curveto(2.285,-1.4642187)(2.385,-1.6142187)(2.46,-1.6642188)
\curveto(2.535,-1.7142187)(2.685,-1.8142188)(2.76,-1.8642187)
\curveto(2.835,-1.9142188)(3.035,-1.9642187)(3.16,-1.9642187)
\curveto(3.285,-1.9642187)(3.51,-1.9142188)(3.61,-1.8642187)
\curveto(3.71,-1.8142188)(3.86,-1.6892188)(3.91,-1.6142187)
\curveto(3.96,-1.5392188)(4.01,-1.3392187)(4.01,-1.2142187)
\curveto(4.01,-1.0892187)(4.035,-0.83921874)(4.06,-0.71421874)
\curveto(4.085,-0.58921874)(4.11,-0.33921874)(4.11,-0.21421875)
\curveto(4.11,-0.08921875)(4.11,0.08578125)(4.11,0.23578125)
}
\psline[linestyle=none,linewidth=0.06]{->}(4.11,0.61078125)(4.01,0.98578125)
\psline[linestyle=none,linewidth=0.06]{->}(4.11,0.48578125)(4.11,0.61078125)
\usefont{T1}{ptm}{m}{n}
\rput(0.32453126,-1.2942188){$K$}
\usefont{T1}{ptm}{m}{n}
\rput(0.27453125,1.7057812){$u$}
\usefont{T1}{ptm}{m}{n}
\rput(3.9345312,1.8057812){$\tilde K$}
\usefont{T1}{ptm}{m}{n}
\rput(2.0045311,-1.5142188){$\tilde u$}
\usefont{T1}{ptm}{m}{n}
\rput(2.7345312,-2.1742187){$P_1$}
\usefont{T1}{ptm}{m}{n}
\rput(1.9745313,0.80578125){$P_2$}
\usefont{T1}{ptm}{m}{n}
\rput(2.1245313,-0.47421876){$\tilde P_1$}
\usefont{T1}{ptm}{m}{n}
\rput(2.6645312,2.2057812){$\tilde P_2$}
\psdots[dotsize=0.12](0.53,1.5557812)
\psdots[dotsize=0.12](2.29,-1.4642187)
\end{pspicture} 
}
}\qquad
\subfloat
{
% Generated with LaTeXDraw 2.0.8
% Wed Apr 18 11:45:46 EDT 2012
% \usepackage[usenames,dvipsnames]{pstricks}
% \usepackage{epsfig}
% \usepackage{pst-grad} % For gradients
% \usepackage{pst-plot} % For axes
\scalebox{1} % Change this value to rescale the drawing.
{
\begin{pspicture}(0,-2.2775)(7.3690624,2.2375)
\psframe[linewidth=0.02,dimen=middle](7.07,2.2375)(3.07,-1.7625)
\psline[linestyle=none,linewidth=0.06]{->}(3.07,-1.7625)(6,-1.7625)
\psline[linestyle=none,linewidth=0.06]{->}(3.07,2.2375)(6,2.2375)
\psline[linestyle=none,linewidth=0.06]{->}(3.07,-1.7625)(3.07,0)
\psline[linestyle=none,linewidth=0.06]{->}(3.07,-1.7625)(3.07,0.2)
\psline[linestyle=none,linewidth=0.06]{->}(7.07,-1.7625)(7.07,0)
\psline[linestyle=none,linewidth=0.06]{->}(7.07,-1.7625)(7.07,0.2)
\usefont{T1}{ptm}{m}{n}
\rput(3.7745314,-0.9525){$p_1$}
\usefont{T1}{ptm}{m}{n}
\rput(5.2745314,0.5475){$p_2$}
\usefont{T1}{ptm}{m}{n}
\rput(4.284531,1.5475){$\tilde p_1$}
\usefont{T1}{ptm}{m}{n}
\rput(6.284531,-0.4525){$\tilde p_2$}
\psdots[dotsize=0.12](3.57,-1.2625)
\psdots[dotsize=0.12](4.57,1.7375)
\psdots[dotsize=0.12](5.57,0.7375)
\psdots[dotsize=0.12](6.57,-0.2625)
\psdots[dotsize=0.12](3.07,-1.7625)
\usefont{T1}{ptm}{m}{n}
\rput(2.9845312,-2.0525){$(\kappa(u),\tilde\kappa(\tilde u))$}
\end{pspicture} 
}
}
\caption
{
\label{tpex1}
A pair of closed Jordan domains $K$ and $\tilde K$ and a torus parametrization for them, drawn with base point $(\kappa(u),\tilde\kappa(\tilde u))$.  The key points to check are that as we vary the first coordinate of $\bbT$ positively starting at $u$, we arrive at $\kappa(P_1)$, $\kappa(\tilde P_1)$, $\kappa(P_2)$, and $\kappa(\tilde P_2)$ in that order, and as we vary the second coordinate of $\bbT$ positively starting at $\tilde\kappa(\tilde u)$, we arrive at $\tilde\kappa(P_1)$, $\tilde \kappa(\tilde P_2)$, $\tilde\kappa(P_2)$, and $\tilde\kappa(P_1)$ in that order.
}
\end{figure}

Suppose that $\phi:\partial K \to \partial \tilde K$ is an orientation-preserving homeomorphism.  Then $\phi$ determines an oriented curve $\gamma$ in $\bbT$ for us, namely its graph $\gamma = \{(\kappa(z),\tilde\kappa(f(z))\}_{z\in \partial K}$, with orientation obtained by traversing $\partial K$ positively.  Note that $\phi$ is fixed-point-free if and only if its associated curve $\gamma$ misses all of the $p_i$ and $\tilde p_i$.  Pick $u\in \partial K$ and denote $\tilde u = \phi(u)$.  Then if we draw the torus parametrization for $K$ and $\tilde K$ using the base point $(\kappa(u),\tilde\kappa(\tilde u))$, the curve $\gamma$ associated to $\phi$ ``looks like the graph of a strictly increasing function.''  The converse is also true: given any such $\gamma$, it determines for us an orientation-preserving homeomorphism $\partial K\to \partial \tilde K$ sending $u$ to $\tilde u$, which is fixed-point-free if and only if $\gamma$ misses all of the $p_i$ and $\tilde p_i$.\medskip

\newcommand{\Num}{\operatorname{Num}}

\newcommand{\nump}{\#p}
\newcommand{\numq}{\#\tilde p}

The nicest thing about torus parametrization is that it allows us to compute $\eta(\phi)$ easily by looking at the curve $\gamma$ associated to $\phi$.  In particular, suppose that $\phi(u) = \tilde u$, equivalently that $(\kappa(u),\tilde\kappa(\tilde u)) \in \gamma$.  The curve $\gamma$ and the horizontal and vertical axes $\{\tilde\kappa(\tilde u)\}\times \bbS^1$ and $\bbS^1\times \{\kappa(u)\}$ divide $\bbT$ into two simply connected open sets $\Delta_\uparrow(u,\gamma)$ and $\Delta_\downarrow(u,\gamma)$ as shown in Figure \ref{fig:proving amazing index}.  We suppress the dependence on $\tilde u$ in the notation because $\tilde u = \phi(u)$.  If neither $u\in \partial \tilde K$ nor $\tilde u\in \partial K$ then every $p_i$ and every $\tilde p_i$ lies in either $\Delta_\downarrow(u,\gamma)$ or $\Delta_\uparrow(u,\gamma)$.  In this case we write $\nump_\downarrow(u,\gamma)$ to denote $|\{p_1,\ldots,p_M\} \cap \Delta_\downarrow(u,\gamma)|$ the number of points $p_i$ which lie in $\Delta_\downarrow(u,\gamma)$, and we define $\nump_\uparrow(u,\gamma)$, $\numq_\downarrow(u,\gamma)$, and $\numq_\uparrow(u,\gamma)$ in the analogous way.  Denote by $\omega(\alpha,z)$ the winding number of the closed curve $\alpha\subset \bbC$ around the point $z\not\in\alpha$.  Then:

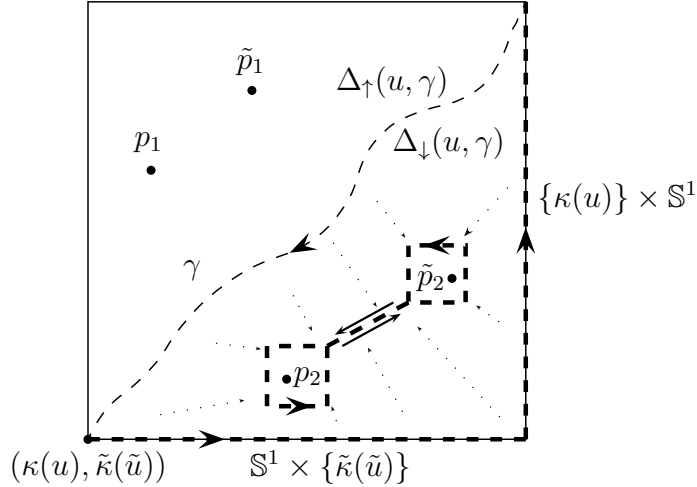
\begin{figure}[t]
\centering

% Generated with LaTeXDraw 2.0.8
% Generated with LaTeXDraw 2.0.8
% Sun Dec 04 17:41:05 EST 2011
% \usepackage[usenames,dvipsnames]{pstricks}
% \usepackage{epsfig}
% \usepackage{pst-grad} % For gradients
% \usepackage{pst-plot} % For axes
\scalebox{1} % Change this value to rescale the drawing.
{
\begin{pspicture}(0,-3.2275)(9.279062,3.1875)
\psframe[linewidth=0.02,dimen=middle](5.82,3.1875)(0.0,-2.6325)

\psdots[dotsize=0.12](0.0,-2.6325)
\rput(0.0,-3.0025){$(\kappa(u),\tilde\kappa(\tilde u))$}

%\psline[linewidth=0.06cm,linestyle=dashed,dash=0.16cm 0.16cm](0.0,-2.6325)(5.82,3.1875)

\pscustom[linewidth=0.02,linestyle=dashed,dash=0.16cm 0.16cm]
{
\newpath
\moveto(0.0,-2.6325)
\lineto(0.08,-2.4725)
\curveto(0.11,-2.4225)(0.165,-2.3375)(0.19,-2.3025)
\curveto(0.215,-2.2675)(0.28,-2.1925)(0.32,-2.1525)
\curveto(0.36,-2.1125)(0.445,-2.0425)(0.49,-2.0125)
\curveto(0.535,-1.9825)(0.65,-1.8825)(0.72,-1.8125)
\curveto(0.79,-1.7425)(0.895,-1.6275)(0.93,-1.5825)
\curveto(0.965,-1.5375)(1.035,-1.4325)(1.07,-1.3725)
\curveto(1.105,-1.3125)(1.19,-1.1925)(1.24,-1.1325)
\curveto(1.29,-1.0725)(1.38,-0.9675)(1.42,-0.9225)
\curveto(1.46,-0.8775)(1.565,-0.7775)(1.63,-0.7225)
\curveto(1.695,-0.6675)(1.825,-0.5675)(1.89,-0.5225)
\curveto(1.955,-0.4775)(2.075,-0.4025)(2.13,-0.3725)
\curveto(2.185,-0.3425)(2.305,-0.2875)(2.37,-0.2625)
\curveto(2.435,-0.2375)(2.56,-0.1925)(2.62,-0.1725)
\curveto(2.68,-0.1525)(2.8,-0.0975)(2.86,-0.0625)
\curveto(2.92,-0.0275)(3.035,0.0525)(3.09,0.0975)
\curveto(3.145,0.1425)(3.235,0.2325)(3.27,0.2775)
\curveto(3.305,0.3225)(3.37,0.4125)(3.4,0.4575)
\curveto(3.43,0.5025)(3.48,0.5875)(3.5,0.6275)
\curveto(3.52,0.6675)(3.555,0.7625)(3.57,0.8175)
\curveto(3.585,0.8725)(3.62,0.9775)(3.64,1.0275)
\curveto(3.66,1.0775)(3.7,1.1625)(3.72,1.1975)
\curveto(3.74,1.2325)(3.795,1.3075)(3.83,1.3475)
\curveto(3.865,1.3875)(3.945,1.4675)(3.99,1.5075)
\curveto(4.035,1.5475)(4.125,1.6075)(4.17,1.6275)
\curveto(4.215,1.6475)(4.325,1.6925)(4.39,1.7175)
\curveto(4.455,1.7425)(4.59,1.7825)(4.66,1.7975)
\curveto(4.73,1.8125)(4.84,1.8425)(4.88,1.8575)
\curveto(4.92,1.8725)(5.015,1.9175)(5.07,1.9475)
\curveto(5.125,1.9775)(5.215,2.0425)(5.25,2.0775)
\curveto(5.285,2.1125)(5.35,2.1925)(5.38,2.2375)
\curveto(5.41,2.2825)(5.47,2.3875)(5.5,2.4475)
\curveto(5.53,2.5075)(5.57,2.6025)(5.58,2.6375)
\curveto(5.59,2.6725)(5.61,2.7375)(5.62,2.7675)
\curveto(5.63,2.7975)(5.655,2.8475)(5.67,2.8675)
\curveto(5.685,2.8875)(5.705,2.9275)(5.71,2.9475)
\curveto(5.715,2.9675)(5.73,3.0025)(5.74,3.0175)
\curveto(5.75,3.0325)(5.765,3.0625)(5.82,3.1875)
}
\usefont{T1}{ptm}{m}{n}
\rput(4.0445313,2.1175){$\Delta_\uparrow(u,\gamma)$}
\usefont{T1}{ptm}{m}{n}
\rput(4.7945313,1.2975){$\Delta_\downarrow(u,\gamma)$}
\usefont{T1}{ptm}{m}{n}
\rput(1.3845313,-0.3825){$\gamma$}

\psline[linewidth=0.02cm,linestyle=dotted,dotsep=0.16cm,arrowsize=0.05291667cm 2.0,arrowlength=1.4,arrowinset=0.4]{->}(5.7,0.8475)(5.0,0.0675)

\psline[linewidth=0.1cm,linestyle=none,arrows=<-](2.68,-0.1525)(2.86,-0.0625)

\psline[linewidth=0.06cm,linestyle=dashed,dash=0.16cm 0.16cm](0.0,-2.6325)(5.82,-2.6325)
\psline[linewidth=0.1cm,linestyle=none,arrows=->](0.0,-2.6325)(1.82,-2.6325)

\psline[linewidth=0.06cm,linestyle=dashed,dash=0.16cm 0.16cm](5.82,-2.6325)(5.82,3.1875)
\psline[linewidth=0.1cm,linestyle=none,arrows=->](5.82,-2.6325)(5.82,0.1875)

\psdots[dotsize=0.12](0.84,0.9475)
\psdots[dotsize=0.12](2.18,2.0075)
\psdots[dotsize=0.12](2.64,-1.8325)
\psdots[dotsize=0.12](4.84,-0.4925)
\usefont{T1}{ptm}{m}{n}
\rput(3.2245312,-3.0025){$\bbS^1\times \{\tilde\kappa(\tilde u)\}$}
\usefont{T1}{ptm}{m}{n}
\rput(7.044531,0.5775){$\{\kappa(u)\}\times \bbS^1$}
\usefont{T1}{ptm}{m}{n}
%\rput(2.7645313,0.5975){$\Delta$}
\usefont{T1}{ptm}{m}{n}
\rput(0.8045313,1.3175){$p_1$}
\usefont{T1}{ptm}{m}{n}
\rput(2.1645312,2.3975){$\tilde p_1$}
\usefont{T1}{ptm}{m}{n}
\rput(2.9245312,-1.8025){$p_2$}
\usefont{T1}{ptm}{m}{n}
\rput(4.5645313,-0.4825){$\tilde p_2$}

\psframe[linewidth=0.06,linestyle=dashed,dash=0.16cm 0.16cm,dimen=middle](5.02,-0.0525)(4.26,-0.8125)

\psline[linewidth=0.1cm,linestyle=none,arrows=->](5.02,-0.0525)(4.46,-0.0525)

\psframe[linewidth=0.06,linestyle=dashed,dash=0.16cm 0.16cm,dimen=middle](3.18,-1.3925)(2.38,-2.1925)

\psline[linewidth=0.1cm,linestyle=none,arrows=<-](2.98,-2.1925)(2.38,-2.1925)

\psline[linewidth=0.06cm,linestyle=dashed,dash=0.16cm 0.16cm](3.18,-1.3925)(4.26,-0.8125)

\psline[linewidth=0.03cm,linestyle=solid,arrows=->](3.38,-1.3925)(4.16,-0.9625)

\psline[linewidth=0.03cm,linestyle=solid,arrows=<-](3.28,-1.2425)(4.06,-0.8125)

\psline[linewidth=0.02cm,linestyle=dotted,dotsep=0.16cm,arrowsize=0.01291667cm 2.0,arrowlength=1.4,arrowinset=0.4]{->}(0.8,-2.3325)(2.08,-2.1125)
\psline[linewidth=0.02cm,linestyle=dotted,dotsep=0.16cm,arrowsize=0.01291667cm 2.0,arrowlength=1.4,arrowinset=0.4]{->}(2.66,-0.5925)(2.98,-1.1925)
\psline[linewidth=0.02cm,linestyle=dotted,dotsep=0.16cm,arrowsize=0.01291667cm 2.0,arrowlength=1.4,arrowinset=0.4]{->}(3.74,0.6675)(4.22,0.0475)
%\psline[linewidth=0.02cm,linestyle=dotted,dotsep=0.16cm,arrowsize=0.01291667cm 2.0,arrowlength=1.4,arrowinset=0.4]{->}(5.28,1.8675)(5.0,0.1875)
\psline[linewidth=0.02cm,linestyle=dotted,dotsep=0.16cm,arrowsize=0.01291667cm 2.0,arrowlength=1.4,arrowinset=0.4]{->}(5.58,-1.1525)(5.16,-0.8525)
\psline[linewidth=0.02cm,linestyle=dotted,dotsep=0.16cm,arrowsize=0.01291667cm 2.0,arrowlength=1.4,arrowinset=0.4]{->}(4.26,-2.4125)(3.72,-1.4125)
\psline[linewidth=0.02cm,linestyle=dotted,dotsep=0.16cm,arrowsize=0.01291667cm 2.0,arrowlength=1.4,arrowinset=0.4]{->}(3.34,-2.5125)(3.24,-2.1925)
\psline[linewidth=0.02cm,linestyle=dotted,dotsep=0.16cm,arrowsize=0.01291667cm 2.0,arrowlength=1.4,arrowinset=0.4]{->}(1.54,-1.3325)(2.3,-1.4325)
\psline[linewidth=0.02cm,linestyle=dotted,dotsep=0.16cm,arrowsize=0.01291667cm 2.0,arrowlength=1.4,arrowinset=0.4]{->}(5.58,-2.4525)(4.38,-1.1325)
\psline[linewidth=0.02cm,linestyle=dotted,dotsep=0.16cm,arrowsize=0.01291667cm 2.0,arrowlength=1.4,arrowinset=0.4]{->}(3.16,0.1275)(3.84,-0.8725)
\end{pspicture} 
}

\caption
{
\label{fig:proving amazing index}
A homotopy from $\partial\Delta_\downarrow(u,\gamma)$ to $\Gamma$.  Here the orientation shown on $\gamma$ is the opposite of the orientation induced by traversing $\partial K$ positively.
}
\end{figure}

\begin{lemma}
\label{prop:computing index from torus}
Let $K$ and $\tilde K$ be closed Jordan domains.  Fix a torus parametrization of $K$ and $\tilde K$ via $\kappa$ and $\tilde\kappa$.  Let $\phi:\partial K \to \partial \tilde K$ be an indexable homeomorphism, with graph $\gamma$ in $\bbT$.  Suppose that $\phi(u) = \tilde u$, where $u\not\in \partial \tilde K$ and $\tilde u\not\in \partial K$.  Then:
\begin{align}
\label{eq 1:prop:computing index from torus eq}\eta(\phi)  = w(\gamma)& = \omega(\partial K, \tilde u) + \omega(\partial \tilde K, u ) - \nump_\downarrow(u,\gamma) + \numq_\downarrow(u,\gamma)\\
\label{eq 2:prop:computing index from torus eq}& = \omega(\partial K, \tilde u) + \omega(\partial \tilde K, u) + \nump_\uparrow(u,\gamma) - \numq_\uparrow(u,\gamma)
\end{align}
\end{lemma}

\begin{proof}
Suppose $\gamma_0$ is any oriented closed curve in $\bbT \setminus \{p_1,\ldots,p_M,\tilde p_1,\ldots,\tilde p_M\}$.  Then the closed curve $\{\tilde\kappa\inv(y) - \kappa\inv(x)\}_{(x,y)\in \gamma_0}$ misses the origin, and has a natural orientation obtained by traversing $\gamma_0$ positively.  We denote by $w(\gamma_0)$ the winding number around the origin of $\{\tilde\kappa\inv(y) - \kappa\inv(x)\}_{(x,y)\in \gamma_0}$.  First:

\begin{observation}
\label{lem:homotopy classes}
If $\gamma_1$ and $\gamma_2$ are homotopic in $\bbT \setminus \{p_1,\ldots,p_M,\tilde p_1,\ldots,\tilde p_M\}$ then $w(\gamma_1) = w(\gamma_2)$.
\end{observation}

\noindent This is because the homotopy between $\gamma_1$ and $\gamma_2$ in $\bbT\setminus \{p_2,\ldots,p_M,\tilde p_1,\ldots,\tilde p_M\}$ induces a homotopy between the closed curves $\{\tilde\kappa\inv(y) - \kappa\inv(x)\}_{(x,y)\in \gamma_1}$ and $\{\tilde\kappa\inv(y) - \kappa\inv(x)\}_{(x,y)\in \gamma_2}$ in the punctured plane $\bbC\setminus \{0\}$.

If $\gamma$ has orientation induced by traversing $\partial K$ and $\partial \tilde K$ positively, then the following is a tautology.

\begin{observation}
\label{ex:never gonna reference this lol}
$\eta(\phi) = w(\gamma)$
\end{observation}

Orient $\partial \Delta_\downarrow(u,\gamma)$ as shown in Figure \ref{fig:proving amazing index}.  Then $\partial \Delta_\downarrow(u,\gamma)$ is the concatenation of the curve $\gamma$ traversed backwards with $\bbS^1\times \{\tilde\kappa(\tilde u)\}$ and $\{\kappa(u)\}\times \bbS^1$, where the two latter curves are oriented according to the positive orientation on $\bbS^1$.

\begin{figure}[t]
\centering

% Generated with LaTeXDraw 2.0.8
% Fri Dec 02 14:45:40 EST 2011
% \usepackage[usenames,dvipsnames]{pstricks}
% \usepackage{epsfig}
% \usepackage{pst-grad} % For gradients
% \usepackage{pst-plot} % For axes
\scalebox{.9} % Change this value to rescale the drawing.
{
\begin{pspicture}(0,-3.953125)(9.065,3.953125)
\definecolor{color308b}{rgb}{0.9215686274509803,0.9215686274509803,0.9215686274509803}
\pscustom[linestyle=none,fillstyle=solid,fillcolor = color308b]
{
\newpath
\moveto(1.69,2.5796876)
\lineto(1.89,2.4196875)
\curveto(1.99,2.3396876)(2.28,2.0546875)(2.47,1.8496875)
\curveto(2.66,1.6446875)(2.89,1.3196875)(2.93,1.1996875)
\curveto(2.97,1.0796875)(3.09,0.8246875)(3.17,0.6896875)
\curveto(3.25,0.5546875)(3.44,0.3296875)(3.55,0.2396875)
\curveto(3.66,0.1496875)(3.845,-0.0403125)(3.92,-0.1403125)
\curveto(3.995,-0.2403125)(4.145,-0.4453125)(4.22,-0.5503125)
\curveto(4.295,-0.6553125)(4.465,-0.8553125)(4.56,-0.9503125)
\curveto(4.655,-1.0453125)(4.82,-1.1903125)(4.89,-1.2403125)
\curveto(4.96,-1.2903125)(5.13,-1.4053125)(5.23,-1.4703125)
\curveto(5.33,-1.5353125)(5.48,-1.6753125)(5.53,-1.7503124)
\curveto(5.58,-1.8253125)(5.665,-1.9453125)(5.7,-1.9903125)
\curveto(5.735,-2.0353124)(5.81,-2.1303124)(5.85,-2.1803124)
\curveto(5.89,-2.2303126)(5.97,-2.3103125)(6.01,-2.3403125)
\curveto(6.05,-2.3703125)(6.13,-2.4153125)(6.17,-2.4303124)
\curveto(6.21,-2.4453125)(6.295,-2.4603126)(6.34,-2.4603126)
\curveto(6.385,-2.4603126)(6.5,-2.4353125)(6.57,-2.4103124)
\curveto(6.64,-2.3853126)(6.82,-2.2553124)(6.93,-2.1503124)
\curveto(7.04,-2.0453124)(7.235,-1.8803124)(7.32,-1.8203125)
\curveto(7.405,-1.7603126)(7.635,-1.5903125)(7.78,-1.4803125)
\curveto(7.925,-1.3703125)(8.155,-1.1853125)(8.24,-1.1103125)
\curveto(8.325,-1.0353125)(8.49,-0.8753125)(8.57,-0.7903125)
\curveto(8.65,-0.7053125)(8.765,-0.5353125)(8.8,-0.4503125)
\curveto(8.835,-0.3653125)(8.88,-0.2203125)(8.89,-0.1603125)
\curveto(8.9,-0.1003125)(8.93,0.0446875)(8.95,0.1296875)
\curveto(8.97,0.2146875)(9.005,0.3746875)(9.02,0.4496875)
\curveto(9.035,0.5246875)(9.045,0.7096875)(9.04,0.8196875)
\curveto(9.035,0.9296875)(8.98,1.1296875)(8.93,1.2196875)
\curveto(8.88,1.3096875)(8.725,1.5196875)(8.62,1.6396875)
\curveto(8.515,1.7596875)(8.305,2.0246875)(8.2,2.1696875)
\curveto(8.095,2.3146875)(7.9,2.5396874)(7.81,2.6196876)
\curveto(7.72,2.6996875)(7.445,2.8696876)(7.26,2.9596875)
\curveto(7.075,3.0496874)(6.76,3.1896875)(6.63,3.2396874)
\curveto(6.5,3.2896874)(6.09,3.3846874)(5.81,3.4296875)
\curveto(5.53,3.4746876)(5.115,3.5346875)(4.98,3.5496874)
\curveto(4.845,3.5646875)(4.53,3.5846875)(4.35,3.5896876)
\curveto(4.17,3.5946875)(3.845,3.5596876)(3.7,3.5196874)
\curveto(3.555,3.4796875)(3.275,3.3746874)(3.14,3.3096876)
\curveto(3.005,3.2446876)(2.79,3.1446874)(2.71,3.1096876)
\curveto(2.63,3.0746875)(2.485,3.0146875)(2.42,2.9896874)
\curveto(2.355,2.9646876)(2.195,2.8796875)(2.1,2.8196876)
\curveto(2.005,2.7596874)(1.89,2.6896875)(1.87,2.6796875)
\curveto(1.85,2.6696875)(1.825,2.6396875)(1.82,2.6196876)
\curveto(1.815,2.5996876)(1.8,2.5646875)(1.79,2.5496874)
\curveto(1.78,2.5346875)(1.77,2.5146875)(1.77,2.4996874)
}

\pscustom[linestyle=none,fillstyle=hlines]
{
\newpath
\moveto(1.47,-1.4403125)
\lineto(1.61,-1.4003125)
\curveto(1.68,-1.3803124)(1.92,-1.2903125)(2.09,-1.2203125)
\curveto(2.26,-1.1503125)(2.49,-1.0453125)(2.55,-1.0103126)
\curveto(2.61,-0.9753125)(2.74,-0.8903125)(2.81,-0.8403125)
\curveto(2.88,-0.7903125)(3.02,-0.6653125)(3.09,-0.5903125)
\curveto(3.16,-0.5153125)(3.295,-0.3603125)(3.36,-0.2803125)
\curveto(3.425,-0.2003125)(3.58,-0.0753125)(3.67,-0.0303125)
\curveto(3.76,0.0146875)(3.965,0.0946875)(4.08,0.1296875)
\curveto(4.195,0.1646875)(4.455,0.2946875)(4.6,0.3896875)
\curveto(4.745,0.4846875)(4.96,0.6396875)(5.03,0.6996875)
\curveto(5.1,0.7596875)(5.24,0.8746875)(5.31,0.9296875)
\curveto(5.38,0.9846875)(5.52,1.1046875)(5.59,1.1696875)
\curveto(5.66,1.2346874)(5.79,1.3496875)(5.85,1.3996875)
\curveto(5.91,1.4496875)(6.055,1.5446875)(6.14,1.5896875)
\curveto(6.225,1.6346875)(6.37,1.7096875)(6.43,1.7396874)
\curveto(6.49,1.7696875)(6.615,1.8246875)(6.68,1.8496875)
\curveto(6.745,1.8746876)(6.86,1.9096875)(6.91,1.9196875)
\curveto(6.96,1.9296875)(7.07,1.9396875)(7.13,1.9396875)
\curveto(7.19,1.9396875)(7.3,1.9396875)(7.35,1.9396875)
\curveto(7.4,1.9396875)(7.5,1.9246875)(7.55,1.9096875)
\curveto(7.6,1.8946875)(7.68,1.8596874)(7.71,1.8396875)
\curveto(7.74,1.8196875)(7.835,1.6996875)(7.9,1.5996875)
\curveto(7.965,1.4996876)(8.115,1.1346875)(8.2,0.8696875)
\curveto(8.285,0.6046875)(8.45,0.2146875)(8.53,0.0896875)
\curveto(8.61,-0.0353125)(8.705,-0.2703125)(8.72,-0.3803125)
\curveto(8.735,-0.4903125)(8.695,-0.7453125)(8.64,-0.8903125)
\curveto(8.585,-1.0353125)(8.41,-1.4003125)(8.29,-1.6203125)
\curveto(8.17,-1.8403125)(7.975,-2.1903124)(7.9,-2.3203125)
\curveto(7.825,-2.4503126)(7.7,-2.6303124)(7.65,-2.6803124)
\curveto(7.6,-2.7303126)(7.43,-2.8553126)(7.31,-2.9303124)
\curveto(7.19,-3.0053124)(6.87,-3.1503124)(6.67,-3.2203126)
\curveto(6.47,-3.2903125)(6.115,-3.3803124)(5.96,-3.4003124)
\curveto(5.805,-3.4203124)(5.51,-3.4453125)(5.37,-3.4503126)
\curveto(5.23,-3.4553125)(4.995,-3.4903126)(4.9,-3.5203125)
\curveto(4.805,-3.5503125)(4.575,-3.5853126)(4.44,-3.5903125)
\curveto(4.305,-3.5953126)(4.07,-3.6003125)(3.97,-3.6003125)
\curveto(3.87,-3.6003125)(3.57,-3.5553124)(3.37,-3.5103126)
\curveto(3.17,-3.4653125)(2.845,-3.3803124)(2.72,-3.3403125)
\curveto(2.595,-3.3003125)(2.34,-3.1453125)(2.21,-3.0303125)
\curveto(2.08,-2.9153125)(1.92,-2.7403126)(1.89,-2.6803124)
\curveto(1.86,-2.6203125)(1.815,-2.5003126)(1.8,-2.4403124)
\curveto(1.785,-2.3803124)(1.745,-2.2553124)(1.72,-2.1903124)
\curveto(1.695,-2.1253126)(1.65,-2.0153124)(1.63,-1.9703125)
\curveto(1.61,-1.9253125)(1.57,-1.8353125)(1.55,-1.7903125)
\curveto(1.53,-1.7453125)(1.505,-1.6703125)(1.5,-1.6403126)
\curveto(1.495,-1.6103125)(1.485,-1.5553125)(1.48,-1.5303125)
\curveto(1.475,-1.5053124)(1.48,-1.4603125)(1.49,-1.4403125)
}

\pscustom[linewidth=0.02]
{
\newpath
\moveto(1.47,-1.4403125)
\lineto(1.61,-1.4003125)
\curveto(1.68,-1.3803124)(1.92,-1.2903125)(2.09,-1.2203125)
\curveto(2.26,-1.1503125)(2.49,-1.0453125)(2.55,-1.0103126)
\curveto(2.61,-0.9753125)(2.74,-0.8903125)(2.81,-0.8403125)
\curveto(2.88,-0.7903125)(3.02,-0.6653125)(3.09,-0.5903125)
\curveto(3.16,-0.5153125)(3.295,-0.3603125)(3.36,-0.2803125)
\curveto(3.425,-0.2003125)(3.58,-0.0753125)(3.67,-0.0303125)
\curveto(3.76,0.0146875)(3.965,0.0946875)(4.08,0.1296875)
\curveto(4.195,0.1646875)(4.455,0.2946875)(4.6,0.3896875)
\curveto(4.745,0.4846875)(4.96,0.6396875)(5.03,0.6996875)
\curveto(5.1,0.7596875)(5.24,0.8746875)(5.31,0.9296875)
\curveto(5.38,0.9846875)(5.52,1.1046875)(5.59,1.1696875)
\curveto(5.66,1.2346874)(5.79,1.3496875)(5.85,1.3996875)
\curveto(5.91,1.4496875)(6.055,1.5446875)(6.14,1.5896875)
\curveto(6.225,1.6346875)(6.37,1.7096875)(6.43,1.7396874)
\curveto(6.49,1.7696875)(6.615,1.8246875)(6.68,1.8496875)
\curveto(6.745,1.8746876)(6.86,1.9096875)(6.91,1.9196875)
\curveto(6.96,1.9296875)(7.07,1.9396875)(7.13,1.9396875)
\curveto(7.19,1.9396875)(7.3,1.9396875)(7.35,1.9396875)
\curveto(7.4,1.9396875)(7.5,1.9246875)(7.55,1.9096875)
}

\psline[linestyle=none,linewidth=0.1,arrows=<-](5.1,0.7596875)(5.31,0.9296875)

\pscustom[linewidth=0.02]
{
\newpath
\moveto(1.69,2.5796876)
\lineto(1.89,2.4196875)
\curveto(1.99,2.3396876)(2.28,2.0546875)(2.47,1.8496875)
\curveto(2.66,1.6446875)(2.89,1.3196875)(2.93,1.1996875)
\curveto(2.97,1.0796875)(3.09,0.8246875)(3.17,0.6896875)
\curveto(3.25,0.5546875)(3.44,0.3296875)(3.55,0.2396875)
\curveto(3.66,0.1496875)(3.845,-0.0403125)(3.92,-0.1403125)
\curveto(3.995,-0.2403125)(4.145,-0.4453125)(4.22,-0.5503125)
\curveto(4.295,-0.6553125)(4.465,-0.8553125)(4.56,-0.9503125)
\curveto(4.655,-1.0453125)(4.82,-1.1903125)(4.89,-1.2403125)
\curveto(4.96,-1.2903125)(5.13,-1.4053125)(5.23,-1.4703125)
\curveto(5.33,-1.5353125)(5.48,-1.6753125)(5.53,-1.7503124)
\curveto(5.58,-1.8253125)(5.665,-1.9453125)(5.7,-1.9903125)
\curveto(5.735,-2.0353124)(5.81,-2.1303124)(5.85,-2.1803124)
\curveto(5.89,-2.2303126)(5.97,-2.3103125)(6.01,-2.3403125)
\curveto(6.05,-2.3703125)(6.13,-2.4153125)(6.17,-2.4303124)
\curveto(6.21,-2.4453125)(6.295,-2.4603126)(6.34,-2.4603126)
}

\psline[linewidth=0.1,linestyle=none,arrows=->](2.97,1.0796875)(3.17,0.6896875)

\usefont{T1}{ptm}{m}{n}
\rput(5.5845313,3.7496874){$K$}
\usefont{T1}{ptm}{m}{n}
\rput(5.7145314,-3.7503126){$\tilde K$}
\psdots[dotsize=0.12](3.77,0.0196875)
\usefont{T1}{ptm}{m}{n}
\rput(3.1945312,0.0496875){$P_i$}
\psdots[dotsize=0.12](2.55,1.7796875)
\psdots[dotsize=0.12](5.25,-1.4603125)
\psdots[dotsize=0.12](2.31,-1.1403126)
\psdots[dotsize=0.12](6.17,1.5996875)
\usefont{T1}{ptm}{m}{n}
\rput(1.8545312,1.6096874){$\kappa\inv(x_0)$}
\usefont{T1}{ptm}{m}{n}
\rput(11.174531,-1.2303125){$\kappa\inv(x_1)$}
\usefont{T1}{ptm}{m}{n}
\rput(5.9045315,2.0296875){$\tilde\kappa\inv(y_0)$}
\usefont{T1}{ptm}{m}{n}
\rput(1.7045312,-0.7903125){$\tilde \kappa\inv(y_1)$}

\pscustom[linewidth=0.01]
{
\newpath
\moveto(5.41,-1.3603125)
\lineto(5.58,-1.2303125)
\curveto(5.665,-1.1653125)(5.835,-1.0853125)(5.92,-1.0703125)
\curveto(6.005,-1.0553125)(6.185,-1.0403125)(6.28,-1.0403125)
\curveto(6.375,-1.0403125)(6.61,-1.0603125)(6.75,-1.0803125)
\curveto(6.89,-1.1003125)(7.21,-1.1853125)(7.39,-1.2503124)
\curveto(7.57,-1.3153125)(7.945,-1.4953125)(8.14,-1.6103125)
\curveto(8.335,-1.7253125)(8.67,-1.8903126)(8.81,-1.9403125)
\curveto(8.95,-1.9903125)(9.33,-2.0553124)(9.57,-2.0703125)
\curveto(9.81,-2.0853126)(10.225,-2.0503125)(10.4,-2.0003126)
\curveto(10.575,-1.9503125)(10.785,-1.8703125)(10.82,-1.8403125)
\curveto(10.855,-1.8103125)(10.915,-1.7553124)(10.94,-1.7303125)
\curveto(10.965,-1.7053125)(11.005,-1.6553125)(11.02,-1.6303124)
\curveto(11.035,-1.6053125)(11.055,-1.5553125)(11.07,-1.4803125)
}
\psline[linewidth=0.07,linestyle=none,arrows=->](5.58,-1.2303125)(5.41,-1.3603125)

\end{pspicture} 
}

\caption
{
\label{fig:proving local winding pi}
The local picture near $P_i$.  This allows us to compute the ``local fixed-point index'' $w(\zeta(p_i))$ of $f$ near $P_i$.
}
\end{figure}
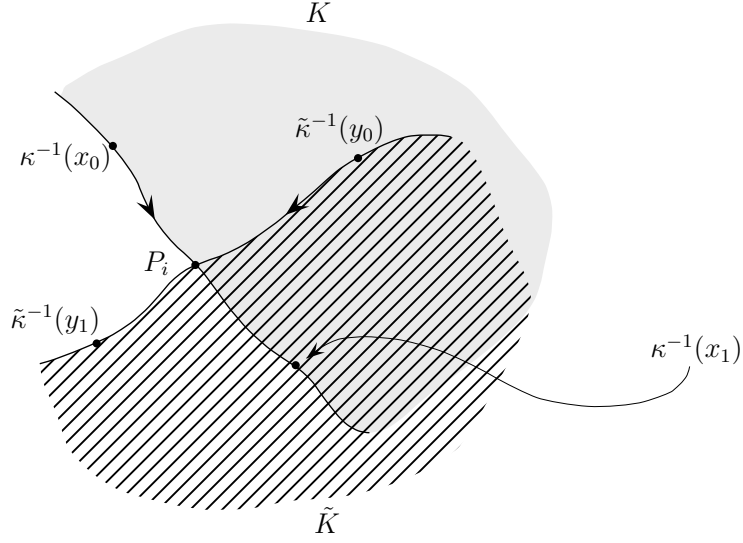

\begin{observation}
If $\bbS^1\times \{\tilde\kappa(\tilde u)\}$ and $\{\kappa(u)\}\times \bbS^1$ are oriented according to the positive orientation on $\bbS^1$, then $w(\bbS^1\times \{\tilde\kappa(\tilde u)\}) = \omega(\partial K, \tilde u)$ and $w(\{\kappa(u)\}\times \bbS^1) = \omega(\partial \tilde K, u)$.
\end{observation}
\noindent It is also easy to see that if we concatenate two closed curves $\gamma_1$ and $\gamma_2$ that meet at a point, we get $w(\gamma_1\circ \gamma_2) = w(\gamma_1) + w(\gamma_2)$.  Thus in light of the orientations on $\partial \Delta_\downarrow(u,\gamma)$ and all other curves concerned we get:
\begin{align*}
\label{prop:computing index from torus eq}
w(\partial \Delta_\downarrow(u,\gamma))& = w(\bbS^1\times \{\tilde\kappa(\tilde u)\}) + w(\{\kappa(u)\}\times \bbS^1) - w(\gamma) \\
& = \omega(\partial K,\tilde u) +  \omega(\partial \tilde K, u) - \eta(\phi)
\end{align*}

For every $i$ let $\zeta(p_i)$ and $\zeta(\tilde p_i)$ be small squares around $p_i$ and $\tilde p_i$ respectively in $\bbT$, oriented as shown in Figure \ref{fig:proving amazing index}.  By \emph{square} we mean a simple closed curve which decomposes into four ``sides,'' so that on a given side one of the two coordinates of $\bbS^1\times \bbS^1 = \bbT$ is constant.  Pick the rectangles small enough so that the closed boxes they bound are pairwise disjoint and do not meet $\partial \Delta_\downarrow(u,\gamma)$.

Let $\Gamma$ be the closed curve in $\Delta_\downarrow(u,\gamma)$ obtained in the following way.  First, start with every loop $\zeta(p_i)$ and $\zeta(\tilde p_i)$ for those $p_i$ and $\tilde p_i$ lying in $\Delta_\downarrow(u,\gamma)$.  Let $\delta_0$ be an arc contained in the interior of $\Delta_\downarrow(u,\gamma)$ which meets each $\zeta(p_i)$ and $\zeta(\tilde p_i)$ contained in $\Delta_\downarrow(u,\gamma)$ at exactly one point.  It is easy to prove inductively that such an arc exists.  Let $\delta$ be the closed curve obtained by traversing $\delta_0$ first in one direction, then in the other.  Then let $\Gamma$ be obtained by concatenating $\delta$ with every $\zeta(p_i)$ and $\zeta(\tilde p_i)$ contained in $\Delta_\downarrow(u,\gamma)$.
\begin{observation}
The curves $\Gamma$ and $\partial \Delta_\downarrow(u,\gamma)$ are homotopic in $\bbT\setminus \{p_1,\ldots,\allowbreak p_M,\allowbreak\tilde p_1,\ldots,\allowbreak\tilde p_M\}$.  Also $w(\delta) = 0$.  It follows that:
\[
w(\partial \Delta_\downarrow(u,\gamma)) = w(\Gamma) = \sum_{p_i\in \Delta_\downarrow(u,\gamma)} w(\zeta(p_i)) + \sum_{\tilde p_i\in \Delta_\downarrow(u,\gamma)} w(\zeta(\tilde p_i))
\]
\end{observation}

\noindent See Figure \ref{fig:proving amazing index} for an example.  On the other hand, the following holds.

\begin{observation}
\label{lem:local windings for pi qi}
$w(\zeta(p_i)) = 1$, $w(\zeta(\tilde p_i)) = -1$
\end{observation}

\noindent To see why, suppose that $\zeta(p_i) = \partial ( [x_0 \to x_1]_{\bbS^1} \times [y_0\to y_1]_{\bbS^1})$.  Then up to orientation-preserving homeomorphism the picture near $P_i$ is as in Figure \ref{fig:proving local winding pi}.  We let $(x,y)$ traverse $\zeta(p_i)$ positively starting at $(x_0,y_0)$, keeping track of the vector $\tilde\kappa\inv(y) - \tilde\kappa\inv(x)$ as we do so.  The vector $\tilde\kappa\inv(y_0) - \tilde\kappa\inv(x_0)$ points to the right.  As $x$ varies from $x_0$ to $x_1$, the vector $\tilde\kappa\inv(y) - \tilde\kappa\inv(x)$ rotates in the positive direction, that is, counter-clockwise, until it arrives at $\tilde \kappa\inv (y_0) - \kappa\inv (x_1)$, which points upward.  Continuing in this fashion, we see that $\tilde\kappa\inv(y) - \tilde\kappa\inv(x)$ makes one full counter-clockwise rotation as we traverse $\zeta(p_i)$.  The proof that $w(\zeta(\tilde p_i)) = -1$ is similar.  Combining all of our observations establishes equation \ref{eq 1:prop:computing index from torus eq}.  The proof that equation \ref{eq 2:prop:computing index from torus eq} holds is similar.
\end{proof}

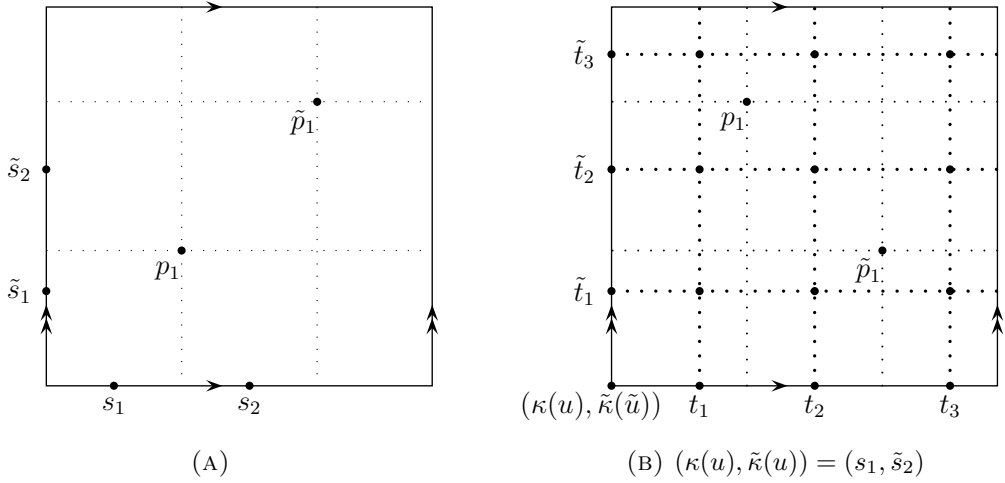
\begin{figure}[t]
\centering
\subfloat[]
{\label{fig:torus 2}
% Generated with LaTeXDraw 2.0.8
% Sat Dec 10 18:58:04 EST 2011
% \usepackage[usenames,dvipsnames]{pstricks}
% \usepackage{epsfig}
% \usepackage{pst-grad} % For gradients
% \usepackage{pst-plot} % For axes
\scalebox{.9} % Change this value to rescale the drawing.
{
\begin{pspicture}(0,-3.1225)(6.8490624,3.0825)%

\usefont{T1}{ptm}{m}{n}%

\psdots[dotsize=0.12](3,-0.6)%
\rput(2.8,-0.9){$p_1$}%

\psdots[dotsize=0.12](5,1.6)
\rput(4.8,1.3){$\tilde p_1$}%

\psframe[linewidth=0.02,dimen=middle](6.7,-2.6)(1,3)%

\psline[linewidth=0.02cm,linestyle=dotted,dotsep=0.16cm](1,1.6)(6.7,1.6)
\psline[linewidth=0.02cm,linestyle=dotted,dotsep=0.16cm](1,-0.6)(6.7,-0.6)%

\psline[linewidth=0.02cm,linestyle=dotted,dotsep=0.16cm](5,3)(5,-2.6)
\psline[linewidth=0.02cm,linestyle=dotted,dotsep=0.16cm](3,3)(3,-2.6)%

\psdots[dotsize=0.12](1,-1.2)
\rput(0.6,-1.2){$\tilde s_1$}
\psdots[dotsize=0.12](1,0.6)
\rput(0.6,0.6){$\tilde s_2$}%
%\psdots[dotsize=0.12](1,2.3)
%\rput(0.6,2.3){$\tilde s_1$}%

\rput(2,-2.9){$s_1$}
\psdots[dotsize=0.12](2,-2.6)
\rput(4,-2.9){$s_2$}
\psdots[dotsize=0.12](4,-2.6)%
%\rput(6,-2.9){$s_1$}
%\psdots[dotsize=0.12](6,-2.6)%

\psline[linestyle=none,linewidth=0.06,arrows=->](1,-2.6)(3.6,-2.6)
\psline[linestyle=none,linewidth=0.06,arrows=->](1,3)(3.6,3)%

\psline[linestyle=none,linewidth=0.06,arrows=->](1,-2.6)(1,-1.6)
\psline[linestyle=none,linewidth=0.06,arrows=->](6.7,-2.6)(6.7,-1.6)%

\psline[linestyle=none,linewidth=0.06,arrows=->](1,-2.6)(1,-1.4)
\psline[linestyle=none,linewidth=0.06,arrows=->](6.7,-2.6)(6.7,-1.4)%

\end{pspicture} 
}
\label{fig:eyes param 2}
}\qquad
\subfloat[$(\kappa(u),\tilde\kappa(u))=(s_1,\tilde s_2)$]
{
\label{fig:chose ess}
% Generated with LaTeXDraw 2.0.8
% Sat Dec 10 18:58:04 EST 2011
% \usepackage[usenames,dvipsnames]{pstricks}
% \usepackage{epsfig}
% \usepackage{pst-grad} % For gradients
% \usepackage{pst-plot} % For axes
\scalebox{.9} % Change this value to rescale the drawing.
{
\begin{pspicture}(0,-3.1225)(6.8490624,3.0825)%

\usefont{T1}{ptm}{m}{n}%

\psdots[dotsize=0.12](3,1.6)%
\rput(2.8,1.3){$p_1$}%

\psdots[dotsize=0.12](5,-0.6)%
\rput(4.8,-0.9){$\tilde p_1$}%

\psdots[dotsize=0.12](1,-2.6)%
\rput(0.7,-2.9){$(\kappa(u),\tilde\kappa(\tilde u))$}%

\psframe[linewidth=0.02,dimen=middle](6.7,-2.6)(1,3)%

\psline[linewidth=0.03cm,linestyle=dotted,dotsep=0.16cm](1,1.6)(6.7,1.6)
\psline[linewidth=0.03cm,linestyle=dotted,dotsep=0.16cm](1,-0.6)(6.7,-0.6)%

\psline[linewidth=0.05cm,linestyle=dotted,dotsep=0.16cm](1,2.3)(6.7,2.3)
\psline[linewidth=0.05cm,linestyle=dotted,dotsep=0.16cm](1,0.6)(6.7,0.6)
\psline[linewidth=0.05cm,linestyle=dotted,dotsep=0.16cm](1,-1.2)(6.7,-1.2)%

\psdots[dotsize=0.12](2.3,2.3)
\psdots[dotsize=0.12](4,2.3)
\psdots[dotsize=0.12](6,2.3)%

\psdots[dotsize=0.12](2.3,0.6)
\psdots[dotsize=0.12](4,0.6)
\psdots[dotsize=0.12](6,0.6)%

%\psdots[dotsize=0.22,dotstyle=triangle*](2.3,-1.2)
\psdots[dotsize=0.12](2.3,-1.2)
\psdots[dotsize=0.12](4,-1.2)
\psdots[dotsize=0.12](6,-1.2)%

\psline[linewidth=0.05cm,linestyle=dotted,dotsep=0.16cm](2.3,3)(2.3,-2.6)
\psline[linewidth=0.05cm,linestyle=dotted,dotsep=0.16cm](4,3)(4,-2.6)
\psline[linewidth=0.05cm,linestyle=dotted,dotsep=0.16cm](6,3)(6,-2.6)%

\psline[linewidth=0.03cm,linestyle=dotted,dotsep=0.16cm](5,3)(5,-2.6)
\psline[linewidth=0.03cm,linestyle=dotted,dotsep=0.16cm](3,3)(3,-2.6)%

\psdots[dotsize=0.12](1,-1.2)
\rput(0.6,-1.2){$\tilde t_1$}
\psdots[dotsize=0.12](1,0.6)
\rput(0.6,0.6){$\tilde t_2$}
\psdots[dotsize=0.12](1,2.3)
\rput(0.6,2.3){$\tilde t_3$}%

\rput(2.3,-2.9){$t_1$}
\psdots[dotsize=0.12](2.3,-2.6)
\rput(4,-2.9){$t_2$}
\psdots[dotsize=0.12](4,-2.6)
\rput(6,-2.9){$t_3$}
\psdots[dotsize=0.12](6,-2.6)%

\psline[linestyle=none,linewidth=0.06,arrows=->](1,-2.6)(3.6,-2.6)
\psline[linestyle=none,linewidth=0.06,arrows=->](1,3)(3.6,3)%

\psline[linestyle=none,linewidth=0.06,arrows=->](1,-2.6)(1,-1.6)
\psline[linestyle=none,linewidth=0.06,arrows=->](6.7,-2.6)(6.7,-1.6)%

\psline[linestyle=none,linewidth=0.06,arrows=->](1,-2.6)(1,-1.4)
\psline[linestyle=none,linewidth=0.06,arrows=->](6.7,-2.6)(6.7,-1.4)%

\end{pspicture} 
}
}
\caption
{
\label{fig:first two}
Two drawings of a torus parametrization for two eyes whose boundaries meet exactly twice.  There is some choice of base point giving the drawing on the left.  The drawing on the right is the same torus parametrization drawn using the base point $(\kappa(u),\tilde\kappa(u))=(s_1,\tilde s_2)$.
}
\end{figure}

\section{Proof of Propositions \ref{prop:nontrivial eye int then done} and \ref{prop:black box 3}\twostars}
\label{sec proof prop}

In this section, we prove the two remaining outstanding propositions to complete the proof of our Main Index Theorem \ref{mainindex} and thus our main rigidity results.\medskip

\begin{proposition}
\label{prop:nontrivial eye int then done}
Let $\{A,B\}$ and $\{\tilde A,\tilde B\}$ be pairs of overlapping closed disks in the plane $\bbC$ in general position.  Suppose that neither of $E = A\cap B$ and $\tilde E = \tilde A \cap \tilde B$ contains the other.  Suppose further that $A\setminus B$ and $\tilde A \setminus \tilde B$ meet, and that $B\setminus A$ and $\tilde B\setminus \tilde A$ meet.  Then there is a faithful indexable homeomorphism $\epsilon:\partial E\to\partial \tilde E$ satisfying $\eta(e) = 0$.
\end{proposition}

\begin{proof}
If $\partial E$ and $\partial \tilde E$ do not meet, we get that $E$ and $\tilde E$ are disjoint.  Then any indexable homeomorphism $\epsilon:\partial E \to \partial \tilde E$ satisfies $\eta(\epsilon) = 0$.  Thus suppose that $\partial E$ and $\partial \tilde E$ meet.  Fix a torus parametrization for $E$ and $\tilde E$ via $\kappa:\partial E\to \bbS^1$ and $\tilde\kappa:\partial \tilde E\to \bbS^1$.  As before denote by $p_i$ the points of $\partial E \cap \partial \tilde E$ where $\partial E$ is entering $\tilde E$, and by $\tilde p_i$ those where $\partial \tilde E$ is entering $E$.  Note that $\partial E$ and $\partial \tilde E$ meet at exactly 2, 4, or 6 points by Lemma \ref{lem1}.  The proof breaks into these three cases.

\begin{starcase}
Suppose that $\partial E$ and $\partial \tilde E$ meet at exactly two points.
\end{starcase}

\noindent Then with an appropriate choice of base point, the torus parametrization for $E$ and $\tilde E$ is as shown in Figure \ref{fig:torus 2}.  The points $s_1,s_2\in \bbS^1$ in Figure \ref{fig:torus 2} are exactly the topologically distinct places where $\kappa(u)$ may be, similarly $\tilde s_1,\tilde s_2\in \bbS^1$ for $\tilde\kappa(\tilde u)$.  A choice of $(s_j,\tilde s_{\tilde j}) = (\kappa(u),\tilde\kappa(\tilde u))$ completely determines the topological configuration of $\{E, \tilde E, u, \tilde u\}$, and conversely every possible topological configuration of those sets is achieved via this procedure.  By Lemma \ref{lem5} we may suppose without loss of generality that $u\not\in \tilde E$, thus that $\kappa(u) = s_1$.

\begin{figure}[t]
\centering
\subfloat[$(\kappa(u),\tilde\kappa(\tilde u)) = (s_1,\tilde s_1)$]
{
\label{fig:two pts 2}
% Generated with LaTeXDraw 2.0.8
% Tue Mar 20 16:20:01 EDT 2012
% \usepackage[usenames,dvipsnames]{pstricks}
% \usepackage{epsfig}
% \usepackage{pst-grad} % For gradients
% \usepackage{pst-plot} % For axes
\scalebox{.9} % Change this value to rescale the drawing.
{
\begin{pspicture}(0,-3.38)(6.14,3.08)
\psframe[linewidth=0.02,dimen=middle](0.06,3.0)(6.06,-3.0)
\rput(-0.28,-3.3){$(\kappa(u),\tilde\kappa(\tilde u))$}
\psline[linestyle=dotted,dotsep=0.16cm,linewidth=0.018](0.06,0)(6.06,0)
\psline[linestyle=dotted,dotsep=0.16cm,linewidth=0.018](0.06,2)(6.06,2)
\psline[linestyle=dotted,dotsep=0.16cm,linewidth=0.018](0.06,-2)(6.06,-2)
\psline[linestyle=dotted,dotsep=0.16cm,linewidth=0.018](1.06,-3.0)(1.06,3.0)
\psline[linestyle=dotted,dotsep=0.16cm,linewidth=0.018](3.06,-3.0)(3.06,3.0)
\psline[linestyle=dotted,dotsep=0.16cm,linewidth=0.018](5.06,-3.0)(5.06,3.0)
\rput(1.16,-3.3){$t_1$}
\rput(3.06,-3.3){$t_2$}
\rput(5.06,-3.3){$t_3$}
\rput(-.2,-2){$\tilde t_1$}
\rput(-.2,0){$\tilde t_2$}
\rput(-.2,2){$\tilde t_3$}
\psdots[dotsize=0.12](0.06,2.0)
\psdots[dotsize=0.12](0.06,0.0)
\psdots[dotsize=0.12](0.06,-2.0)
\psdots[dotsize=0.12](0.06,-3.0)
\psdots[dotsize=0.12](1.06,-3.0)
\psdots[dotsize=0.12](3.06,-3.0)
\psdots[dotsize=0.12](5.06,-3.0)
\psdots[dotsize=0.12](2.06,-1.0)
\psdots[dotsize=0.12](4.06,1.0)
\psdots[dotsize=0.18,dotstyle=diamond*](1.06,2.0)
\psdots[dotsize=0.18,dotstyle=asterisk](3.06,2.0)
\psdots[dotsize=0.12](5.06,2.0)
\psdots[dotsize=0.12](1.06,0.0)
\psdots[dotsize=0.12](3.06,0.0)
\psdots[dotsize=0.12](5.06,0.0)
\psdots[dotsize=0.12](1.06,-2.0)
\psdots[dotsize=0.18,dotstyle=asterisk](3.06,-2.0)
\psdots[dotsize=0.18,dotstyle=diamond*](5.06,-2.0)
\usefont{T1}{ptm}{m}{n}
\rput(2.2645311,-1.29){$p_1$}
\usefont{T1}{ptm}{m}{n}
\rput(3.7745314,1.41){$\tilde p_1$}
\psbezier[linewidth=0.02](0.06,-3.0)(0.76,-2.7)(0.96,-2.4)(1.06,-2.0)(1.16,-1.6)(1.6466855,-0.9072548)(2.06,-0.8)(2.4733145,-0.69274515)(4.7215776,-0.3743383)(5.06,0.0)(5.3984227,0.3743383)(5.86,1.9)(6.06,3.0)
\psbezier[linewidth=0.02](0.06,-3.0)(0.26,-2.2)(0.86,-1.2)(1.36,-0.8)(1.86,-0.4)(2.5207036,-0.2431068)(3.06,0.0)(3.5992963,0.2431068)(4.46638,0.6915863)(4.66,0.9)(4.85362,1.1084137)(5.66,2.0)(6.06,3.0)
\psbezier[linewidth=0.02](0.06,-3.0)(0.16,-1.8)(0.3082702,-0.6594712)(1.06,0.0)(1.8117298,0.6594712)(3.632032,0.5392249)(4.16,0.8)(4.6879683,1.0607752)(4.832477,1.721402)(5.06,2.0)(5.2875233,2.278598)(5.56,2.5)(6.06,3.0)
\end{pspicture} 
}
}
\quad
\subfloat[$(\kappa(u),\tilde\kappa(\tilde u)) = (s_1,\tilde s_2)$]
{
\label{fig:two pts 1}
% Generated with LaTeXDraw 2.0.8
% Tue Mar 20 16:16:13 EDT 2012
% \usepackage[usenames,dvipsnames]{pstricks}
% \usepackage{epsfig}
% \usepackage{pst-grad} % For gradients
% \usepackage{pst-plot} % For axes
\scalebox{.9} % Change this value to rescale the drawing.
{
\begin{pspicture}(0,-3.38)(6.14,3.08)
\psframe[linewidth=0.02,dimen=middle](0.06,3.0)(6.06,-3.0)
\rput(-0.28,-3.3){$(\kappa(u),\tilde\kappa(\tilde u))$}
\psline[linestyle=dotted,dotsep=0.16cm,linewidth=0.018](0.06,0)(6.06,0)
\psline[linestyle=dotted,dotsep=0.16cm,linewidth=0.018](0.06,2)(6.06,2)
\psline[linestyle=dotted,dotsep=0.16cm,linewidth=0.018](0.06,-2)(6.06,-2)
\psline[linestyle=dotted,dotsep=0.16cm,linewidth=0.018](1.06,-3.0)(1.06,3.0)
\psline[linestyle=dotted,dotsep=0.16cm,linewidth=0.018](3.06,-3.0)(3.06,3.0)
\psline[linestyle=dotted,dotsep=0.16cm,linewidth=0.018](5.06,-3.0)(5.06,3.0)
\rput(1.16,-3.3){$t_1$}
\rput(3.06,-3.3){$t_2$}
\rput(5.06,-3.3){$t_3$}
\rput(-.2,-2){$\tilde t_1$}
\rput(-.2,0){$\tilde t_2$}
\rput(-.2,2){$\tilde t_3$}
\psdots[dotsize=0.12](0.06,2.0)
\psdots[dotsize=0.12](0.06,0.0)
\psdots[dotsize=0.12](0.06,-2.0)
\psdots[dotsize=0.12](0.06,-3.0)
\psdots[dotsize=0.12](1.06,-3.0)
\psdots[dotsize=0.12](3.06,-3.0)
\psdots[dotsize=0.12](5.06,-3.0)
\psdots[dotsize=0.12](2.06,1.0)
\psdots[dotsize=0.12](4.06,-1.0)
\psdots[dotsize=0.12](1.06,2.0)
\psdots[dotsize=0.12](3.06,2.0)
\psdots[dotsize=0.12](5.06,2.0)
\psdots[dotsize=0.12](1.06,0.0)
\psdots[dotsize=0.18,dotstyle=asterisk](3.06,0.0)
\psdots[dotsize=0.12](5.06,0.0)
\psdots[dotsize=0.12](1.06,-2.0)
\psdots[dotsize=0.12](3.06,-2.0)
\psdots[dotsize=0.12](5.06,-2.0)
\psbezier[linewidth=0.02](0.06,-3.0)(0.16,-0.2)(0.66,1.5)(1.06,2.0)(1.46,2.5)(2.56,2.9)(6.06,3.0)
\psbezier[linewidth=0.02](0.06,-3.0)(0.46,-1.1)(0.86,-0.6)(1.06,0.0)(1.26,0.6)(1.46,1.2)(1.66,1.4)(1.86,1.6)(2.66,1.9)(3.06,2.0)(3.46,2.1)(4.86,2.7)(6.06,3.0)
\psbezier[linewidth=0.02](0.06,-3.0)(1.56,-2.5)(1.96,-2.4)(3.06,-2.0)(4.16,-1.6)(4.76,-1.0)(5.06,0.0)(5.36,1.0)(5.56,1.5)(6.06,3.0)
\psbezier[linewidth=0.02](0.06,-3.0)(1.76,-2.7)(4.552893,-2.5071068)(5.06,-2.0)(5.5671067,-1.4928932)(5.76,0.8)(6.06,3.0)
\psbezier[linewidth=0.02](0.06,-3.0)(0.86,-2.5)(0.8830243,-2.2789164)(1.06,-2.0)(1.2369757,-1.7210835)(1.2528932,0.8928932)(1.76,1.3)(2.2671068,1.7071068)(4.6810837,1.8230243)(5.06,2.0)(5.4389167,2.1769757)(5.76,2.5)(6.06,3.0)
\usefont{T1}{ptm}{m}{n}
\rput(2.1645312,0.71){$p_1$}
\usefont{T1}{ptm}{m}{n}
\rput(3.9745312,-0.69){$\tilde p_1$}
\end{pspicture} 
}
}
\caption
{
\label{fig:two pts}
Graphs of homeomorphisms $\epsilon$ giving $\eta(\epsilon) = 0$ for a pair of eyes whose boundaries meet twice.  The torus parametrizations are drawn using the indicated choice of base point.
}
\end{figure}
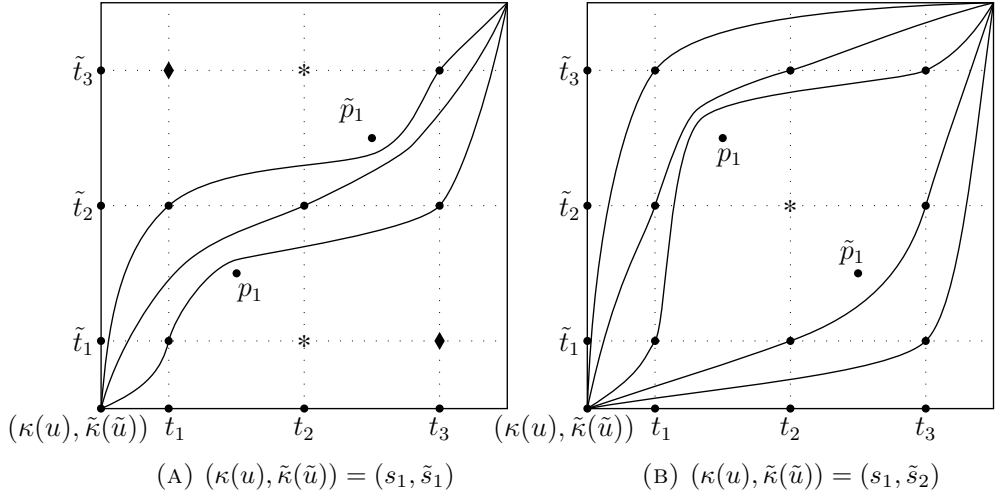

Suppose first that $\tilde\kappa(\tilde u) = \tilde s_2$.  In Figure \ref{fig:chose ess} we redraw the torus parametrization for $E$ and $\tilde E$ using the base point $(s_1,\tilde s_2) = (\kappa(u),\tilde\kappa(\tilde u))$.  Then the points $t_1,t_2,t_3\in \bbS^1$ are exactly the topologically distinct places where $\kappa(v)$ may be, similarly $\tilde t_1,\tilde t_2,\tilde t_3\in \bbS^1$ for $\tilde\kappa(\tilde v)$.
\begin{observation}
A choice of $(s_j,\tilde s_{\tilde j}) = (\kappa(u),\tilde\kappa(\tilde u))$ and a subsequent choice of $(t_k,\tilde t_{\tilde k}) = (\kappa(v),\tilde\kappa(\tilde v))$ together completely determine the topological configuration of $\{E,  \tilde E,  u,  \tilde u,  v,  \tilde v\}$.  Conversely every possible topological configuration of $\{E, \tilde E, u, \tilde u, v, \allowbreak \tilde v\}$ is achieved by some choice of $(s_j,\tilde s_{\tilde j})$, and then a subsequent choice of $(t_k,\tilde t_{\tilde k})$, for $(\kappa(u),\tilde\kappa(\tilde u))$ and $(\kappa(v),\tilde\kappa(\tilde v))$ respectively.
\end{observation}

\begin{figure}[t]
\centering
\subfloat[]
{
\label{fig:eyes param 4}
% Generated with LaTeXDraw 2.0.8
% Sat Dec 10 18:58:04 EST 2011
% \usepackage[usenames,dvipsnames]{pstricks}
% \usepackage{epsfig}
% \usepackage{pst-grad} % For gradients
% \usepackage{pst-plot} % For axes
\scalebox{1} % Change this value to rescale the drawing.
{
\begin{pspicture}(0,-3.1225)(6.8490624,3.0825)
\usefont{T1}{ptm}{m}{n}%
\psdots[dotsize=0.12](2.3,-1.4)
\rput(2.1,-1.7){$p_1$}%
\psdots[dotsize=0.12](3.4,-0.3)
\rput(3.2,-0.6){$\tilde p_1$}%
\psdots[dotsize=0.12](4.5,0.8)
\rput(4.3,0.5){$p_2$}%
\psdots[dotsize=0.12](5.6,1.9)
\rput(5.4,1.6){$\tilde p_2$}%
\psframe[linewidth=0.02,dimen=middle](6.7,-2.6)(1,3)
\psline[linewidth=0.02cm,linestyle=dotted,dotsep=0.16cm](1,-1.4)(6.7,-1.4)
\psline[linewidth=0.02cm,linestyle=dotted,dotsep=0.16cm](1,-0.3)(6.7,-0.3)
\psline[linewidth=0.02cm,linestyle=dotted,dotsep=0.16cm](1,0.8)(6.7,0.8)
\psline[linewidth=0.02cm,linestyle=dotted,dotsep=0.16cm](1,1.9)(6.7,1.9)%
\psline[linewidth=0.02cm,linestyle=dotted,dotsep=0.16cm](5.6,3)(5.6,-2.6)
\psline[linewidth=0.02cm,linestyle=dotted,dotsep=0.16cm](4.5,3)(4.5,-2.6)
\psline[linewidth=0.02cm,linestyle=dotted,dotsep=0.16cm](3.4,3)(3.4,-2.6)
\psline[linewidth=0.02cm,linestyle=dotted,dotsep=0.16cm](2.3,3)(2.3,-2.6)%
\rput(0.6,-2.1){$\tilde s_1$}
\psdots[dotsize=0.12](1,-2.1)
\rput(0.6,-0.8){$\tilde s_2$}
\psdots[dotsize=0.12](1,-0.8)
\rput(0.6,0.2){$\tilde s_3$}
\psdots[dotsize=0.12](1,0.2)
\rput(0.6,1.4){$\tilde s_4$}
\psdots[dotsize=0.12](1,1.4)%
%\rput(0.6,2.5){$\tilde s_1$}%
\rput(1.8,-2.9){$s_1$}
\psdots[dotsize=0.12](1.8,-2.6)
\rput(2.9,-2.9){$s_2$}
\psdots[dotsize=0.12](2.9,-2.6)
\rput(4,-2.9){$s_3$}
\psdots[dotsize=0.12](4,-2.6)
\rput(5.1,-2.9){$s_4$}
\psdots[dotsize=0.12](5.1,-2.6)%
%\rput(6.1,-2.9){$s_1$}%
\psline[linestyle=none,linewidth=0.06,arrows=->](1,-2.6)(3.6,-2.6)
\psline[linestyle=none,linewidth=0.06,arrows=->](1,3)(3.6,3)%
\psline[linestyle=none,linewidth=0.06,arrows=->](1,-2.6)(1,-1.6)
\psline[linestyle=none,linewidth=0.06,arrows=->](6.7,-2.6)(6.7,-1.6)%
\psline[linestyle=none,linewidth=0.06,arrows=->](1,-2.6)(1,-1.4)
\psline[linestyle=none,linewidth=0.06,arrows=->](6.7,-2.6)(6.7,-1.4)%
\end{pspicture} 
}%scalebox
}%subfloat
\qquad
\subfloat[$(\kappa(u),\tilde\kappa(\tilde u)) = (s_1,\tilde s_1)$]
{\label{fig:four pts 1}% Generated with LaTeXDraw 2.0.8
% Tue Apr 10 00:46:04 EDT 2012
% \usepackage[usenames,dvipsnames]{pstricks}
% \usepackage{epsfig}
% \usepackage{pst-grad} % For gradients
% \usepackage{pst-plot} % For axes
\scalebox{.6} % Change this value to rescale the drawing.
{
\begin{pspicture}(0,-5.08)(10.14,5.08)
\usefont{T1}{ppl}{m}{n}
\rput(2.1045313,-3.37){$p_1$}
\usefont{T1}{ppl}{m}{n}
\rput(4.2945313,-1.19){$\tilde p_1$}
\usefont{T1}{ppl}{m}{n}
\rput(6.304531,0.75){$p_2$}
\usefont{T1}{ppl}{m}{n}
\rput(7.8945312,3.35){$\tilde p_2$}
\psframe[linewidth=0.03,dimen=middle](0.06,-5)(10.06,5)
\psline[linestyle=dotted,dotsep=0.16cm,linewidth=0.018](0.06,-4)(10.06,-4)
\psline[linestyle=dotted,dotsep=0.16cm,linewidth=0.018](0.06,-2)(10.06,-2)
\psline[linestyle=dotted,dotsep=0.16cm,linewidth=0.018](0.06,0)(10.06,0)
\psline[linestyle=dotted,dotsep=0.16cm,linewidth=0.018](0.06,2)(10.06,2)
\psline[linestyle=dotted,dotsep=0.16cm,linewidth=0.018](0.06,4)(10.06,4)
\psline[linestyle=dotted,dotsep=0.16cm,linewidth=0.018](1.06,-5)(1.06,5)
\psline[linestyle=dotted,dotsep=0.16cm,linewidth=0.018](3.06,-5)(3.06,5)
\psline[linestyle=dotted,dotsep=0.16cm,linewidth=0.018](5.06,-5)(5.06,5)
\psline[linestyle=dotted,dotsep=0.16cm,linewidth=0.018](7.06,-5)(7.06,5)
\psline[linestyle=dotted,dotsep=0.16cm,linewidth=0.018](9.06,-5)(9.06,5)
\rput(1.06,-5.3){$t_1$}
\rput(3.06,-5.3){$t_2$}
\rput(5.06,-5.3){$t_3$}
\rput(7.06,-5.3){$t_4$}
\rput(9.06,-5.3){$t_5$}
\rput(-.2,-4){$\tilde t_1$}
\rput(-.2,-2){$\tilde t_2$}
\rput(-.2,0){$\tilde t_3$}
\rput(-.2,2){$\tilde t_4$}
\rput(-.2,4){$\tilde t_5$}
\psdots[dotsize=0.12](0.06,4.0)
\psdots[dotsize=0.12](0.06,2.0)
\psdots[dotsize=0.12](0.06,0.0)
\psdots[dotsize=0.12](0.06,-2.0)
\psdots[dotsize=0.12](0.06,-4.0)
\psdots[dotsize=0.12](1.06,-5.0)
\psdots[dotsize=0.12](3.06,-5.0)
\psdots[dotsize=0.12](5.06,-5.0)
\psdots[dotsize=0.12](7.06,-5.0)
\psdots[dotsize=0.12](9.06,-5.0)
\psdots[dotsize=0.2,fillstyle=solid,dotstyle=o](1.06,4.0)
\psdots[dotsize=0.2,fillstyle=solid,dotstyle=o](1.06,2.0)
\psdots[dotsize=0.2,fillstyle=solid,dotstyle=o](3.06,4.0)
\psdots[dotsize=0.12](3.06,2.0)
\psdots[dotsize=0.2,dotstyle=diamond*](5.06,4.0)
\psdots[dotsize=0.12](5.06,2.0)
\psdots[dotsize=0.2,fillstyle=solid,dotstyle=o](7.06,4.0)
\psdots[dotsize=0.2,fillstyle=solid,dotstyle=o](7.06,2.0)
\psdots[dotsize=0.2,fillstyle=solid,dotstyle=o](9.06,4.0)
\psdots[dotsize=0.2,fillstyle=solid,dotstyle=o](9.06,2.0)
\psdots[dotsize=0.12](9.06,0.0)
\psdots[dotsize=0.2,dotstyle=asterisk](7.06,0.0)
\psdots[dotsize=0.12](5.06,0.0)
\psdots[dotsize=0.2,dotstyle=asterisk](3.06,0.0)
\psdots[dotsize=0.12](1.06,0.0)
\psdots[dotsize=0.2,fillstyle=solid,dotstyle=o](1.06,-2.0)
\psdots[dotsize=0.2,fillstyle=solid,dotstyle=o](3.06,-2.0)
\psdots[dotsize=0.12](5.06,-2.0)
\psdots[dotsize=0.12](7.06,-2.0)
\psdots[dotsize=0.2,fillstyle=solid,dotstyle=o](9.06,-2.0)
\psdots[dotsize=0.2,fillstyle=solid,dotstyle=o](9.06,-4.0)
\psdots[dotsize=0.2,fillstyle=solid,dotstyle=o](7.06,-4.0)
\psdots[dotsize=0.2,dotstyle=diamond*](5.06,-4.0)
\psdots[dotsize=0.2,fillstyle=solid,dotstyle=o](3.06,-4.0)
\psdots[dotsize=0.2,fillstyle=solid,dotstyle=o](1.06,-4.0)
\psdots[dotsize=0.12](4.06,-1.0)
\psdots[dotsize=0.12](2.06,-3.0)
\psdots[dotsize=0.12](6.06,1.0)
\psdots[dotsize=0.12](8.06,3.0)
\psbezier[linewidth=0.02](0.06,-5.0)(0.36,-2.2)(0.5881421,-0.8816746)(1.06,0.0)(1.531858,0.8816746)(2.16799,1.5479844)(3.06,2.0)(3.95201,2.4520156)(7.591614,2.5041113)(8.16,2.8)(8.728386,3.0958889)(9.36,3.8)(10.06,5.0)
\psbezier[linewidth=0.02](0.06,-5.0)(0.46,-2.5)(1.26,-0.4)(2.36,0.7)(3.46,1.8)(4.06,1.9)(5.06,2.0)(6.06,2.1)(7.513807,1.9671241)(8.36,2.5)(9.206193,3.0328758)(9.46,3.9)(10.06,5.0)
\psbezier[linewidth=0.02](0.06,-5.0)(0.66,-3.4)(1.69473,-2.8013062)(2.16,-2.6)(2.62527,-2.3986938)(3.86,-2.3)(5.06,-2.0)(6.26,-1.7)(8.66,0.2)(10.06,5.0)
\psbezier[linewidth=0.02](0.06,-5.0)(0.66,-3.9)(1.5868326,-2.9874206)(1.96,-2.9)(2.3331673,-2.8125796)(5.760099,-2.536094)(7.06,-2.0)(8.359901,-1.4639062)(8.86,-0.5)(9.06,0.0)(9.26,0.5)(9.76,1.6)(10.06,5.0)
\psbezier[linewidth=0.02](0.06,-5.0)(0.56,-2.6)(4.76,-0.2)(5.06,0.0)(5.36,0.2)(5.66,1.0)(5.86,1.2)(6.06,1.4)(8.027381,1.5193613)(8.56,2.2)(9.092619,2.8806388)(9.46,3.4)(10.06,5.0)
\end{pspicture} 
}%scalebox
}%subfloat

\subfloat[$(\kappa(u),\tilde\kappa(\tilde u)) = (s_1,\tilde s_2)$]
{\label{fig:four pts 2}
% Generated with LaTeXDraw 2.0.8
% Tue Apr 10 00:56:49 EDT 2012
% \usepackage[usenames,dvipsnames]{pstricks}
% \usepackage{epsfig}
% \usepackage{pst-grad} % For gradients
% \usepackage{pst-plot} % For axes
\scalebox{.6} % Change this value to rescale the drawing.
{
\begin{pspicture}(0,-5.08)(10.14,5.08)
\usefont{T1}{ppl}{m}{n}
\rput(2.2045312,2.73){$p_1$}
\usefont{T1}{ppl}{m}{n}
\rput(3.7945313,-2.69){$\tilde p_1$}
\usefont{T1}{ppl}{m}{n}
\rput(5.804531,-0.75){$p_2$}
\usefont{T1}{ppl}{m}{n}
\rput(7.7945313,1.25){$\tilde p_2$}
\psframe[linewidth=0.03,dimen=middle](0.06,-5)(10.06,5)
\psline[linestyle=dotted,dotsep=0.16cm,linewidth=0.018](0.06,-4)(10.06,-4)
\psline[linestyle=dotted,dotsep=0.16cm,linewidth=0.018](0.06,-2)(10.06,-2)
\psline[linestyle=dotted,dotsep=0.16cm,linewidth=0.018](0.06,0)(10.06,0)
\psline[linestyle=dotted,dotsep=0.16cm,linewidth=0.018](0.06,2)(10.06,2)
\psline[linestyle=dotted,dotsep=0.16cm,linewidth=0.018](0.06,4)(10.06,4)
\psline[linestyle=dotted,dotsep=0.16cm,linewidth=0.018](1.06,-5)(1.06,5)
\psline[linestyle=dotted,dotsep=0.16cm,linewidth=0.018](3.06,-5)(3.06,5)
\psline[linestyle=dotted,dotsep=0.16cm,linewidth=0.018](5.06,-5)(5.06,5)
\psline[linestyle=dotted,dotsep=0.16cm,linewidth=0.018](7.06,-5)(7.06,5)
\psline[linestyle=dotted,dotsep=0.16cm,linewidth=0.018](9.06,-5)(9.06,5)
\rput(1.06,-5.3){$t_1$}
\rput(3.06,-5.3){$t_2$}
\rput(5.06,-5.3){$t_3$}
\rput(7.06,-5.3){$t_4$}
\rput(9.06,-5.3){$t_5$}
\rput(-.2,-4){$\tilde t_1$}
\rput(-.2,-2){$\tilde t_2$}
\rput(-.2,0){$\tilde t_3$}
\rput(-.2,2){$\tilde t_4$}
\rput(-.2,4){$\tilde t_5$}
\psdots[dotsize=0.12](0.06,4.0)
\psdots[dotsize=0.12](0.06,2.0)
\psdots[dotsize=0.12](0.06,0.0)
\psdots[dotsize=0.12](0.06,-2.0)
\psdots[dotsize=0.12](0.06,-4.0)
\psdots[dotsize=0.12](1.06,-5.0)
\psdots[dotsize=0.12](3.06,-5.0)
\psdots[dotsize=0.12](5.06,-5.0)
\psdots[dotsize=0.12](7.06,-5.0)
\psdots[dotsize=0.12](9.06,-5.0)
\psdots[dotsize=0.2,fillstyle=solid,dotstyle=o](1.06,4.0)
\psdots[dotsize=0.2,fillstyle=solid,dotstyle=o](1.06,2.0)
\psdots[dotsize=0.2,fillstyle=solid,dotstyle=o](3.06,4.0)
\psdots[dotsize=0.2,fillstyle=solid,dotstyle=o](3.06,2.0)
\psdots[dotsize=0.12](5.06,4.0)
\psdots[dotsize=0.12](5.06,2.0)
\psdots[dotsize=0.2,fillstyle=solid,dotstyle=o](7.06,4.0)
\psdots[dotsize=0.2,dotstyle=asterisk](7.06,2.0)
\psdots[dotsize=0.2,fillstyle=solid,dotstyle=o](9.06,4.0)
\psdots[dotsize=0.2,fillstyle=solid,dotstyle=o](9.06,2.0)
\psdots[dotsize=0.12](9.06,0.0)
\psdots[dotsize=0.12](7.06,0.0)
\psdots[dotsize=0.12](5.06,0.0)
\psdots[dotsize=0.12](3.06,0.0)
\psdots[dotsize=0.12](1.06,0.0)
\psdots[dotsize=0.2,fillstyle=solid,dotstyle=o](1.06,-2.0)
\psdots[dotsize=0.2,dotstyle=asterisk](3.06,-2.0)
\psdots[dotsize=0.12](5.06,-2.0)
\psdots[dotsize=0.2,dotstyle=asterisk](7.06,-2.0)
\psdots[dotsize=0.2,fillstyle=solid,dotstyle=o](9.06,-2.0)
\psdots[dotsize=0.2,fillstyle=solid,dotstyle=o](9.06,-4.0)
\psdots[dotsize=0.2,fillstyle=solid,dotstyle=o](7.06,-4.0)
\psdots[dotsize=0.12](5.06,-4.0)
\psdots[dotsize=0.2,fillstyle=solid,dotstyle=o](3.06,-4.0)
\psdots[dotsize=0.2,fillstyle=solid,dotstyle=o](1.06,-4.0)
\psdots[dotsize=0.12](4.06,-3.0)
\psdots[dotsize=0.12](2.06,3.0)
\psdots[dotsize=0.12](6.06,-1.0)
\psdots[dotsize=0.12](8.06,1.0)
\psbezier[linewidth=0.02](0.06,-5.0)(0.46,-1.0)(0.86,-1.3)(1.06,0.0)(1.26,1.3)(1.46,2.8)(1.86,3.2)(2.26,3.6)(4.069666,3.8612988)(5.06,4.0)(6.0503345,4.1387014)(6.16,4.7)(10.06,5.0)
\psbezier[linewidth=0.02](0.06,-5.0)(2.26,-4.7)(3.5783255,-4.671858)(5.06,-4.0)(6.5416746,-3.3281422)(8.56,-1.5)(9.06,0.0)(9.56,1.5)(9.86,2.9)(10.06,5.0)
\psbezier[linewidth=0.02](0.06,-5.0)(1.76,-2.5)(4.1898003,-0.49269882)(5.06,0.0)(5.9302,0.49269882)(7.8986745,0.15162991)(8.66,0.8)(9.421326,1.4483701)(9.76,3.1)(10.06,5.0)
\psbezier[linewidth=0.02](0.06,-5.0)(0.56,-3.0)(2.229902,-0.55761766)(3.06,0.0)(3.8900979,0.55761766)(7.26,0.5)(8.26,0.9)(9.26,1.3)(9.76,3.3)(10.06,5.0)
\psbezier[linewidth=0.02](0.06,-5.0)(2.06,-3.9)(3.9376483,-3.7827377)(4.26,-3.3)(4.5823517,-2.8172624)(4.66,1.3)(5.06,2.0)(5.46,2.7)(8.06,4.4)(10.06,5.0)
\psbezier[linewidth=0.02](0.06,-5.0)(2.96,-4.3)(4.0188575,-3.8713474)(4.46,-3.5)(4.9011426,-3.1286526)(4.6650157,-2.419025)(5.06,-2.0)(5.454984,-1.5809749)(6.16,-1.6)(6.46,-1.3)(6.76,-1.0)(6.8187675,-0.27124837)(7.06,0.0)(7.3012323,0.27124837)(8.46,0.3)(8.86,0.7)(9.26,1.1)(9.76,2.6)(10.06,5.0)
\end{pspicture} 
}
}%subfloat
\qquad
\subfloat[$(\kappa(u),\tilde\kappa(\tilde u)) = (s_1,\tilde s_3)$]
{\label{fig:four pts 3}
% Generated with LaTeXDraw 2.0.8
% Tue Apr 10 01:05:40 EDT 2012
% \usepackage[usenames,dvipsnames]{pstricks}
% \usepackage{epsfig}
% \usepackage{pst-grad} % For gradients
% \usepackage{pst-plot} % For axes
\scalebox{0.6} % Change this value to rescale the drawing.
{
\begin{pspicture}(0,-5.08)(10.14,5.08)
\usefont{T1}{ppl}{m}{n}
\rput(2.2045312,0.73){$p_1$}
\usefont{T1}{ppl}{m}{n}
\rput(3.8945312,3.31){$\tilde p_1$}
\usefont{T1}{ppl}{m}{n}
\rput(6.204531,-3.35){$p_2$}
\usefont{T1}{ppl}{m}{n}
\rput(7.7945313,-0.65){$\tilde p_2$}
\psframe[linewidth=0.03,dimen=middle](0.06,-5)(10.06,5)
\psline[linestyle=dotted,dotsep=0.16cm,linewidth=0.018](0.06,-4)(10.06,-4)
\psline[linestyle=dotted,dotsep=0.16cm,linewidth=0.018](0.06,-2)(10.06,-2)
\psline[linestyle=dotted,dotsep=0.16cm,linewidth=0.018](0.06,0)(10.06,0)
\psline[linestyle=dotted,dotsep=0.16cm,linewidth=0.018](0.06,2)(10.06,2)
\psline[linestyle=dotted,dotsep=0.16cm,linewidth=0.018](0.06,4)(10.06,4)
\psline[linestyle=dotted,dotsep=0.16cm,linewidth=0.018](1.06,-5)(1.06,5)
\psline[linestyle=dotted,dotsep=0.16cm,linewidth=0.018](3.06,-5)(3.06,5)
\psline[linestyle=dotted,dotsep=0.16cm,linewidth=0.018](5.06,-5)(5.06,5)
\psline[linestyle=dotted,dotsep=0.16cm,linewidth=0.018](7.06,-5)(7.06,5)
\psline[linestyle=dotted,dotsep=0.16cm,linewidth=0.018](9.06,-5)(9.06,5)
\rput(1.06,-5.3){$t_1$}
\rput(3.06,-5.3){$t_2$}
\rput(5.06,-5.3){$t_3$}
\rput(7.06,-5.3){$t_4$}
\rput(9.06,-5.3){$t_5$}
\rput(-.2,-4){$\tilde t_1$}
\rput(-.2,-2){$\tilde t_2$}
\rput(-.2,0){$\tilde t_3$}
\rput(-.2,2){$\tilde t_4$}
\rput(-.2,4){$\tilde t_5$}
\psdots[dotsize=0.12](0.06,4.0)
\psdots[dotsize=0.12](0.06,2.0)
\psdots[dotsize=0.12](0.06,0.0)
\psdots[dotsize=0.12](0.06,-2.0)
\psdots[dotsize=0.12](0.06,-4.0)
\psdots[dotsize=0.12](1.06,-5.0)
\psdots[dotsize=0.12](3.06,-5.0)
\psdots[dotsize=0.12](5.06,-5.0)
\psdots[dotsize=0.12](7.06,-5.0)
\psdots[dotsize=0.12](9.06,-5.0)
\psdots[dotsize=0.2,fillstyle=solid,dotstyle=o](1.06,4.0)
\psdots[dotsize=0.2,fillstyle=solid,dotstyle=o](1.06,2.0)
\psdots[dotsize=0.2,fillstyle=solid,dotstyle=o](3.06,4.0)
\psdots[dotsize=0.12](3.06,2.0)
\psdots[dotsize=0.2,dotstyle=diamond*](5.06,4.0)
\psdots[dotsize=0.12](5.06,2.0)
\psdots[dotsize=0.2,fillstyle=solid,dotstyle=o](7.06,4.0)
\psdots[dotsize=0.12](7.06,2.0)
\psdots[dotsize=0.2,fillstyle=solid,dotstyle=o](9.06,4.0)
\psdots[dotsize=0.2,fillstyle=solid,dotstyle=o](9.06,2.0)
\psdots[dotsize=0.12](9.06,0.0)
\psdots[dotsize=0.2,dotstyle=asterisk](7.06,0.0)
\psdots[dotsize=0.2,dotstyle=pentagon*](5.06,0.0)
\psdots[dotsize=0.2,dotstyle=asterisk](3.06,0.0)
\psdots[dotsize=0.12](1.06,0.0)
\psdots[dotsize=0.2,fillstyle=solid,dotstyle=o](1.06,-2.0)
\psdots[dotsize=0.12](3.06,-2.0)
\psdots[dotsize=0.12](5.06,-2.0)
\psdots[dotsize=0.12](7.06,-2.0)
\psdots[dotsize=0.2,fillstyle=solid,dotstyle=o](9.06,-2.0)
\psdots[dotsize=0.2,fillstyle=solid,dotstyle=o](9.06,-4.0)
\psdots[dotsize=0.2,fillstyle=solid,dotstyle=o](7.06,-4.0)
\psdots[dotsize=0.2,dotstyle=diamond*](5.06,-4.0)
\psdots[dotsize=0.2,fillstyle=solid,dotstyle=o](3.06,-4.0)
\psdots[dotsize=0.2,fillstyle=solid,dotstyle=o](1.06,-4.0)
\psdots[dotsize=0.12](4.06,3.0)
\psdots[dotsize=0.12](2.06,1.0)
\psdots[dotsize=0.12](6.06,-3.0)
\psdots[dotsize=0.12](8.06,-1.0)
\psbezier[linewidth=0.02](0.06,-5.0)(3.16,-3.1)(2.5730214,-2.2618103)(3.06,-2.0)(3.5469787,-1.7381896)(8.397127,-1.3054208)(8.66,-1.1)(8.9228735,-0.89457923)(8.96,-0.5)(9.06,0.0)(9.16,0.5)(9.76,3.6)(10.06,5.0)
\psbezier[linewidth=0.02](0.06,-5.0)(3.46,-3.8)(3.3361204,-2.2826834)(5.06,-2.0)(6.7838798,-1.7173165)(8.56,-1.6)(8.96,-1.2)(9.36,-0.8)(9.76,2.2)(10.06,5.0)
\psbezier[linewidth=0.02](0.06,-5.0)(4.06,-4.6)(6.328431,-2.0635653)(7.06,-2.0)(7.7915688,-1.9364347)(8.86,-1.7)(9.16,-1.4)(9.46,-1.1)(9.86,2.2)(10.06,5.0)
\psbezier[linewidth=0.02](0.06,-5.0)(0.46,-2.7)(0.8773166,-0.42387953)(1.06,0.0)(1.2426834,0.42387953)(1.36,1.1)(1.76,1.3)(2.16,1.5)(6.56,1.7)(7.06,2.0)(7.56,2.3)(7.86,3.9)(10.06,5.0)
\psbezier[linewidth=0.02](0.06,-5.0)(0.36,-2.3)(0.76,1.0)(1.36,1.4)(1.96,1.8)(3.96,1.8)(5.06,2.0)(6.16,2.2)(7.06,4.1)(10.06,5.0)
\psbezier[linewidth=0.02](0.06,-5.0)(0.16,-1.2)(0.76,1.1)(1.06,1.4)(1.36,1.7)(1.6200029,1.7588178)(3.06,2.0)(4.499997,2.2411823)(7.56,4.4)(10.06,5.0)
\end{pspicture} 
}
}%subfloat
\caption%
{%
\label{def}%
The situation if two eyes' boundaries meet four times.  Figure {\ref{fig:eyes param 4}} shows the torus parametrization for $E$ and $\tilde{E}$ with some suitable choice of base point.  Figures {\ref{fig:four pts 1}}--{\ref{fig:four pts 3}} give graphs of homeomorphisms $\epsilon$ giving $\eta(\epsilon) = 0$, with torus parametrizations drawn using base point $(\kappa(u),\tilde\kappa(\tilde u)) = (s_j,\tilde{s}_{\tilde{j}})$ as indicated.
}
\end{figure}
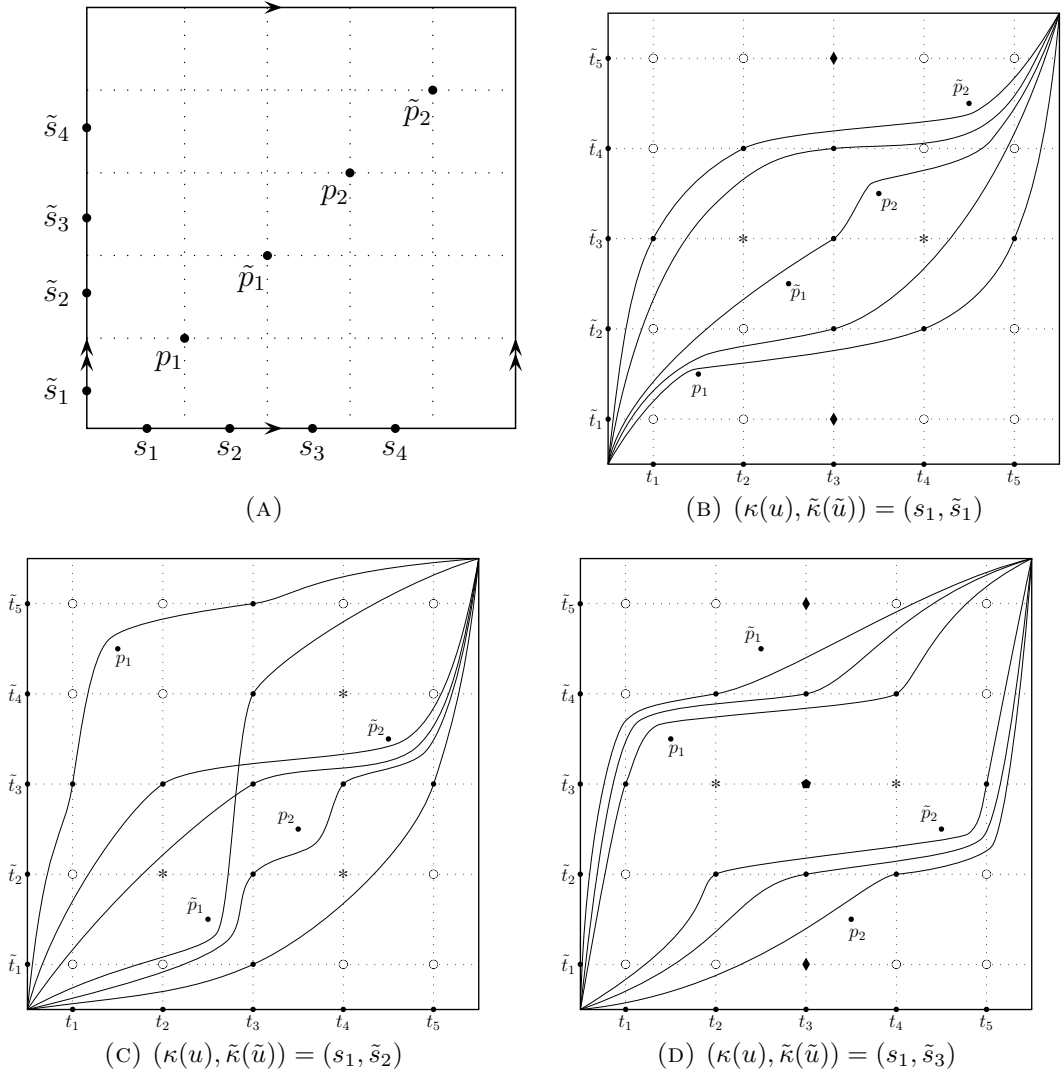

\noindent We are currently working under the assumption that $(\kappa(u),\tilde\kappa(\tilde u)) = (s_1,\tilde s_2)$.  For every choice of $(t_k,\tilde t_{\tilde k}) = (\kappa(v),\tilde\kappa(\tilde v))$ we hope to find a faithful indexable homeomorphism $\epsilon:\partial E \to \partial \tilde E$ so that $\eta(\epsilon) = 0$.
\begin{observation}
Suppose we have drawn the parametrization for $E$ and $\tilde E$ using $(s_j,\tilde s_{\tilde j}) = (\kappa(u),\tilde\kappa(\tilde u))$ as the base point.  Then finding a faithful indexable homeomorphism $\epsilon:\partial E \to \partial \tilde E$ amounts to finding a curve $\gamma$ in $\bbT\setminus \{p_1,\ldots,\tilde p_M,\tilde p_1,\ldots,\tilde p_M\}$ which ``looks like the graph of a strictly increasing function,'' from the lower-left-hand corner $(\kappa(u),\tilde\kappa(u))$ to the upper-right-hand corner, passing through $(\kappa(v),\tilde\kappa(\tilde v)) = (t_k,\tilde t_{\tilde k})$.  Having fixed such a curve $\gamma$, we may compute $\eta(\epsilon)$, where $\epsilon$ is the homeomorphism associated to $\gamma$, using Lemma \ref{prop:computing index from torus}.
\end{observation}

\noindent In our current situation $\kappa(u) = s_1$ implies that $u\not\in \tilde K$, and $\tilde \kappa(\tilde u) = \tilde s_2$ implies that $\tilde u\not\in K$.  Thus by Lemma \ref{prop:computing index from torus} we wish to find curves $\gamma$ so that both $p_1$ and $\tilde p_1$ lie in the upper diagonal $\Delta_\uparrow(u,\gamma)$, or both lie in the lower diagonal $\Delta_\downarrow(u,\gamma)$.  Figure \ref{fig:two pts 1} depicts such a $\gamma$ for every $(t_k,\tilde t_{\tilde k})$ except for $(t_2,\tilde t_2)$.  Suppose $(t_2,\tilde t_2) = (\kappa(v),\tilde\kappa(\tilde v))$.  Then $v\in \tilde K$ and $\tilde v\in K$, so we get a contradiction by Lemma \ref{lem5}.  From now on points $(t_k,\tilde t_{\tilde k})$ which are handled via Lemma \ref{lem5} will be labeled with an asterisk, as in Figure \ref{fig:two pts 1}.

Next suppose that $(\kappa(u),\tilde\kappa(\tilde u)) = (s_1,\tilde s_1)$.  The situation is depicted in Figure \ref{fig:two pts 2}.  Then $u\not\in \tilde K$ and $\tilde u\in K$, so to achieve $\eta(\epsilon) = 0$ we wish to find curves $\gamma$ so that $p_1\in \Delta_\downarrow(u,\gamma)$ and $\tilde p_1\in \Delta_\uparrow(u,\gamma)$.  This time there are four $(t_k,\tilde t_{\tilde k})$ for which this is not possible.  For $(\kappa(v),\tilde\kappa(\tilde v) = (t_2,\tilde t_1), (t_2,\tilde t_3)$ we again get contradictions via Lemma \ref{lem5}.  The following observation will be helpful for $(\kappa(v),\tilde\kappa(\tilde v) = (t_1,\tilde t_3), (t_3,\tilde t_1)$.

\begin{observation}
Choose $(s_j,\tilde s_{\tilde j}) = (\kappa(u),\tilde\kappa(\tilde u))$ and draw our torus parametrization for $E$ and $\tilde E$ using $(\kappa(u),\tilde\kappa(\tilde u))$ as the base point.  Then a choice of $(t_k,\tilde t_{\tilde k}) = (\kappa(v),\tilde\kappa(\tilde v))$ defines for us four ``quadrants,'' namely $[\kappa(u)\to \kappa(v)]_{\bbS^1} \times [\tilde\kappa(\tilde u)\to \tilde\kappa(\tilde v)]_{\bbS^1}$ the points ``below and to the left of'' $(t_k,\tilde t_{\tilde k})$, etc.  Then which of the two arcs $\partial A\cap \partial E$ and $\partial B\cap \partial E$, and which of $\partial \tilde A\cap \partial \tilde E$ and $\partial \tilde B\cap \partial \tilde E$, a point $P_i$ or $\tilde P_i$ lies on is determined by which quadrant $p_i$ or $\tilde p_i$ lies in.
\end{observation}

\noindent For example, suppose $(\kappa(v),\tilde\kappa(\tilde v) = (t_1,\tilde t_3)$.  Then $p_1$ and $\tilde p_1$ lie in the lower-right-hand quadrant $[\kappa(v)\to \kappa(u)]_{\bbS^1} \times [\tilde\kappa(\tilde u)\to \tilde\kappa(\tilde v)]_{\bbS^1}$, so both $P_1$ and $\tilde P_1$ lie on $\partial E \cap \partial B = [v\to u]_{\partial E}$ and on $\partial \tilde E\cap \partial \tilde A = [\tilde u\to \tilde v]_{\partial \tilde E}$.  Also $[\tilde v\to \tilde u]_{\partial E}$ is contained in $E$, because both $\tilde v$ and $\tilde u$ are, and no $p_i$ nor any $\tilde p_i$ lies the two upper quadrants $[\tilde v\to \tilde u]_{\bbS^1} \times \bbS^1$.  Then we get a contradiction via Lemma \ref{lem3}.  A similar argument gives us a contradiction via Lemma \ref{lem3} for $(\kappa(v),\tilde\kappa(\tilde v) = (t_3,\tilde t_1)$.  From now on points $(t_k,\tilde t_{\tilde k})$ which are handled via Lemma \ref{lem3} in this way will be labeled with a diamond, as in Figure \ref{fig:two pts 2}.  This completes the proof of Proposition \ref{prop:nontrivial eye int then done} when $\partial E$ and $\partial \tilde E$ meet at exactly two points.

\begin{starcase}
Suppose that $\partial E$ and $\partial \tilde E$ meet at exactly four points.
\end{starcase}

\noindent Lemma \ref{prop:topo confo compact conv} guarantees that with a correct choice of base point, the torus parametrization for $E$ and $\tilde E$ is as in Figure \ref{fig:eyes param 4}.  As before, we may suppose without loss of generality that $u\not\in \tilde K$, thus $\kappa(u) = s_1$, by Lemma \ref{lem5} and relabeling the $s_i$ if necessary.  Thus we have the possibilities $\tilde \kappa(\tilde u) = \tilde s_1, \tilde s_2, \tilde s_3, \tilde s_4$ to consider.  The cases $(\kappa(u),\tilde\kappa(\tilde u)) = (s_1,\tilde s_2)$ and $(\kappa(u),\tilde\kappa(\tilde u)) = (s_1,\tilde s_4)$ are symmetric by Figure \ref{abc}.  Figures \ref{fig:four pts 1}--\ref{fig:four pts 3} give the solutions for $\tilde \kappa(u) = \tilde s_1, \tilde s_2, \tilde s_3$, modulo some remaining special cases.

\begin{figure}[t]
\centering
\subfloat
{
% Generated with LaTeXDraw 2.0.8
% Sat Mar 17 21:00:25 EDT 2012
% \usepackage[usenames,dvipsnames]{pstricks}
% \usepackage{epsfig}
% \usepackage{pst-grad} % For gradients
% \usepackage{pst-plot} % For axes
\scalebox{1} % Change this value to rescale the drawing.
{
\begin{pspicture}(0,-1.3025)(7.3290625,1.2625)
\psframe[linewidth=0.02,dimen=outer](4.83,0.7625)(2.83,-0.2375)
\psframe[linewidth=0.02,linestyle=dashed](4.33,1.2625)(3.33,-0.7375)
\psdots[dotsize=0.12](2.83,0.2625)
\psdots[dotsize=0.12](3.85,-0.7175)
\usefont{T1}{ptm}{m}{n}
\rput(1.6545313,0.2325){$u=\kappa\inv(s_1)$}
\usefont{T1}{ptm}{m}{n}
\rput(3.8245313,-1.0675){$\tilde u = \tilde\kappa\inv (\tilde s_2)$}
\end{pspicture} 
}%scalebox
}%subfloat
\qquad
\subfloat
{
% Generated with LaTeXDraw 2.0.8
% Sat Mar 17 21:00:54 EDT 2012
% \usepackage[usenames,dvipsnames]{pstricks}
% \usepackage{epsfig}
% \usepackage{pst-grad} % For gradients
% \usepackage{pst-plot} % For axes
\scalebox{1} % Change this value to rescale the drawing.
{
\begin{pspicture}(0,-1.2925)(7.0490627,1.2525)
\psframe[linewidth=0.02,linestyle=dashed](6.21,0.7525)(4.21,-0.2475)
\psframe[linewidth=0.02,dimen=outer](5.71,1.2525)(4.71,-0.7475)
\psdots[dotsize=0.12](4.21,0.2525)
\psdots[dotsize=0.12](5.23,-0.7275)
\usefont{T1}{ptm}{m}{n}
\rput(5.1945314,-1.0575){$u=\kappa\inv(s_1)$}
\usefont{T1}{ptm}{m}{n}
\rput(3.0045312,0.2825){$\tilde u = \tilde\kappa\inv (\tilde s_4)$}
\end{pspicture} 
}%scalebox
}%subfloat
\caption
{
\label{abc}
The topological configurations of $\{E, u, \tilde u\}$ leading to the cases $(\kappa(u),\tilde\kappa(u) = (s_1,\tilde s_2), (s_1,\tilde s_4)$.  We see that these are equivalent via a rotation, because $\eta(\epsilon) = \eta(\epsilon\inv)$.
}
\end{figure}
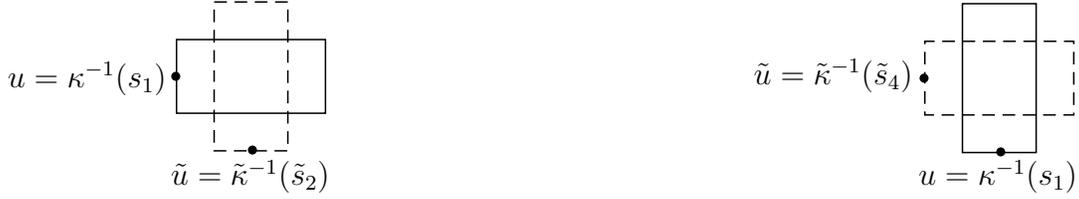

%\afterpage{\clearpage}

Points $(t_k,\tilde t_{\tilde k})$ labeled with an asterisk or a diamond are handled via Lemma \ref{lem5} or \ref{lem3} respectively as before.  Suppose $(\kappa(u),\tilde\kappa(\tilde u)) = (s_1,\tilde s_1)$, and $(\kappa(v),\tilde\kappa(\tilde v)) = (t_1,\tilde t_1)$ in Figure \ref{fig:four pts 1}.  Then the upper-right-hand quadrant defined for us by $(t_1,\tilde t_1)$ contains all four points $p_1,p_2,\tilde p_1,\tilde p_2$, thus the circular arcs $[v\to u]_{\partial E}$ and $[\tilde v\to \tilde u]_{\partial \tilde E}$ meet four times, a contradiction.  All points that are handled in this way are labeled with a small circle.  Finally, if $(\kappa(u),\tilde\kappa(\tilde u)) = (s_1,\tilde s_3)$ and $(\kappa(v), \tilde\kappa(\tilde v)) = (t_3,\tilde t_3)$ in Figure \ref{fig:four pts 3}, we get a contradiction via Lemma \ref{lem4}.

\begin{starcase}
Suppose that $\partial E$ and $\partial \tilde E$ meet at exactly six points.
\end{starcase}

\noindent Then Lemma \ref{lem2} restricts us to two cases to consider.  These are handled in Figure \ref{fig:six pts}.  This completes the proof of Proposition \ref{prop:nontrivial eye int then done}.
\end{proof}

\begin{figure}[t]
\centering
\subfloat[]
{
\label{fig:six pts 1}
% Generated with LaTeXDraw 2.0.8
% Tue Mar 20 15:45:20 EDT 2012
% \usepackage[usenames,dvipsnames]{pstricks}
% \usepackage{epsfig}
% \usepackage{pst-grad} % For gradients
% \usepackage{pst-plot} % For axes
\scalebox{.9} % Change this value to rescale the drawing.
{
\begin{pspicture}(0,-3.08)(6.1590624,3.08)
\psframe[linewidth=0.02,dimen=middle](0.06,3.0)(6.06,-3.0)
\psline[linestyle=dotted,dotsep=0.16cm,linewidth=0.018](0.06,.5)(6.06,.5)
\psline[linestyle=dotted,dotsep=0.16cm,linewidth=0.018](0.06,1.5)(6.06,1.5)
\psline[linestyle=dotted,dotsep=0.16cm,linewidth=0.018](0.06,2.5)(6.06,2.5)
\psline[linestyle=dotted,dotsep=0.16cm,linewidth=0.018](0.06,-.5)(6.06,-.5)
\psline[linestyle=dotted,dotsep=0.16cm,linewidth=0.018](0.06,-1.5)(6.06,-1.5)
\psline[linestyle=dotted,dotsep=0.16cm,linewidth=0.018](0.06,-2.5)(6.06,-2.5)
\psline[linestyle=dotted,dotsep=0.16cm,linewidth=0.018](.56,3.0)(.56,-3.0)
\psline[linestyle=dotted,dotsep=0.16cm,linewidth=0.018](1.56,3.0)(1.56,-3.0)
\psline[linestyle=dotted,dotsep=0.16cm,linewidth=0.018](2.56,3.0)(2.56,-3.0)
\psline[linestyle=dotted,dotsep=0.16cm,linewidth=0.018](3.56,3.0)(3.56,-3.0)
\psline[linestyle=dotted,dotsep=0.16cm,linewidth=0.018](4.56,3.0)(4.56,-3.0)
\psline[linestyle=dotted,dotsep=0.16cm,linewidth=0.018](5.56,3.0)(5.56,-3.0)
\rput(0.06,-3.3){$(\kappa(u),\tilde\kappa(\tilde u))$}
\psdots[dotsize=0.12](0.56,0.5)
\psdots[dotsize=0.12](1.56,1.5)
\psdots[dotsize=0.12](2.56,2.5)
\psdots[dotsize=0.12](3.56,-2.5)
\psdots[dotsize=0.12](4.56,-1.5)
\psdots[dotsize=0.12](5.56,-0.5)
\psdots[dotsize=0.12](2.06,-1.0)
\psdots[dotsize=0.12](0.06,-3.0)
\psbezier[linewidth=0.02](0.06,-3.0)(0.66,-2.3)(1.7580551,-1.131851)(2.06,-1.0)(2.361945,-0.868149)(5.5616255,-0.7157326)(5.66,-0.6)(5.7583747,-0.48426744)(5.96,1.2)(6.06,3.0)
\usefont{T1}{ptm}{m}{n}
\rput(0.76453125,0.21){$p_1$}
\usefont{T1}{ptm}{m}{n}
\rput(1.7645313,1.21){$p_2$}
\usefont{T1}{ptm}{m}{n}
\rput(2.7645311,2.21){$p_3$}
\usefont{T1}{ptm}{m}{n}
\rput(3.2745314,-2.19){$\tilde p_1$}
\usefont{T1}{ptm}{m}{n}
\rput(4.1745315,-1.39){$\tilde p_2$}
\usefont{T1}{ptm}{m}{n}
\rput(5.074531,-0.19){$\tilde p_3$}
\usefont{T1}{ptm}{m}{n}
\rput(1.9845313,-0.73){$(v,\tilde v)$}
\end{pspicture} 
}
}
\qquad
\subfloat[]
{
\label{fig:six pts 2}
% Generated with LaTeXDraw 2.0.8
% Tue Mar 20 15:50:31 EDT 2012
% \usepackage[usenames,dvipsnames]{pstricks}
% \usepackage{epsfig}
% \usepackage{pst-grad} % For gradients
% \usepackage{pst-plot} % For axes
\scalebox{.9} % Change this value to rescale the drawing.
{
\begin{pspicture}(0,-3.08)(6.14,3.08)
\psframe[linewidth=0.02,dimen=middle](0.06,3.0)(6.06,-3.0)
\rput(0.06,-3.3){$(\kappa(u),\tilde\kappa(\tilde u))$}
\psline[linestyle=dotted,dotsep=0.16cm,linewidth=0.018](0.06,.5)(6.06,.5)
\psline[linestyle=dotted,dotsep=0.16cm,linewidth=0.018](0.06,1.5)(6.06,1.5)
\psline[linestyle=dotted,dotsep=0.16cm,linewidth=0.018](0.06,2.5)(6.06,2.5)
\psline[linestyle=dotted,dotsep=0.16cm,linewidth=0.018](0.06,-.5)(6.06,-.5)
\psline[linestyle=dotted,dotsep=0.16cm,linewidth=0.018](0.06,-1.5)(6.06,-1.5)
\psline[linestyle=dotted,dotsep=0.16cm,linewidth=0.018](0.06,-2.5)(6.06,-2.5)
\psline[linestyle=dotted,dotsep=0.16cm,linewidth=0.018](.56,3.0)(.56,-3.0)
\psline[linestyle=dotted,dotsep=0.16cm,linewidth=0.018](1.56,3.0)(1.56,-3.0)
\psline[linestyle=dotted,dotsep=0.16cm,linewidth=0.018](2.56,3.0)(2.56,-3.0)
\psline[linestyle=dotted,dotsep=0.16cm,linewidth=0.018](3.56,3.0)(3.56,-3.0)
\psline[linestyle=dotted,dotsep=0.16cm,linewidth=0.018](4.56,3.0)(4.56,-3.0)
\psline[linestyle=dotted,dotsep=0.16cm,linewidth=0.018](5.56,3.0)(5.56,-3.0)
\psdots[dotsize=0.12](0.56,-1.5)
\psdots[dotsize=0.12](1.56,-0.5)
\psdots[dotsize=0.12](2.56,0.5)
\psdots[dotsize=0.12](3.56,1.5)
\psdots[dotsize=0.12](4.56,2.5)
\psdots[dotsize=0.12](5.56,-2.5)
\psdots[dotsize=0.12](2.06,1.0)
\psdots[dotsize=0.12](0.06,-3.0)
\usefont{T1}{ptm}{m}{n}
\rput(0.8645313,-1.69){$p_1$}
\usefont{T1}{ptm}{m}{n}
\rput(2.7645311,0.21){$p_2$}
\usefont{T1}{ptm}{m}{n}
\rput(4.164531,2.41){$p_3$}
\usefont{T1}{ptm}{m}{n}
\rput(1.7745312,-0.79){$\tilde p_1$}
\usefont{T1}{ptm}{m}{n}
\rput(3.2745314,1.91){$\tilde p_2$}
\usefont{T1}{ptm}{m}{n}
\rput(4.8745313,-2.19){$\tilde p_3$}
\usefont{T1}{ptm}{m}{n}
\rput(1.8845312,1.27){$(v,\tilde v)$}
\psbezier[linewidth=0.02](0.06,-3.0)(0.36,-0.9)(0.48018706,-0.114749625)(0.66,0.2)(0.83981293,0.51474965)(1.36,0.9)(2.06,1.0)(2.76,1.1)(3.855023,1.206694)(4.56,1.5)(5.264977,1.793306)(5.66,2.4)(6.06,3.0)
\end{pspicture} 
}
}
\caption
{
\label{fig:six pts}
Torus parametrizations for the eyes depicted in Figure \ref{fig:possible sixers}.  Both are drawn with base point $(\kappa(u),\tilde\kappa(\tilde u))$.  Each of the two curves is the graph of a faithful indexable homeomorphism $\epsilon:\partial E\to \partial \tilde E$ satisfying $\eta(\epsilon) = 0$.
}
\end{figure}
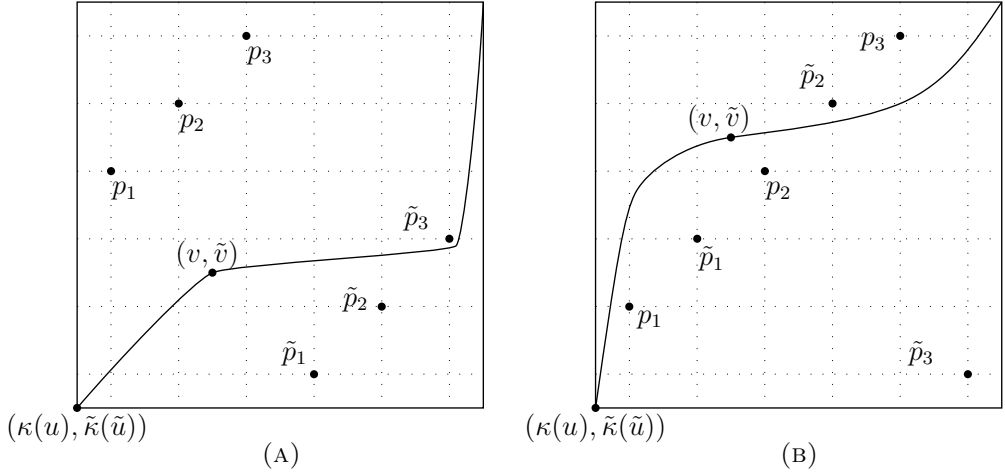

\begin{proposition}
\label{prop:black box 3}
Let $\calD = \{ D_1,\ldots,D_n \}$ and $\tilde\calD = \{\tilde D_1,\ldots,\tilde D_n\}$ be as in the statement of our Main Index Theorem \ref{mainindex}.  That is, they are thin disk configurations in the plane $\bbC$ in general position, realizing the same pair $(G,\Theta)$ where $G = (V,E)$ is a graph and $\Theta:E\to [0,\pi)$.  In addition, suppose that for all $i,j$ the sets $D_i\setminus D_j$ and $\tilde D_i\setminus \tilde D_j$ meet.  Suppose that there is no $i$ so that one of $D_i$ and $\tilde D_i$ contains the other.  Suppose that for every disjoint non-empty $I,J\subset \{1,\ldots,n\}$ so that $I\sqcup J = \{1,\ldots,n\}$, there exists an eye $E_{ij}$ with $i\in I$ and $j\in J$ so that one of $E_{ij}$ and $\tilde E_{ij}$ contains the other.  Then for every $i$ we have that any faithful indexable homeomorphism $\delta_i:\partial D_i \to \partial \tilde D_i$ satisfies $\eta(\delta_i) \ge 1$.  Furthermore there is a $k$ so that $D_i$ and $D_k$ overlap for all $i$, and so that one of $E_{ij}$ and $\tilde E_{ij}$ contains the other if and only if either $i=k$ or $j=k$.
\end{proposition}

\noindent For the proof of Proposition \ref{prop:black box 3}, we need to establish three geometric lemmas:

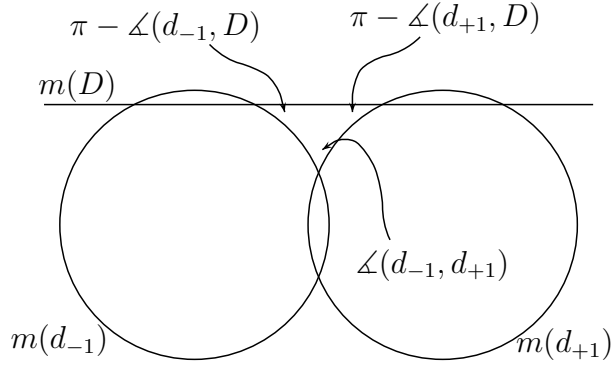
\begin{figure}[t]
\centering
% Generated with LaTeXDraw 2.0.8
% Mon Apr 23 03:11:01 EDT 2012
% \usepackage[usenames,dvipsnames]{pstricks}
% \usepackage{epsfig}
% \usepackage{pst-grad} % For gradients
% \usepackage{pst-plot} % For axes
\scalebox{1} % Change this value to rescale the drawing.
{
\begin{pspicture}(0,-2.3817186)(9.049063,2.3967187)
\psline[linewidth=0.02cm](0.87,1.0432812)(8.17,1.0432812)
\pscircle[linewidth=0.02,dimen=outer](2.87,-0.55671877){1.8}
\pscircle[linewidth=0.02,dimen=outer](6.17,-0.55671877){1.8}
\usefont{T1}{ptm}{m}{n}
\rput(2.4845312,2.0732813){$\pi - \aext(d_{-1},D)$}
\usefont{T1}{ptm}{m}{n}
\rput(6.204531,2.1932812){$\pi- \aext(d_{+1},D)$}
\usefont{T1}{ptm}{m}{n}
\rput(6.034531,-1.0867188){$\aext(d_{-1},d_{+1})$}
\usefont{T1}{ptm}{m}{n}
\rput(1.0745312,-2.0867188){$m(d_{-1})$}
\usefont{T1}{ptm}{m}{n}
\rput(7.7945313,-2.1467187){$m(d_{+1})$}
\usefont{T1}{ptm}{m}{n}
\rput(1.3145312,1.2932812){$m(D)$}
\pscustom[linewidth=0.02]
{
\newpath
\moveto(4.07,0.94328123)
\lineto(4.02,1.1432812)
\curveto(3.995,1.2432812)(3.895,1.3932812)(3.82,1.4432813)
\curveto(3.745,1.4932812)(3.57,1.5682813)(3.47,1.5932813)
\curveto(3.37,1.6182812)(3.22,1.6932813)(3.07,1.8432813)
}
\psline[linestyle=none,linewidth=0.02]{<-}(4.07,0.94328123)(4.02,1.1432812)
\pscustom[linewidth=0.02]
{
\newpath
\moveto(4.97,0.94328123)
\lineto(4.97,1.1432812)
\curveto(4.97,1.2432812)(5.02,1.4182812)(5.07,1.4932812)
\curveto(5.12,1.5682813)(5.295,1.6682812)(5.42,1.6932813)
\curveto(5.545,1.7182813)(5.72,1.7932812)(5.87,1.9432813)
}
\psline[linestyle=none,linewidth=0.02]{<-}(4.97,0.94328123)(4.97,1.1432812)
\pscustom[linewidth=0.02]
{
\newpath
\moveto(4.57,0.44328126)
\lineto(4.72,0.54328126)
\curveto(4.795,0.59328127)(4.945,0.59328127)(5.02,0.54328126)
\curveto(5.095,0.49328125)(5.22,0.31828126)(5.27,0.19328125)
\curveto(5.32,0.06828125)(5.37,-0.18171875)(5.37,-0.30671874)
\curveto(5.37,-0.43171874)(5.395,-0.6067188)(5.47,-0.75671875)
}
\psline[linestyle=none,linewidth=0.02]{<-}(4.57,0.44328126)(4.72,0.54328126)
\end{pspicture} 
}

\caption
{
\label{fig:mam}
The image of the M\"obius transformation described in the proof of Lemma \ref{lem:hat}.
}
\end{figure}%
\begin{lemma}
\label{lem:hat}
Suppose that $D, d_{-1}, d_{+1}$ are closed disks in the plane $\bbC$ in the topological configuration depicted in Figure \ref{pop2}.  Then $\pi + \aext(d_{-1}, d_{+1}) < \aext(d_{-1},D) + \aext(d_{+1},D)$.
\end{lemma}

\begin{proof}
Let $m$ be a M\"obius transformation sending a point on the bottom arc of $\partial D\setminus d_{-1} \cup d_{+1}$ to $\infty$, so that $m(D)$ is the lower half plane.  Then the images of the disks under $m$ are as depicted in Figure \ref{fig:mam}.  We see that $(\pi - \aext(d_{-1},D)) + (\pi - \aext(d_{+1},D)) + \aext(d_{-1},d_{+1}) < \pi$ and the desired inequality follows.
\end{proof}

\begin{lemma}
\label{lem:shoes}
Suppose that $D, d_{-1}, d_{+1}$ are closed disks in the plane $\bbC$ in the topological configuration depicted in Figure \ref{shoes}.  Then $\aext(d_{-1},D) + \aext(d_{+1},D) < \pi + \aext(d_{-1},d_{+1})$.
\end{lemma}

\begin{proof}
This is proved similarly to Lemma \ref{lem:hat}, see Figure \ref{shoes}.
\end{proof}

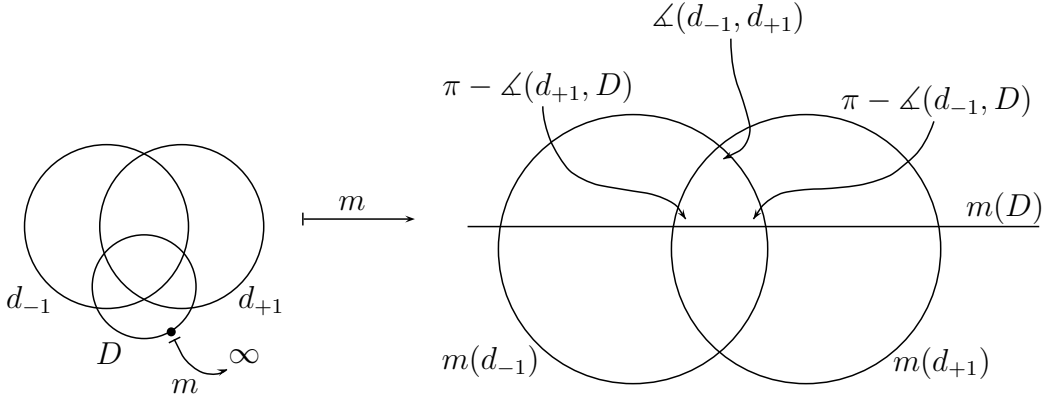
\begin{figure}[t]
\centering
% Generated with LaTeXDraw 2.0.8
% Mon Apr 23 03:26:44 EDT 2012
% \usepackage[usenames,dvipsnames]{pstricks}
% \usepackage{epsfig}
% \usepackage{pst-grad} % For gradients
% \usepackage{pst-plot} % For axes
\scalebox{1} % Change this value to rescale the drawing.
{
\begin{pspicture}(0,-2.653125)(14.889063,2.653125)
\psline[linewidth=0.02cm](6.59,-0.3203125)(14.19,-0.3203125)
\pscircle[linewidth=0.02,dimen=outer](8.79,-0.6203125){1.8}
\pscircle[linewidth=0.02,dimen=outer](11.09,-0.6203125){1.8}
\usefont{T1}{ptm}{m}{n}
\rput(12.824532,1.3496875){$\pi - \aext(d_{-1},D)$}
\usefont{T1}{ptm}{m}{n}
\rput(7.5245314,1.5296875){$\pi- \aext(d_{+1},D)$}
\usefont{T1}{ptm}{m}{n}
\rput(10.054531,2.4496875){$\aext(d_{-1},d_{+1})$}
\usefont{T1}{ptm}{m}{n}
\rput(6.8945312,-2.1103125){$m(d_{-1})$}
\usefont{T1}{ptm}{m}{n}
\rput(12.894531,-2.1303124){$m(d_{+1})$}
\usefont{T1}{ptm}{m}{n}
\rput(13.734531,-0.0703125){$m(D)$}
\pscircle[linewidth=0.02,dimen=outer](1.79,-0.3203125){1.1}
\pscircle[linewidth=0.02,dimen=outer](2.79,-0.3203125){1.1}
\pscircle[linewidth=0.02,dimen=outer](2.29,-1.1203125){0.7}
\usefont{T1}{ptm}{m}{n}
\rput(0.76453125,-1.3103125){$d_{-1}$}
\usefont{T1}{ptm}{m}{n}
\rput(3.8445313,-1.3103125){$d_{+1}$}
\usefont{T1}{ptm}{m}{n}
\rput(1.8245312,-2.0103126){$D$}
\psline[linewidth=0.02cm,tbarsize=0.07055555cm 5.0,arrowsize=0.05291667cm 2.0,arrowlength=1.4,arrowinset=0.4]{|->}(4.39,-0.2203125)(5.89,-0.2203125)
\usefont{T1}{ptm}{m}{n}
\rput(5.054531,-0.0103125){$m$}
\pscustom[linewidth=0.02]
{
\newpath
\moveto(10.39,-0.2203125)
\lineto(10.49,-0.0703125)
\curveto(10.54,0.0046875)(10.69,0.1046875)(10.79,0.1296875)
\curveto(10.89,0.1546875)(11.115,0.1796875)(11.24,0.1796875)
\curveto(11.365,0.1796875)(11.64,0.2046875)(11.79,0.2296875)
\curveto(11.94,0.2546875)(12.165,0.3296875)(12.24,0.3796875)
\curveto(12.315,0.4296875)(12.465,0.5546875)(12.54,0.6296875)
\curveto(12.615,0.7046875)(12.715,0.8546875)(12.79,1.0796875)
}
\psline[linestyle=none,linewidth=0.02]{->}(10.49,-0.0703125)(10.39,-0.2203125)
\pscustom[linewidth=0.02]
{
\newpath
\moveto(9.49,-0.2203125)
\lineto(9.39,-0.0703125)
\curveto(9.34,0.0046875)(9.115,0.1046875)(8.94,0.1296875)
\curveto(8.765,0.1546875)(8.465,0.2046875)(8.34,0.2296875)
\curveto(8.215,0.2546875)(8.015,0.3546875)(7.94,0.4296875)
\curveto(7.865,0.5046875)(7.74,0.7046875)(7.69,0.8296875)
\curveto(7.64,0.9546875)(7.59,1.1296875)(7.59,1.2796875)
}
\psline[linestyle=none,linewidth=0.02]{->}(9.39,-0.0703125)(9.49,-0.2203125)
\pscustom[linewidth=0.02]
{
\newpath
\moveto(9.99,0.5796875)
\lineto(10.19,0.6796875)
\curveto(10.29,0.7296875)(10.365,0.9046875)(10.34,1.0296875)
\curveto(10.315,1.1546875)(10.24,1.4046875)(10.19,1.5296875)
\curveto(10.14,1.6546875)(10.09,1.8796875)(10.09,2.1796875)
}
\psline[linestyle=none,linewidth=0.02]{->}(10.19,0.6796875)(9.99,0.5796875)
\psdots[dotsize=0.12](2.65,-1.7203125)
\pscustom[linewidth=0.02]
{
\newpath
\moveto(2.69,-1.8203125)
\lineto(2.79,-2.0203125)
\curveto(2.84,-2.1203125)(2.99,-2.2453125)(3.09,-2.2703125)
\curveto(3.19,-2.2953124)(3.315,-2.2953124)(3.39,-2.2203126)
}
\psline[linestyle=none,linewidth=0.02]{|-}(2.69,-1.8203125)(2.79,-2.0203125)
\psline[linestyle=none,linewidth=0.02]{->}(3.19,-2.2953124)(3.39,-2.2203126)
\usefont{T1}{ptm}{m}{n}
\rput(3.6245313,-2.0503125){$\infty$}
\usefont{T1}{ptm}{m}{n}
\rput(2.8345313,-2.4503126){$m$}
\end{pspicture} 
}

\caption
{
\label{shoes}
A M\"obius transformation chosen to prove Lemma \ref{lem:shoes}.  Here $\partial m(D) = \bbR$, and $m(D)$ is the lower half-plane.
}
\end{figure}

\begin{figure}[t]
\centering
\subfloat[]
{\label{pop2}
% Generated with LaTeXDraw 2.0.8
% Mon Apr 23 02:22:43 EDT 2012
% \usepackage[usenames,dvipsnames]{pstricks}
% \usepackage{epsfig}
% \usepackage{pst-grad} % For gradients
% \usepackage{pst-plot} % For axes
\scalebox{1} % Change this value to rescale the drawing.
{
\begin{pspicture}(0,-1.1617187)(4.7890625,1.1767187)
\pscircle[linewidth=0.02,dimen=outer](1.67,-0.23671874){0.8}
\pscircle[linewidth=0.02,dimen=outer](2.32,-0.18671875){0.95}
\pscircle[linewidth=0.02,dimen=outer](2.97,-0.23671874){0.8}
\usefont{T1}{ptm}{m}{n}
\rput(0.86453125,-0.9267188){$d_{-1}$}
\usefont{T1}{ptm}{m}{n}
\rput(3.8445313,-0.9267188){$d_{+1}$}
\usefont{T1}{ptm}{m}{n}
\rput(2.3045313,0.97328126){$D$}
\end{pspicture} 
}
}
\qquad
\subfloat[]
{\label{pop1}
% Generated with LaTeXDraw 2.0.8
% Mon Apr 23 02:25:09 EDT 2012
% \usepackage[usenames,dvipsnames]{pstricks}
% \usepackage{epsfig}
% \usepackage{pst-grad} % For gradients
% \usepackage{pst-plot} % For axes
\scalebox{1} % Change this value to rescale the drawing.
{
\begin{pspicture}(0,-1.0325)(4.5690627,0.9925)
\pscircle[linewidth=0.02,dimen=outer](1.54,0.0425){0.95}
\pscircle[linewidth=0.02,dimen=outer](2.14,0.0425){0.55}
\pscircle[linewidth=0.02,dimen=outer](2.74,0.0425){0.95}
\usefont{T1}{ptm}{m}{n}
\rput(0.76453125,-0.7975){$d_{-1}$}
\usefont{T1}{ptm}{m}{n}
\rput(3.6245313,-0.7975){$d_{+1}$}
\usefont{T1}{ptm}{m}{n}
\rput(2.8445313,0.0025){$D$}
\end{pspicture} 
}

}
\caption
{
\label{fig:pop}
The topological configurations for which we prove Lemma \ref{lem:pop}.
}
\end{figure}
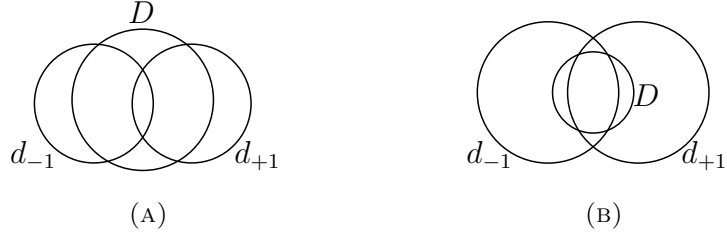

\begin{lemma}
\label{lem:pop}
Suppose that $D, d_{-1}, d_{+1}$ are closed disks in the plane $\bbC$ in one of the two topological configurations depicted in Figure \ref{fig:pop}.  In either case, we get that both $\aext(d_{-1}, D)$ and $\aext(d_{+1},D)$ are strictly greater than $\aext(d_{-1},d_{+1})$.
\end{lemma}

\begin{proof}
Suppose that the disks are in the configuration depicted in Figure \ref{pop2}.  Let $m$ be a M\"obius transformation sending a point on $\partial d_{+1} \setminus D$ to $\infty$.  We may suppose without loss of generality that $m(d_{+1})$ is the lower half-plane.  Then the image of our disks under $m$ is as in Figure \ref{fig:whoa}, where $\theta_1 = \aext(d_{+1}, D)$ and $\theta_2 = \aext(d_{-1},d_{+1})$.  It is then an easy exercise to show that $\theta_2 < \theta_1$ because the two circles $\partial m(d_{-1})$ and $\partial m(D)$ meet in the upper half-plane.  The other inequality follows by symmetry.  The case where the disks are in the configuration depicted in Figure \ref{pop1} follows from the first case after applying a M\"obius transformation sending a point in the interior of $D\cap d_{-1} \cap d_{+1}$ to $\infty$.
\end{proof}

\begin{proof}[Proof of Proposition \ref{prop:black box 3}]
Recalling notation from before, if $D_i$ and $D_j$ overlap then $E_{ij} = D_i\cap D_j$, similarly $\tilde E_{ij}$, and a homeomorphism $\delta_i:\partial D_i\to \partial \tilde D_i$ is called \emph{faithful} if it restricts to homeomorphisms $D_j \cap \partial D_i \to \tilde D_j \cap \partial D_i$ for all $j$.

\begin{claim}
\label{obs:final lem}
Let $i,j$ be so that $\tilde E_{ij} \subset E_{ij}$.  Denote $A = D_i$, $B = D_j$, $\tilde A = \tilde D_i$, $\tilde B = \tilde D_j$.  Then both \ref{fig:possible5c} and \ref{fig:possible6c} occur.  Also one of \ref{fig:possible3d}, \ref{fig:possible3e}, \ref{fig:possible3g} occurs, and one of \ref{fig:possible4d}, \ref{fig:possible4e}, \ref{fig:possible4g} occurs.  Furthermore at least one of \ref{fig:possible3g} and \ref{fig:possible4g} occurs.
\end{claim}

\noindent To see why, note first that both \ref{fig:possible5c} and \ref{fig:possible6c} occur, because these are the only candidates in Figure \ref{fig:possibleii} where $\tilde A\cap \tilde B$ is contained in the respective one of $A$ and $B$.  Note the following by Lemmas \ref{lem:pop}
\begin{equation}
\aext(A,B) = \aext(\tilde A,\tilde B) < \aext(\tilde A,B) \label{c1v}
\end{equation}
and the following by Lemma \ref{lem:hat}.
\begin{equation}
\pi+\aext(\tilde A,\tilde B) < \aext(A,\tilde A) + \aext(\tilde B, A), \qquad \pi+\aext(\tilde A,\tilde B) < \aext(\tilde A,B) + \aext(\tilde B, B) \label{c2v}
\end{equation}

Next, because $\tilde A\cap \tilde B$ contains part of $\partial \tilde A$ and part of $\partial \tilde B$, both of these circles must pass through $A\cap B$.  Noting that \ref{fig:possible3f} cannot occur because $\tilde A\not\subset A$, we conclude that one of \ref{fig:possible3a}, \ref{fig:possible3d}, \ref{fig:possible3e}, \ref{fig:possible3g}, and \ref{fig:possible3h} occurs.  If either of \ref{fig:possible3a} and \ref{fig:possible3h} occurs, then Lemma \ref{lem:pop} implies that $\aext(\tilde A,B) < \aext(A,B)$, contradicting \ref{c1v}.  This leaves us with only the claimed possibilities \ref{fig:possible3d}, \ref{fig:possible3e}, and \ref{fig:possible3g}.  By symmetry we also get that one of \ref{fig:possible4d}, \ref{fig:possible4e}, and \ref{fig:possible4g} occurs.

Finally, note by Lemma \ref{lem:shoes} that if \ref{fig:possible3d} or \ref{fig:possible3e} occurs then we get $\aext(\tilde A, A) + \aext(\tilde A, B) < \pi + \aext(A, B)$, and if \ref{fig:possible3d} or \ref{fig:possible3e} occurs then we get $\aext(\tilde B,A) + \aext(\tilde B,B) < \pi + \aext(A,B)$.  We get that if neither of \ref{fig:possible3g} and \ref{fig:possible4g} occurs, then we may combine these two inequalities with \ref{c2v} to arrive at a contradiction, establishing Claim \ref{obs:final lem}.\medskip

\begin{figure}[t]
\centering
% Generated with LaTeXDraw 2.0.8
% Sun Apr 22 23:23:58 EDT 2012
% \usepackage[usenames,dvipsnames]{pstricks}
% \usepackage{epsfig}
% \usepackage{pst-grad} % For gradients
% \usepackage{pst-plot} % For axes
\scalebox{1} % Change this value to rescale the drawing.
{
\begin{pspicture}(0,2.5130377)(10.199062,4.266719)
\psline[linewidth=0.02cm](0.7,2.9732811)(9.4,2.9732811)
\psarc[linewidth=0.02](4.85,2.4232812){1.65}{3.8140748}{175.60129}
\psarc[linewidth=0.02](4.1,-0.12671874){4.1}{40.914383}{139.08562}
\usefont{T1}{ptm}{m}{n}
\rput(1.2745312,3.2032812){$\theta_1$}
\usefont{T1}{ptm}{m}{n}
\rput(3.1345313,3.2032812){$\theta_2$}
\usefont{T1}{ptm}{m}{n}
\rput(8.314531,3.2032812){$m(d_{+1})$}
\usefont{T1}{ptm}{m}{n}
\rput(6.114531,4.063281){$m(d_{-1})$}
\usefont{T1}{ptm}{m}{n}
\rput(2.7245312,4.063281){$m(D)$}
\end{pspicture} 
}

\caption
{
\label{fig:whoa}
The image of the M\"obius transformation described in the proof of Lemma \ref{lem:pop}.
}
\end{figure}
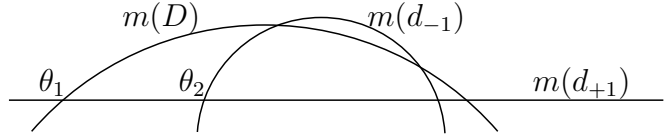

Moving on, pick $1\le i\le n$.  By the hypotheses of Proposition \ref{prop:black box 3} there is a $j$ so that one of $E_{ij}$ and $\tilde E_{ij}$ contains the other, without loss of generality so that $\tilde E_{ij} \subset E_{ij}$.  Let $\delta_i:\partial D_i\to \partial \tilde D_i$ be a faithful indexable homeomorphism.  Continuing with the notation of Claim \ref{obs:final lem}, regardless of which of \ref{fig:possible3d}, \ref{fig:possible3e}, and \ref{fig:possible3g} occurs, there is a point $z\in \partial A \cap \partial E$ so that $z$ lies in the interior of $\tilde A$.  Furthermore note that $\delta_i(z) \in \partial \tilde E$ by the faithfulness condition, and that $\tilde E\subset A$ by our hypotheses, so $\delta_i(z)$ lies in the interior of $A$.  Thus if we draw a torus parametrization for $A$ and $\tilde A$ using $(\kappa(z),\tilde\kappa(\delta_i(z)))$ as the base point, Lemma \ref{prop:computing index from torus} implies that $\eta(\delta_i) \ge 1$, because $\partial A$ and $\partial \tilde A$ meet exactly twice.  This establishes the first part of Proposition \ref{prop:black box 3}.\medskip

Next, let $H_{\mathrm{u}}$ be the undirected simple graph having $\{1,\ldots,n\}$ as its vertex set, so that $\left<i,j\right>$ is an edge in $H_{\mathrm{u}}$ if and only if $D_i$ and $D_j$ overlap and one of $E_{ij}$ and $\tilde E_{ij}$ contains the other.  Note that $H_{\mathrm{u}}$ is connected, otherwise we could pick $I$ to be the vertex set of one connected component of $H_{\mathrm{u}}$ and $J$ to be $\{1,\ldots,n\} \setminus I$ to contradict the hypotheses of Proposition \ref{prop:black box 3}.

Let $H$ be the directed graph obtained from $H_{\mathrm{u}}$ in the following way.  Suppose $\left<i,j\right>$ is an edge in $H_{\mathrm{u}}$.  Denote $A = D_i$, $B = D_j$, $\tilde A = \tilde D_i$, $\tilde B = \tilde D_j$.  Then $\left<i\to j\right>$ is an edge in $H$ if and only if one of \ref{fig:possible3g} and \ref{fig:possible5g} occurs.  In particular Claim \ref{obs:final lem} implies that if $\left<i,j\right>$ is an edge in $H_{\mathrm{u}}$ then at least one of $\left<i\to j\right>$ and $\left<j\to i\right>$ is an edge in $H$, and possibly both are.

\begin{claim}
\label{lem:absolutely last}
Suppose that $\left<i\to j\right>$ is an edge in $H$.  Then $\left<i,j\right>$ is the only edge in $H_{\mathrm{u}}$ having $i$ as a vertex.
\end{claim}

\noindent To see why, observe first that if \ref{fig:possible3d} or \ref{fig:possible3e} occurs then one intersection point $\partial A\cap \partial \tilde A$ lie in the interior of $B$, and if \ref{fig:possible3g} occurs then both do.  Suppose without loss of generality that $\tilde D_i\cap \tilde D_j\subset D_i\cap D_j$.  Then both intersection points $\partial D_i\cap \partial \tilde D_i$ lie in the interior of $D_j$.  For contradiction let $k\ne j$ so that $\left<i,k\right>$ is an edge in $H_{\mathrm{u}}$.  There are two cases.

\begin{keycase}Suppose that $\tilde D_i\cap \tilde D_k \subset D_i\cap D_k$.\end{keycase}

\noindent Then one or both points $\partial D_i\cap \partial \tilde D_i$ lie in the interior of $D_k$.  Then there is a point in the interior of $D_i$ which lies in the interiors of both $D_j$ and $D_k$, a contradiction.

\begin{keycase}Suppose that $D_i\cap D_k\subset \tilde D_i\cap \tilde D_k$.\end{keycase}

\noindent Then by a symmetric restatement of Claim \ref{obs:final lem} we get that both points $\partial D_i\cap \partial D_k$ lie in $\tilde D_i$.  On the other hand $\tilde D_i\cap \partial D_i$ is contained in the interior of $D_j$ by \ref{fig:possible3a}.  Thus there are points interior to all of $D_i,D_j,D_k$, a contradiction.\medskip

\noindent This establishes Caim \ref{lem:absolutely last}.\medskip

Thus $H$ is either the graph on two vertices $\{i,j\}$ having one or both of $\left<i\to j\right>$ and $\left<j\to i\right>$ as edges, or is a graph having $\{k,i_1,\ldots,i_{n-1}\}$ as vertices and exactly the edges $\left<i_\ell \to k\right>$ for $1\le \ell < n$.  The last part of Proposition \ref{prop:black box 3} follows.
\end{proof}
\medskip

This completes the proofs of the main results of this article.

\section{Generalizations, open problems, and conjectures\twostars}
\label{sec conjectures}

We conclude the article with some general conjectures which are directly related to the new results of this article.\medskip

First, we discuss eliminating the thinness condition from the hypotheses of our theorem statements.  Most simply, it seems likely that our Main Theorems \ref{mainrigidity} and \ref{mainuniformization} should continue to hold with the thinness condition completely omitted.  In this direction, we conjecture the following fixed-point index statement for non-thin configurations of disks:

\begin{conjecture}
\label{conj index}
Suppose that $\calD$ and $\tilde\calD$ are disk configurations in $\bbC$ realizing the same incidence data.  Then any faithful indexable homeomorphism $\phi : \partial \calD \to \partial \tilde \calD$ satisfies $\eta(\phi) \ge 0$.
\end{conjecture}

\noindent Note that something stronger than Conjecture \ref{conj index} would be required to prove the corresponding generalizations of our main results on disk configurations using the methods of this article: in particular, we would probably need to generalize the notion of an isolated subsumptive subset of the common index set of $\calD$ and $\tilde\calD$.  However, it is plausible that this is a workable approach.\medskip

We remark at this point that in the present author's thesis, see \cite{mishchenko-thesis}, the definition of \emph{thin} used in the statements of our Main Theorems \ref{mainrigidity} and \ref{mainuniformization} is slightly weaker than the one we have given here: there we call a disk configuration \emph{thin} if given three disks from it, the intersection of their interiors is empty.  The proofs there are essentially the same, without any interesting new ideas, but there are technically annoying degenerate situations to deal with, so we do not work at this level of generality here.\medskip

More strongly, we make two conjectures which together would subsume all other currently known rigidity and uniformization statements on disk configurations.  First:

\begin{conjecture}
\label{velll}
Suppose that $\calC$ and $\tilde\calC$ are disk configurations, locally finite in $\bbG$ and $\tilde \bbG$ respectively, where each of $\bbG$ and $\tilde \bbG$ is equal to one of $\bbC$ and $\bbH^2$, with the \emph{a priori} possibility that $\bbG\ne \tilde\bbG$.  Suppose that $\calC$ and $\tilde \calC$ share a contact graph $G = (V,E)$.  Suppose further that $calC$ and $\tilde\calC$ fill their respective spaces, in sense that every connected component of $\bbG \setminus \cup_{D\in \calC} D$ or of $\tilde\bbG \setminus \cup_{\tilde D\in \tilde\calC} \tilde D$ is bounded.  Then $\bbG = \tilde\bbG$.
\end{conjecture}

\noindent Second:

\begin{conjecture}
\label{conj:ultimate rigidity}
Suppose that $\calC$ and $\tilde\calC$ are disk configurations, both locally finite in $\bbG$, where $\bbG$ is equal to one of $\bbC$ and $\bbH^2$.  Suppose that $\calC$ and $\tilde \calC$ realize the same incidence data $(G,\Theta)$.  Suppose further that some maximal planar subgraph of $G$ is the 1-skeleton of a triangulation of a topological open disk.  Then $\calP$ and $\tilde\calP$ differ by a Euclidean similarity if $\bbG = \bbC$ or by a hyperbolic isometry if $\bbG = \bbH^2$.
\end{conjecture}

\noindent We also make the natural conjecture analogous to Conjecture \ref{conj:ultimate rigidity} for disk configurations on the Riemann sphere.  It seems plausible that a fixed-point index approach could work to prove Conjectures \ref{velll} and \ref{conj:ultimate rigidity}.  An alternative approach to try to prove Conjecture \ref{velll} is via vertex extremal length arguments, along the lines of \cite{MR1680531}*{Uniformization Theorem 1.3} and \cite{MR1331923}.\medskip

Finally, we conjecture that Conjecture \ref{conj:ultimate rigidity} is the best possible uniqueness statement of its type, in the following precise sense:

\begin{conjecture}
\label{conj:strongest rigidity}
Let $\calC$ be a disk configuration which is locally finite in $\bbG$, where $\bbG$ is one of $\hat \bbC$, $\bbC$, or $\bbH^2$.  Let $(G,\Theta)$ be the incidence data of $\calC$.  Suppose that no maximal planar subgraph of $G$ is the 1-skeleton of a triangulation of a topological open disk.  Then there are other locally finite disk configurations in $\bbG$ realizing $(G,\Theta)$ which are not images of $\calC$ under any conformal or anti-conformal automorphism of $\bbG$.
\end{conjecture}

\noindent The most promising tool to prove Conjecture \ref{conj:strongest rigidity} would be a good existence statement taking incidence data $(G,\Theta)$ as input.\medskip

\begin{figure}[t]
\centering
% Generated with LaTeXDraw 2.0.8
% Sat Dec 17 18:09:27 EST 2011
% \usepackage[usenames,dvipsnames]{pstricks}
% \usepackage{epsfig}
% \usepackage{pst-grad} % For gradients
% \usepackage{pst-plot} % For axes
\scalebox{1} % Change this value to rescale the drawing.
{
\begin{pspicture}(0,-2.21)(8.509063,2.16)
\usefont{T1}{ppl}{m}{n}
\rput(2.0645313,-0.01){$D_1$}
\usefont{T1}{ppl}{m}{n}
\rput(6.0245314,0.17){$D_2$}
\usefont{T1}{ppl}{m}{n}
\rput(1.0145313,-0.05){$\tilde D_1$}
\usefont{T1}{ppl}{m}{n}
\rput(7.3945312,0.15){$\tilde D_2$}
\pscircle[linewidth=0.02,dimen=outer](3.18,-0.05){1.51}
\pscircle[linewidth=0.02,dimen=outer](5.02,0.05){1.43}
\pscircle[linewidth=0.02,linestyle=dashed,dimen=outer](5.54,0.05){1.57}
\pscircle[linewidth=0.02,linestyle=dashed,dimen=outer](2.85,-0.02){1.46}
\psdots[dotsize=0.15](3.75,1.32)
\psdots[dotsize=0.15](4.45,1.34)
\psdots[dotsize=0.15](1.75,0.88)
\psdots[dotsize=0.15](6.79,0.98)
\psdots[dotsize=0.15](3.89,-1.36)
\psdots[dotsize=0.15](4.45,-1.24)
\psdots[dotsize=0.15](6.73,-0.94)
\psdots[dotsize=0.15](1.71,-0.9)
\pscustom[linewidth=0.02]
{
\newpath
\moveto(4.37,1.52)
\lineto(4.33,1.65)
\curveto(4.31,1.715)(4.32,1.835)(4.35,1.89)
\curveto(4.38,1.945)(4.48,2.03)(4.55,2.06)
\curveto(4.62,2.09)(4.845,2.13)(5.0,2.14)
\curveto(5.155,2.15)(5.465,2.15)(5.62,2.14)
\curveto(5.775,2.13)(6.095,2.075)(6.26,2.03)
\curveto(6.425,1.985)(6.69,1.88)(6.79,1.82)
\curveto(6.89,1.76)(7.045,1.64)(7.1,1.58)
\curveto(7.155,1.52)(7.21,1.425)(7.21,1.39)
\curveto(7.21,1.355)(7.195,1.3)(7.18,1.28)
\curveto(7.165,1.26)(7.125,1.22)(7.1,1.2)
\curveto(7.075,1.18)(7.03,1.145)(7.01,1.13)
\curveto(6.99,1.115)(6.955,1.09)(6.94,1.08)
\curveto(6.925,1.07)(6.905,1.06)(6.89,1.06)
}
\psline[linewidth=0.02,linestyle=none,arrows=|*-](4.37,1.52)(4.33,1.65)
\psline[linewidth=0.02,linestyle=none,arrows=->](6.94,1.08)(6.89,1.06)
\pscustom[linewidth=0.02]
{
\newpath
\moveto(3.93,-1.48)
\lineto(3.96,-1.68)
\curveto(3.975,-1.78)(3.885,-1.975)(3.78,-2.07)
\curveto(3.675,-2.165)(3.3,-2.2)(3.03,-2.14)
\curveto(2.76,-2.08)(2.285,-1.93)(2.08,-1.84)
\curveto(1.875,-1.75)(1.555,-1.595)(1.44,-1.53)
\curveto(1.325,-1.465)(1.25,-1.325)(1.29,-1.25)
\curveto(1.33,-1.175)(1.405,-1.07)(1.44,-1.04)
\curveto(1.475,-1.01)(1.525,-0.975)(1.57,-0.96)
}
\psline[linewidth=0.02,linestyle=none,arrows=|*-](3.93,-1.48)(3.96,-1.68)
\psline[linewidth=0.02,linestyle=none,arrows=->](1.475,-1.01)(1.57,-0.96)
\pscustom[linewidth=0.02]
{
\newpath
\moveto(4.45,-1.38)
\lineto(4.42,-1.53)
\curveto(4.405,-1.605)(4.57,-1.745)(4.75,-1.81)
\curveto(4.93,-1.875)(5.34,-1.95)(5.57,-1.96)
\curveto(5.8,-1.97)(6.17,-1.95)(6.31,-1.92)
\curveto(6.45,-1.89)(6.65,-1.805)(6.71,-1.75)
\curveto(6.77,-1.695)(6.83,-1.58)(6.83,-1.52)
\curveto(6.83,-1.46)(6.825,-1.36)(6.82,-1.32)
\curveto(6.815,-1.28)(6.805,-1.21)(6.8,-1.18)
\curveto(6.795,-1.15)(6.79,-1.105)(6.79,-1.06)
}
\psline[linewidth=0.02,linestyle=none,arrows=|*-](4.45,-1.38)(4.42,-1.53)
\psline[linewidth=0.02,linestyle=none,arrows=->](6.795,-1.15)(6.79,-1.06)
\pscustom[linewidth=0.02]
{
\newpath
\moveto(3.77,1.46)
\lineto(3.78,1.56)
\curveto(3.785,1.61)(3.695,1.715)(3.6,1.77)
\curveto(3.505,1.825)(3.3,1.895)(3.19,1.91)
\curveto(3.08,1.925)(2.805,1.94)(2.64,1.94)
\curveto(2.475,1.94)(2.145,1.895)(1.98,1.85)
\curveto(1.815,1.805)(1.595,1.71)(1.54,1.66)
\curveto(1.485,1.61)(1.415,1.52)(1.4,1.48)
\curveto(1.385,1.44)(1.385,1.345)(1.4,1.29)
\curveto(1.415,1.235)(1.455,1.145)(1.48,1.11)
\curveto(1.505,1.075)(1.545,1.015)(1.56,0.99)
\curveto(1.575,0.965)(1.605,0.935)(1.65,0.92)
}
\psline[linewidth=0.02,linestyle=none,arrows=|*-](3.77,1.46)(3.78,1.56)
\psline[linewidth=0.02,linestyle=none,arrows=->](1.56,0.99)(1.65,0.92)
\end{pspicture} 
}
\caption
{
\label{fig:ex bad configuration circles}
A counterexample to Theorem \ref{mainindex} if we allow $\aext(D_1,D_2) \ne \aext(\tilde D_1,\tilde D_2)$.  Any indexable $\phi:\partial (D_1\cup D_2) \to \partial (\tilde D_1\cup \tilde D_2)$ making the shown identifications gives $\eta(f) = -1$.
}
\end{figure}
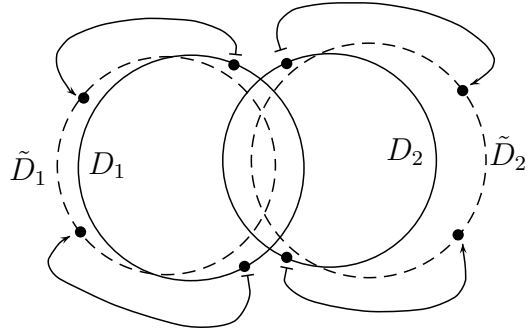

Finally, we consider other directions in which our Main Index Theorem \ref{mainindex} could be generalized.  First, one may hope to weaken the condition that $\calD$ and $\tilde \calD$ realize the same incidence data, insisting only that they share a contact graph.  Figure \ref{fig:ex bad configuration circles} provides an explicit small-scale counterexample.  Alternatively, we may hope to prove a theorem analogous to Theorem \ref{mainindex} for collections of shapes other than metric closed disks.  For example, if $K$ and $\tilde K$ are compact convex sets in $\bbC$ having smooth boundaries, one of which is the image of the other by translation and scaling, then $\partial K$ and $\partial \tilde K$ meet at most twice, so the Circle Index Lemma \ref{cil} applies.  This gives hope for a generalization of Theorem \ref{mainindex} in this direction.  Schramm has proved rigidity theorems for packings by shapes other than circles using related ideas, for example in \cite{MR1076089}.

% bibliography


\begin{bibdiv}
\begin{biblist}

%\bib{MR0259734}{article}{
%   author={Andreev, E. M.},
%   title={Convex polyhedra in Loba\v cevski\u\i\ spaces},
%   language={Russian},
%   journal={Mat. Sb. (N.S.)},
%   volume={81 (123)},
%   date={1970},
%   pages={445--478},
%   review={\MR{0259734 (41 \#4367)}},
%}

\bib{MR0273510}{article}{
   author={Andreev, E. M.},
   title={Convex polyhedra of finite volume in Loba\v cevski\u\i\ space},
   language={Russian},
   journal={Mat. Sb. (N.S.)},
   volume={83 (125)},
   date={1970},
   pages={256--260},
   review={\MR{0273510 (42 \#8388)}},
}

\bib{MR1087197}{article}{
   author={Beardon, Alan F.},
   author={Stephenson, Kenneth},
   title={The uniformization theorem for circle packings},
   journal={Indiana Univ. Math. J.},
   volume={39},
   date={1990},
   number={4},
   pages={1383--1425},
   issn={0022-2518},
   review={\MR{1087197 (92b:52038)}},
   doi={10.1512/iumj.1990.39.39062},
}


\bib{MR2022715}{article}{
   author={Bobenko, Alexander I.},
   author={Springborn, Boris A.},
   title={Variational principles for circle patterns and Koebe's theorem},
   journal={Trans. Amer. Math. Soc.},
   volume={356},
   date={2004},
   number={2},
   pages={659--689},
   issn={0002-9947},
   review={\MR{2022715 (2005b:52054)}},
   doi={10.1090/S0002-9947-03-03239-2},
}

\bib{MR1189006}{article}{
   author={Br{\"a}gger, Walter},
   title={Kreispackungen und Triangulierungen},
   language={German},
   journal={Enseign. Math. (2)},
   volume={38},
   date={1992},
   number={3-4},
   pages={201--217},
   issn={0013-8584},
   review={\MR{1189006 (94b:52032)}},
}

\bib{MR1511735}{article}{
   author={Carath{\'e}odory, C.},
   title={\"Uber die gegenseitige Beziehung der R\"ander bei der konformen
   Abbildung des Inneren einer Jordanschen Kurve auf einen Kreis},
   language={German},
   journal={Math. Ann.},
   volume={73},
   date={1913},
   number={2},
   pages={305--320},
   issn={0025-5831},
   review={\MR{1511735}},
   doi={10.1007/BF01456720},
}


\bib{MR1106755}{article}{
   author={Colin de Verdi{\`e}re, Yves},
   title={Un principe variationnel pour les empilements de cercles},
   language={French},
   journal={Invent. Math.},
   volume={104},
   date={1991},
   number={3},
   pages={655--669},
   issn={0020-9910},
   review={\MR{1106755 (92h:57020)}},
   doi={10.1007/BF01245096},
}

\bib{MR1680531}{article}{
   author={He, Zheng-Xu},
   title={Rigidity of infinite disk patterns},
   journal={Ann. of Math. (2)},
   volume={149},
   date={1999},
   number={1},
   pages={1--33},
   issn={0003-486X},
   review={\MR{1680531 (2000j:30068)}},
   doi={10.2307/121018},
}

\bib{MR1207210}{article}{
   author={He, Zheng-Xu},
   author={Schramm, Oded},
   title={Fixed points, Koebe uniformization and circle packings},
   journal={Ann. of Math. (2)},
   volume={137},
   date={1993},
   number={2},
   pages={369--406},
   issn={0003-486X},
   review={\MR{1207210 (96b:30015)}},
   doi={10.2307/2946541},
}

\bib{MR1331923}{article}{
   author={He, Zheng-Xu},
   author={Schramm, O.},
   title={Hyperbolic and parabolic packings},
   journal={Discrete Comput. Geom.},
   volume={14},
   date={1995},
   number={2},
   pages={123--149},
   issn={0179-5376},
   review={\MR{1331923 (96h:52017)}},
   doi={10.1007/BF02570699},
}

\bib{MR1193599}{article}{
   author={Hodgson, Craig D.},
   author={Rivin, Igor},
   title={A characterization of compact convex polyhedra in hyperbolic
   $3$-space},
   journal={Invent. Math.},
   volume={111},
   date={1993},
   number={1},
   pages={77--111},
   issn={0020-9910},
   review={\MR{1193599 (93j:52015)}},
   doi={10.1007/BF01231281},
}

\bib{koebe-1908}{article}{
	author={Koebe, Paul},
	title={\"Uber die Uniformisierung beliebiger analytischer Kurven, III},
	language={German},
	journal={Nachr. Ges. Wiss. Gott.},
	year={1908},
	pages={337--358},
}

\bib{koebe-1936}{article}{
	author={Koebe, Paul},
	title={Kontaktprobleme der Konformen Abbildung},
	language={German},
	journal={Ber. Verh. S\"achs. Akad. Wiss. Leipzig},
	volume={88},
	date={1936},
	pages={141--164},
}

\bib{MR2258757}{article}{
   author={Kojima, Sadayoshi},
   author={Mizushima, Shigeru},
   author={Tan, Ser Peow},
   title={Circle packings on surfaces with projective structures: a survey},
   conference={
      title={Spaces of Kleinian groups},
   },
   book={
      series={London Math. Soc. Lecture Note Ser.},
      volume={329},
      publisher={Cambridge Univ. Press},
      place={Cambridge},
   },
   date={2006},
   pages={337--353},
   review={\MR{2258757 (2008a:52027)}},
}

\bib{MR2900233}{article}{
   author={Merenkov, Sergei},
   title={Planar relative Schottky sets and quasisymmetric maps},
   journal={Proc. Lond. Math. Soc. (3)},
   volume={104},
   date={2012},
   number={3},
   pages={455--485},
   issn={0024-6115},
   review={\MR{2900233}},
   doi={10.1112/plms/pdr038},
}

\bib{mishchenko-thesis}{thesis}{
	author={Mishchenko, Andrey},
	title={Rigidity of thin disk configurations},
	date={2012},
	type={Ph.D.\ thesis},
	organization={University of Michigan},
	note={Available online at \url{http://hdl.handle.net/2027.42/95930}.},
}

\bib{MR1283870}{article}{
   author={Rivin, Igor},
   title={Euclidean structures on simplicial surfaces and hyperbolic volume},
   journal={Ann. of Math. (2)},
   volume={139},
   date={1994},
   number={3},
   pages={553--580},
   issn={0003-486X},
   review={\MR{1283870 (96h:57010)}},
   doi={10.2307/2118572},
}

\bib{MR1370757}{article}{
   author={Rivin, Igor},
   title={A characterization of ideal polyhedra in hyperbolic $3$-space},
   journal={Ann. of Math. (2)},
   volume={143},
   date={1996},
   number={1},
   pages={51--70},
   issn={0003-486X},
   review={\MR{1370757 (96i:52008)}},
   doi={10.2307/2118652},
}

\bib{MR1985831}{article}{
   author={Rivin, Igor},
   title={Combinatorial optimization in geometry},
   journal={Adv. in Appl. Math.},
   volume={31},
   date={2003},
   number={1},
   pages={242--271},
   issn={0196-8858},
   review={\MR{1985831 (2004i:52005)}},
   doi={10.1016/S0196-8858(03)00093-9},
}

\bib{MR906396}{article}{
   author={Rodin, Burt},
   author={Sullivan, Dennis},
   title={The convergence of circle packings to the Riemann mapping},
   journal={J. Differential Geom.},
   volume={26},
   date={1987},
   number={2},
   pages={349--360},
   issn={0022-040X},
   review={\MR{906396 (90c:30007)}},
}

\bib{MR2336832}{article}{
   author={Roeder, Roland K. W.},
   author={Hubbard, John H.},
   author={Dunbar, William D.},
   title={Andreev's theorem on hyperbolic polyhedra},
   language={English, with English and French summaries},
   journal={Ann. Inst. Fourier (Grenoble)},
   volume={57},
   date={2007},
   number={3},
   pages={825--882},
   issn={0373-0956},
   review={\MR{2336832 (2008e:51011)}},
}

\bib{MR2884870}{article}{
   author={Rohde, Steffen},
   title={Oded Schramm: from circle packing to SLE},
   journal={Ann. Probab.},
   volume={39},
   date={2011},
   number={5},
   pages={1621--1667},
   issn={0091-1798},
   review={\MR{2884870}},
   doi={10.1007/978-1-4419-9675-6\_1},
}

\bib{MR1303402}{article}{
   author={Sachs, Horst},
   title={Coin graphs, polyhedra, and conformal mapping},
   note={Algebraic and topological methods in graph theory (Lake Bled, 1991)},
   journal={Discrete Math.},
   volume={134},
   date={1994},
   number={1-3},
   pages={133--138},
   issn={0012-365X},
   review={\MR{1303402 (95j:52020)}},
   doi={10.1016/0012-365X(93)E0068-F},
}

\bib{MR1076089}{article}{
   author={Schramm, Oded},
   title={Rigidity of infinite (circle) packings},
   journal={J. Amer. Math. Soc.},
   volume={4},
   date={1991},
   number={1},
   pages={127--149},
   issn={0894-0347},
   review={\MR{1076089 (91k:52027)}},
   doi={10.2307/2939257},
}

\bib{MR2131318}{book}{
   author={Stephenson, Kenneth},
   title={Introduction to circle packing: the theory of discrete analytic functions},
   publisher={Cambridge University Press},
   place={Cambridge},
   date={2005},
   pages={xii+356},
   isbn={978-0-521-82356-2},
   isbn={0-521-82356-0},
   review={\MR{2131318 (2006a:52022)}},
}

\bib{MR0051934}{article}{
   author={Strebel, Kurt},
   title={\"Uber das Kreisnormierungsproblem der konformen Abbildung},
   language={German},
   journal={Ann. Acad. Sci. Fennicae. Ser. A. I. Math.-Phys.},
   volume={1951},
   date={1951},
   number={101},
   pages={22},
   review={\MR{0051934 (14,549j)}},
}

\bib{thurston-gt3m-notes}{misc}{
	author={Thurston, William},
	title={The Geometry and Topology of Three-Manifolds},
	organization={Princeton University},
	status={unpublished lecture notes, version 1.1},
	year={1980},
	note={Available online at \url{http://library.msri.org/books/gt3m/}.},
}

\end{biblist}
\end{bibdiv}
\end{document}